\documentclass[11pt,oneside]{book}
\usepackage[T1]{fontenc}
\usepackage[utf8]{inputenc}
\usepackage{newcent}
\usepackage[reqno]{amsmath}
\usepackage{amsfonts, amscd, amsthm, amssymb, mathrsfs, setspace, fancyhdr, times}
\usepackage{lmodern}
\usepackage{enumerate}
\usepackage[shortlabels]{enumitem}
\usepackage[all,2cell]{xy}\UseAllTwocells
\usepackage{mathtools}
\usepackage{thmtools}
\usepackage{parskip}
\usepackage{graphicx}
\usepackage{xcolor}
\usepackage{array}
\usepackage{hhline}
\usepackage{multicol}
\usepackage{booktabs}
\usepackage{mathdots}
\usepackage{bm}
\usepackage{mathabx}
\usepackage{mathdots}
\usepackage{eufrak}
\usepackage{yfonts}
\DeclareMathAlphabet{\EuRoman}{U}{eur}{sb}{n}
\SetMathAlphabet{\EuRoman}{bold}{U}{eur}{sb}{n}
\usepackage{titlesec}
\usepackage{framed}
\usepackage{geometry}
\usepackage{tikz}
\usepackage{pgfplots}
\usepackage{index}
\makeindex
\newindex{not}{ndx}{nnd}{List of Notation} %makeindex -o ECsSpecSeqs.nnd ECsSpecSeqs.ndx
\usepackage[colorlinks=true,
	hyperindex=true,
  pagebackref=true,
  anchorcolor=black,
   breaklinks=true,
  linkcolor=blue,
  citecolor=OliveGreen,
  filecolor=magenta,
  menucolor=red,
  pagecolor=red,
  urlcolor=blue,
  pdflinkmargin=5pt,
  pdfstartview={FitBH -32768},
  pdfpagemode=UseOutlines,
  plainpages=false,
  bookmarks=true,
  bookmarksnumbered=true,
   ]{hyperref}
\hypersetup{pdfauthor=George Peschke }
\hypersetup{pdftitle=Exact Couples and Their Spectral Sequences}
\newcommand{\MSComp}{\nobreak\hspace{-.2ex}}
\newcounter{intro}
\usepackage{titlesec}	% package to enable formatting of parts, chapters, sections, subsections, etc.
\titleformat*{\section}{\large\bfseries\sffamily} 
\titleformat{\chapter}[display]{\thispagestyle{empty}\Large\bfseries\sffamily}{\large Chapter\ \ \thechapter}{0.5ex}{}
\titleformat{\part}[display]{\thispagestyle{empty}\Huge\bfseries\sffamily}{\Large Part\ \ \thepart}{0.5ex}{}

\definecolor{theoremcaption}{rgb}{0,0,0}		%{.98,.98,1}  % lightblue
\definecolor{proofcaption}{rgb}{0,0,0}				%black

\newtheoremstyle{captionstyle}%
{2ex}{}% space above, and below
{\sffamily}{0pt}% body font, indent amount
{\bfseries\sffamily}     % thm head font
{}                  % punctuation after theorem head
{\newline}% space after thm head
{\color{theoremcaption}\bfseries\thmnumber{#2} \thmname{#1}\hfill\thmnote{#3} } % thm head spec

\theoremstyle{captionstyle}  % to typeset the body with captions
\newtheorem{theorem}{\sffamily Theorem}[section]
\newtheorem{definition}[theorem]{Definition}
\newtheorem{proposition}[theorem]{Proposition}
\newtheorem{corollary}[theorem]{Corollary}
\newtheorem{lemma}[theorem]{Lemma}
\newtheorem{example}[theorem]{Example}
\newtheorem{remark}[theorem]{Remark}

\newtheorem{notation}[theorem]{Notation}
\newtheorem{convention}[theorem]{Convention}
\newtheorem{terminology}[theorem]{Terminology}

\makeatletter
\renewenvironment{proof}[1][\proofname]{\vspace{-2ex}\par
  \pushQED{\qed}%
	\sffamily\topsep6\p@\@plus6\p@\relax
  \trivlist
  \item[\hskip\labelsep
        \color{proofcaption}\bfseries\sffamily
    #1\@addpunct{\quad}]\ignorespaces
}{%
  \popQED\endtrivlist\@endpefalse
}
\makeatother

\newcommand{\NoProof}{{\unskip\nobreak\hfil\penalty 50\hskip 2em\hbox{}
 \nobreak\hfil$\lozenge$\parfillskip=0pt\finalhyphendemerits=0\par}}
\definecolor{lightgreen}{rgb}{.98,1,.98}%{.95,1,.95}
\newenvironment{subordinate}{%
\MakeFramed{\hsize=0.95\linewidth\advance\hsize-\width\FrameRestore
\small{ }}
}
{\endMakeFramed}
\definecolor{light-gray}{gray}{0.95}
\newenvironment{colbox}{%
   \MakeFramed{\advance\hsize-\width \FrameRestore}}
 {\endMakeFramed}
\definecolor{deepmagenta}{rgb}{0.8, 0.0, 0.8}
\definecolor{chestnut}{rgb}{0.8, 0.36, 0.36}
\definecolor{blanchedalmond}{rgb}{1.0, 0.92, 0.8}
\definecolor{Maroon}{cmyk}{0,0.87,0.68,0.32}
\definecolor{ForestGreen}{cmyk}{0.91,0,0.88,0.12}
\definecolor{OliveGreen}{cmyk}{0.64,0,0.95,0.40}
\newcommand{\Defn}[1]{{\em #1}}	% identifies #1 as something being defined
\newcommand{\SqBr}[1]{\left[#1\right]}						% square brackets
\newcommand{\DefEq}{\coloneq}   										% := by definition equal
\newcommand{\EqDef}{\eqcolon}   										% =: equal by definition
\newcommand{\XRA}[1]{\xrightarrow{\ #1\ }}

\newcommand{\prdct}{\hskip -.1em\times\hskip -.1em}  				% product in a given category
\newcommand{\Prdct}[2]{#1\hskip -.1em\times\hskip -.1em #2} % #1 product #2
\newcommand{\FamPrdct}[2]{\prod_{#1}#2}							% Product over #1 of objects #2
\newcommand{\PrjctnOnto}[1]{\textit{pr}_{#1}} 			% Projection onto #1
\newcommand{\Set}[1]{\left\{#1\right\}}		% set of #1 object
\newcommand{\union}{\cup}									% union of sets
\newcommand{\FamUnion}[2]{\bigcup_{#1}#2}	% #1 union with #2
\newcommand{\intrsctn}{\cap}							% intersection of sets
\newcommand{\SetIntrsctn}[2]{#1\cap #2}			% intersection of sets
\newcommand{\FamIntrsctn}[2]{\bigcap_{#1}#2}	% intersection of sets in display mode
\newcommand{\from}{\colon}				% Colon in description of a function
\newcommand{\Comp}{\circ}					%{\comp}	% Composition of functions / morphisms
\newcommand{\Img}[1]{\mathit{im}\left(#1\right)}		% image of the map #1
\newcommand{\IdMap}{\mathit{Id}}						% Identity on #1
\newcommand{\IdMapOn}[1]{\mathit{Id}_{#1}}		% Identity on #1
\newcommand{\InclsnOf}[1]{\mathit{inc}_{#1}}	% Inclusion of #1
\newcommand{\NNr}[1][]{\mathbb{N}^{#1}}			% natural numbers [of dimension #1]
\newcommand{\ZNr}[1][]{\mathbb{Z}^{#1}}			% integers [of dimension #1]
\newcommand{\QNr}{\mathbb{Q}}			% rational numbers
\newcommand{\RNrSpc}[1]{\mathbb{R}^{#1}}		% Vector space R^{#1}
\newcommand{\CPrSpc}[1]{\mathbb{C}P^{#1}}
\newcommand{\ZMod}[1]{\mathbb{Z}\hskip -.2em/\hskip -.1em #1}	% Z/#1
\newcommand{\ZPoly}[1]{\mathbb{Z}\hskip -.2em\SqBr{#1}}	% PolynomialRing Z[#1]
\newcommand{\CyclcGrp}[1]{C_{#1}}					% Cyclic group of order #1
\newcommand{\GLGrp}[2]{\mathit{GL}_{#1}(#2)}		% General linear group, (#1,#1)-size over #2
\newcommand{\Ker}[1]{\mathit{ker}\left(#1\right)}			% Kernel of #1
\newcommand{\CoKer}[1]{\mathit{coker}\left(#1\right)}	% Cokernel of #1
\newcommand{\DtrmnntOf}[1]{\mathit{det}(#1)}			% determinant of #1
\newcommand{\DtrmnntOfMtrx}[1]{\mathit{det}\, #1}	% determinant of #1
\newcommand{\Dtrmnnt}{\mathit{det}} 							% determinant operation
\newcommand{\uSphr}[1]{\mathbb{S}^{#1}}		% unit sphere in R^{#1}, #1 optional
\newcommand{\ZDiagImg}[3]{I^{#1}_{#2}#3}        % I^{#1}_{#2}#3
\newcommand{\ZDiagQuo}[3]{Q^{#1}_{#2}#3}		% Q^{#1}_{#2}#3
\newcommand{\ZDiagKer}[3]{K^{#1}_{#2}#3}		% K^{#1}_{#2}#3
\newcommand{\ZDiagCoKer}[3]{R^{#1}_{#2}#3}		% R^{#1}_{#2}#3

\newcommand{\SpecSeqsInModOver}[1]{\mathit{SpecSeq}(#1)}
\newcommand{\SSPage}[1]{E^{#1}}
\newcommand{\SSObjct}[3]{E^{#1}_{#2,#3}}
\newcommand{\SSObjctBB}[2]{E^{#1}_{#2}}

\newcommand{\SSDffrntl}[1]{d^{#1}}
\newcommand{\SSDffrntlBB}[2]{d^{#1}_{#2}}
\newcommand{\SSDffrntlAt}[3]{d^{#1}_{#2,#3}}

\newcommand{\SSBidegreeVect}[2]{\mathbf{#1}_{#2}}
\newcommand{\SSMap}[2]{{#1}^{#2}}
\newcommand{\SSMapBB}[3]{{#1}^{#2}_{#3}}
\newcommand{\SSMapAt}[4]{{#1}^{#2}_{#3,#4}}
\newcommand{\ExctCpl}[1]{\mathcal{#1}}
\newcommand{\ExctCplEPage}[1]{E^{#1}}
\newcommand{\ExctCplEObjct}[3]{E^{#1}_{#2,#3}}
\newcommand{\ExctCplEObjctBB}[2]{E^{#1}_{#2}}

\newcommand{\ExctCplStblEObjct}{\bar{E}}
\newcommand{\ExctCplStblEObjctBB}[1]{\bar{E}_{#1}}
\newcommand{\ExctCplDPage}[1]{D^{#1}}
\newcommand{\ExctCplDObjctBB}[2]{D^{#1}_{#2}}
\newcommand{\ExctCplDObjct}[3]{D^{#1}_{#2,#3}}

\newcommand{\ExctCplIMap}[1]{i^{#1}}
\newcommand{\ExctCplIMapBB}[2]{i^{#1}_{#2}}

\newcommand{\ExctCplIBdg}[1]{\mathbf{#1}}

\newcommand{\ExctCplIMapItrtdBB}[2]{i^{#1}_{#2}}

\newcommand{\ExctCplJMap}[1]{j^{#1}}
\newcommand{\ExctCplJMapBB}[2]{j^{#1}_{#2}}

\newcommand{\ExctCplJBdg}[1]{\mathbf{#1}}
\newcommand{\ExctCplKMap}[1]{k^{#1}}
\newcommand{\ExctCplKMapBB}[2]{k^{#1}_{#2}}

\newcommand{\ExctCplKBdg}[1]{\mathbf{#1}}
\newcommand{\ExctCplCycles}[3]{Z^{#1}_{#2,#3}}
\newcommand{\ExctCplCyclesBB}[2]{Z^{#1}_{#2}}

\newcommand{\ExctCplBndrs}[3]{B^{#1}_{#2,#3}}
\newcommand{\ExctCplBndrsBB}[2]{B^{#1}_{#2}}

\newcommand{\ExctCplCoLimAbut}[1]{L_{#1}}

\newcommand{\ExctCplCoLimAbutMapBB}[1]{\pi_{#1}}

\newcommand{\ExctCplCoLimAbutFltrtn}[2]{F_{#1,#2}}
\newcommand{\ExctCplCoLimAbutFltrtnBB}[1]{F_{#1}}

\newcommand{\ExctCplCoLimAbutFltrtnQtntBB}[1]{\varepsilon_{#1}}

\newcommand{\ExctCplLimAbut}[1]{L^{#1}}

\newcommand{\ExctCplLimAbutMapBB}[1]{\rho^{#1}}

\newcommand{\ExctCplLimAbutFltrtn}[2]{F^{#1,#2}}
\newcommand{\ExctCplLimAbutFltrtnBB}[1]{F^{#1}}

\newcommand{\ExctCplLimAbutFltrtnQtntBB}[1]{\varepsilon^{#1}}

\newcommand{\LModules}[1]{#1\text{-$\cal M$\hskip -.09em\it od\hskip .1em}} % left #1 modules
\newcommand{\ModulesOver}[1]{\EuRoman{M{\kern-0.15ex}od}\!\left(#1\right)}					% Modules over #1
\newcommand{\OrdOmega}{\omega}
\newcommand{\OrdOmegaOp}{\bar{\omega}}

\newcommand{\ZCat}{\mathcal{Z}}
\newcommand{\ExctCplsCatModulesOver}[1]{\mathcal{EC}(#1)}

\newcommand{\ZImgDiagItrtd}[2]{I^{#2}#1}
\newcommand{\ZImgDiagItrtdAt}[3]{I^{#2}_{#3}#1}
\newcommand{\ZStableImg}[1]{\bar{I}#1}
\newcommand{\ZStableImgAt}[2]{\bar{I}_{#2}#1}
\newcommand{\ZPointImg}[2]{J^{#2}#1}

\newcommand{\ZPointKer}[2]{K^{#2}#1}
\newcommand{\ZPointKerAt}[3]{K^{#2}_{#3}#1}
\newcommand{\Vect}[1]{\mathbf{#1}}        % Vector notation for #1
\newcommand{\Sum}[2]{#1\oplus #2}											% #1 + #2
\newcommand{\FamSum}[2]{\bigoplus_{#1} #2}						% sum over #1 of family #2
\newcommand{\GrpHmlgyDimOf}[2]{H_{#1}#2}
\newcommand{\GrpHmlgyDimOfCoeffs}[3]{H_{#1}\left(#2;#3\right)}
\newcommand{\FuncCat}[2]{\mathit{Func}\left(#1,#2\right)}			% category of functors #1 to #2
\newcommand{\CoLim}{\mathit{colim}}					% colimit functor
\newcommand{\CoLimOf}[1]{\mathit{colim}\left( #1\right)}		% colimit of functor #2
\newcommand{\CoLimOver}[1]{\mathit{colim}^{#1}}
\newcommand{\CoLimOfOver}[2]{\mathit{colim}^{#2}\!\left( #1\right)}		% colimit of #1 of functor #2
\newcommand{\Lim}{\mathit{lim}}														% inverse limit functor
\newcommand{\LimOf}[1]{\mathit{lim}\left( #1 \right)}			% inverse limit over #1 of functor #2
\newcommand{\LimOver}[1]{\mathit{lim}_{#1}}								% inverse limit over #1 
\newcommand{\LimOfOver}[2]{\mathit{lim}_{#2}\left( #1 \right)}	% inverse limit in #1 of functor #2
\newcommand{\LimDrvdOfOver}[3]{\mathit{lim}^{#1}_{#3}\ #2}		% #1 derived colimit of functor #2 over cat #3
\newcommand{\LimOne}{\mathit{lim}^{1}}
\newcommand{\LimOneOf}[1]{\mathit{lim}^{1}\left(#1\right)}
\newcommand{\LimOneOver}[1]{\mathit{lim}^{1}_{#1}}
\newcommand{\LimOneOfOver}[2]{\mathit{lim}^{1}_{#2}\left(#1\right)} % limit^1 over #1 of functor #2
\newdir{>>}{{}*!/3.5pt/:(1,-.2)@^{>}*!/3.5pt/:(1,+.2)@_{>}*!/7pt/:(1,-.2)@^{>}*!/7pt/:(1,+.2)@_{>}}
\newdir{ >>}{{}*!/8pt/@{|}*!/3.5pt/:(1,-.2)@^{>}*!/3.5pt/:(1,+.2)@_{>}}
\newdir{|>}{!/4.5pt/@{|}*:(1,-.2)@^{>}*:(1,+.2)@_{>}}
\newdir{ |>}{{}*!/-3.5pt/@{|}*!/-8pt/:(1,-.2)@^{>}*!/-8pt/:(1,+.2)@_{>}}
\newdir{ >}{{}*!/-8pt/@{>}}
\newdir{>}{{}*:(1,-.2)@^{>}*:(1,+.2)@_{>}}
\newdir{<}{{}*:(1,+.2)@^{<}*:(1,-.2)@_{<}}
\newcommand{\PullLU}[1]{\ar@{}[#1]|-{%
\begin{picture}(10,10)%
\put(1,1){\line(1,0){8}}%
\put(9,1){\line(0,1){8}}%
\put(1,1){\circle*{2}}%
\put(9,1){\circle*{2}}%
\put(9,9){\circle*{2}}%
\put(1,9){\circle{2}}
\end{picture} } }
\newcommand{\PullRU}[1]{\ar@{}[#1]|-{%
\begin{picture}(10,10)%
\put(1,1){\line(1,0){8}}%
\put(1,1){\line(0,1){8}}%
\put(1,1){\circle*{2}}%
\put(9,1){\circle*{2}}%
\put(9,9){\circle{2}}%
\put(1,9){\circle*{2}}
\end{picture} } }
\newcommand{\PullLD}[1]{\ar@{}[#1]|-{%
\begin{picture}(10,10)%
\put(1,9){\line(1,0){8}}%
\put(9,1){\line(0,1){8}}%
\put(1,1){\circle{2}}%
\put(9,1){\circle*{2}}%
\put(9,9){\circle*{2}}%
\put(1,9){\circle*{2}}
\end{picture} } }
\newcommand{\PullRD}[1]{\ar@{}[#1]|-{%
\begin{picture}(10,10)%
\put(1,9){\line(1,0){8}}%
\put(1,1){\line(0,1){8}}%
\put(1,1){\circle*{2}}%
\put(9,1){\circle{2}}%
\put(9,9){\circle*{2}}%
\put(1,9){\circle*{2}}
\end{picture} } }
\newcommand{\PushLU}[1]{\ar@{}[#1]|-{%
\begin{picture}(10,10)%
\put(1,1){\line(1,0){8}}%
\put(9,1){\line(0,1){8}}%
\put(1,1){\circle*{2}}%
\put(9,1){\circle*{2}}%
\put(9,9){\circle*{2}}%
\put(1,9){\circle{2}}
\end{picture} } }
\newcommand{\PushRU}[1]{\ar@{}[#1]|-{%
\begin{picture}(10,10)%
\put(1,1){\line(1,0){8}}%
\put(1,1){\line(0,1){8}}%
\put(1,1){\circle*{2}}%
\put(9,1){\circle*{2}}%
\put(9,9){\circle{2}}%
\put(1,9){\circle*{2}}
\end{picture} } }
\newcommand{\PushLD}[1]{\ar@{}[#1]|-{%
\begin{picture}(10,10)%
\put(1,9){\line(1,0){8}}%
\put(9,1){\line(0,1){8}}%
\put(1,1){\circle{2}}%
\put(9,1){\circle*{2}}%
\put(9,9){\circle*{2}}%
\put(1,9){\circle*{2}}
\end{picture} } }
\newcommand{\PushRD}[1]{\ar@{}[#1]|-{%
\begin{picture}(10,10)%
\put(1,9){\line(1,0){8}}%
\put(1,1){\line(0,1){8}}%
\put(1,1){\circle*{2}}%
\put(9,1){\circle{2}}%
\put(9,9){\circle*{2}}%
\put(1,9){\circle*{2}}
\end{picture} } }
\newcommand{\BiCart}[1]{\ar@{}[#1]|-{%
\begin{picture}(10,10)%
\put(1,1){\line(1,0){8}}%
\put(1,9){\line(1,0){8}}%
\put(1,1){\line(0,1){8}}%
\put(9,1){\line(0,1){8}}%
\put(1,1){\circle*{2}}%
\put(9,1){\circle*{2}}%
\put(9,9){\circle*{2}}%
\put(1,9){\circle*{2}}
\end{picture} } }
\begin{document}
%
%%%%%%%%%%%%%%%%%%  Begin title page
\title{\Huge $\EuRoman{Exact\ \ Couples\ \ and}$ \\ \vskip 1ex $\EuRoman{Their\ \  Spectral\ \  Sequences}$}
\vskip 2ex

\author{\Large $\EuRoman{George\ \ Peschke}$}

%\date{\sffamily{\tiny\today}}
\date{\sffamily{\tiny May 21, 2022}}
\maketitle
%%%%%%%%%%%%%%%%%%  End title page
\sffamily{% applies to the entire document

\begin{center}
\bfseries{Abstract}
\end{center}

\small
Given a bigraded exact couple of modules over some ring, we determine the meaning of the $E^{\infty}$-terms of its associated spectral sequence: Let $L^{\ast}$ and $L_{\ast}$ denote the limit and colimit abutting objects of the exact couple, filtered by the kernel and image objects to the associated cone and cocone diagrams. Then the unstable E-infinite extension theorem states how adjacent filtration quotients of the colimit filtration are extended by $E^{\infty}$ objects over corresponding adjacent filtration quotients of the kernel filtration. Depending on the exact couple, this extension may involve a $\LimOne$-corrective term.
 
The stable E-infinity extension theorem is based on the fact that the derivation process of the exact couple admits a transfinite recursion which is beyond the scope of the traditional spectral sequence perspective. The transfinite recursion always stabilizes at some ordinal. The resulting stable $E$-objects are (a) always subobjects of $E^{\infty}$ and (b) extend adjacent filtration quotients of the colimit filtration over corresponding adjacent filtration quotients of the kernel filtration {\em without} the need for $\LimOne$ corrective terms.

From the perspective of these E-infinity extensions, classical convergence results focus on sufficient conditions under which one of the end terms of these extensions vanishes: So $E^{\infty}$ objects are isomorphic to adjacent filtration quotients of exactly one $L_{\ast}$, or $L^{\ast}$ - in the latter case, with potential $\LimOne$-complication. If $\LimOne$-terms are needed for a particular spectral sequence, then this is because its E-infinity objects are unstable.

The E-infinity extension theorems enable conclusions about the filtered limit/colimit abutments even in cases where the spectral sequence is far from converging in any traditional sense. We develop such results in the context of 'comparing' the spectral sequences via the morphism that is induced by a morphism of underlying exact couples.

We also contribute to 'reverse comparison' in a spectral sequence; that is using information about the universal abutment(s) of the underlying exact couple to extract information about one or more pages of the spectral sequence. These results overlap with Zeeman's comparison theorems \cite{ECZeeman1957} in a generalizing fashion.

\vfill
\begin{center}
{\bfseries E-infinity extensions}
\end{center}

\begin{equation*}
\resizebox{0.35\textwidth}{!}{$
\xymatrix@R=.5ex@C=3em{
\ExctCplCoLimAbutFltrtnBB{\Vect{x}-\ExctCplIBdg{a}} \ar@{{ |>}->}[r] &
	\ExctCplCoLimAbutFltrtnBB{\Vect{x}} \ar@{-{ >>}}[r] &
  \ExctCplCoLimAbutFltrtnQtntBB{\Vect{x}} \\
}
$}
\end{equation*}

\begin{equation*}
\resizebox{0.7\textwidth}{!}{$
\xymatrix@R=5ex@C=5em{
\ExctCplCoLimAbutFltrtnQtntBB{\Vect{x}} \ar@{{ |>}->}[r] \ar@{=}[d] &
	\ExctCplStblEObjctBB{\Vect{x}+\ExctCplJBdg{b}} \ar@{-{ >>}}[r] \ar@{{ |>}->}[d] \PullLU{rd}&
	\ExctCplLimAbutFltrtnQtntBB{\Vect{x}+\ExctCplJBdg{b}+\ExctCplKBdg{c}} \ar@{{ |>}->}[d]^{M} \\
\ExctCplCoLimAbutFltrtnQtntBB{\Vect{x}} \ar@{{ |>}->}[r] &
	\ExctCplEObjctBB{\infty}{\Vect{x}+\ExctCplJBdg{b}} \ar@{-{ >>}}[r] &
	\frac{ \ExctCplCyclesBB{\infty}{\Vect{x}+\ExctCplJBdg{b} } }{ \Img{ \ExctCplJMapBB{}{\Vect{x}} } }
}
$}
\end{equation*}

\begin{equation*}
\resizebox{0.5\textwidth}{!}{$
\xymatrix@R=.5ex@C=3em{
\ExctCplLimAbutFltrtnBB{\Vect{x}+\ExctCplJBdg{b}+\ExctCplKBdg{c}} \ar@{{ |>}->}[r] &
	\ExctCplLimAbutFltrtnBB{\Vect{x}+\ExctCplIBdg{a}+\ExctCplJBdg{b}+\ExctCplKBdg{c}} \ar@{-{ >>}}[r] &
	\ExctCplLimAbutFltrtnQtntBB{\Vect{x}+\ExctCplJBdg{b}+\ExctCplKBdg{c}}
}
$}
\end{equation*}
\begin{equation*}
\xymatrix@R=5ex@C=2.5em{
\ExctCplLimAbutFltrtnBB{\Vect{x}+\ExctCplJBdg{b}+\ExctCplKBdg{c}} \ar@{{ |>}->}[r] &
	\ExctCplLimAbutFltrtnBB{\Vect{x}+\ExctCplIBdg{a}+\ExctCplJBdg{b}+\ExctCplKBdg{c}} \ar[r]^-{t} &
	\frac{ \ExctCplCyclesBB{\infty}{\Vect{x}+\ExctCplJBdg{b} } }{ \Img{ \ExctCplJMapBB{}{\Vect{x}} } } \ar[r] &
	\LimOneOfOver{ \Ker{ \ExctCplIMapItrtdBB{r}{\Vect{x}+\ExctCplJBdg{b}+\ExctCplKBdg{c}-r\ExctCplIBdg{a} } } }{\ZCat} \ar[r] &
	\LimOneOfOver{ \Ker{ \ExctCplIMapItrtdBB{r+1}{\Vect{x}+\ExctCplJBdg{b}+\ExctCplKBdg{c} -r\ExctCplIBdg{a} } } }{\ZCat} 
}
\end{equation*}
\vfill\vfill

\thispagestyle{empty}
%%%%%%%%%%%%%%%%%%
\fancyhead[LO]{\sffamily\bfseries Introduction}
\fancyhead[LE]{\sffamily\bfseries Introduction}
\fancyhead[RO]{}
\fancyhead[RE]{}
\fancyfoot[C]{\sffamily\bfseries\footnotesize $\blacktriangleleft$\qquad \thepage\qquad $\blacktriangleright$}
%%%%%%%%%%%%%%%%%%  Begin TOC
\begin{footnotesize}
\tableofcontents
\end{footnotesize}
\thispagestyle{empty}
%%%%%%%%%%%%%%%%%%  End TOC
%
%
\pagenumbering{Roman}
\setcounter{page}{0}
\newpage
%%%%%%%%%%%  set page styling parameters
\pagestyle{fancy}   %This sets the page layout  (fancy heading)

%%%%%%%%%%%%%%%%%%%%%%%%%%%%%%%%%%%%%%%%%%%%%%%%%%%%%%%%%%%%%%%
%%%%%%%%%%%%%%%%%%%%%%%%%%%%%%%%%%%%%%%%%%%%%%%%%%%%%%%%%%%%%%%
\newpage
%%%%%%%%%%%%%%%%%%%%%%%%%%%%%%%%%%%%%%%%%%%%%%
\newpage
\section*{Introduction}
\label{sec:Introduction-Global}

Spectral sequences were invented by Jean Leray in the 1940's as a tool to compute sheaf cohomology; see \cite{JLeray2_1946} and, for a historic perspective of this development, \cite{HMiller2000}. Subsequently, spectral sequences have become a standard tool in homological/homotopical algebra and related subjects. By now, a rich supply of works on spectral sequences are available for introduction and for reference; to mention just a few: \cite{JMBoardman1999} \cite{HPCartanSEilenberg1956} \cite{PJHiltonUStammbach1971} \cite{SMacLane1975-Homology} \cite{JMcCleary1985.1} \cite{JJRotman2009} \cite{CAWeibel1994}. An account of the early history of the subject can be found in \cite{JMcCleary1999}.

The emphasis of these works is on applications of spectral sequences. In contrast, here we explore in greater depth the relationship between bigraded exact couples and their associated spectral sequences. To explain what we found, let's just work with Abelian groups for now.

The purpose of a spectral sequence is to help describe an abelian group $X$ via a filtration. So, consider an ascending sequence of subgroups indexed by the integers:
\begin{equation*}
\tag{\text{\sffamily F}}\label{eq:Filtration}%
\cdots \subseteq F_{p-1} \subseteq F_p \subseteq F_{p+1}\subseteq \cdots \subseteq X
\end{equation*}
Each pair of adjacent filtration stages determines a short exact sequence
\begin{equation*}
\xymatrix@R=5ex@C=4em{
F_{p-1} \ar@{{ |>}->}[r] &
	F_{p} \ar@{-{ >>}}[r] &
	\varepsilon_{p}
}
\end{equation*}
Whenever a spectral sequence is `matched' to such a filtration, its purpose is to provide information about the adjacent filtration quotients $\varepsilon_{p}$. Consequence: whenever we know something about $F_{p-1}$, such information helps inferring properties of $F_{p}$. Thus a spectral sequence matched to the given filtration of $X$ enables us to work toward information about $X$ inductively.

The process we just sketched out is the nucleus of working with spectral sequences. Unfortunately, in practice one might have to deal with significant complicating issues such as (a) the quality of the filtration of $X$, and (b) the quality of the match between spectral sequence and filtration. Here is a snapshot of what is involved.

{\bfseries Quality of the filtration}\quad Let's begin with a really nice filtration of $X$:
\begin{equation*}
0=F_0 \subseteq F_1\subseteq \cdots \subseteq F_{p-1}\subseteq F_{p}\subseteq \cdots \subseteq F_n=X
\end{equation*}
Here, the $0$-group forms a canonical entry point for an inductive analysis of $X$, and one reaches $X$ in finitely many extension steps. In general, a given filtration of $X$ may not present a canonical entry point. The intersection of the filtration stages $F_p$ need not be $0$, and the union of the filtration stages need not be $X$. - If a spectral sequence is matched to the given filtration of $X$, then it only provides information about the terms $\varepsilon_{\ast}$. So, a discussion pertaining to the quality of the filtration is beyond the scope of the spectral sequence, hence requires separate attention.

{\bfseries Quality of the match between spectral sequence and filtration}\quad Every spectral sequence delivers objects denoted $E^{\infty}_{\ast}$. A really nice situation is when $\varepsilon_{p}\cong E^{\infty}_{p}$. This is the case for many spectral sequences, especially the onces which have been used to great effect in the past. However, in general, the relationship between the E-infinity terms computed from the spectral sequence and adjacent filtration quotients can be more complicated.

{\bfseries Sources of spectral sequences / matched filtered objects}\quad How does one obtain filtered objects with a spectral sequence matched to them? - One standard source of such sets of interlinked data are exact couples, introduced and discussed by W.S. Massey in \cite{WSMassey1952},\cite{WSMassey1953},\cite{WSMassey1954}, see Chapter \ref{chap:ExactCouples} for details.  Attractive is that standard constructions in algebraic topology, homological algebra, homotopical algebra, etc provide a rich supply of exact couples.

\newpage
Most common are $(\Prdct{\ZNr}{\ZNr})$-bigraded exact couples; see Chapter \ref{chap:ExactCouples}. Such an exact couple determines at once:
\begin{enumerate}[(I)]
\item Two $\ZNr$-graded objects, each canonically filtered. These are (a) the colimit abutment $L_{\ast}$, and (b) the limit abutment $L^{\ast}$. Thus, each object $L_n$ respectively $L^n$ is filtered as in diagram \eqref{eq:Filtration}, and each of these filtrations is subject to potential filtration quality complications as outlined above. Furthermore: %
\index{colimit abutment}\index{limit abutment}
\item An exact couple determines a spectral sequence which, in turn, yields objects $\ExctCplEObjct{\infty}{p}{q}$. These objects are matched to these filtrations, however with quality complications, as outlined above.
\end{enumerate}
Well documented in the literature are situations where one of the two universal abutments vanishes, and the other coincides with the (co-)homology/homotopy groups/etc associated to underlying objects of interest. In this situation the nonzero universal abutment is known as  \Defn{the abutment} of the spectral sequence. %
\index{abutment}%

Here, we ask in general: Given a $(\Prdct{\ZNr}{\ZNr})$-bigraded exact couple, what is the meaning of the E-infinity objects of its spectral sequence. This question is answered Chapter \ref{chap:ExactCouples}. As immediate applications we offer an assortment of comparison theorems.

{\bfseries Organization}\quad In Chapter \ref{chap:SpectralSequences} we discuss spectral sequences as an algebraic structure in its own right. It is quite remarkable that the material of that section alone enables the use of theorems which read like this: 'Under <conditions> there is a spectral sequence which converges to <objects> via a filtration whose adjacent filtration quotients are isomorphic to <E-infinity objects of the spectral sequence>'. This said, we echo a recommendation commonly found in the literature: Appreciation of spectral sequences benefits from using them. Accordingly, we offer some recommendations in this direction; see Section \ref{sec:SS-Computations}.

On the other hand, setting up a spectral sequence and analyzing its properties is an entirely different story. Here, we rely on background from Chapter \ref{chap:ExactCouples}. We recall the concept of exact couple, and introduce its limit/colimit abutments, their canonical filtrations, and a spectral sequence matched to these filtrations. We discuss the quality of these filtrations, and we explain completely how the E-infinity objects of the spectral sequence are matched to the appropriate adjacent filtration quotients of these filtrations. The main results here are the stable extension theorem (\ref{thm:Stable-E-Extension}) and the E-infinity extension theorem (\ref{thm:E-InfinityExtensionThm}).

Then we turn to applications of this development by establishing two types of comparison mechanisms\footnote{\sffamily The use of the term `comparison' in this context goes back at least as far as Zeeman in \cite{ECZeeman1957}}. In both cases, we start from a morphism of exact couples.
\begin{enumerate}
\item In one type of comparison, we analyze how properties of the induced map of spectral sequences affect properties of the induced maps of universal abutments; see Section \ref{sec:SpecSeq-Comparison-I}.
\item An opposite type of comparison is named after Zeeman. In one variant, we assume that the limit abutments vanish, and that the induced map of colimit abutments is an isomorphism. If now certain initial conditions in the spectral sequence are met, then we conclude that the induced map of spectral sequences is an isomorphism; see Section \ref{sec:SpecSec-Comparison-II}.
\end{enumerate}
The discussion in \ref{chap:ExactCouples} relies heavily on understanding the sometimes subtle properties of certain constructions on $\ZCat$-shaped diagrams
\begin{equation*}
\cdots \longrightarrow X_{-2} \longrightarrow X_{-1} \longrightarrow X_0\longrightarrow X_1 \longrightarrow X_2 \longrightarrow \cdots
\end{equation*}
We collect suitable background in an appendix; see Chapter \ref{chap:CoLimsInModuleCategories}. This material is based on Boardman's \cite{JMBoardman1999}.

An earlier rendition of some of the results presented here was used by S. Rahmati to show how spectral sequence methods can be applied to  work with transfinitely filtered objects; see \cite{SRahmati2013}.

{\bfseries Contact}\footnote{\sffamily George Peschke, University of Alberta, Edmonton, Canada T6G 2G1,  george.peschke@ualberta.ca}
\newpage
%%%%%%%%%%%%%%%%%%%%%%%%%%%%%%%%%%%%%%%%%%%%%%
\renewcommand{\sectionmark}[1]{\markright{ #1}}
\renewcommand{\chaptermark}[1]{\markboth{ #1}{}}
\fancyhead{}
\fancyfoot{}
\fancyhead[LO]{\sffamily\bfseries\footnotesize \thechapter\ \leftmark}
\fancyhead[LE]{\sffamily\bfseries\footnotesize \thechapter\ \leftmark}
\fancyhead[C]{}
\fancyhead[RO]{\sffamily\bfseries\footnotesize \thesection\ \rightmark}
\fancyhead[RE]{\sffamily\bfseries\footnotesize \thesection\ \rightmark}
\fancyfoot[C]{\sffamily\bfseries\footnotesize $\blacktriangleleft$\qquad \thepage\qquad $\blacktriangleright$}
%%%%%%%%%%%%%%%%%%%%%%%%%%%%%%%%%%%%%%%%%%%%%%
\setcounter{page}{0}
\newpage
\part[Exact Couples and Their Spectral Sequences]{Exact Couples \vspace{2ex} \\  \hphantom{\ } and  \vspace{2ex} \\ Their Spectral Sequences}
\label{part:ExactCouples-Spectral Sequences}
\pagenumbering{arabic}
\newpage
%
%%%%%%%%%%%%%%%%%%%%%%%%%%%%%%%%%%%%%%%%%
\chapter{Spectral Sequences}
\label{chap:SpectralSequences}

Given a ring $R$, a $(\ZNr\prdct \ZNr)$-bigraded $R$-module is a family of $R$-modules $A=(A_{p,q} | (p,q)\in \ZNr\prdct \ZNr)$. A morphism $f\from A\to B$ of $(\ZNr\prdct \ZNr)$-bigraded $R$-modules of bidegree $(\mu,\nu)\in (\ZNr\prdct \ZNr)$ is given by a family of $R$-module maps %
\index[not]{$\ZNr$ - integers}%
\begin{equation*}
f_{p,q}\from A_{p,q} \longrightarrow B_{p+\mu,q+\nu},\qquad (p,q)\in \ZNr\prdct \ZNr.
\end{equation*}
With $\Vect{x}=(p,q)$ and $\Vect{u}= (\mu,\nu)\in (\ZNr\prdct \ZNr)$, we compress the notation for a morphism of bigraded $R$-modules to
\begin{equation*}
f_{\Vect{x}}\from A_{\Vect{x}} \longrightarrow B_{\Vect{x}+\Vect{u}}.
\end{equation*}
A spectral sequence (\ref{def:SpectralSequence}) in the category of $R$-modules is an algebraic structure given by a family $(E^r,d^r)$ of $(\ZNr\prdct \ZNr)$-bigraded $R$-modules, and endomorphisms $d^r\from E^r\to E^r$, which act as differentials; i.e. $d^rd^r=0$. The members of the family $(E^r)$ are linked by the requirement that $E^{r+1}$ is the homology of $E^r$ with respect to $d^r$.

We show that a spectral sequence determines a bigraded $R$-module $\SSPage{\infty}$; see (\ref{def:SpectralSequence-LimitPage}). To explain the role of these $\SSPage{\infty}$-terms, let us consider the following situation encountered in many classical  applications. We are dealing with $\ZNr$\MSComp-graded $R$-module $L_n$,  which is filtered by submodules:
\begin{equation*}
\cdots\subseteq F_{\Vect{x}(n)-\Vect{a}}(L_n) \subseteq F_{\Vect{x}(n)}(L_n)\subseteq F_{\Vect{x}(n)+\Vect{a}}(L_n) \subseteq \cdots \subseteq L_n=\FamUnion{r\in \ZNr}{F_{\Vect{x}(n)+r\Vect{a}}}
\end{equation*}
Up to now, these two structures (SS) spectral sequence, and (FO) filtered objects were unrelated. However, they are linked:
each adjacent filtration stage of a filtered object fits into a short exact sequence like the following:
\begin{equation*}
\xymatrix@R=5ex@C=2em{
0 \ar[r] &
	F_{\Vect{x}(n)+(r-1)\Vect{a}}(L_n) \ar@{{ |>}->}[rr] &&
	F_{\Vect{x}(n)+r\Vect{a}}(L_n) \ar@{-{ >>}}[rr] &&
	\SSObjctBB{\infty}{\Vect{x}(n) + r\Vect{a}} \ar[r] &
	0
}
\end{equation*}
Remarkable is that, in many applications of spectral sequences, this structural relationship between (SS) and (FO) alone is sufficient to extract valuable information. In contrast, the hard work needed to set up the spectral sequence, to analyze its properties,  and to establish its link to the filtered objects is a one-time effort which, once completed, tends to dwell in the background.

As a consequence, based on the information provided here, the reader is ready for computations with any of the classical spectral sequences; e.g. the Serre spectral sequence, the spectral sequence of a filtered chain complex, the spectral sequences of a double chain complex, etc.

\section{A Spectral Sequence and its Limit Page}
\label{sec:SpecSeq-E-infinity}

In this section we introduce the algebraic structure of  $(\Prdct{\ZNr}{\ZNr})$-bigraded spectral sequence, and we show how to compute its E-infinity objects. For notation and basic facts on occurring (co-)limits, we refer the reader to the Appendix, \ref{chap:CoLimsInModuleCategories}.

\begin{definition}[Spectral Sequence]
\label{def:SpectralSequence}%
A $(\Prdct{\ZNr}{\ZNr})$-bigraded \Defn{spectral sequence} of $R$-modules consists of %
\index{spectral sequence}\index[not]{$\ExctCplEObjctBB{r}{\Vect{x}}$ - spec seq object: page $r$, position $\Vect{x}\in \Prdct{\ZNr}{\ZNr}$}%
\begin{enumerate}[(a)]
\item a sequence of $(\Prdct{\ZNr}{\ZNr})$-bigraded $R$-modules $\SSPage{r} = (\SSObjctBB{r}{\Vect{x}}\, |\, \Vect{x}\in \ZNr \prdct \ZNr)$, $0\leq r_0\leq r \in \NNr$;
\item for each $r\geq r_0$, a family of $R$-module maps $\SSDffrntlBB{r}{\Vect{x}}\from \SSPage{r}_{\Vect{x}}\to \SSPage{r}_{\Vect{x}+\SSBidegreeVect{v}{r}}$ of bidegree $\SSBidegreeVect{v}{r}$.
\end{enumerate}
These are to satisfy the following conditions:
\begin{enumerate}[(1)]
\item $\SSDffrntl{r}$ is a differential on $\SSPage{r}$; i.e. any composite $\SSObjctBB{r}{\Vect{x}-\SSBidegreeVect{v}{r}} \XRA{\SSDffrntlBB{r}{\Vect{x}-\SSBidegreeVect{v}{r}} } \SSObjctBB{r}{\Vect{x}} \XRA{ \SSDffrntlBB{r}{\Vect{x}} } \SSObjctBB{r}{\Vect{x}+\Vect{v}_r}$ vanishes. %
\index{differential!of a spectral sequence}
\item $\SSPage{r+1}$ is the homology of $\SSPage{r}$ with respect to the differential $\SSDffrntl{r}$; i.e. for each $\Vect{x}\in \Prdct{\ZNr}{\ZNr}$,
\begin{equation*}
\SSObjctBB{r+1}{\Vect{x}} = \dfrac{\Ker{\SSDffrntlBB{r}{\Vect{x}}}}{\Img{\SSDffrntlBB{r}{\Vect{x}-\SSBidegreeVect{v}{r}} } }
\end{equation*}
\end{enumerate}
\end{definition}

\begin{terminology}[Naming conventions]
\label{term:SpectralSequence-Namings}
The following naming of terms associated to a spectral sequence $(\SSPage{r},\SSDffrntl{r} | r\geq r_0)$ are found in the literature:
\begin{enumerate}[$\bullet$]
\item The bigraded module $\SSPage{r}$, together with its differential $\SSDffrntl{r}$ form the $r$-th \Defn{layer/page/sheet/stage/step} of the spectral sequence. %
\index{spectral sequence!page}%
\index{spectral sequence!layer}%
\index{spectral sequence!stage}%
\index{spectral sequence!step}%
\index{spectral sequence!sheet}%
\index{page of a spectral sequence}
\item $(\SSPage{r},\SSDffrntl{r} | r\geq r_0)$ is a \Defn{first quadrant spectral sequence} if $\SSObjctBB{r_0}{p,q}\neq 0$ implies $p,q\geq 0$. - Second/third/fourth quadrant spectral sequences are defined analogously. %
\index{spectral sequence!first quadrant} %
\index{spectral sequence!second quadrant} %
\index{spectral sequence!third quadrant} %
\index{spectral sequence!fourth quadrant} %
\index{first quadrant spectral sequence}
\item $(\SSPage{r},\SSDffrntl{r} | r\geq r_0)$ is an \Defn{upper half plane spectral sequence} if $\SSObjctBB{r_0}{p,q}\neq 0$ implies $q\geq 0$. - Lower/left/right half plane spectral sequences are defined analogously. Any one of these spectral sequences is called a \Defn{half plane spectral sequence}. %
\index{spectral sequence!upper half plane}
\index{spectral sequence!half plane}
\index{spectral sequence!lower half plane}
\index{spectral sequence!right half plane}
\index{spectral sequence!left half plane}
\index{half plane spectral sequence}
\item $(\SSPage{r},\SSDffrntl{r} | r\geq r_0)$ is called a \Defn{full plane spectral sequence} if, in any position $(p,q)$, the object $\SSObjctBB{r_0}{p,q}$ is potentially nonzero.%
\index{spectral sequence!full plane}%
\index{full plane!spectral sequence}%
\item $(\SSPage{r},\SSDffrntl{r} | r\geq r_0)$ is called \Defn{homological} if its differential $\SSDffrntl{r}$ on page $\SSPage{r}$ has bidegree $(-r,r-1)$. It is called  \Defn{cohomological} if its differential $\SSDffrntl{r}$ has bidegree $(r,-r+1)$. %
\index{homological spectral sequence}\index{spectral sequence!homological}%
\index{cohomological spectral sequence}\index{spectral sequence!cohomological}%
\end{enumerate}
\end{terminology}

Thus the $r$-th page of a spectral sequence, whose differential has bidegree $\SSBidegreeVect{v}{r}=(\mu_r,\nu_r)$, may be visualized via this diagram:

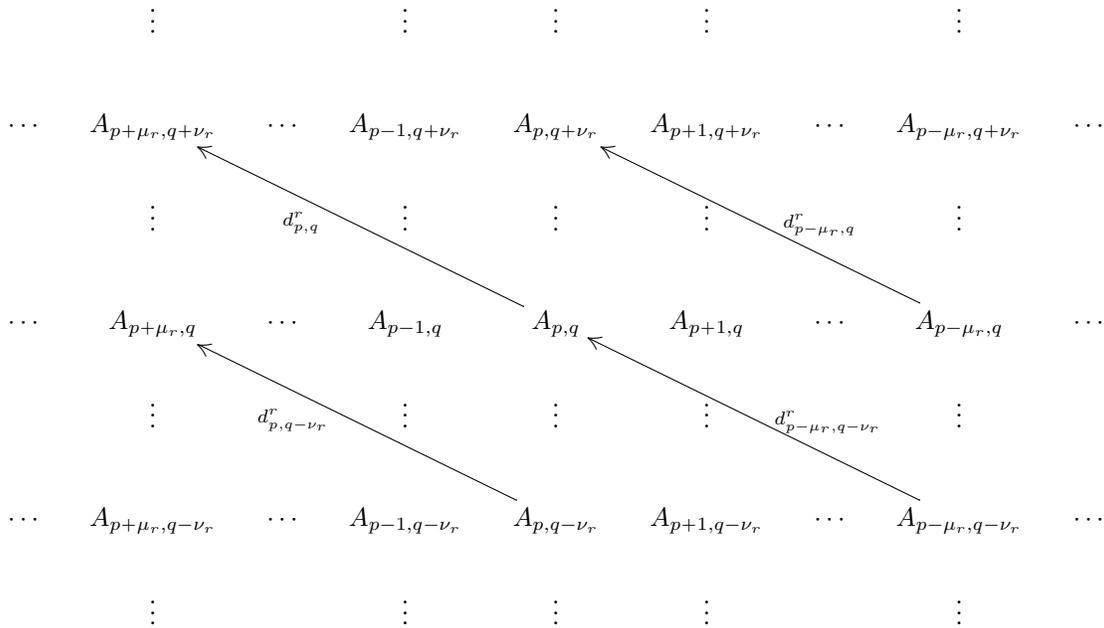
\begin{equation*}
\begin{xy}
\xymatrix@!R=3ex@C=3ex{
& \vdots && \vdots & \vdots & \vdots && \vdots \\
\cdots & A_{p+\mu_r,q+\nu_r} & \cdots & A_{p-1,q+\nu_r} & A_{p,q+\nu_r} & A_{p+1,q+\nu_r} & \cdots & A_{p-\mu_r,q+\nu_r} & \cdots \\
	& \vdots && \vdots & \vdots & \vdots && \vdots \\
\cdots & A_{p+\mu_r,q} & \cdots & A_{p-1,q} & A_{p,q} \ar[llluu]^(.6){\SSDffrntlAt{r}{p}{q}} & A_{p+1,q} & \cdots & A_{p-\mu_r,q} \ar[llluu]_(.4){\SSDffrntlAt{r}{p-\mu_r}{q} } & \cdots \\
	& \vdots && \vdots & \vdots & \vdots && \vdots \\
\cdots & A_{p+\mu_r,q-\nu_r} & \cdots & A_{p-1,q-\nu_r} & A_{p,q-\nu_r} \ar[llluu]^(.6){\SSDffrntlAt{r}{p}{q-\nu_r} } & A_{p+1,q-\nu_r} & \cdots & A_{p-\mu_r,q-\nu_r} \ar[llluu]_(.4){\SSDffrntlAt{r}{p-\mu_r}{q-\nu_r}} & \cdots \\
& \vdots && \vdots & \vdots & \vdots && \vdots \\
}
\end{xy}
\end{equation*}

Every spectral sequence has, what is called its limit page $\SSPage{\infty}$. We use the following lemma to compute it.

\begin{lemma}[Collecting `cycles' and `boundaries' of a spectral sequence]
\label{thm:E-infinityPreparation}%
A spectral sequence $(\SSPage{r},\SSDffrntl{r} | r\geq r_0)$ determines in each position $\Vect{x}\in \ZNr\prdct \ZNr$ a nesting of submodules of $\SSObjctBB{r_0}{\Vect{x}}$:
\begin{equation*}
0\EqDef B^{r_0-1}_{\Vect{x}} \subseteq \cdots\cdots \subseteq B^{r}_{\Vect{x}}  \subseteq \cdots \subseteq \FamUnion{r}{B^{r}_{\Vect{x}}}\subseteq \dots  \subseteq\FamIntrsctn{r}{Z^{r}_{\Vect{x}}}\subseteq\cdots \subseteq Z^{r}_{\Vect{x}}  \subseteq\cdots\cdots \subseteq Z^{r_0-1}_{\Vect{x}}\DefEq \SSObjctBB{r_0}{\Vect{x}}
\end{equation*}
such that $\SSObjctBB{r+1}{\Vect{x}}\cong Z^{r}_{\Vect{x}}/B^{r}_{\Vect{x}}$.
\end{lemma}
\begin{proof}
If $\SSDffrntl{r}$ has bidegree $\Vect{v}_r$, then we inductively construct the submodules $\ExctCplCyclesBB{r}{\Vect{x}}$ and $\ExctCplBndrsBB{r}{\Vect{x}}$: Let
\begin{equation*}
\ExctCplBndrsBB{r_0-1}{\Vect{x}}\DefEq 0  \qquad \text{and}\qquad \ExctCplCyclesBB{r_0-1}{\Vect{x}}\DefEq \SSObjctBB{r_0}{\Vect{x}}
\end{equation*}
and assume that, for $r_0-1\leq k\leq r-1$, and $\Vect{x}\in \Prdct{\ZNr}{\ZNr}$, we have short exact sequences plotted vertically in black in this diagram:
\begin{equation*}
\xymatrix@R=5ex@C=0em{
0& = & \ExctCplBndrsBB{r_0-1}{\Vect{x}} \ar@{{ |>}->}[d] & \subseteq & \ExctCplBndrsBB{r_0}{\Vect{x}} \ar@{{ |>}->}[d] & \subseteq & \cdots & \subseteq & \ExctCplBndrsBB{r-1}{\Vect{x}} \ar@{{ |>}->}[d]  & \subseteq & {\color{Maroon}\ExctCplBndrsBB{r}{\Vect{x}}} \ar@{{ |>}->}@[Maroon][d] & {\color{Maroon}\DefEq} & {\color{Maroon}\left(\tau^{r-1}_{\Vect{x}}\right)^{-1}\Img{\SSDffrntlBB{r}{\Vect{x}-\Vect{v}_r}}} \\
\SSObjctBB{r_0}{\Vect{x}} & = & \ExctCplCyclesBB{r_0-1}{\Vect{x}} \ar@{-{ >>}}[d]^{\tau^{r_0-1}_{\Vect{x}}}_{\IdMap} & \supseteq & \ExctCplCyclesBB{r_0}{\Vect{x}} \ar@{-{ >>}}[d]^{\tau^{r_0}_{\Vect{x}}} & \supseteq & \cdots & \supseteq & \ExctCplCyclesBB{r-1}{\Vect{x}} \ar@{-{ >>}}[d]^{\tau^{r-1}_{\Vect{x}}} & \supseteq & {\color{Maroon}\ExctCplCyclesBB{r}{\Vect{x}}} \ar@{-{ >>}}@[Maroon][d]^{\color{Maroon} \tau^{r}_{\Vect{x}}}& {\color{Maroon}\DefEq} & {\color{Maroon} \left(\tau^{r-1}_{\Vect{x}}\right)^{-1}\Ker{\SSDffrntlBB{r}{\Vect{x}} }} \\
&& \SSObjctBB{r_0}{\Vect{x}} && \SSObjctBB{r_0+1}{\Vect{x}} &&&& \SSObjctBB{r}{\Vect{x}} && {\color{Maroon}\SSObjctBB{r+1}{\Vect{x}}} & {\color{Maroon}\cong} & {\color{Maroon}\ExctCplCyclesBB{r}{\Vect{x}}/\ExctCplBndrsBB{r}{\Vect{x}}}
}
\end{equation*}
Then the induction is completed via the colored portion of the diagram on the right.
\end{proof}

\begin{definition}[Limit page of a spectral sequence]
\label{def:SpectralSequence-LimitPage}%
In the setting of (\ref{thm:E-infinityPreparation}), put $B^{\infty}_{\Vect{x}}\DefEq \FamUnion{r}{ \ExctCplBndrsBB{r}{\Vect{x}}} \subseteq \FamIntrsctn{r}{ \ExctCplCyclesBB{r}{\Vect{x}} } \EqDef Z^{\infty}_{\Vect{x}}$. Then the limit page of $(\SSPage{r},\SSDffrntl{r})$ is the $(\ZNr\prdct \ZNr)$-bigraded $R$-module $\SSPage{\infty}$ with %
\index{spectral sequence!limit page}\index[not]{$\SSPage{r}$}
\begin{equation*}
\SSObjctBB{\infty}{\Vect{x}} \DefEq \dfrac{ \ExctCplCyclesBB{\infty}{\Vect{x}}}{ \ExctCplBndrsBB{\infty}{\Vect{x}} }\ .
\end{equation*}
\end{definition}

In categorial terms, $\ExctCplBndrsBB{\infty}{\Vect{x}}\cong \CoLimOf{\ExctCplBndrsBB{r}{\Vect{x}} }$, and $\ExctCplCyclesBB{\infty}{\Vect{x}}\cong \LimOf{\ExctCplCyclesBB{r}{\Vect{x}} }$. While limits and colimits do not commute in general, they do so in the construction of $\SSPage{\infty}$:

\begin{proposition}[$\SSPage{\infty}$ as a (co-)limit]
\label{thm:E-infinityAs(Co-)Limit}%
The $\SSPage{\infty}$-terms of a spectral sequence $(\SSPage{r},\SSDffrntl{r})$ satisfy for each $\Vect{x}\in \ZNr\prdct \ZNr$:
\begin{enumerate}[(i)]
\item For each $s\geq r_0$, $\ExctCplCyclesBB{s}{\Vect{x}}/\ExctCplBndrsBB{\infty}{\Vect{x}} \cong \CoLimOfOver{\ExctCplCyclesBB{s}{\Vect{x}}/\ExctCplBndrsBB{r}{\Vect{x}} }{r}$ is a colimit of epimorphisms.
\item For each $r\geq r_0$, $\ExctCplCyclesBB{\infty}{\Vect{x}}/\ExctCplBndrsBB{r}{\Vect{x}} \cong \LimOfOver{\ExctCplCyclesBB{s}{\Vect{x}}/\ExctCplBndrsBB{r}{\Vect{x}}}{s}$ is a limit of monomorphisms.
\end{enumerate}
Consequently,
\begin{equation*}
\LimOfOver{\CoLimOfOver{\ExctCplCyclesBB{s}{\Vect{x}}/\ExctCplBndrsBB{r}{\Vect{x}} }{r}}{s} \cong \SSObjctBB{\infty}{\Vect{x}} \cong \CoLimOfOver{ \LimOfOver{\ExctCplCyclesBB{s}{\Vect{x}}/\ExctCplBndrsBB{r}{\Vect{x}} }{s} }{r}
\end{equation*}
\end{proposition}
\begin{proof}
To see (i), fix $s\geq r_0$, and consider the short exact sequence of $\OrdOmega$-diagrams $\ExctCplBndrsBB{r}{\Vect{x}}\to \ExctCplCyclesBB{s}{\Vect{x}} \to  \ExctCplCyclesBB{s}{\Vect{x}}/\ExctCplBndrsBB{r}{\Vect{x}}$. The claim follows because $\CoLimOver{\OrdOmega}$ is exact; see (\ref{thm:CoLim^ZExactFunctor}). To see (ii), fix $r\geq r_0$, and consider the short exact sequence of $\OrdOmegaOp$-diagrams $\ExctCplBndrsBB{r}{\Vect{x}} \to \ExctCplCyclesBB{s}{\Vect{x}} \to \ExctCplCyclesBB{s}{\Vect{x}} / \ExctCplBndrsBB{r}{\Vect{x}}$. The kernel-diagram has vanishing $\LimOne$. So $\Lim$ applied to this sequence is exact, and this implies (ii).

Applying similar reasoning to the $\omega$-diagram of short exact sequences $\ExctCplBndrsBB{r}{\Vect{x}} \to \ExctCplCyclesBB{\infty}{\Vect{x}} \to \ExctCplCyclesBB{\infty}{\Vect{x}}/\ExctCplBndrsBB{r}{\Vect{x}}$, respectively the $\bar{\omega}$-diagram of short exact sequences $\ExctCplBndrsBB{\infty}{\Vect{x}}\to \ExctCplCyclesBB{s}{\Vect{x}} \to \ExctCplCyclesBB{s}{\Vect{x}} / \ExctCplBndrsBB{\infty}{\Vect{x}}$, yields the two descriptions of $\SSObjctBB{\infty}{\Vect{x}}$.
\end{proof}

\begin{definition}[Collapsing spectral sequence]
\label{def:SpectralSequence-Collapse}%
A spectral sequence $(\SSPage{r},\SSDffrntl{r}\, |\, r\geq r_0)$ collapses on page $c\geq r_0$ if $\SSDffrntl{c+k}=0$ for all $k\geq 0$. %
\index{collapsing spectral sequence}%
\index{spectral sequence!collapse}%
\end{definition}

\begin{proposition}[$\SSPage{\infty}$ of a collapsing spectral sequence]
\label{thm:E-infinityOfCollapsingSpectralSequence}%
If a spectral sequence $(\SSPage{r},\SSDffrntl{r})$ collapses on page $c$, then $\SSPage{c}\cong \SSPage{c+1}\cong \cdots \cong \SSPage{c+k}\cong \cdots \cong \SSPage{\infty}$. \NoProof
\end{proposition}

The information provided in this section already supports successful computations with some spectral sequences. The interested reader is encouraged to take a look at Chapter \ref{sec:SS-Computations}.

\section[Morphisms of Spectral Sequences]{Morphisms of Spectral Sequences}
\label{sec:SpectralSequences-Morphisms}

A morphism $(f^r\from (\SSPage{r}(1),\SSDffrntl{r}(1)) \to (\SSPage{r}(2),\SSDffrntl{r}(2)\, |\, r\geq r_0)$ of spectral sequences, see (\ref{def:Morphism-SpectralSequences}), is given by a family $(f^r)$ of morphisms of bigraded $R$-modules which are required to (a) commute with differentials and (b) $f^r$ induces $f^{r+1}$ via passage to homology. A morphism of spectral sequences induces a morphism of $\SSPage{\infty}$-limit pages; see (\ref{thm:EInfinityMapFromMorphismOfSpecSeqs}). We are particularly interested in conditions under which a morphism of spectral sequences induces an isomorphism of $\SSPage{\infty}$-pages; see (\ref{thm:ErIsoGivesEInftyIso}).

\begin{definition}[Morphism of spectral sequences]
\label{def:Morphism-SpectralSequences}%
Let $(\SSPage{r}(1),\SSDffrntl{r}(1)\, |\, r\geq r_0)$ and $(\SSPage{r}(2),\SSDffrntl{r}(2)\, |\, r\geq r_0)$ be spectral sequences whose differentials $\SSDffrntl{r}(1)$ and $\SSDffrntl{r}(2)$ have the same bidegree  $\SSBidegreeVect{v}{r}$, $r\geq r_0$. A \Defn{morphism of spectral sequences} from $(\SSPage{r}(1),\SSDffrntl{r}(1))$ to $(\SSPage{r}(2),\SSDffrntl{r}(2))$ is given by a family of morphisms of bigraded modules $(\SSMap{f}{r}\from \SSPage{r}(1)\to \SSPage{r}(2) | r\geq r_0)$ such that the diagram below commutes for each $r\geq r_0$ and $\Vect{x}\in \Prdct{\ZNr}{\ZNr}$: %
\index{morphism!of spectral sequences}%
\index{spectral sequence!morphism}
\begin{equation*}
\begin{xy}
\xymatrix@R=5ex@C=5em{
\SSObjctBB{r}{\Vect{x}}(1) \ar[r]^-{\SSDffrntlBB{r}{\Vect{x}}(1)} \ar[d]_{\SSMapBB{f}{r}{\Vect{x}} } &
	\SSObjctBB{r}{\Vect{x}+\Vect{v}_r}(1) \ar[d]^{\SSMapBB{f}{r}{\Vect{x}+\Vect{v}_r} } \\
\SSObjctBB{r}{\Vect{x}}(2) \ar[r]_-{\SSDffrntlBB{r}{\Vect{x}}(2) } &
	\SSObjctBB{r}{\Vect{x}+\SSBidegreeVect{v}{r}}(2)
}
\end{xy}
\end{equation*}
Moreover, $\SSMap{f}{r}$ is to induce $\SSMap{f}{r+1}$ in homology; i.e. $\SSMapBB{f}{r+1}{}=H(\SSMapBB{f}{r}{})$.
\end{definition}

\begin{subordinate}
Thus the spectral sequences in $\ModulesOver{R}$ form a category $\SpecSeqsInModOver{R}$ whose connected components are determined by the collection of families $(r_0;\SSBidegreeVect{b}{r}\, |\, r\in \ZNr,\ r\geq r_0)$. Given such a family, the associated connected component $\SpecSeqsInModOver{R}(r_0;\SSBidegreeVect{b}{r}\, |\, r\geq r_0)$ is additive with sums defined positionwise.

More generally, it is possible to describe the category of all $(\Prdct{\ZNr}{\ZNr})$-bigraded $R$-modules with a differential of a given bidegree as a full subcategory of a functor category. It follows that the category of spectral sequences of $R$-modules is an additive category which is closed under those (co-)limits which are exact, hence commute with homology.
\end{subordinate}%
\index[not]{$\SpecSeqsInModOver{R}$ - cat of spectral sequences in $\ModulesOver{R}$}\index[not]{$\SpecSeqsInModOver{R}(r_0;\SSBidegreeVect{b}{r})$}%

\begin{proposition}[$E^{\infty}$-map from morphism of spectral sequences]
\label{thm:EInfinityMapFromMorphismOfSpecSeqs}%
A morphism $(\SSMap{f}{r}\from \SSPage{r}(1)\to \SSPage{r}(2) | r\geq r_0)$ of spectral sequences functorially induces a morphism of $\SSPage{\infty}$-pages. %
\begin{equation*}
\SSMap{f}{\infty}\from \SSPage{\infty}(1) \longrightarrow \SSPage{\infty}(2)
\end{equation*}
\end{proposition}
\begin{proof}
An inductive argument shows that $f^{r_0}$ restricts, for each $\Vect{x}\in \ZNr\prdct \ZNr$, to the vertical arrows in the diagram below.
\begin{equation*}
\begin{xy}
\xymatrix@C=0mm{
B^{r_0}_{\Vect{x}}(1) \ar[d] &
	\subseteq &
	B^{r_0+1}_{\Vect{x}}(1) \ar[d] &
  \subseteq\ \cdots\ \subseteq &
  B^{\infty}_{\Vect{x}}(1) \ar@{..>}[d] &
  \subseteq &
  Z^{\infty}_{\Vect{x}}(1) \ar@{..>}[d] &
  \subseteq\ \cdots\ \subseteq &
  Z^{r_0+1}_{\Vect{x}}(1) \ar[d] &
  \subseteq &
  Z^{r_0}_{\Vect{x}}(1) \ar[d] \\
B^{r_0}_{\Vect{x}}(2) &
  \subseteq &
  B^{r_0+1}_{\Vect{x}}(2) &
  \subseteq\ \cdots\ \subseteq &
  B^{\infty}_{\Vect{x}}(2) &
  \subseteq &
  Z^{\infty}_{\Vect{x}}(2) &
  \subseteq\ \cdots\ \subseteq &
  Z^{r_0+1}_{\Vect{x}}(2) &
  \subseteq &
  Z^{r_0}_{\Vect{x}}(2)
}
\end{xy}
\end{equation*}
On cokernels, $f^{r_0}$ induces this morphism of $E^{\infty}$-objects: $Z^{\infty}_{\Vect{x}}(1)/B^{\infty}_{\Vect{x}}(1)\cong \SSPage{\infty}_{\Vect{x}}(1) \XRA{f^{\infty}_{\Vect{x}}} \SSPage{\infty}_{\Vect{x}}(2)\cong Z^{\infty}_{\Vect{x}}(2)/B^{\infty}_{\Vect{x}}(2)$. This was to be shown.
\end{proof}

We turn to conditions under which a morphism of spectral sequence induces an isomorphism of $\SSPage{\infty}$-pages.

\begin{lemma}[Epi/mono/iso propagation]
\label{thm:E^rEpiMonoIsoPropagation}%
Given a morphism $(\SSMap{f}{r}\from \SSPage{r}(1)\to \SSPage{r}(2)\, |\, r\geq r_0)$ of spectral sequences, consider the induced diagram below.
\begin{equation*}
\begin{xy}
\xymatrix@R=6ex@C=4em{
\SSObjctBB{r}{\Vect{x}-\Vect{b}_r}(1) \ar[d]_{\SSMapBB{f}{r}{\Vect{x}-\Vect{b}_r}} \ar[r]^-{\SSDffrntl{r}(1)} &
    \SSObjctBB{r}{\Vect{x}}(1) \ar[d]_{f^{r}_{\Vect{x}}} \ar[r]^-{d^{r}(1)} &
    \SSObjctBB{r}{\Vect{x}+\Vect{b}_r}(1) \ar[d]^{\SSMapBB{f}{r}{\Vect{x}+\Vect{b}_r}} \\
\SSObjctBB{r}{\Vect{x}-\Vect{b}_r}(2) \ar[r]_-{\SSDffrntl{r}(2)} &
    \SSObjctBB{r}{\Vect{x}}(2) \ar[r]_-{d^{r}(2)} &
    \SSObjctBB{r}{\Vect{x}+\Vect{b}_r}(2) 
     }
\end{xy}
\end{equation*}
Then the following hold:
\begin{enumerate}[(i)]
\item $\SSMapBB{f}{r+1}{\Vect{x}}$ is a monomorphism whenever $\SSMapBB{f}{r}{\Vect{x}-\Vect{b}_r}$ is an epimorphism and $\SSMapBB{f}{r}{\Vect{x}}$ is a monomorphism. 
\item $\SSMapBB{f}{r+1}{\Vect{x}}$ is an epimorphism whenever $\SSMapBB{f}{r}{\Vect{x}}$ is an epimorphism and $\SSMapBB{f}{r}{\Vect{x}+\Vect{b}_r}$ is a monomorphism.
\item $\SSMapBB{f}{r+1}{\Vect{x}}$ is an isomorphism whenever $\SSMapBB{f}{r}{\Vect{x}-\Vect{b}_r}$ is an epimorphism, $\SSMapBB{f}{r}{\Vect{x}}$ is an isomorphism, and $\SSMapBB{f}{r}{\Vect{x}+\Vect{b}_r}$ is a monomorphism.
\end{enumerate}
\end{lemma}
\begin{proof}
Consider this morphism of defining short exact sequences for the objects $\SSObjctBB{r+1}{\Vect{x}}(i)$:
\begin{equation*}
\xymatrix@R=5ex@C=4em{
\Img{\SSDffrntlBB{r}{\Vect{x}-\Vect{v}_r}}(1) \ar@{{ |>}->}[r] \ar[d]_{\SSMapBB{f}{r}{\Vect{x}|}} &
	\Ker{\SSDffrntlBB{r}{\Vect{x}}}(1) \ar@{-{ >>}}[r] \ar[d]_{\SSMapBB{f}{r}{\Vect{x}|}} &
	\ExctCplEObjctBB{r+1}{\Vect{x}}(1) \ar[d]^{\SSMapBB{f}{r+1}{\Vect{x}}} \\
\Img{\SSDffrntlBB{r}{\Vect{x}-\Vect{v}_r}}(2) \ar@{{ |>}->}[r] &
	\Ker{\SSDffrntlBB{r}{\Vect{x}}}(2) \ar@{-{ >>}}[r] &
	\ExctCplEObjctBB{r+1}{\Vect{x}}(2)
}
\end{equation*}
With the snake lemma, we see that $\SSMapBB{f}{r+1}{\Vect{x}}$ is
\begin{enumerate}[$\bullet$]
\item monic if the map of boundaries is epic, and the map of cycles is monic;
\item epic if the map of cycles is epic.
\end{enumerate}
To analyze these properties, consider these morphisms of exact sequences:
\begin{equation*}
\xymatrix@R=5ex@C=4em{
\Ker{\SSDffrntlBB{r}{\Vect{x}}}(1) \ar@{{ |>}->}[r] \ar[d]_{\SSMapBB{f}{r}{\Vect{x}}|} &
    \SSObjctBB{r}{\Vect{x}}(1) \ar[r]^-{d^r(1)} \ar[d]_{\SSMapBB{f}{r}{\Vect{x}} } &
    \SSObjctBB{r}{\Vect{x}+\Vect{b}_r}(1) \ar[d]^{\SSMapBB{f}{r}{\Vect{x}+\Vect{b}_r} } \\
\Ker{\SSDffrntlBB{r}{\Vect{x}}}(2) \ar@{{ |>}->}[r] &
    \SSObjctBB{r}{\Vect{x}}(2) \ar[r]_-{\SSDffrntl{r}(2)} &
    \SSObjctBB{r}{\Vect{x}+\Vect{b}_r}(2)
}\qquad\qquad
\xymatrix@R=5ex@C=3em{
\SSObjctBB{r}{\Vect{x}-\Vect{b}_r}(1)  \ar@{-{ >>}}[r]^-{\SSDffrntlBB{r}{}(1)} \ar[d]_{\SSMapBB{f}{r}{\Vect{x}-\Vect{b}_r}} &
    \Img{\SSDffrntlBB{r}{\Vect{x}-\Vect{v}_r}}(1) \ar[d]^{\SSMapBB{f}{r}{\Vect{x}}| } \\
\SSObjctBB{r}{\Vect{x}-\Vect{b}_r}(2) \ar@{-{ >>}}[r]_-{\SSDffrntlBB{r}{}(2)} &
    \Img{\SSDffrntlBB{r}{\Vect{x}-\Vect{v}_r}}(2)
}
\end{equation*}
We find: The map of cycles is a monomorphism whenever $\SSMapBB{f}{r}{\Vect{x}}$ is a monomorphism, and is an epimorphism whenever $\SSMapBB{f}{r}{\Vect{x}}$ is an epimorphism and $\SSMapBB{f}{r}{\Vect{x}+\Vect{b}_r}$ is a monomorphism. Similarly, the map of boundaries is an epimorphism whenever $\SSMapBB{f}{r}{\Vect{x}-\Vect{b}_r}$ is an epimorphism. - This implies the claim.
\end{proof}

\begin{corollary}[$\SSMap{f}{r}\from \SSPage{r}(1)\xrightarrow{\cong}\SSPage{r}(2)$ gives $\SSMap{f}{\infty}\from \SSPage{\infty}(1)\xrightarrow{\cong} \SSPage{\infty}(2)$]
\label{thm:ErIsoGivesEInftyIso}%
Suppose for a morphism $(\SSMap{f}{r}\from \SSPage{r}(1)\to \SSPage{r}(2)\, |\, r\geq r_0)$ of spectral sequences there exists $r\geq r_0$ such that $\SSMap{f}{r}\from \SSPage{r}(1)\xrightarrow{\ \cong\ } \SSPage{r}(2)$ is an isomorphism. Then $\SSMap{f}{r+k}$ is an isomorphism for each $0\leq k\leq \infty$.
\end{corollary}
\begin{proof}
Via (\ref{thm:E^rEpiMonoIsoPropagation}) we see that $\SSMap{f}{r+k}$ is an isomorphism for $0\leq k<\infty$. It remains to show that $f^{\infty}$ is an isomorphism. First, the diagram below represents the induction step toward showing that the vertical arrow in the middle is an isomorphism for all $l\geq 0$:
\begin{equation*}
\xymatrix@R=5ex@C=4em{
B^{r+l-1}_{\Vect{x}}(1)/B^{r-1}_{\Vect{x}}(1) \ar@{{ |>}->}[r] \ar[d]_{\cong} &
	B^{r+l}_{\Vect{x}}(1)/B^{r-1}_{\Vect{x}}(1) \ar@{-{ >>}}[r] \ar[d] &
	\Img{d^{r+l}_{\Vect{x}-\Vect{v}_{r+l}}(1)} \ar[d]^{\cong}_{f^{r+l}_{\Vect{x}}|} \\
B^{r+l-1}_{\Vect{x}}(2)/B^{r-1}_{\Vect{x}}(2) \ar@{{ |>}->}[r] &
	B^{r+l}_{\Vect{x}}(2)/B^{r-1}_{\Vect{x}}(2) \ar@{-{ >>}}[r] &
	\Img{d^{r+l}_{\Vect{x}-\Vect{v}_{r+l}}(2)}
}
\end{equation*}
With the exactness of $\CoLimOver{\omega}$ we conclude
\begin{equation*}
\dfrac{B^{\infty}_{\Vect{x}}(1)}{B^{r-1}_{\Vect{x}}(1)} \cong \CoLimOfOver{\dfrac{B^{r+l}_{\Vect{x}}(1)}{B^{r-1}_{\Vect{x}}(1)}}{\omega} \cong \CoLimOfOver{\dfrac{B^{r+l}_{\Vect{x}}(2)}{B^{r-1}_{\Vect{x}}(2)}}{\omega} \cong \dfrac{B^{\infty}_{\Vect{x}}(2)}{B^{r-1}_{\Vect{x}}(2)}
\end{equation*}
Next, the diagram below shows that the vertical arrow in the middle is an isomorphism for every $k\geq 0$.
\begin{equation*}
\xymatrix@R=5ex@C=4em{
B^{r+k-1}_{\Vect{x}}(1)/B^{r-1}_{\Vect{x}}(1) \ar@{{ |>}->}[r] \ar[d]_{\cong} &
	Z^{r+k}_{\Vect{x}}(1)/B^{r-1}_{\Vect{x}}(1) \ar@{-{ >>}}[r] \ar[d] &
	\Ker{d^{r+k}_{\Vect{x}}(1)} \ar[d]^{\cong}_{f^{r+k}_{\Vect{x}}|} \\
B^{r+k-1}_{\Vect{x}}(2)/B^{r-1}_{\Vect{x}}(2) \ar@{{ |>}->}[r] &
	Z^{r+k}_{\Vect{x}}(2)/B^{r-1}_{\Vect{x}}(2) \ar@{-{ >>}}[r] &
	\Ker{d^{r+k}_{\Vect{x}}(2)}
}
\end{equation*}
Combined, for every $k\geq 0$, we obtain isomorphisms
\begin{equation*}
\dfrac{Z^{r+k}_{\Vect{x}}(1)}{B^{\infty}_{\Vect{x}}(1)} \cong \dfrac{Z^{r+k}_{\Vect{x}}(1)/B^{r-1}_{\Vect{x}}(1)}{B^{\infty}_{\Vect{x}}(1)/B^{r-1}_{\Vect{x}}(1)} \cong \dfrac{Z^{r+k}_{\Vect{x}}(2)/B^{r-1}_{\Vect{x}}(2)}{B^{\infty}_{\Vect{x}}(2)/B^{r-1}_{\Vect{x}}(2)} \cong \dfrac{Z^{r+k}_{\Vect{x}}(2)}{B^{\infty}_{\Vect{x}}(2)}
\end{equation*}
Finally, the limit description (\ref{thm:E-infinityAs(Co-)Limit}) of $\SSPage{\infty}$ yields isomorphisms $\SSObjctBB{\infty}{\Vect{x}}(1)\cong \SSObjctBB{\infty}{\Vect{x}}(1)$ for $\Vect{x}\in \Prdct{\ZNr}{\ZNr}$.
\end{proof}

\begin{subordinate}
\begin{remark}[Mono- / epimorphisms don't pass through spectral sequence pages]
\label{rem:Mono/Epi-NoPassThroughPages}
We can not expect a component-wise monomorphism $f^r$ to induce a component-wise monomorphism $f^{r+1}$. This is so because a monomorphism of chain complexes need not induce a monomorphism in homology. Similarly, a component-wise epimorphism $f^r$ need not propagate to a component-wise epimorphism $f^{r+1}$.
\end{remark}
\end{subordinate}

\section{Computing with Spectral Sequences}
\label{sec:SS-Computations}

Appreciation for spectral sequences develops with using them. To facilitate the opportunity for practice, we present here selected classical spectral sequences. Assuming only relevant background knowledge from the subject they are taken from, we suggest immediate applications.  More applications of spectral sequences can be found in the literature, notably \cite{JMcCleary1999}.

In the following sample computations with spectral sequences, notice that the algebraic structures governing computations in the spectral sequence alone suffice. No further subject knowledge whatsoever is involved! 

{\bfseries Background from group homology}\quad The homology of a group $G$ with coefficients in an abelian group $A$ is given by the family of abelian groups $H_n(G;A)$, with $n\geq 0$. These groups depend functorially on $G$ and $A$ and satisfy the following properties:
\begin{enumerate}[(i)]
\item $H_0(G;A)=A$  for every group $G$
\item $H_1(G;A)\cong G_{\mathit{ab}}\otimes_{\ZNr} A$
\item $H_{\geq 2}(\ZNr;A)=0$
\end{enumerate}
It is customary to  use the shorthand $H_nG \DefEq H_n(G;\ZNr)$.

\begin{theorem}[Lyndon-Hochschild-Serre spectral sequence]
\label{thm:Lyndon-Hochschild-Serre-SS}%
Let  $A\to E\to G$  be a central group extension. Then $H_nE$  has a filtration $F_{p,q}$, $p+q=n$, satisfying
\begin{equation*}
0=F_{-1,n+1}\subseteq F_{0,n}\subseteq \cdots \subseteq F_{p-1,q+1}\subseteq F_{p,q}\subseteq \cdots \subseteq F_{n,0}=H_nE
\end{equation*}
There is a spectral sequence $(E^r,d^r)$, $r\geq 2$, whose differential $d^r$  has bidegree $(-r,r-1)$. Moreover:
\begin{equation*}
E^{2}_{p,q} = H_p(G;H_qA)\qquad \text{and}\qquad  \xymatrix{
F_{p-1,q+1} \ar@{{ |>}->}[r] &
    F_{p,q} \ar@{-{ >>}}[r] &
    E^{\infty}_{p,q} }
\end{equation*}
All of these data depend functorially morphisms between central extensions. \NoProof
\end{theorem}

\begin{theorem}[Homology of cyclic groups]
\label{thm:Homology-C_n}
For $k\geq 2$, the homology of the cyclic group $\CyclcGrp{k}$  of order $k$ is given by
\begin{equation*}
\GrpHmlgyDimOf{n}{\CyclcGrp{k}}\ \cong\ \left\{
\begin{array}{rcl}
\ZNr & \text{if} & n=0 \\
\ZMod{k} & \text{if} & n \quad \text{odd} \\
0 & \text{if} & n\geq 2 \quad \text{even}
\end{array}\right.
\end{equation*}
\end{theorem}
\begin{proof}
To apply the LHS-spectral sequence, use the short exact sequence
\begin{equation*}
\xymatrix@R=5ex@C=4em{
\ZNr \ar@{{ |>}->}[r]^-{\times k} &
    \ZNr \ar@{-{ >>}}[r]^-{f} &
   \CyclcGrp{k}	}
\end{equation*}
Then the $E^2$-page of the SS satisfies $\SSObjct{2}{p}{q} = \GrpHmlgyDimOfCoeffs{p}{ \ZMod{k} }{ \GrpHmlgyDimOf{q}{\ZNr} }$. With the information provided above, we conclude:
\begin{equation*}
\SSObjct{2}{p}{0} = \GrpHmlgyDimOf{p}{\CyclcGrp{k}} = \SSObjct{2}{p}{1}  \qquad \text{and}\qquad \SSObjct{2}{p}{q}=0\quad \text{if}\quad q\neq 0,1\quad\text{or}\quad p\leq -1
\end{equation*}
Thus the $E^2$-page of the spectral sequence takes the following form:
$$
\begin{array}{c|cccccc}
q & 0 & 0 & \cdots & \cdots & 0 & \cdots \\
1 & \colorbox{blanchedalmond}{$\color{blue} \ZNr$} & \colorbox{blanchedalmond}{$\color{red} \ZMod{k}$} & {\color{ForestGreen} \GrpHmlgyDimOf{2}{\CyclcGrp{k}}} & \GrpHmlgyDimOf{3}{\CyclcGrp{k}} & \GrpHmlgyDimOf{4}{\CyclcGrp{k}} & \cdots \\
0 & \colorbox{blanchedalmond}{$\ZNr$} & \colorbox{blanchedalmond}{$\ZMod{k}$} & {\color{blue} \GrpHmlgyDimOf{2}{\CyclcGrp{k}}} & {\color{red} \GrpHmlgyDimOf{3}{\CyclcGrp{k}}} & {\color{ForestGreen} \GrpHmlgyDimOf{4}{\CyclcGrp{k}} } & \cdots \\ \hline
-1/-1 & 0 & 1 & 2 & 3 & 4 & p
\end{array}
$$
Now let's take a look at the differentials: We know that the bidegree of $d^r$ is $(-r,r-1)$. This implies:
\begin{enumerate}
\item On $E^2$: $d^{2}_{p,q}\neq 0$, is only possible if $p\geq 2$, and $q=0$.
\item For $r\geq 3$, $d^r=0$; that is: the spectral  sequence collapses on page $3$.
\end{enumerate}
With this information, will compute  $\GrpHmlgyDimOf{n}{\CyclcGrp{k}}$  via the following steps:
\begin{enumerate}[(a)]
\item We use our knowledge of $\GrpHmlgyDimOf{n}{\ZNr}$  to compute $\SSObjct{\infty}{p}{q}$.
\item Knowledge of the $\SSObjct{\infty}{p}{q}$  will determine all differentials on $\SSPage{2}$.
\item Knowledge of the differentials on $\SSPage{2}$, combined with the known terms (colored background) will enable us to complete the computation.
\end{enumerate}
To find the terms $E^{\infty}_{p,q}$, recall how they determine $\GrpHmlgyDimOf{n}{\ZNr}$:
\begin{equation*}
\xymatrix@R=4ex@C=2em{
0=F_{-1,n+1} \ar@{{ |>}->}[r] &
    F_{0,n} \ar@{-{ >>}}[d] \ar@{{ |>}->}[r] &
    \cdots \ar@{{ |>}->}[r] &
    F_{p-1,q+1} \ar@{{ |>}->}[r] \ar@{-{ >>}}[d] &
    F_{p,q} \ar@{{ |>}->}[r] \ar@{-{ >>}}[d] &
    \cdots \ar@{{ |>}->}[r] &
    F_{n,0} = \GrpHmlgyDimOf{n}{\ZNr} \ar@{-{ >>}}[d] \\
& E^{\infty}_{0,n} & &
    E^{\infty}_{p-1,q+1} &
    E^{\infty}_{p,q} &&
    E^{\infty}_{n,0}
}
\end{equation*}
Thus, for a fixed $n$, terms contributing to $\GrpHmlgyDimOf{n}{\ZNr}$ lie on the antidiagonal joining $\SSObjct{\infty}{0}{n}$ to $\SSObjct{\infty}{n}{0}$.
\begin{equation*}
\begin{array}{c|ccccccccc}
n & {\color{red} \SSObjct{\infty}{0}{n} } &&&&&&&&  \\
n-1 & {\color{blue} \SSObjct{\infty}{0}{n-1} } & {\color{red} \SSObjct{\infty}{1}{n-1} } & &&&&&& \\
\vdots & \vdots  & {\color{blue} \SSObjct{\infty}{1}{n-2} } & {\color{red} \ddots } & &&&&& \\
\vdots & \vdots  & & {\color{blue} \ddots } & {\color{red} \ddots } &&&&& \\
p & {\color{ForestGreen} \SSObjct{\infty}{0}{p} } & & & {\color{blue} \ddots } & {\color{red} \SSObjct{\infty}{p}{n-p} } &&&& \\
\vdots & \vdots  & & &&  {\color{blue} \SSObjct{\infty}{p}{n-p-1} } & {\color{red} \ddots } &&& \\
2 & {\color{deepmagenta} \SSObjct{\infty}{0}{2} } & \SSObjct{\infty}{1}{2} & &&& & && \\
1 & {\color{chestnut} \SSObjct{\infty}{0}{1} } & {\color{deepmagenta} \SSObjct{\infty}{1}{1} } & &&&& {\color{blue} \ddots } & {\color{red} \SSObjct{\infty}{n-1}{1} } & \\
0 & {\color{red} \SSObjct{\infty}{0}{0} } & {\color{chestnut} \SSObjct{\infty}{1}{0} } & {\color{deepmagenta} \SSObjct{\infty}{2}{0} } & \cdots & {\color{ForestGreen} \SSObjct{\infty}{p}{0} } & \cdots & \cdots & {\color{blue} \SSObjct{\infty}{n-1}{0} } & {\color{red} \SSObjct{\infty}{n}{0} } \\ \hline
  & 0 & 1 & 2 & \cdots & p & \cdots & \cdots & n-1 & n \\ \hline
  & {\color{red} \ZNr } & {\color{chestnut} \ZNr } & {\color{deepmagenta} 0 } & \cdots & {\color{ForestGreen} 0 } & \cdots & \cdots & {\color{blue} 0 } & {\color{red} 0 }
\end{array}
\end{equation*}
As $H_n\ZNr=0$ for $n\geq 0$, most $E^{\infty}$-objects must vanish as indicated here.
\begin{equation*}
\begin{array}{c|ccccc}
0 & \vdots  & \vdots & \vdots & \SSObjct{\infty}{p}{q} \\
2 & 0 & 0 & \ddots & \cdots \\
1 & ? & 0 & 0 & \cdots \\
0 & \ZNr & \ZMod{k} & 0 & \cdots  \\ \hline
  & 0 & 1 & 2 & \cdots & p \\ \hline
  & \ZNr & \ZNr & 0 & \cdots
\end{array}
\end{equation*}
Now, let's put the pieces together: First of all:
\begin{equation*}
\ZNr \cong E^{2}_{0,0} \cong E^{\infty}_{0,0}\cong \GrpHmlgyDimOf{0}{\CyclcGrp{k}}
\end{equation*}
Next, we have $\ZMod{k}\cong \SSObjct{2}{1}{0}\cong \SSObjct{\infty}{1}{0}\cong \GrpHmlgyDimOf{1}{\CyclcGrp{k}}$ because $\SSDffrntlAt{2}{1}{0}=0$. To see how this fact constrains the action of the differentials, consider the diagram below (Beginning with the short exact sequence on the right, already established facts are color coded blue. New facts, inferred using this diagram are color coded red.):
\begin{equation*}
\xymatrix@R=4ex@C=3em{
& {\color{blue} H_2\CyclcGrp{k}} & {\color{blue} \ZNr} & {\color{red} \ZNr} & {\color{blue} \ZNr} & {\color{blue} \ZMod{k}} \\
E^{\infty}_{2,0} \ar@{=}@[blue][d] \ar@{{ |>}->}[r] &
    E^{2}_{2,0} \ar[r]^-{d^{2}_{2,0}}_-{\color{red} 0}  \ar@{=}@[red][d] \ar@{=}@[blue][u] &
    E^{2}_{0,1} \ar@{=}@[blue][u] \ar@{-{ >>}}[d] &
    F_{0,1} \ar@{=}@[red][u] \ar@{=}[d] \ar@{{ |>}->}[r]^-{\color{red} \times k} &
    F_{1,0} \ar@{=}@[blue][d] \ar@{=}@[blue][u] \ar@{-{ >>}}[r] &
    E^{\infty}_{1,0} \ar@{=}@[blue][d] \ar@{=}@[blue][u] \\
{\color{blue} 0} &
    {\color{red} 0} &
    E^{3}_{0,1} \ar@{=}[r] &
    E^{\infty}_{0,1} &
    {\color{blue} H_1\ZNr} &
    {\color{blue} E^{2}_{1,0}}
}
\end{equation*}
Combining that $F_{1,0}=H_1\ZNr = \ZNr$ and that $F_{0,1}$ is a quotient of $\ZNr$ which must map in $F_{1,0}$ with cokernel $\ZMod{k}$ forces $d^{2}_{2,0}=0$ and, hence, $E^{\infty}_{2,0}=\Ker{d^{2}_{2,0}}=0$. But then $H_2\CyclcGrp{k}=E^{2}_{2,0}=0$.
It follows that  $E^{2}_{2,1}\cong E^{2}_{2,0}=0$. To compute  $\GrpHmlgyDimOf{3}{\ZMod{k}}$, consider the exact sequence
\begin{equation*}
\xymatrix@R=5ex@C=3em{
{\color{blue} 0}=\SSObjct{\infty}{3}{0} \ar@{{ |>}->}[r] &
    \SSObjct{2}{3}{0}=\GrpHmlgyDimOf{3}{\ZMod{k}} \ar[r]^-{d^{2}_{3,0}}_-{\color{red} \cong} &
    \SSObjct{2}{1}{1}=\ZMod{k} \ar@{-{ >>}}[r] &
    \SSObjct{\infty}{1}{1}=0	}
\end{equation*}
We conclude that $d^{2}_{3,0}$ is an isomorphism and, hence,  $\GrpHmlgyDimOf{3}{\ZMod{k}}\cong \ZMod{k}$. This implies:  $\SSObjct{2}{3}{1}=\SSObjct{2}{3}{0}\cong \ZMod{k}$. To compute $\GrpHmlgyDimOf{4}{\ZMod{k}}$, consider the exact sequence
\begin{equation*}
\xymatrix@R=5ex@C=3em{
{\color{blue} 0}=\SSObjct{\infty}{4}{0} \ar@{{ |>}->}[r] &
    \SSObjct{2}{4}{0}=\GrpHmlgyDimOf{4}{\ZMod{k}} \ar[r]^-{d^{2}_{4,0}}_-{\color{red} \cong} &
    \SSObjct{2}{2}{1}=0 \ar@{-{ >>}}[r] &
    \SSObjct{\infty}{2}{1}=0	}
\end{equation*}
We conclude $d^{2}_{4,0}$ is an isomorphism and, hence, $\GrpHmlgyDimOf{4}{\ZMod{k}}\cong 0$, along with $\SSObjct{2}{4}{1}=\SSObjct{2}{4}{0}\cong 0$. Continuing inductively in this fashion completes the proof of the theorem.
\end{proof}

\begin{theorem}[$5$-term exact sequence of a group extension]
\label{thm:GroupExtension-5-TermExactSequence}
Associated to a central group extension:
\begin{equation*}
\xymatrix@R=5ex@C=3em{
A \ar@{{ |>}->}[r]^-{i} &
    E \ar@{-{ >>}}[r]^-{\pi} &
    G
}
\end{equation*}
is the following exact sequence of homology groups. It is functorial with respect to morphisms of central extensions.
\begin{equation*}
\xymatrix@R=5ex@C=3em{
\GrpHmlgyDimOf{2}{E} \ar[r]^-{\pi_{\ast}} &
	\GrpHmlgyDimOf{2}{G} \ar[r]^-{\partial} &
	A \ar[r]^-{i_{\ast}} &
	\GrpHmlgyDimOf{1}{E} \ar@{-{ >>}}[r]^-{\pi_{\ast}} &
	\GrpHmlgyDimOf{1}{G}
}
\end{equation*}
\end{theorem}
\begin{proof}
We analyze the Lyndon-Hochschild-Serre spectral sequence for small $p$ and $q$.
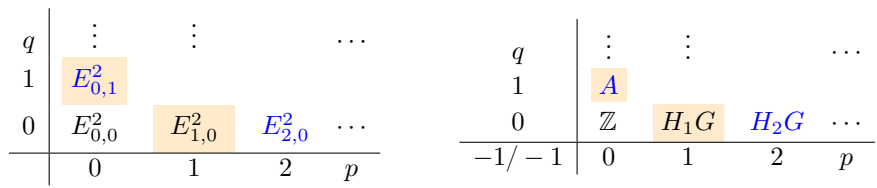
\begin{equation*}
\begin{array}{cc}
\begin{array}{c|cccc}
q & \vdots & \vdots &  & \cdots  \\
1 & \colorbox{blanchedalmond}{${\color{blue} \SSObjct{2}{0}{1} }$} & & & \\
0 & \SSObjct{2}{0}{0} & \colorbox{blanchedalmond}{ $\SSObjct{2}{1}{0}$ } & {\color{blue} \SSObjct{2}{2}{0} } & \cdots \\ \hline
  & 0 & 1 & 2 & p
\end{array} &\qquad
\begin{array}{c|cccc}
q & \vdots & \vdots &  & \cdots  \\
1 & \colorbox{blanchedalmond}{${\color{blue} A}$} & & & \\
0 & \ZNr & \colorbox{blanchedalmond}{$\GrpHmlgyDimOf{1}{G}$} & {\color{blue} \GrpHmlgyDimOf{2}{G}} & \cdots \\ \hline
-1/-1 & 0 & 1 & 2 & p
\end{array}
\end{array}
\end{equation*}
Now work  through the diagram below, starting from the right:
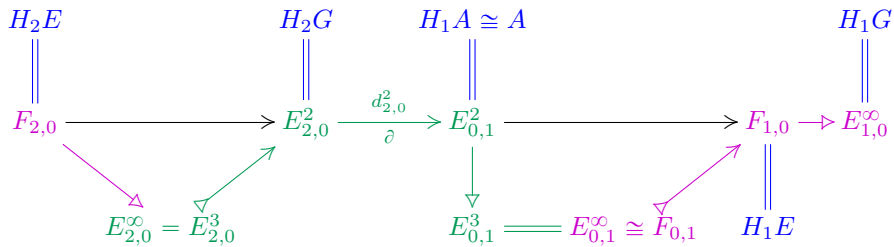
\begin{equation*}
\xymatrix@R=5ex@C=1em{
{\color{blue} \GrpHmlgyDimOf{2}{E}} \ar@[blue]@{=}[d] &
    & {\color{blue} \GrpHmlgyDimOf{2}{G} } \ar@[blue]@{=}[d] &&
	{\color{blue} \GrpHmlgyDimOf{1}{A}\cong A} \ar@[blue]@{=}[d] &&&
   {\color{blue} \GrpHmlgyDimOf{1}{G}} \ar@[blue]@{=}[d] \\
{\color{deepmagenta} F_{2,0}} \ar[rr] \ar@[deepmagenta]@{-{ >>}}[rd] &&
	{\color{ForestGreen}  \SSObjct{2}{2}{0} } \ar@[ForestGreen][rr]^-{\color{ForestGreen} \SSDffrntlAt{2}{2}{0}}_-{\color{ForestGreen} \partial} &&
	{\color{ForestGreen}  \SSObjct{2}{0}{1} } \ar[rr] \ar@[ForestGreen]@{-{ >>}}[d] &&
	{\color{deepmagenta} F_{1,0}} \ar@[deepmagenta]@{-{ >>}}[r] &
	{\color{deepmagenta} \SSObjct{\infty}{1}{0}} \\
& {\color{ForestGreen} \SSObjct{\infty}{2}{0}=\SSObjct{3}{2}{0}} \ar@[ForestGreen]@{{ |>}->}[ru] &&&
    {\color{ForestGreen}  \SSObjct{3}{0}{1}} \ar@[ForestGreen]@{=}[r]  &
	{\color{deepmagenta} \SSObjct{\infty}{0}{1} \cong F_{0,1}} \ar@[deepmagenta]@{{ |>}->}[ru] &
    {\color{blue} \GrpHmlgyDimOf{1}{E}} \ar@[blue]@{=}[u]
}
\end{equation*}
Using the indicated image factorizations, the diagram of horizontal arrows is seen to be exact, and that's the stated $5$-term exact sequence.
\end{proof}

A parent of the Lyndon-Hochschild-Serre spectral sequence in group homology is the Serre spectral sequence associated to a Hurewicz fibration of topological spaces. One of its most convenient variants is:

\begin{theorem}[Serre homological spectral sequence]
\label{thm:SerreSS-Homology}
Let $F\to E\to B$ be a (Hurewicz) fibration in which $B$ is simply connected. Then the singular homology of $E$ has a filtration
\begin{equation*}
0=F_{-1,n+1}\subseteq F_{0,n}\subseteq \cdots \subseteq F_{p-1,q+1}\subseteq F_{p,q}\subseteq \cdots \subseteq F_{n,0}=H_nE
\end{equation*}
There is a spectral sequence $(E^r,d^r)$, $r\geq 2$ whose differential have bidegree of $d^r$  is $(-r,r-1)$. It satisfies 
\begin{equation*}
E^{2}_{p,q} = H_p(B;H_qF)  \qquad \text{and}\qquad  \xymatrix{
F_{p-1,q+1} \ar@{{ |>}->}[r] &
    F_{p,q} \ar@{-{ >>}}[r] &
    E^{\infty}_{p,q} }
\end{equation*}
All of these data are functorial with respect to morphisms of central extensions. \NoProof
\end{theorem}

In the following, $\uSphr{k}$ denotes the unit sphere in $\RNrSpc{k+1}$, and $\CPrSpc{r}$ denotes complex projective space of complex dimension $r$. Adapting the reasoning of the previous two examples, we encourage the reader to prove:

\begin{theorem}[Homology of complex projective spaces]
\label{thm:H_*CP}
For each $r\geq 1$, there is a fibration $\uSphr{1}\to \uSphr{2r+1}\to \CPrSpc{r}$. It enables the computation:
\begin{equation*}
H_n\CPrSpc{r} \cong \left\{
\begin{array}{rcl}
\ZNr & \text{if} & 0\leq n\leq 2r\ \text{is even} \\
0 & \text{else} & 
\end{array}\right.
\end{equation*}
Further, there is a fibration $\uSphr{1}\to \uSphr{\infty}\to \CPrSpc{\infty}$. Using that $\uSphr{\infty}$ is contractible, this enables the computation:
\begin{equation*}
H_n\CPrSpc{\infty} \cong \left\{
\begin{array}{rcl}
\ZNr & \text{if} & 0\leq n\ \text{is even} \\
0 & \text{else} & 
\end{array}\right.
\end{equation*}
\end{theorem}
%%
%%%%%%%%%%%%%%%%%%%%%%%%%%%%%%%%%%%%%%%%%
\chapter{Exact Couples and their Spectral Sequences}
\label{chap:ExactCouples}

{\bfseries Introduction}\quad We owe to  W.S. Massey the ingenious insight which linked earlier work of J.H.C. Whitehead and Chern-Spanier to Leray's approach to spectral sequences; see \cite[p.~7, FN11]{HMiller2000}. This link consists of a new algebraic object, called an exact couple; see \cite{WSMassey1952},\cite{WSMassey1953},\cite{WSMassey1954}. 

In a compressed view of an exact couple we see a triangle of morphisms of $R$-modules which is exact in each corner.
\begin{equation*}
\xymatrix@R=2em@C=1.6em{
\ExctCplDObjctBB{}{-} \ar[rr]^{\ExctCplIMap{}} &&
	\ExctCplDObjctBB{}{-} \ar[ld]^{\ExctCplJMap{}} \\
&	\ExctCplEObjctBB{}{-} \ar[lu]^{\ExctCplKMap{}}
}
\end{equation*}
As advocated by Eckmann-Hilton \cite[Sec.~I]{BEckmannPJHilton1966}, the simplicity of this view is most suitable to present a categorical description of `deriving the exact couple' and, thus, constructing an associated spectral sequence, along with its $E^{\infty}$-object. Their analysis results in the prediction \cite[p.~40]{BEckmannPJHilton1966} that gradation is likely to prove very important in establishing the meaning of the $E^{\infty}$-object.

One decisive step in this direction was taken by Boardman \cite{JMBoardman1999}. He `unrolled' the exact couple into a diagram of $\ZNr$-graded modules:
\begin{equation*}
\xymatrix@R=5ex@C=1em{
\cdots \ar[r] & 
	A^{s+2} \ar[rr]^-{i} &&
	A^{s+1} \ar[rr]^-{i} \ar[ld]^-{j} &&
	A^{s} \ar[rr]^-{i} \ar[ld]^{j} &&
	A^{s-1} \ar[r] \ar[ld]^{j} &
	\cdots \\
&& E^{s+1} \ar[lu]^{k} &&
	E^{s} \ar[lu]^{k} &&
	E^{s} \ar[lu]^{k} &&
}
\end{equation*}
in which each triangle $\to A^{s+1} \to A^s \to E^s\to A^{s+1}\to$ is a long exact sequence. The resulting convergence discussion is very refined and fully adequate for many applications of spectral sequences.

On the other hand, most spectral sequences are bigraded, and come from a bigraded exact couple. In such a setting the maps $i$ and $kj$ may have linearly independent bigradings, and the full effect of that is not easily visible if the exact couple is merely unrolled. If we follow the prescient remark in \cite[p.~40]{BEckmannPJHilton1966} and flatten the exact couple into a plane spanned by $i$ and $kj$, we are able to make further progress toward establishing the meaning of the $E^{\infty}$-objects.

Accordingly, we work with exact couples which are bigraded over $\Prdct{\ZNr}{\ZNr}$. So, the morphisms $\ExctCplIMap{},\ExctCplJMap{},\ExctCplKMap{}$ have bidegrees $\ExctCplIBdg{a}=(a_1,a_2)$, $\ExctCplJBdg{b}=(b_1,b_2)$, $\ExctCplKBdg{c}=(c_1,c_2)$, respectively. If these bidegrees satisfy the regularity condition (\ref{def:(ZxZ)-BigradedExactCouple}), such a bigraded exact couple flattens to a planar diagram, a portion of which is presented below. 

\newpage
A bigraded  regular exact couple presented as a flattened diagram of interlocking long exact sequences, arranged in a staircase: %
\label{fig:ExactCouple-ZxZBigraded}%
\begin{equation*}
\label{fig:ExactCouple-Flattened}%
\resizebox{1\textwidth}{!}{$
\begin{xy}
\xymatrix@R=3em@C=2em{
& {\color{deepmagenta} \ExctCplLimAbut{n}} \ar@[deepmagenta]@{.>}[d]_-{\color{deepmagenta} \ExctCplLimAbutMapBB{\Vect{x} - r\Vect{a}} } & &
  {\color{deepmagenta} \ExctCplLimAbut{n+\sigma}} \ar@[deepmagenta]@{.>}[d]^-{\color{deepmagenta} \ExctCplLimAbutMapBB{\Vect{x} +\Vect{b}+\Vect{c}-r\Vect{a}}} &
     & \\
% & & & & \\
\ExctCplEObjctBB{}{\Vect{x} - \Vect{c} - r\Vect{a} } \ar[r] &
     \ExctCplDObjctBB{}{\Vect{x} -r\Vect{a}} \ar[d] \ar[r] &
     \ExctCplEObjctBB{}{\Vect{x}+\Vect{b}  -r\Vect{a} } \ar[r] &
     \ExctCplDObjctBB{}{\Vect{x} +\Vect{b}+\Vect{c}-r\Vect{a}} \ar[d] \ar[r] &
     \ExctCplEObjctBB{}{\Vect{x} +2\Vect{b}+\Vect{c}-r\Vect{a}} \\
% & & & & \\
\ExctCplEObjctBB{}{\Vect{x} - \Vect{c} - (r-1)\Vect{a} } \ar[r] &
     \ExctCplDObjctBB{}{\Vect{x} -(r-1)\Vect{a} } \ar@{..>}[d] \ar[r] &
     \ExctCplEObjctBB{}{\Vect{x} + \Vect{b} - (r-1)\Vect{a} } \ar[r] &
     \ExctCplDObjctBB{}{\Vect{x} + \Vect{b}+\Vect{c} -(r-1)\Vect{a} } \ar@{..>}[d] \ar[r]^{j} &
     \ExctCplEObjctBB{}{\Vect{x} + 2\Vect{b}+\Vect{c} - (r-1)\Vect{a} } \\
% & \ar@{..>}[d] & & & \\
%
{\color{Maroon} \ExctCplEObjctBB{}{\Vect{x} - \Vect{c} - \Vect{a} } } \ar@[Maroon][r] &
     {\color{Maroon} \ExctCplDObjctBB{}{\Vect{x} -\Vect{a}} } \ar@[Maroon][d]_{\color{Maroon} i_{\Vect{x} - \Vect{a}}} \ar[r]^{j} &
     \ExctCplEObjctBB{}{\Vect{x} + \Vect{b} - \Vect{a} } \ar[r]^{k} &
     \ExctCplDObjctBB{}{\Vect{x} + \Vect{b}+\Vect{c} - \Vect{a} } \ar[r] \ar[d]^{i_{\Vect{x} +\Vect{b}+\Vect{c} -\Vect{a} }} &
     \ExctCplEObjctBB{}{\Vect{x}+ 2\Vect{b}+\Vect{c} -\Vect{a} } \\
% & & & & \\
%
\ExctCplEObjctBB{}{\Vect{x} -\Vect{c}} \ar@[blue]@<+2pt> `d[rr] `[rr]_(.3){\color{blue} d^{1}_{\Vect{x} - \Vect{c}}} \ar[r] &
	{\color{Maroon} \ExctCplDObjctBB{}{\Vect{x}} } \ar@[Maroon][r]^{\color{Maroon} j_{\Vect{x}}} \ar[d]|(.35)\hole^(.6){i_{\Vect{x}}} &
     {\color{Maroon} \ExctCplEObjctBB{}{\Vect{x} + \Vect{b}} } \ar@[blue]@<+2pt> `d[rr] `[rr]_(.75){\color{blue} d^{1}_{\Vect{x} +\Vect{b}}=\bigtriangleup^{1}_{\Vect{x} + \Vect{b}}} \ar@[Maroon][r]^{\color{Maroon} k_{\Vect{x}+\Vect{b}} } \ar@[red]@{.>}@/^25pt/[rruu]_(.65){\color{red} \bigtriangleup^{r}_{\Vect{x} +\Vect{b}}} &
     {\color{Maroon} \ExctCplDObjctBB{}{\Vect{x} + \Vect{b}+\Vect{c}} } \ar@[Maroon][d]|(.38)\hole \ar[r] &
     \ExctCplEObjctBB{}{\Vect{x}\,+2\Vect{b}+\Vect{c}}  \\
%& & & & \\
%
\ExctCplEObjctBB{}{\Vect{x} - \Vect{c} +\Vect{a} } \ar[r] &
     \ExctCplDObjctBB{}{\Vect{x}+\Vect{a}} \ar[r] \ar@{..>}[d] &
     \ExctCplEObjctBB{}{\Vect{x}+\Vect{b} +\Vect{a} } \ar[r] &
     {\color{Maroon} \ExctCplDObjctBB{}{\Vect{x} + \Vect{b}+\Vect{c} +\Vect{a} } } \ar@[Maroon][r] \ar@{..>}[d] &
     {\color{Maroon} \ExctCplEObjctBB{}{\Vect{x}+ 2\Vect{b}+\Vect{c} +\Vect{a} } } \\
% & & & \ar@{..>}[d] & \\
%
\ExctCplEObjctBB{}{\Vect{x}- \Vect{c} +(r-1)\Vect{a} } \ar[r] \ar@[red]@{.>}@/_25pt/[rruu]_(.55){\color{red} \bigtriangleup^{r}_{\Vect{x} - (r-1)\Vect{a}-\Vect{c}}} &
     \ExctCplDObjctBB{}{\Vect{x}+(r-1)\Vect{a}} \ar[d]^{i_{\Vect{x}+(r-1)\Vect{a}}} \ar[r] &
     \ExctCplEObjctBB{}{\Vect{x}+\Vect{b} +(r-1)\Vect{a} } \ar[r] &
     \ExctCplDObjctBB{}{\Vect{x}+\Vect{b}+\Vect{c} +(r-1)\Vect{a} } \ar[d] \ar[r] &
     \ExctCplEObjctBB{}{\Vect{x}+ 2\Vect{b}+\Vect{c} +(r-1)\Vect{a} } \\
% & & & & \\
%
\ExctCplEObjctBB{}{\Vect{x} - \Vect{c} +r\Vect{a} } \ar[r] &
     \ExctCplDObjctBB{}{\Vect{x}+r\Vect{a}} \ar[r] \ar@[deepmagenta]@{.>}[d]_{\color{deepmagenta} \ExctCplCoLimAbutMapBB{\Vect{x}+r\Vect{a}}} &
     \ExctCplEObjctBB{}{\Vect{x}+\Vect{b} +r\Vect{a} } \ar[r] &
     \ExctCplDObjctBB{}{\Vect{x} + \Vect{b}+\Vect{c} +r\Vect{a} } \ar[r] \ar@[deepmagenta]@{.>}[d]^-{\color{deepmagenta} \ExctCplCoLimAbutMapBB{\Vect{x}+ \Vect{b} + \Vect{c} + r\Vect{a}}} &
     \ExctCplEObjctBB{}{\Vect{x}+2\Vect{b}+\Vect{c} +r\Vect{a} }  \\
% & {\ } & & {\ }    
% & & & & \\
& {\color{deepmagenta} \ExctCplCoLimAbut{n}} &&
  {\color{deepmagenta} \ExctCplCoLimAbut{n+\sigma}} &
}
\end{xy}
$}
\end{equation*}

Above, $n\DefEq \DtrmnntOfMtrx{[\ExctCplIBdg{a}\ \ \Vect{x}]}\in \ZNr$. - Associated to such an exact couple are the two $\ZNr$-graded filtered objects:
\begin{enumerate}[$\bullet$]
\item $\ExctCplCoLimAbut{n} \DefEq \CoLimOfOver{ \ExctCplDObjctBB{}{\Vect{x}+r\Vect{a}} }{r\in\ZNr}$: \Defn{colimit abutment}, filtered by  $\ExctCplCoLimAbutFltrtnBB{\Vect{x}+r\Vect{a}}\DefEq\Img{\ExctCplCoLimAbutMapBB{\Vect{x}+r\Vect{a}} \from \ExctCplDObjctBB{}{\Vect{x}+r\Vect{a}}\to \ExctCplCoLimAbut{n} }$%
\index[not]{$\ExctCplCoLimAbut{\ast}$ - colimit abutment}%
\item $\ExctCplLimAbut{n} \DefEq \LimOfOver{ \ExctCplDObjctBB{}{\Vect{x}+r\Vect{a}} }{r\in\ZNr}$: \Defn{limit abutment}, filtered by  $\ExctCplLimAbutFltrtnBB{\Vect{x}+r\Vect{a}}\DefEq\Ker{ \ExctCplLimAbutMapBB{\Vect{x}+r\Vect{a}}\from \ExctCplLimAbut{n} \to \ExctCplDObjctBB{}{\Vect{x}+r\Vect{a}}}$%
\index[not]{$\ExctCplLimAbut{\ast}$ - limit abutment}%
\end{enumerate}
In Section \ref{sec:SS-From-EC} we explain how a spectral sequence $(\ExctCplEPage{r},\SSDffrntl{r})$ is derived from an exact couple. This spectral sequence has a limit page $\SSPage{\infty}$ which may be computed as is explained in Section \ref{sec:SpecSeq-E-infinity}. The $\SSObjctBB{\infty}{}$-objects of the spectral sequence are related to the quotients of adjacent filtration stages of the limit and colimit abutment objects in a non-obvious way. We describe this relationship completely from two mutually complementary points of view:
\begin{enumerate}[(A)]
\item the E-infinity extension theorem (\ref{thm:E-InfinityExtensionThm}): It is based on spectral sequence computations which are directly mapped into constructions on the underlying exact couple, and are carried through the relevant (co-)limit constructions.
\item the stable E-infinity extension theorem (\ref{thm:Stable-E-Extension}): It is based upon the observation that the recursive passage between the pages of the spectral sequence of an exact couple admits transfinite continuation. By its nature, this transfinite continuation is beyond the scope of classical spectral sequence computations. It results in a stable E-infinity object whose meaning is beautifully simple.
\end{enumerate}

We digress briefly to explain how the E-infinity extension theorem interacts with what are called convergence results for spectral sequences. In many classical spectral sequence constructions, we start from a graded set $X_n$ of $R$-modules about which we want information; e.g. homotopy, (co-)homology groups of an object of interest. The $X_n$ are canonically filtered in a manner which yields an exact couple and, with it, a spectral sequence. The objects $X_n$ of original interest act as the \Defn{abutment} for the spectral sequence, and one hopes that the E-infinity objects of the spectral sequence match adjacent filtration quotients of the filtered abutment.

In terms of the flattened exact couple above, we always find ourselves in one of the following two situations:
\begin{enumerate}[(1)]
\item Either each $X_n$ is the terminal vertex of a cocone from an appropriate $D$-column of the exact couple, 
\item or each $X_n$ is the initial vertex of a cone into an appropriate $D$-column of the exact couple.
\end{enumerate}
If each $X_n$ is the terminal vertex of such a cocone, then there is a unique comparison map from the appropriate colimit abutting object into the abutting object $X_n$. Conversely, if each $X_n$ is the initial vertex of such a cone, then there is a unique comparison map from the abutting object $X_n$ into the appropriate limit abutting object of the exact couple.

From this perspective classical convergence results, notably Boardman's in \cite{JMBoardman1999}, see also \cite[Sec. 5.5]{CAWeibel1994}, provide sufficient conditions under which (a) the abutment of the spectral sequence equals either the exact couple's limit abutment or its colimit abutment, and (b) E-infinity objects from the spectral sequence are isomorphic to appropriate adjacent filtration quotients of the objects $X_n$.

Thus, our approach, differs from classical developments in that we {\em do not} start from a given graded set of filtered objects. Rather, we start from an exact couple and explain the relationship between its E-infinity objects and the filtrations of the universal limit/colimit abutments. If this exact couple is constructed from a given abutment, then we respond with these two steps:
\begin{enumerate}[(S.1)]
\item We investigate the relationship between the given abutment and whichever one of the universal limit/colimit abutments is relevant.
\item We combine the outcome of the investigation in (A) with what the E-infinity extension theorem tells us about the meaning of the  E-infinity objects of the spectral sequence.
\end{enumerate}
The homotopy spectral sequence of a tower of fibrations provides a nice example in which the steps in this approach can be observed; see Example \ref{exa:HomotopySS-Fibrations}.

Thus many classical convergence results can be seen as applications of the E-infinity extension theorems; see Section \ref{sec:ConvergenceII}. To demonstrate the utility of the approach taken here, we identify a class of spectral sequences which do not converge in any classical sense, yet they support meaningful comparison results. We close this chapter with Zeeman style (reverse) comparison results; see \ref{sec:SpecSec-Comparison-II}.

\section[Exact Couples]{Exact Couples}
\label{sec:ExactCouples}

From Massey \cite{WSMassey1952} we recall the notion of a $(\Prdct{\ZNr}{\ZNr})$-bigraded  exact couple (\ref{def:(ZxZ)-BigradedExactCouple}). Such an exact couple has an associated spectral sequence and our ultimate objective is to clarify the meaning of its objects $\ExctCplEObjctBB{r}{\Vect{x}}$ for all $r$. For reasons explained in the introduction to this chapter, we specialize here to $(\ZNr\prdct\ZNr)$-bigraded exact couples which flatten out as shown in diagram (\ref{fig:ExactCouple-ZxZBigraded}). This is achieved by imposing a `regularity condition' upon the bidegrees of its structure maps $i,j,k$.

We observe that exact couples whose structure maps $\ExctCplIMap{},\ExctCplJMap{},\ExctCplKMap{}$ have matching bidegrees form  category, with morphisms as defined in (\ref{def:ExactCouplesMorphism}). This category is additive, and admits arbitrary sums and products; see (\ref{thm:ExactCouplesAdditiveCat}). It fails to be an abelian category as few morphisms admit kernels and cokernels.

Every exact couple has an associated spectral sequence. We present two ways of constructing this spectral sequence. First, via the recursive process of exact couple derivation (\ref{thm:DerivedExactCouples}) and, second, via the construction internal to the exact couple described in Theorem \ref{thm:SpecSequFromExactCouple}. Both approaches are classic well document; see for example \cite{WSMassey1952} and \cite[XI.5]{SMacLane1975-Homology}.

The recursive process of exact couple derivation is formally elegant. However, for the most part, we work with the internal exact couple construction of the spectral sequence objects. New here is that this internal construction extends {\em transfinitely} to arbitrary ordinals. We show that it stabilizes at a sufficiently large ordinal $\alpha$ so as to yield objects $\ExctCplStblEObjct\DefEq \ExctCplEObjctBB{\alpha}{}$; see (\ref{thm:E^tau,tau>omega}). Examples show that the stabilizing ordinal $\alpha$ can be arbitrarily large. The stable E-objects $\ExctCplStblEObjct$ are always subobjects of $\ExctCplEPage{\infty} = \ExctCplEPage{\omega+1}$.

We will see in the next section (\ref{sec:Convergence-I}) on `convergence'  that the stable objects $\ExctCplStblEObjct$ are what one would actually want from the spectral sequence. With this hindsight perspective, all classical convergence considerations involve the search for conditions under which stabilizations happens at the ordinal $\omega$. Accordingly, in (\ref{thm:SpecSeqs-(Co-)Mittag-Leffler}) we tackle more generally conditions under which $\ExctCplStblEObjct=\ExctCplEPage{\infty}$.

We close this section with a brief discussion on eligible bidegrees of the structure maps of a regular exact couple. Even though regular exact couples impose considerable constraints upon those bidegrees, there is still a wide range of possibilities. However, via the canonical $\GLGrp{2}{\ZNr}$-action on $\Prdct{\ZNr}{\ZNr}$, it is possible to re-index any regular exact couple so that its associated spectral sequence is homological or, alternatively, cohomological; see (\ref{exa:ExactCouple-ReindexTo(Co-)Homological}).

\begin{definition}[$(\ZNr\prdct \ZNr)$-bigraded exact couple]
\label{def:(ZxZ)-BigradedExactCouple}%
A $(\ZNr\prdct \ZNr)$-\Defn{bigraded exact couple} $\ExctCpl{C}=\left(\ExctCplEPage{}, \ExctCplDPage{},\ExctCplIMap{},\ExctCplJMap{},\ExctCplKMap{}\right)$ in $\ModulesOver{R}$ consists of %
\index{exact couple}%
\vspace{-1.7ex}
\begin{enumerate}[({EC-}1)\ ]
\setlength{\itemindent}{6ex}
\item A $(\Prdct{\ZNr}{\ZNr})$-bigraded module $\ExctCplDObjctBB{}{-}$,
\item A $(\Prdct{\ZNr}{\ZNr})$-bigraded module $\ExctCplEObjctBB{}{-}$,
\item A morphism $\ExctCplIMap{}\from \ExctCplDObjctBB{}{-}\to \ExctCplDObjctBB{}{-}$ of bidegree $\ExctCplIBdg{a}=(a_1,a_2)\in \Prdct{\ZNr}{\ZNr}$. %
\index[not]{$\ExctCplIBdg{a}$ - bidegree of EC-map $i$}
\item A morphism $\ExctCplJMap{}\from \ExctCplDObjctBB{}{-}\to \ExctCplEObjctBB{}{-}$ of bidegree $\ExctCplJBdg{b}=(b_1,b_2)\in \Prdct{\ZNr}{\ZNr}$. %
\index[not]{$\ExctCplJBdg{b}$ - bidegree of EC-map $j$}
\item A morphism $\ExctCplKMap{}\from \ExctCplEObjctBB{}{-}\to \ExctCplDObjctBB{}{-}$ of bidegree $\ExctCplKBdg{c}=(c_1,c_2)\in \Prdct{\ZNr}{\ZNr}$. %
\index[not]{$\ExctCplKBdg{c}$ - bidegree of EC-map $k$}
\end{enumerate}
The \Defn{structure maps} $\ExctCplIMap{},\ExctCplJMap{},\ExctCplKMap{}$ are to form a triangle which is exact in each corner:
\begin{equation*}
\begin{xy}
\xymatrix@R=2em@C=1.6em{
\ExctCplDObjctBB{}{-} \ar[rr]^{\ExctCplIMap{}} &&
	\ExctCplDObjctBB{}{-} \ar[ld]^{\ExctCplJMap{}} \\
&	\ExctCplEObjctBB{}{-} \ar[lu]^{\ExctCplKMap{}}
}
\end{xy}
\end{equation*}
We say that $\ExctCpl{C}$ is \Defn{regular} if the bidegrees of its structure maps satisfy $\sigma\DefEq \DtrmnntOfMtrx{[\ExctCplIBdg{a}\ \ (\ExctCplJBdg{b}+\ExctCplKBdg{c})]} \in \Set{\pm 1}$. %
\index{regular exact couple}\index{exact couple!regular}\index{exact couple!structure maps}\index[not]{$\sigma=\DtrmnntOfMtrx{[\ExctCplIBdg{a}\ \ (\ExctCplJBdg{b}+\ExctCplKBdg{c})]}$}
\end{definition}

\begin{subordinate}
\begin{remark}[Purpose of the constraint $\sigma\DefEq \DtrmnntOfMtrx{[\ExctCplIBdg{a}\ \ (\Vect{b}+\Vect{c})]} \in \Set{\pm 1}$.]
\label{rem:Sigma-Constraint}%
In \cite[Sec.~6]{WSMassey1952}, Massey briefly addresses $(\ZNr\prdct \ZNr)$-bigraded exact couples. There, he does not impose any constraints on the bidegrees of their structure maps. Via (\ref{thm:BigradingPropertiesOfExactCouple}), the regularity condition ensures that the exact couple is connected via its structure maps and, furthermore, that it flattens out to a diagram like the one displayed in (\ref{fig:ExactCouple-ZxZBigraded}).

To see the potential differences between an arbitrary $(\Prdct{\ZNr}{\ZNr})$-bigraded exact couple and a regular one, the reader may contemplate an exact couple all of whose structure maps have bidegree $(0,0)$: Such an exact couple decomposes into a disjoint union of a $(\Prdct{\ZNr}{\ZNr})$-indexed family of exact couples.
\end{remark}
\end{subordinate}

\begin{lemma}[Bigrading properties of an exact couple]
\label{thm:BigradingPropertiesOfExactCouple}%
For an exact couple $\ExctCpl{C}$ whose structure maps have bidegrees $\ExctCplIBdg{a},\ExctCplJBdg{b}, \ExctCplKBdg{c}$ satisfying $\DtrmnntOfMtrx{[\ExctCplIBdg{a}\ \ (\ExctCplJBdg{b}+\ExctCplKBdg{c})]} = \xi\in \ZNr$ the following hold:
\begin{enumerate}[(i)]
\item \label{thm:BigradingPropertiesOfExactCouple-UniqueConnect}%
For arbitrary $\Vect{x},\Vect{x}'\in \ZNr\prdct \ZNr$, there exist unique $\alpha,\beta\in\ZNr$ with
\begin{equation*}
\Vect{x}' = \Vect{x} + \alpha \ExctCplIBdg{a} + \beta(\ExctCplJBdg{b}+\ExctCplKBdg{c})
\end{equation*}
if and only if $\xi\in\Set{\pm 1}$.
\item \label{thm:BigradingPropertiesOfExactCouple-UniqueZDiagram}%
Suppose $\xi\in \Set{\pm 1}$. Given $n\in\ZNr$, put $\Vect{x}(n)\DefEq \xi n (\ExctCplJBdg{b}+\ExctCplKBdg{c})$. Then $\Vect{x}\in \Prdct{\ZNr}{\ZNr}$ satisfies $\DtrmnntOfMtrx{[\ExctCplIBdg{a}\ \ \Vect{x}]} = n$ if and only if there exists $r\in \ZNr$ with %
\index[not]{$\Vect{x}(n) \DefEq \xi n (\ExctCplJBdg{b}+\ExctCplKBdg{c})$}%
\begin{equation*}
\Vect{x} = \Vect{x}(n) + r\ExctCplIBdg{a}
\end{equation*}
\end{enumerate}
\end{lemma}
\begin{proof}
An elementary computation.
\end{proof}

\begin{corollary}[$\ZCat$-diagrams of a regular exact couple]
\label{thm:ExactCouple-ZGradedZ-Diagrams}%
Let $\ExctCpl{C}$ be a regular exact couple whose structure maps have bidegrees $\ExctCplIBdg{a},\ExctCplJBdg{b},\ExctCplKBdg{c}$. Then every $n\in \ZNr$ determines a $\ZCat$-diagram
\begin{equation*}
D(n)\qquad \qquad \cdots \longrightarrow \ExctCplDObjctBB{}{\Vect{x}(n) - \ExctCplIBdg{a}} \XRA{ i } \ExctCplDObjctBB{}{\Vect{x}(n) } \XRA{ i } \ExctCplDObjctBB{}{\Vect{x}(n) + \ExctCplIBdg{a}} \longrightarrow \cdots
\end{equation*}
Moreover, $\Vect{x}(n')$ belongs to $D(n)$ if and only if $n=n'$. \NoProof%
\index[not]{$D(n)$ - $n$-th $\ZCat$-diagram of regular exact couple}%
\end{corollary}

Next, we state basic categorical features of exact couples. 

\begin{definition}[Morphism of exact couples]
\label{def:ExactCouplesMorphism}%
Consider exact couples $\ExctCpl{C}(1)=\left(\ExctCplEPage{}(1), \ExctCplDPage{}(1),\ExctCplIMap{}(1),\ExctCplJMap{}(1),\ExctCplKMap{}(1)\right)$ and $\ExctCpl{C}(2)=\left(\ExctCplEPage{}(2), \ExctCplDPage{}(2),\ExctCplIMap{}(2),\ExctCplJMap{}(2),\ExctCplKMap{}(2)\right)$ in $\ModulesOver{R}$ whose structure maps have matching bidegrees: $\ExctCplIBdg{a}(1)=\ExctCplIBdg{a}(2)$, $\ExctCplJBdg{b}(1)=\ExctCplJBdg{b}(2)$, and $\ExctCplKBdg{c}(1)=\ExctCplKBdg{c}(2)$. A \Defn{morphism $(f,g)\from \ExctCpl{C}(1)\to \ExctCpl{C}(2)$ of exact couples} is given by a family of $R$-module maps  %
\index{morphism!of exact couples}\index{exact couple!morphism}
\begin{equation*}
f_{\Vect{x}}\from \ExctCplDObjctBB{}{\Vect{x}}(1) \longrightarrow \ExctCplDObjctBB{}{\Vect{x}}(2)\qquad \text{and}\qquad g_{\Vect{x}}\from \ExctCplEObjctBB{}{\Vect{x}}(1)\longrightarrow \ExctCplEObjctBB{}{\Vect{x}}(2)
\end{equation*}
which render the resulting diagram of exact couples commutative.
\end{definition}

Thus the $(\Prdct{\ZNr}{\ZNr})$-bigraded exact couples in $\ModulesOver{R}$ form a category $\ExctCplsCatModulesOver{R}$ in which each triple of bidegrees $(\ExctCplIBdg{a},\ExctCplJBdg{b},\ExctCplKBdg{c})$ determines a connected component. Given a triple of bidegrees $(\ExctCplIBdg{a},\ExctCplJBdg{b},\ExctCplKBdg{c})$, let  $\ExctCplsCatModulesOver{R}(\ExctCplIBdg{a},\ExctCplJBdg{b},\ExctCplKBdg{c})$ denote the full subcategory of $\ExctCplsCatModulesOver{R}$ of all those exact couples whose structure maps $i,j,k$ have bidegrees $(\ExctCplIBdg{a},\ExctCplJBdg{b},\ExctCplKBdg{c})$. %
\index{$\ExctCplsCatModulesOver{R}$ - cat of exact couples in $\ModulesOver{R}$}%
\index{$\ExctCplsCatModulesOver{R}(\ExctCplIBdg{a},\ExctCplJBdg{b},\ExctCplKBdg{c})$ - EC's in $\ModulesOver{R}$ with structure map bidegrees $\ExctCplIBdg{a},\ExctCplJBdg{b},\ExctCplKBdg{c}$}

\begin{proposition}[Exact couples form an additive category]
\label{thm:ExactCouplesAdditiveCat}%
The category $\ExctCplsCatModulesOver{R}(\ExctCplIBdg{a},\ExctCplJBdg{b},\ExctCplKBdg{c})$ is additive with objectwise defined addition of morphisms. Moreover, it has arbitrary sums and products which are defined objectwise. \NoProof
\end{proposition}

\section{Spectral Sequence From an Exact Couple}
\label{sec:SS-From-EC}%

Let us now turn to the process of extracting a spectral sequence from an exact couple $\ExctCpl{C}$. The key is that its $(\ZNr\prdct\ZNr)$-graded $R$-module of $E$-objects is canonically equipped with the differential $d\DefEq j\Comp k$. The homology of this differential yields the next page in the spectral sequence. There are two ways to compute subsequent pages of the spectral sequence:
\begin{enumerate}[(1)]
\item {\em Recursive derivation of $\ExctCpl{C}$:} Every exact couple $\ExctCpl{C}$ admits two constructions called 'derivation' which result in a new exact couple $\ExctCpl{C}'$. Iterating either derivation yields a sequence of $\ExctCplEObjctBB{-}{}$-objects with associated differentials. This is the spectral sequence associated to the exact couple. The essence of this construction is stated as Proposition (\ref{thm:DerivedExactCouples}).
\item {\em Direct computation:} As outlined in \cite[XI.5]{SMacLane1975-Homology}, \cite{GWWhitehead1978}, it is possible to obtain the recursively constructed spectral sequence from (1) above, by working `internal' to a given exact couple, with additive relations. This approach delivers the desired spectral sequence, but the underlying construction may be extended {\em transfinitely}. Beyond the desired spectral sequence, we obtain a new reference point for the meaning of its E-infinity objects; see the combination of the stable E-extension theorem \ref{thm:Stable-E-Extension} and the E-infinity extension theorem \ref{thm:E-InfinityExtensionThm}.
\end{enumerate}

\begin{proposition}[Derived exact couples]
\label{thm:DerivedExactCouples}%
An exact couple $\ExctCpl{C}=\left(\ExctCplEPage{}, \ExctCplDPage{},\ExctCplIMap{},\ExctCplJMap{},\ExctCplKMap{}\right)$ whose structure maps have bidegrees $\ExctCplIBdg{a},\ExctCplJBdg{b},\ExctCplKBdg{c}$, yields two new exact couples denoted $Q\ExctCpl{C}$ and $I\ExctCpl{C}$ defined as follows:
\begin{center}
\renewcommand{\arraystretch}{2.1}
\begin{tabular}{|l|l|}
\hline 
$Q\ExctCpl{C}=\left(Q\ExctCplEPage{}, Q\ExctCplDPage{}, Q\ExctCplIMap{}, Q\ExctCplJMap{}, Q\ExctCplKMap{}\right)$ & 
	$I\ExctCpl{C}=\left(I\ExctCplEPage{}, I\ExctCplDPage{}, I\ExctCplIMap{}, I\ExctCplJMap{}, I\ExctCplKMap{}\right)$ \\ 
\hline 
$Q\ExctCplEObjctBB{}{\Vect{x}+\ExctCplJBdg{b}} \DefEq \dfrac{\Ker{ \ExctCplJMapBB{}{\Vect{x}+\ExctCplJBdg{b}+\ExctCplKBdg{c} }\Comp \ExctCplKMapBB{}{\Vect{x}+\ExctCplJBdg{b} } } }{ \Img{ \ExctCplJMapBB{}{\Vect{x}} \Comp \ExctCplKMapBB{}{\Vect{x}-\ExctCplKBdg{c} } } }$ & $I\ExctCplEObjctBB{}{\Vect{x}+\ExctCplJBdg{b}} \DefEq \dfrac{\Ker{ \ExctCplJMapBB{}{\Vect{x}+\ExctCplJBdg{b}+\ExctCplKBdg{c} }\Comp \ExctCplKMapBB{}{\Vect{x}+\ExctCplJBdg{b} } } }{ \Img{ \ExctCplJMapBB{}{\Vect{x}} \Comp \ExctCplKMapBB{}{\Vect{x}-\ExctCplKBdg{c} } } }$ \\
\hline
$Q\ExctCplDObjctBB{}{\Vect{x}}\DefEq \Img{ \ExctCplIMapBB{}{\Vect{x}}\from \ExctCplDObjctBB{}{\Vect{x}}\to \ExctCplDObjctBB{}{\Vect{x} + \ExctCplIBdg{a}}}$ &
 $I\ExctCplDObjctBB{}{\Vect{x} +\ExctCplIBdg{a} }\DefEq \Img{ \ExctCplIMapBB{}{\Vect{x}}\from \ExctCplDObjctBB{}{\Vect{x}}\to \ExctCplDObjctBB{}{\Vect{x} + \ExctCplIBdg{a}}}$ \\ 
\hline 
$Q\ExctCplIMapBB{}{\Vect{x}} \DefEq \ExctCplIMapBB{}{\Vect{x}+ \ExctCplIBdg{a} }|$ &
 $I\ExctCplIMapBB{}{\Vect{x}} \DefEq \ExctCplIMapBB{}{\Vect{x}+ \ExctCplIBdg{a} }|$ \\ 
\hline 
$Q\ExctCplJMapBB{}{\Vect{x}}( \ExctCplIMapBB{}{\Vect{x}}(t) ) \DefEq [\ExctCplJMapBB{}{\Vect{x}}(t)] \in{} Q\ExctCplEObjctBB{}{\Vect{x}+\ExctCplJBdg{b}}$ &
 $I\ExctCplJMapBB{}{\Vect{x}+\ExctCplIBdg{a} }( \ExctCplIMapBB{}{\Vect{x}}(t) ) \DefEq [\ExctCplJMapBB{}{\Vect{x}}(t)] \in I\ExctCplEObjctBB{}{\Vect{x}+\ExctCplJBdg{b}}$ \\ 
\hline 
$Q\ExctCplKMapBB{}{\Vect{x}+\ExctCplJBdg{b} }[z] \DefEq \ExctCplKMapBB{}{\Vect{x}+\ExctCplJBdg{b} }(z) \in Q\ExctCplDObjctBB{}{\Vect{x} +\ExctCplJBdg{b} + \ExctCplKBdg{c} - \ExctCplIBdg{a} }$ & 
	$I\ExctCplKMapBB{}{\Vect{x}+\ExctCplJBdg{b} }[z] \DefEq \ExctCplKMapBB{}{\Vect{x}+\ExctCplJBdg{b} }(z) \in I\ExctCplDObjctBB{}{\Vect{x}+\ExctCplJBdg{b} + \ExctCplKBdg{c} }$ \\ 
\hline  
\end{tabular} \renewcommand{\arraystretch}{1}
\end{center}
\smallskip
The structure maps of $Q\ExctCpl{C}$ and $I\ExctCpl{C}$ have bidegrees
\begin{equation*}
Q\ExctCplIBdg{a}= \ExctCplIBdg{a} = I\ExctCplIBdg{a},\quad Q\ExctCplJBdg{b}=\ExctCplJBdg{b}\ \ \text{and}\ \ I\ExctCplJBdg{b}=\ExctCplJBdg{b}- \ExctCplIBdg{a},\quad  Q\ExctCplKBdg{c} = \ExctCplJBdg{c}- \ExctCplIBdg{a}\ \ \text{and}\ \ I\ExctCplKBdg{c} = \ExctCplKBdg{c}
\end{equation*}
and satisfy $\DtrmnntOfMtrx{[\ExctCplIBdg{a}\ (\ExctCplJBdg{b}+\ExctCplKBdg{c})]} =  \DtrmnntOfMtrx{[Q\ExctCplIBdg{a}\ (Q\ExctCplJBdg{b}+Q\ExctCplKBdg{c})]} = \DtrmnntOfMtrx{[I\ExctCplIBdg{a}\ (I\ExctCplJBdg{b}+I\ExctCplKBdg{c})]}$. If $\mathcal{C}$ is regular, then so are $Q\mathcal{C}$ and $I\mathcal{C}$.
\end{proposition}
\begin{proof}[Comment on proof]
The derivation $I\ExctCpl{C}$ of $\ExctCpl{C}$ is the one originally proposed by Massey in \cite[Sec.~6]{WSMassey1952}. It has been widely adopted in contemporary literature; see for example \cite[XI.5]{SMacLane1975-Homology},\cite[10.2]{JJRotman2009},\cite{CAWeibel1994}.  Abelian category perspectives go back at least as far as  Eilenberg-Moore \cite[Sec.~5]{SEilenbergJCMoore1961} and Eckmann-Hilton \cite{BEckmannPJHilton1966}.
\end{proof}

\begin{remark}[On exact couple derivation]
\label{rem:ExactCoupleDerivation-Comment}%
While the derived exact couples $Q\ExctCpl{C}$ and $I\ExctCpl{C}$ have identical objects, they play quite different roles: $Q\ExctCpl{C}$ is based upon the cokernels in the image factorizations of $i$, and $I\ExctCpl{C}$ is based on their kernels. The difference between the two manifests itself under (co-)limit processes; see Section (\ref{sec:ImageFactorization-Z-Diagrams}).

Both exact couples are related via the commutative diagram below. Visibly, neither index positions nor bidegrees of the structure maps match up. Therefore, both derived exact couples belong to distinct categories. In particular, the commutative diagram below is not part of an isomorphism of exact couples.
\begin{equation*}
\begin{xy}
\xymatrix@R=6ex@C=4.5em{
{Q\ExctCplDObjctBB{}{\Vect{x}}} \ar[r]^-{Q\ExctCplJMapBB{}{\Vect{x}} } \ar@{=}[d] &
	{Q\ExctCplEObjctBB{}{\Vect{x} +\ExctCplJBdg{b} }} \ar[r]^-{Q\ExctCplKMapBB{}{\Vect{x} + \ExctCplJBdg{b} } } \ar@{=}[d] &
	{Q\ExctCplDObjctBB{}{\Vect{x} + \ExctCplJBdg{b}+\ExctCplKBdg{c} - \ExctCplIBdg{a} }} \ar[r]^-{ Q\ExctCplIMapBB{}{\Vect{x} + \ExctCplJBdg{b}+\ExctCplKBdg{c} - \ExctCplIBdg{a} } } \ar@{=}[d] &
	{Q\ExctCplDObjctBB{}{ \Vect{x} + \ExctCplJBdg{b}+\ExctCplKBdg{c} }} \ar@{=}[d] \\
I\ExctCplDObjctBB{}{\Vect{x} + \ExctCplIBdg{a} } \ar[r]_-{I\ExctCplJMapBB{}{\Vect{x}+\ExctCplIBdg{a} } } &
	I\ExctCplEObjctBB{}{\Vect{x} +\ExctCplJBdg{b} } \ar[r]_-{I\ExctCplKMapBB{}{\Vect{x} + \ExctCplJBdg{b} } } &
	I\ExctCplDObjctBB{}{\Vect{x} + \ExctCplJBdg{b} +\ExctCplKBdg{c} } \ar[r]_-{ I\ExctCplIMapBB{}{\Vect{x} + \ExctCplJBdg{b}+\ExctCplKBdg{c} } } &
	I\ExctCplDObjctBB{}{\Vect{x} + \ExctCplJBdg{b}+\ExctCplKBdg{c} + \ExctCplIBdg{a} }
}
\end{xy}
\end{equation*}
\end{remark}

As an alternate to exact couple derivation we describe a \Defn{transfinitely recursive process} by which we compute objects $\ExctCplEPage{\tau}$, for every ordinal $\tau$. For cardinality reasons this process stabilizes at some ordinal $\alpha$, hence yields stable E-objects $\ExctCplStblEObjct\DefEq \ExctCplEPage{\alpha}$. The stabilization ordinal of the E-objects matches the stabilization ordinal of the image subdiagrams in the $\ZCat$-diagrams $D(n)$, hence can be arbitrarily large; see (\ref{sec:ImageFactorization-Z-Diagrams}).

For finite ordinals $r\geq 1$, the objects $\ExctCplEPage{r}$ are the ones computed via a spectral sequence from exact couple derivation (\ref{thm:DerivedExactCouples}). Thus, for finite $r$, $\SSPage{r+1}_{\Vect{x}}$ is a subquotient $\SSPage{r}_{\Vect{x}}$, for infinite ordinals we have
\begin{equation*}
\ExctCplStblEObjct \subseteq \cdots \subseteq \ExctCplEPage{\omega+\kappa}\subseteq\cdots \subseteq \ExctCplEPage{\omega+1} = \ExctCplEPage{\infty}
\end{equation*}
The following constructions rely on the image/quotient subdiagrams of a $\ZCat$-diagram; see Section (\ref{sec:ImageFactorization-Z-Diagrams}).

\begin{notation}[Exact couple objects]
\label{def:ExactCoupleObjects}%
\label{def:ExactCoupleItems}% deprecated
Let $\ExctCpl{C}=(\ExctCplEPage{},\ExctCplDPage{})$ be a regular exact couple whose structure maps $\ExctCplIMap{},\ExctCplJMap{},\ExctCplKMap{}$ have bidegrees $\ExctCplIBdg{a},\ExctCplIBdg{b},\ExctCplIBdg{c}$. For $n\in \ZNr$ and $\Vect{x}\in \Prdct{\ZNr}{\ZNr}$ with $\DtrmnntOfMtrx{[\ExctCplIBdg{a}\ \Vect{x}]}=n$\*, we know  (\ref{thm:BigradingPropertiesOfExactCouple}) that the object $\ExctCplDObjctBB{}{\Vect{x}}$ belongs to the $\ZCat$-diagram $D(n)$. For an ordinal $\tau$, define:
\vspace{-1.7ex}%
\begin{enumerate}[$\bullet$]
\item The \Defn{image quotient object of order $\tau$} of $D(n)$ in position $\Vect{x}$: $\ZDiagQuo{\tau}{\Vect{x}}{} \DefEq  \ZDiagQuo{\tau}{\Vect{x}}{D(n)}$
\item The \Defn{image subobject of order $\tau$} of $D(n+\sigma)$ in position $\Vect{x}+\ExctCplJBdg{b} + \ExctCplKBdg{c}$: $\ZDiagImg{\tau}{\Vect{x}+\ExctCplJBdg{b} + \ExctCplKBdg{c}}{} \DefEq \ZDiagImg{\tau}{\Vect{x}+\ExctCplJBdg{b} + \ExctCplKBdg{c}}{D(n+\sigma)}$
\item The \Defn{object of cycles} of order $\tau$ in position $\Vect{x}+\ExctCplJBdg{b}$: $\ExctCplCyclesBB{\tau}{\Vect{x}+\ExctCplJBdg{b}}\DefEq \left( \ExctCplKMapBB{}{\Vect{x}+\ExctCplJBdg{b}}\right)^{-1} \ZDiagImg{\tau}{\Vect{x}+\ExctCplJBdg{b} + \ExctCplKBdg{c}}{}$
\index{object!of cycles}
\item the \Defn{object of boundaries} of order $\tau$ in position  $\Vect{x}+\ExctCplJBdg{b}$: $\ExctCplBndrsBB{\tau}{\Vect{x}+\ExctCplJBdg{b}}\DefEq  \ExctCplJMapBB{}{\Vect{x}} \Ker{ \ExctCplDObjctBB{}{\Vect{x}}\to \ZDiagQuo{\tau}{\Vect{x}}{} }$ %
\index{object!of boundaries}
\end{enumerate}
For ordinals $\tau<\lambda$ we then have the following inclusions:
\begin{equation*}
\ExctCplBndrsBB{\tau}{\Vect{x}+\ExctCplJBdg{b}}\subseteq \ExctCplBndrsBB{\lambda}{\Vect{x}+\ExctCplJBdg{b}}\subseteq \Img{\ExctCplJMapBB{}{\Vect{x}} } = \Ker{\ExctCplKMapBB{}{\Vect{x}+\ExctCplJBdg{b}}} \subseteq \ExctCplCyclesBB{\lambda}{\Vect{x}+\ExctCplJBdg{b}} \subseteq \ExctCplCyclesBB{\tau}{\Vect{x}+\ExctCplJBdg{b}}.
\end{equation*}
Thus, for any successor ordinal $\tau+1$, the object below is defined
\begin{equation*}
\SSObjctBB{\tau+1}{\Vect{x}+\Vect{b}} \DefEq \dfrac{ \ExctCplCyclesBB{\tau}{\Vect{x}+\ExctCplJBdg{b}} }{ \ExctCplBndrsBB{\tau}{\Vect{x}+\ExctCplJBdg{b}} }
\end{equation*}
\end{notation}

\begin{lemma}[$\ExctCplEPage{\tau}$ for $\tau>\omega$]
\label{thm:E^tau,tau>omega}
In the setting of (\ref{def:ExactCoupleItems}) the following hold for ordinals $\lambda>\tau\geq \omega$
\begin{enumerate}[(i)]
\item $\ExctCplBndrsBB{\lambda}{} = \ExctCplBndrsBB{\tau}{} = \ExctCplBndrsBB{\omega}{}$
\item $\ExctCplEPage{\lambda}\subseteq\ExctCplEPage{\tau}\subseteq \ExctCplEPage{\omega}$
\item There exists an ordinal $\alpha$ such that, for $\lambda\geq \alpha$, $\ExctCplEPage{\lambda}=\ExctCplEPage{\alpha}$. 
\end{enumerate}
\end{lemma}
\begin{proof}
Part (i) follows from (\ref{thm:Q^(omega)_pA-All-tau}). Then the inclusions $\ZDiagImg{\lambda}{\Vect{x}+\ExctCplJBdg{b} + \ExctCplKBdg{c}}{}\subseteq \ZDiagImg{\tau}{\Vect{x}+\ExctCplJBdg{b} + \ExctCplKBdg{c}}{}\subseteq \ZDiagImg{\omega}{\Vect{x}+\ExctCplJBdg{b} + \ExctCplKBdg{c}}{}$ yield $\ExctCplCyclesBB{\lambda}{\Vect{x}+\ExctCplJBdg{b}} \subseteq \ExctCplCyclesBB{\tau}{\Vect{x}+\ExctCplJBdg{b}} \subseteq \ExctCplCyclesBB{\omega}{\Vect{x}+\ExctCplJBdg{b}}$, which implies (ii). To see why the objects $\ExctCplEPage{\tau}$ stabilize, recall from Section \ref{sec:ImageFactorization-Z-Diagrams} that, for cardinality reasons, there is an ordinal $\alpha$ such that $\ZDiagImg{\alpha}{}{D(n+\sigma)} = \ZDiagImg{\lambda}{}{D(n+\sigma)}$, whenever $\alpha\leq \lambda$. Consequently, the descending sequence of cycle objects $\ExctCplCyclesBB{\tau}{}$ stabilizes at $\alpha$. Choosing $\alpha>\omega$ sufficiently large, we see via (i), that the objects $\ExctCplEPage{\tau}$ stabilize at $\alpha$ as well.
\end{proof}

\begin{terminology}[Stable E-objects]
\label{def:Stable-E-Objects}%
The \Defn{stable E-objects} of a regular exact couple, are the objects $\ExctCplStblEObjct\DefEq \ExctCplEPage{\alpha}$ established in (\ref{thm:E^tau,tau>omega}.iii). %
\index{stable!E-objects of exact couple}\index[not]{$\ExctCplStblEObjctBB{\Vect{x}}$ -- stable E-object of exact couple}%
\end{terminology}

Let us now turn to matching the objects $\ExctCplEPage{r}$, $r<\infty$, to the familiar spectral sequence objects associated with a given exact couple. We use the following notation.

\begin{notation}[Additive relation differentials]
\label{def:AdditiveRelationDifferential}
The \Defn{additive relation differentials} $\bigtriangleup^{r}_{\Vect{x}+\ExctCplJBdg{b}}\from \ExctCplCyclesBB{r-1}{\Vect{x} + \ExctCplJBdg{b} } \longrightarrow \ExctCplCyclesBB{r-1}{\Vect{x} +2\ExctCplJBdg{b} + \ExctCplKBdg{c} - (r-1)\ExctCplIBdg{a} }$ are composites %
\index{additive relation!differentials}\index[not]{$\bigtriangleup^{r}_{\Vect{x}+\ExctCplJBdg{b}}$ - additive relation differentials}
\begin{equation*}
\xymatrix@R=5ex@C=7em{
\ExctCplCyclesBB{r-1}{\Vect{x} + \ExctCplJBdg{b} } \ar[r]^-{\ExctCplKMapBB{}{\Vect{x} + \ExctCplJBdg{b} }} \ar[r] \ar@{-<} `d[rrr]`[rrr]_{\bigtriangleup^{r}_{\Vect{x}+\ExctCplJBdg{b}} }[rrr] &
	\ExctCplDObjctBB{}{\Vect{x} + \ExctCplJBdg{b} + \ExctCplKBdg{c}} \ar[r]^-{\left( \ExctCplIMapItrtdBB{r-1}{\Vect{x} +\ExctCplJBdg{b} + \ExctCplKBdg{c} - (r-1)\ExctCplIBdg{a} }\right)^{-1} } &
	\ExctCplDObjctBB{}{\Vect{x} + \ExctCplJBdg{b} + \ExctCplKBdg{c}-(r-1)\ExctCplIBdg{a}} \ar[r]^-{\ExctCplJMapBB{}{\Vect{x} +\ExctCplJBdg{b} + \ExctCplKBdg{c} - (r-1)\ExctCplIBdg{a} }} &
	\ExctCplCyclesBB{r-1}{\Vect{x} + 2\ExctCplJBdg{b} + \ExctCplKBdg{c}-(r-1)\ExctCplIBdg{a} }
}
\end{equation*}
\end{notation}

Referring to the additive relations $\bigtriangleup^r$ as differentials is justified by the following theorem. For finite ordinals, it identifies the objects $\ExctCplEPage{r}$, as  constructed in (\ref{def:ExactCoupleItems}), with the spectral sequence objects that are classically constructed from the exact couple.

\begin{theorem}[Spectral sequence of an exact couple]
\label{thm:SpecSequFromExactCouple}%
For a regular exact couple, see (\ref{def:(ZxZ)-BigradedExactCouple}), the additive relation differentials $\bigtriangleup^{r}$ factor to $R$-module morphisms as shown in the commutative diagram below.
\begin{equation*}
\xymatrix@R=5ex@C=7em{
\ExctCplCyclesBB{r-1}{\Vect{x}- \ExctCplKBdg{c}+(r-1)\ExctCplIBdg{a}} \ar@{-{ >>}}[d] \ar[r]^-{\bigtriangleup^{r}_{\Vect{x}- \ExctCplKBdg{c}+(r-1)\ExctCplIBdg{a} } } &
	\ExctCplCyclesBB{r-1}{\Vect{x}+\ExctCplJBdg{b}} \ar@{-{ >>}}[d]_{\tau^{r-1}_{\Vect{x}+\ExctCplJBdg{b}}} \ar[r]^-{\bigtriangleup^{r}_{\Vect{x}+\Vect{b}} }  &
	\ExctCplCyclesBB{r-1}{\Vect{x}+2\ExctCplJBdg{b} + \ExctCplKBdg{c}-(r-1)\ExctCplIBdg{a}} \ar@{-{ >>}}[d] \\
\SSObjctBB{r}{\Vect{x}- \ExctCplKBdg{c}+(r-1)\ExctCplIBdg{a} } \ar[r]_-{\SSDffrntlBB{r}{\Vect{x}- \ExctCplKBdg{c}+(r-1)\ExctCplIBdg{a} } } &
	\SSObjctBB{r}{\Vect{x}+\Vect{b}} \ar[r]_-{\SSDffrntlBB{r}{\Vect{x}+\ExctCplJBdg{b}}} &
	\SSObjctBB{r}{\Vect{x}+2\Vect{b}+\Vect{c}-(r-1)\Vect{a}}
}
\end{equation*}
Further, the morphisms $\SSDffrntl{r}$, $1\leq r< \infty$, form the differentials of a spectral sequence which is related to the construction (\ref{def:ExactCoupleItems}) via the identities
\begin{equation*}
\ExctCplCyclesBB{r}{\Vect{x}+\ExctCplJBdg{b}} = (\tau^{r-1}_{\Vect{x}+\ExctCplJBdg{b}})^{-1}\Ker{\SSDffrntlBB{r}{\Vect{x}+\Vect{b}}}  \qquad \text{and}\qquad  \ExctCplBndrsBB{r}{\Vect{x}+\ExctCplJBdg{b}} = (\tau^{r-1}_{\Vect{x}+\ExctCplJBdg{b}})^{-1}\Img{\SSDffrntlBB{r}{\Vect{x} -\ExctCplKBdg{c}+(r-1)\ExctCplIBdg{a}}}.
\end{equation*}
Both, the differential $\SSDffrntl{r}$ and the additive relation differential $\bigtriangleup^{r}$, have bidegree $\SSBidegreeVect{v}{r}=\ExctCplKBdg{c}-(r-1)\ExctCplIBdg{a}+\ExctCplJBdg{b}$.
\end{theorem}
\begin{proof}
We include the lengthy, but straight forward argument because Theorem (\ref{thm:SpecSequFromExactCouple}) is the link between objects of primary interest to us: the spectral sequence construction of the objects $\ExctCplEPage{r}$\MSComp, $r<\infty$, and the recursive, exact couple based construction, of the same objects in (\ref{def:ExactCoupleItems}). We begin by showing  that
\begin{equation*}
\SSDffrntlBB{r}{\Vect{x}+\ExctCplJBdg{b}}\from \ExctCplEObjctBB{r}{\Vect{x}+\ExctCplJBdg{b}} \longrightarrow \ExctCplEObjctBB{r}{\Vect{x}+2\ExctCplJBdg{b}+\ExctCplKBdg{c}-(r-1)\ExctCplIBdg{a}}
\end{equation*}
is a morphism of $R$-modules. It suffices to show that the additive relation $\bigtriangleup$ factors to a well defined set theoretic function. If $e\in \ExctCplCyclesBB{r-1}{\Vect{x}+\ExctCplJBdg{b}}$, then $\left(\tau^{r-1}_{\Vect{x}+\ExctCplJBdg{b}}\right)^{-1}(\tau^{r-1}_{\Vect{x}+\ExctCplJBdg{b}}(e)) = e + \ExctCplBndrsBB{r-1}{\Vect{x}+\ExctCplJBdg{b}}$ and, omitting some obvious subscripts,
\begin{equation*}
\begin{array}{rcl}
\bigtriangleup^{r}_{\Vect{x}+\ExctCplJBdg{b}}(e + \ExctCplBndrsBB{r-1}{\Vect{x}+\ExctCplJBdg{b}}) & = & j(i^{r-1})^{-1}k (e + \ExctCplBndrsBB{r-1}{\Vect{x}+\ExctCplJBdg{b}}) \\
	& = & j(i^{r-1})^{-1}k(e) \\
	& = & j(e' + \Ker{i^{r-1}}) \\
	& = & j(e') + \ExctCplBndrsBB{r-1}{\Vect{x}+2\ExctCplJBdg{b}+\ExctCplKBdg{c}-(r-1)\ExctCplIBdg{a}}
\end{array}
\end{equation*}
Here, $e'$ is arbitrary with $i^{r-1}(e') =k(e)$; such $e'$ exists by the construction of $\ExctCplCyclesBB{r-1}{\Vect{x}+\ExctCplJBdg{b}}$. It follows that  $\bigtriangleup^r$ factors to a well defined morphism of modules. Next, we show that $\ExctCplCyclesBB{r}{\Vect{x}+\ExctCplJBdg{b}} = (\tau^{r-1}_{\Vect{x}+\ExctCplJBdg{b}})^{-1}\Ker{\SSDffrntlBB{r}{\Vect{x}+\Vect{b}}}$. By commutativity:
\begin{equation*}
\begin{array}{rcl}
(\tau^{r-1}_{\Vect{x}+\ExctCplJBdg{b}})^{-1}\Ker{\SSDffrntlBB{r}{\Vect{x}+\Vect{b}}} & = & \left( \bigtriangleup^{r}_{\Vect{x}+\ExctCplJBdg{b} } \right)^{-1} \ExctCplBndrsBB{r-1}{\Vect{x} + 2\ExctCplJBdg{b}+\ExctCplKBdg{c} - (r-1)\ExctCplIBdg{a}} \\
	& = & k^{-1} i^{r-1} j^{-1}\left( j(\Ker{i^{r-1}}) \right) \\
	& = & k^{-1} i^{r-1}\left( \Ker{i^{r-1} + \Ker{j}}\right) \\
	& = & k^{-1} i^{r-1}\left( \Ker{i^{r-1} + \Img{\ExctCplIMapBB{}{\Vect{x}+\ExctCplJBdg{b}+\ExctCplKBdg{c} - r\ExctCplIBdg{a}}}} \right) \\
	& = & k^{-1}\Img{\ExctCplIMapBB{r}{\Vect{x}+\ExctCplJBdg{b}+\ExctCplKBdg{c} - r\ExctCplIBdg{a}}} \\
	& = & \ExctCplCyclesBB{r}{\Vect{x}+\ExctCplJBdg{b}}
\end{array}
\end{equation*}
Then, we need to know that $\ExctCplBndrsBB{r}{\Vect{x}+\ExctCplJBdg{b}} = (\tau^{r-1}_{\Vect{x}+\ExctCplJBdg{b}})^{-1}\Img{\SSDffrntlBB{r}{\Vect{x} -\ExctCplKBdg{c}+(r-1)\ExctCplIBdg{a}}}$. By commutativity,
\begin{equation*}
\begin{array}{rcl}
(\tau^{r-1}_{\Vect{x}+\ExctCplJBdg{b}})^{-1}\Img{\SSDffrntlBB{r}{\Vect{x} -\ExctCplKBdg{c}+(r-1)\ExctCplIBdg{a}}} & = &  \bigtriangleup^{r}_{\Vect{x} -\ExctCplKBdg{c}+ (r-1)\ExctCplIBdg{a} }\left( \ExctCplCyclesBB{r-1}{\Vect{x} -\ExctCplKBdg{c}+ (r-1)\Vect{a}} \right)  \\
	& = &  j \left( \ExctCplIMapBB{r-1}{\Vect{x}}\right)^{-1} k \ExctCplCyclesBB{r-1}{\Vect{x}-\ExctCplKBdg{c} + (r-1)\Vect{a} }  \\
	& = &  j \left( \ExctCplIMapBB{r-1}{\Vect{x}}\right)^{-1} k k^{-1} \Img{\ExctCplIMapBB{r-1}{\Vect{x}}}   \\
	& = &  j \left( \ExctCplIMapBB{r-1}{\Vect{x}}\right)^{-1} \left( \Img{\ExctCplIMapBB{r-1}{\Vect{x}}} \intrsctn \Img{k}\right)  \\
	& = & j \left( \ExctCplIMapBB{r-1}{\Vect{x}}\right)^{-1} \left( \Img{\ExctCplIMapBB{r-1}{\Vect{x}}} \intrsctn \Ker{ \ExctCplIMapBB{}{\Vect{x} + (r-1)\ExctCplIBdg{a}}}\right)  \\
	& = & j \Ker{\ExctCplIMapBB{r}{\Vect{x}}} \\
	& = & \ExctCplBndrsBB{r}{\Vect{x}+\ExctCplJBdg{b}}
\end{array}
\end{equation*}
Finally, we confirm that  $\SSDffrntl{r}\SSDffrntl{r} = 0$. Setting $\tau\DefEq \tau^{r-1}_{\Vect{x} + 2\ExctCplJBdg{b} +\ExctCplKBdg{c} - (r-1)\ExctCplIBdg{a}}$, by commutativity:
\begin{equation*}
\begin{array}{rcl}
\Img{\SSDffrntlBB{r}{\Vect{x}+\ExctCplJBdg{b} } \Comp \SSDffrntlBB{r}{\Vect{x} -\ExctCplKBdg{c}+ (r-1) \ExctCplIBdg{a} }} 
	& = & \tau\Comp \bigtriangleup^{r}_{\Vect{x}+\ExctCplJBdg{b}}\Comp \bigtriangleup^{r}_{\Vect{x}-\ExctCplKBdg{c}+(r-1)\ExctCplIBdg{a}} \left(\ExctCplCyclesBB{r-1}{ \Vect{x}-\ExctCplKBdg{c}+(r-1) \ExctCplIBdg{a} } \right)  \\
	& \subseteq & \tau\Comp \ExctCplJMapBB{}{ \Vect{x}+\ExctCplJBdg{b}+\ExctCplKBdg{c}- (r-1)\ExctCplIBdg{a}}  (i^{r-1}_{\Vect{x}+\ExctCplJBdg{b}+\ExctCplKBdg{c}-(r-1)\ExctCplIBdg{a}})^{-1} k_{\Vect{x}+\ExctCplJBdg{b}} j_{\Vect{x}} ( \ExctCplDObjctBB{}{\Vect{x} })  \\
	& = &  \tau\Comp \ExctCplJMapBB{}{ \Vect{x}+\ExctCplJBdg{b}+\ExctCplKBdg{c}- (r-1)\ExctCplIBdg{a}}  (i^{r-1}_{\Vect{x}+\ExctCplJBdg{b}+\ExctCplKBdg{c}-(r-1)\ExctCplIBdg{a}})^{-1}(0) \\
	& = &  \tau \left(\ExctCplBndrsBB{r-1}{\Vect{x}+2\ExctCplJBdg{b}+\ExctCplKBdg{c}- (r-1)\ExctCplIBdg{a}  }\right) = 0
\end{array}
\end{equation*}
This implies the claim.
\end{proof}

\begin{corollary}[$\ExctCplEPage{\omega+1} = \SSObjctBB{\infty}{}$]
\label{thm:SSE-infty=E^(omega+1)}
Given a regular exact couple $\ExctCpl{C}$, the object $\SSObjctBB{\infty}{}$ of its associated spectral sequence equals the object $\ExctCplEPage{\omega+1}$ as constructed in (\ref{def:ExactCoupleItems}). \NoProof
\end{corollary}

\begin{corollary}[Functoriality of spectral sequence construction]
\label{thm:SpecSequFromExactCouple-Functoriality}%
The spectral sequence construction in (\ref{thm:SpecSequFromExactCouple}) establishes an additive functor
\begin{equation*}
\ExctCplsCatModulesOver{R}(\ExctCplIBdg{a},\ExctCplJBdg{b},\ExctCplKBdg{c}) \longrightarrow \SpecSeqsInModOver{R}(1,\SSBidegreeVect{v}{r}\, |\, r\geq 1)
\end{equation*}
Here $(\ExctCplIBdg{a},\ExctCplJBdg{b},\ExctCplKBdg{c})$ are the fixed bidegrees of structure maps of exact couples meeting the regularity condition and, for $r\geq 1$, the bidegree of the differential $\SSDffrntl{r}$ is $\SSBidegreeVect{v}{r}\DefEq \ExctCplJBdg{b}+\ExctCplKBdg{c}-(r-1)\ExctCplIBdg{a}$.
\end{corollary}

We will explain the meaning of the objects $\ExctCplEPage{\tau}$ in the next section. For now, we identify some exact couples whose $E$-objects are related in a simple manner.

\begin{proposition}[Spectral sequences and (co-)Mittag-Leffler condition]
\label{thm:SpecSeqs-(Co-)Mittag-Leffler}
For the $E$\MSComp-objects of a regular exact couple $\ExctCpl{C}$ the following hold:
\vspace{-1.5ex}
\begin{enumerate}[(i)]
\item If all of the $\ZCat$-diagrams $D(n)$, $n\in \ZNr$, satisfy the co-Mittag-Leffler condition (\ref{def:Mittag-Leffler/CoMittag-Leffler-Condition}), then each $\Vect{x}\in \Prdct{\ZNr}{\ZNr}$ admits $1\leq r< \infty$, along with an an associated sequence of $E$\MSComp-subobjects
\begin{equation*}
\xymatrix@R=5ex@C=2em{
\ExctCplStblEObjctBB{\Vect{x}+\ExctCplJBdg{b}} \ar@{{ |>}->}[r] & \cdots 
	\ar@{{ |>}->}[r] &
	\ExctCplEObjctBB{\tau}{\Vect{x}+\ExctCplJBdg{b}} \ar@{{ |>}->}[r] &
	\cdots \ar@{{ |>}->}[r] & \ExctCplEObjctBB{\omega+1}{\Vect{x}+\ExctCplJBdg{b}}=\ExctCplEObjctBB{\infty}{\Vect{x}+\ExctCplJBdg{b}} \ar@{{ |>}->}[r] & \cdots \ar@{{ |>}->}[r] &
	\ExctCplEObjctBB{r+k}{\Vect{x}+\ExctCplJBdg{b}} \ar@{{ |>}->}[r] &
	\cdots \ar@{{ |>}->}[r] &
	\ExctCplEObjctBB{r}{\Vect{x}+\ExctCplJBdg{b}}.
}
\end{equation*}
\begin{equation*}
\ExctCplStblEObjctBB{\Vect{x}+\ExctCplJBdg{b}}\subseteq \cdots \subseteq \ExctCplEObjctBB{\tau}{\Vect{x}+\ExctCplJBdg{b}}\subseteq \cdots \subseteq \ExctCplEObjctBB{\omega+1}{\Vect{x}+\ExctCplJBdg{b}}=\ExctCplEObjctBB{\infty}{\Vect{x}+\ExctCplJBdg{b}} \subseteq \cdots \subseteq \ExctCplEObjctBB{r+k}{\Vect{x}+\ExctCplJBdg{b}}\subseteq \cdots \subseteq \ExctCplEObjctBB{r}{\Vect{x}+\ExctCplJBdg{b}}.
\end{equation*}
\item If all of the $\ZCat$-diagrams $D(n)$, $n\in \ZNr$, satisfy the Mittag-Leffler condition (\ref{def:Mittag-Leffler/CoMittag-Leffler-Condition}), then each $\Vect{x}\in \Prdct{\ZNr}{\ZNr}$ admits $1\leq r< \infty$, along with an an associated sequence of $E$\MSComp-quotient objects
\begin{equation*}
\xymatrix@R=5ex@C=1.5em{
\ExctCplEObjctBB{r}{\Vect{x}+\ExctCplJBdg{b}} \ar@{-{ >>}}[r] &
	\cdots \ar@{-{ >>}}[r] &
	\ExctCplEObjctBB{r+k}{\Vect{x}+\ExctCplJBdg{b}}\ar@{-{ >>}}[r] &
	\cdots \ar@{-{ >>}}[r] &
	\ExctCplEObjctBB{\infty}{\Vect{x}+\ExctCplJBdg{b}} = \ExctCplStblEObjctBB{\Vect{x}+\ExctCplJBdg{b}}.
}
\end{equation*}
\item If all of the $\ZCat$-diagrams $D(n)$, $n\in \ZNr$, satisfy both the Mittag-Leffler condition and the co-Mittag-Leffler condition, then each $\Vect{x}\in \Prdct{\ZNr}{\ZNr}$ admits $1\leq r< \infty$, such that $\ExctCplStblEObjctBB{\Vect{x}+\ExctCplJBdg{b}}=	\ExctCplEObjctBB{\lambda}{\Vect{x}+\ExctCplJBdg{b}} = \ExctCplEObjctBB{r}{\Vect{x}+\ExctCplJBdg{b}}$ for each successor ordinal $\lambda\geq r$.
\end{enumerate}
\end{proposition}
\begin{proof}
(i)\quad If the $\ZCat$-diagram $\ExctCplDObjctBB{}{\Vect{x} + s\ExctCplIBdg{a}}$, $s\in \ZNr$, satisfies the co-Mittag-Leffler condition, then $\ZDiagQuo{r}{\Vect{x}}{} = \ZDiagQuo{\tau}{\Vect{x}}{}$ for sufficiently large $r\leq \tau$. On boundary objects, the constructions in (\ref{def:ExactCoupleItems}) respond by $\ExctCplBndrsBB{\tau}{\Vect{x}+\ExctCplJBdg{b}}=\ExctCplBndrsBB{r}{\Vect{x}+\ExctCplJBdg{b}}$, for $\tau\geq r$. This implies (i).

(ii)\quad If the $\ZCat$-diagram $\ExctCplDObjctBB{}{\Vect{x} + \ExctCplJBdg{b}+\ExctCplKBdg{c} + s\ExctCplIBdg{a}}$, $s\in \ZNr$, satisfies the Mittag-Leffler condition, then $\ZDiagImg{\tau}{\Vect{x}+\ExctCplJBdg{b}+\ExctCplKBdg{c}}{}=\ZDiagImg{r}{\Vect{x}+\ExctCplJBdg{b}+\ExctCplKBdg{c}}{}$ for sufficiently large $r\leq \tau$. On cycle objects, the constructions in (\ref{def:ExactCoupleItems}) respond by $\ExctCplCyclesBB{r}{\Vect{x}+\ExctCplJBdg{b}}=\ExctCplCyclesBB{\tau}{\Vect{x}+\ExctCplJBdg{b}}$, for $r\leq \tau$. Recalling (\ref{thm:E^tau,tau>omega}) that, for $\lambda\geq \omega$, $\ExctCplBndrs{\lambda}{\Vect{x}+\ExctCplJBdg{b}}{}=\ExctCplBndrs{\omega}{\Vect{x}+\ExctCplJBdg{b}}{}$, claim (ii) follows.

(iii)\quad follows by combining parts (i) and (ii).
\end{proof}

\begin{corollary}[Spectral sequences of originally / eventually stable exact couples]
\label{thm:SpecSeqsOfOriginally/EventuallyStableExactCouples}
For the $E$\MSComp-objects of a regular exact couple $\ExctCpl{C}$ the following hold: %
\index{originally!stable exact couple}\index{eventually!stable exact couple}\index{spectral sequence!orginally stable exact couple} \index{spectral sequence!eventually stable exact couple}%
\begin{enumerate}[(i)]
\item Suppose the $\ZCat$-diagrams $D(n)$, $n\in \ZNr$, are all eventually stable (\ref{def:Z-DiagramTypes}). Then, for each $\Vect{x}\in \ZNr\prdct\ZNr$, there exists $r\geq 1$ such that, for all $k\geq 0$, the differential $\SSDffrntlBB{r+k}{\Vect{x}-\ExctCplKBdg{c}+(r-1+k)\ExctCplIBdg{a}}$ arriving at $\ExctCplEObjctBB{r+k}{\Vect{x}+\ExctCplJBdg{b}}$ has domain $0$. Consequently, we obtain a sequence of subobjects
\begin{equation*}
\ExctCplStblEObjctBB{\Vect{x}+\ExctCplJBdg{b}}\subseteq \cdots \subseteq \ExctCplEObjctBB{\infty}{\Vect{x}+\ExctCplJBdg{b}} \subseteq \cdots \subseteq \ExctCplEObjctBB{r+k}{\Vect{x}+\ExctCplJBdg{b}}\subseteq \cdots \subseteq \ExctCplEObjctBB{r}{\Vect{x}+\ExctCplJBdg{b}}.
\end{equation*}
\item Suppose the $\ZCat$-diagrams $D(n)$, $n\in\ZNr$, are all originally stable. Then for each $\Vect{x}\in \ZNr\prdct\ZNr$, there exists $r\geq 1$ such that, for all $k\geq 0$, the differential $\SSDffrntlBB{r+k}{\Vect{x}+\ExctCplJBdg{b}}$ exiting from $\ExctCplEObjctBB{r+k}{\Vect{x}+\ExctCplJBdg{b}}$ has codomain $0$. Consequently, we obtain a sequence of quotient objects
\begin{equation*}
\xymatrix@R=5ex@C=2em{
\ExctCplEObjctBB{r}{\Vect{x}+\ExctCplJBdg{b}} \ar@{-{ >>}}[r] &
	\cdots \ar@{-{ >>}}[r] &
	\ExctCplEObjctBB{r+k}{\Vect{x}+\ExctCplJBdg{b}}\ar@{-{ >>}}[r] &
	\cdots \ar@{-{ >>}}[r] &
	\ExctCplEObjctBB{\infty}{\Vect{x}+\ExctCplJBdg{b}} = \ExctCplStblEObjctBB{\Vect{x}+\ExctCplJBdg{b}}.
} \tag*{$\lozenge$}
\end{equation*}
\end{enumerate}
\end{corollary}

Criteria which may help recognize the presence of the (co-)Mittag-Leffler property in the $\ZCat$-diagrams $D(n)$ are collected in (\ref{thm:MittagLefflerExamples}) and (\ref{thm:CoMittagLefflerExamples}).

We close this section with information about re-indexing a given exact couple. Indeed, such re-indexing may be advantageous so that the differentials of the associated  spectral sequence acquire familiar bidegrees. On a regular exact couple this may be accomplished via the action of the general linear group $\GLGrp{2}{\ZNr}$ on $(\ZNr\prdct\ZNr)$: %
\index[not]{$\GLGrp{2}{\ZNr}$ - general linear group on $\Prdct{\ZNr}{\ZNr}$}%

\begin{lemma}[Re-indexing an exact couple]
\label{thm:ExactCouple-ReIndexed}
The general linear group $\GLGrp{2}{\ZNr}$ right acts on the category of $(\ZNr\prdct\ZNr)$-bigraded exact couples by precomposition: if $T\in \GLGrp{2}{\ZNr}$, and $\ExctCpl{C}=(D,E,i,j,k)$ is an exact couple whose structure maps $i,j,k$ have bidegrees $\ExctCplIBdg{a},\ExctCplJBdg{b},\ExctCplKBdg{c}$, respectively, then $\ExctCpl{C}.T$ is given by: %
\index{exact couple!re-indexing}\index{regular exact couple!re-indexing}%
\begin{enumerate}[(i)]
\item $\ZNr\prdct \ZNr \XRA{T} \ZNr\prdct \ZNr \XRA{D,E} \LModules{R}$, that is $(D.T)_{\Vect{x}}\DefEq D_{T\Vect{x}}$ and $(E.T)_{\Vect{x}}\DefEq E_{T\Vect{x}}$.
\item $i.T\DefEq i\Comp T$, $j.T\DefEq j\Comp T$, and $k.T\DefEq k\Comp T$  have bidegrees $T(\Vect{a})$,  $T(\Vect{b})$, and  $T(\Vect{c})$, respectively.
\end{enumerate}
Moreover, $\ExctCpl{C}$ is regular if and only if $\ExctCpl{C}.T$ is regular.
\end{lemma}
\begin{proof}
The right action properties follow from properties of composition of functors. To see that $T$ preserves regular exact couples, set $\Vect{z}\DefEq \ExctCplJBdg{b} + \ExctCplKBdg{c}$. Then:
\begin{equation*}
\DtrmnntOfMtrx{[T(\ExctCplIBdg{a})\ \ T(\Vect{z})]} = \DtrmnntOf{T}\cdot \DtrmnntOfMtrx{[\ExctCplIBdg{a}\ \ \Vect{z}]}
\end{equation*}
As $\DtrmnntOf{T}$ is a unit in $\ZNr$, the product on the right is a unit in $\ZNr$ if and only if $\DtrmnntOfMtrx{[\ExctCplIBdg{a}\ \ \Vect{z}]}$ is a unit in $\ZNr$; i.e. if and only if $\ExctCpl{C}$ is regular.
\end{proof}

Via (\ref{thm:ExactCouple-ReIndexed}) we may re-index any regular exact couple so that its associated spectral sequence is homological or cohomological (\ref{term:SpectralSequence-Namings}).

\begin{example}[Re-indexing to (co-)homological spectral sequence]
\label{exa:ExactCouple-ReindexTo(Co-)Homological}
If $\ExctCpl{C}=(D,E,i,j,k)$ is a regular exact couple whose structure maps have bidegrees $\ExctCplIBdg{a}=(a_1,a_2)$, $\ExctCplJBdg{b}=(b_1,b_2)$, $\ExctCplKBdg{c}=(c_1,c_2)$, with $\Vect{z}\DefEq \ExctCplIBdg{a}+\ExctCplJBdg{b}$, then the $\GLGrp{2}{\ZNr}$ matrix
\begin{equation*}
T\DefEq \DtrmnntOfMtrx{[\ExctCplIBdg{a}\ \ \Vect{z}]}\cdot \left[
\begin{array}{cc}
a_2+z_2 & -(a_1+z_1) \\
-z_2 & z_1
\end{array}
\right]
\end{equation*}
transforms $\ExctCpl{C}$ into an exact couple whose spectral sequence is homological, and $-T$ transforms $\ExctCpl{C}$ into an exact couple whose spectral sequence is cohomological.
\end{example}
\begin{proof}
One checks that $T(\ExctCplIBdg{a})=(1,-1)$ and that $T(\Vect{z})=(-1,0)$. Consequently, the differential on the $r$-th page of the spectral sequence of $\ExctCpl{C}.T$ has bidegree
\begin{equation*}
\Vect{v}_r = T(\Vect{z}) - (r-1)T(\ExctCplIBdg{a}) = (-r,r-1)
\end{equation*}
So, $\ExctCpl{C}.T$ has a homological spectral sequence. This also implies the proposed re-indexing of a regular exact couple so that its associated spectral sequence is cohomological.
\end{proof}

\section[Convergence I]{Convergence I: The E-Infinity Extension Theorems}
\label{sec:Convergence-I}

When working with the spectral sequence associated to an exact couple, we need to understand the meaning of the $\SSObjctBB{\infty}{}$-objects. We take the inquiry about such meaning to be the question about \Defn{convergence} of spectral sequences in the most general sense. This question has two components: (I) What is the meaning of each individual object $\SSObjctBB{\infty}{\Vect{x}+\ExctCplJBdg{b}}$? and (II) What is the combined meaning of the objects $\SSObjctBB{\infty}{}$?

In this section, we describe the meaning of each  individual object $\SSObjctBB{\infty}{\Vect{x}+\ExctCplJBdg{b}}$. This is the content of the E-infinity extension theorem (\ref{thm:E-InfinityExtensionThm}). It applies to the spectral sequence of any regular $(\Prdct{\ZNr}{\ZNr})$-bigraded exact couple  (\ref{def:(ZxZ)-BigradedExactCouple}). Let us explain how the  $\ExctCplEPage{\infty}$\MSComp-objects are related to the universal abutment objects of the exact couple, and what the underlying phenomena are. %
\index{convergence!of a spectral sequence}

In subsequent sections we develop consequences of the E-infinity extension theorem. Via comparison theorems we approach question (II) about the combined meaning of the objects $\SSObjctBB{\infty}{}$.

{\em Outline of results}\quad We use notation as in Diagram \eqref{fig:ExactCouple-Flattened}, the flattened diagram of a $(\ZNr\prdct \ZNr)$-bigraded exact couple. Such an exact couple has structure maps $i,j,k$ with bidegrees $\ExctCplIBdg{a},\ExctCplJBdg{b},\ExctCplKBdg{c}$,  respectively, and it has limit and colimit abutment objects defined in (\ref{def:ExactCouple-UniversalLim/CoLimAbutment}): %
\index[not]{$\ExctCplLimAbut{n+\sigma}$ - colimit abutment of exact couple}\index[not]{$\ExctCplCoLimAbut{n}$ - limit abutment of exact couple}%
\begin{equation*}
\begin{array}{rrcll}
\text{colimit abutment}\qquad & \ExctCplCoLimAbut{n} & \DefEq & \CoLimOfOver{\ExctCplDObjctBB{}{\Vect{x}+r\ExctCplIBdg{a}}}{r} &\qquad n\DefEq \DtrmnntOf{[\ExctCplIBdg{a}\ \ \Vect{x}]} \\
\text{limit abutment}\qquad & \ExctCplLimAbut{n+\sigma} & \DefEq & \LimOfOver{\ExctCplDObjctBB{}{\Vect{x} +\ExctCplJBdg{b} + \ExctCplKBdg{c} + r\ExctCplIBdg{a}}}{r} &\qquad \sigma\DefEq \DtrmnntOf{[\ExctCplIBdg{a}\ \ (\ExctCplJBdg{b} + \ExctCplKBdg{c})]} \in \Set{\pm 1}
\end{array}
\end{equation*}
For $r\in \ZNr$, let $\ExctCplCoLimAbutMapBB{\Vect{x}+r\ExctCplIBdg{a}}\from \ExctCplDObjctBB{}{\Vect{x}+r\ExctCplIBdg{a}}\to \ExctCplCoLimAbut{n}$ and $\ExctCplLimAbutMapBB{\Vect{x}+\ExctCplJBdg{b}+\ExctCplKBdg{c}+r\ExctCplIBdg{a}}\from \ExctCplLimAbut{n+\sigma} \to \ExctCplDObjctBB{}{\Vect{x}+\ExctCplJBdg{b}+\ExctCplKBdg{c}+r\ExctCplIBdg{a}}$ denote the maps of the associate cocone and cone, respectively. The universal abutment objects are canonically filtered by %
\index[not]{$\ExctCplCoLimAbutMapBB{\Vect{x}}\from \ExctCplDObjctBB{}{\Vect{x}}\to \ExctCplCoLimAbut{n}$}%
\index[not]{$\ExctCplLimAbutMapBB{\Vect{x}+\ExctCplJBdg{b}+\ExctCplKBdg{c}}\from \ExctCplLimAbut{n+\sigma} \to \ExctCplDObjctBB{}{\Vect{x}+\ExctCplJBdg{b}+\ExctCplKBdg{c}}$}
\begin{equation*}
\ExctCplCoLimAbutFltrtnBB{\Vect{x}+r\ExctCplIBdg{a}} \DefEq \Img{ \ExctCplCoLimAbutMapBB{\Vect{x}+r\ExctCplIBdg{a} }\from \ExctCplDObjctBB{}{\Vect{x}+r\ExctCplIBdg{a}}\to \ExctCplCoLimAbut{n} }  \qquad \text{and}\qquad \ExctCplLimAbutFltrtnBB{\Vect{x} +\ExctCplJBdg{b} + \ExctCplKBdg{c} +r\ExctCplIBdg{a}} \DefEq \Ker{\rho^{\Vect{x}+\ExctCplJBdg{b} + \ExctCplKBdg{c} +r\ExctCplIBdg{a} }\from  \ExctCplLimAbut{n+\sigma}\to \ExctCplDObjctBB{}{\Vect{x}+\ExctCplJBdg{b} + \ExctCplKBdg{c}+r\ExctCplIBdg{a} } }
\end{equation*}
Next, we tie quotients of adjacent stages of these filtrations into short exact sequences:
\begin{equation*}
\xymatrix@R=.5ex@C=3em{
\ExctCplCoLimAbutFltrtnBB{\Vect{x}-\ExctCplIBdg{a}} \ar@{{ |>}->}[r] &
	\ExctCplCoLimAbutFltrtnBB{\Vect{x}} \ar@{-{ >>}}[r] &
  \ExctCplCoLimAbutFltrtnQtntBB{\Vect{x}} \\
\ExctCplLimAbutFltrtnBB{\Vect{x}+\ExctCplJBdg{b}+\ExctCplKBdg{c}} \ar@{{ |>}->}[r] &
	\ExctCplLimAbutFltrtnBB{\Vect{x}+\ExctCplIBdg{a}+\ExctCplJBdg{b}+\ExctCplKBdg{c}} \ar@{-{ >>}}[r] &
	\ExctCplLimAbutFltrtnQtntBB{\Vect{x}+\ExctCplJBdg{b}+\ExctCplKBdg{c}}
}
\end{equation*}
Notice that the objects $\ExctCplCoLimAbutFltrtnQtntBB{\Vect{x}}$ are constructed from the $\ZCat$-diagram $\ExctCplDObjctBB{}{\Vect{x}+r\ExctCplIBdg{a}}$, $r\in \ZNr$, while the objects $\ExctCplLimAbutFltrtnQtntBB{\Vect{x}+\ExctCplJBdg{b}+\ExctCplKBdg{c}}$ are constructed from the neighboring $\ZCat$-diagram $\ExctCplDObjctBB{}{\Vect{x}+\ExctCplJBdg{b}+\ExctCplKBdg{c}+r\ExctCplIBdg{a}}$, $r\in \ZNr$. As opposite as these adjacent filtration stages may appear, they {\em are always related via the stable E-extension Theorem} (\ref{thm:Stable-E-Extension}). It asserts that the top row in the commutative diagram below is short exact.
\begin{equation*}
\xymatrix@R=5ex@C=5em{
\ExctCplCoLimAbutFltrtnQtntBB{\Vect{x}} \ar@{{ |>}->}[r] \ar@{=}[d] &
	\ExctCplStblEObjctBB{\Vect{x}+\ExctCplJBdg{b}} \ar@{-{ >>}}[r] \ar@{{ |>}->}[d] \PullLU{rd}&
	\ExctCplLimAbutFltrtnQtntBB{\Vect{x}+\ExctCplJBdg{b}+\ExctCplKBdg{c}} \ar@{{ |>}->}[d]^{M} \\
\ExctCplCoLimAbutFltrtnQtntBB{\Vect{x}} \ar@{{ |>}->}[r] &
	\ExctCplEObjctBB{\infty}{\Vect{x}+\ExctCplJBdg{b}} \ar@{-{ >>}}[r] &
	\frac{ \ExctCplCyclesBB{\infty}{\Vect{x}+\ExctCplJBdg{b} } }{ \Img{ \ExctCplJMapBB{}{\Vect{x}} } }
}
\end{equation*}
We recall from (\ref{def:Stable-E-Objects}) that the stable E-objects $\ExctCplStblEObjctBB{\Vect{x}+\ExctCplJBdg{b}}$ are constructed from the underlying exact couple via a transfinite a process. Below the least infinite ordinal this process delivers the spectral sequence objects $\SSObjctBB{r}{\Vect{x}+\ExctCplJBdg{b}}$ from which we obtain $\SSObjctBB{\infty}{\Vect{x}+\ExctCplJBdg{b}}=\SSObjctBB{\omega+1}{\Vect{x}+\ExctCplJBdg{b}}$. For ordinals $\lambda>\omega+1$, $\SSObjctBB{\lambda}{\Vect{x}+\ExctCplJBdg{b}}$ is, in general, a proper subobject of $\SSObjctBB{\infty}{\Vect{x}+\ExctCplJBdg{b}}$. The vertical map in the middle of the diagram factors through $\SSObjctBB{\infty}{\Vect{x}+\ExctCplJBdg{b}}$.

In the E-infinity extension theorem (\ref{thm:E-InfinityExtensionThm}), we show that the bottom row of the diagram is short exact, and we express the cokernel of $M$ using $\Lim/\LimOne$-methods. On this foundation, we tackle the following two questions:
\begin{enumerate}[(D.1)]
\setlength{\itemindent}{2ex}
\item Under which condition(s) is the submodule $\ExctCplStblEObjctBB{\Vect{x}+\ExctCplJBdg{b}}$ equal to the spectral sequence object $\ExctCplEObjctBB{\infty}{\Vect{x}+\ExctCplJBdg{b}}$? - As we explain, this happens if and only if $\ExctCplEObjctBB{\infty}{\Vect{x}+\ExctCplJBdg{b}}$ is stable. If so, then it fits into the module extension
\begin{equation*}
\xymatrix@R=.5ex@C=3em{
\ExctCplCoLimAbutFltrtnQtntBB{\Vect{x}} \ar@{{ |>}->}[r] &
	\SSObjctBB{\infty}{\Vect{x}+\ExctCplJBdg{b}} \ar@{-{ >>}}[r] &
	\ExctCplLimAbutFltrtnQtntBB{\Vect{x}+\ExctCplJBdg{b}+\ExctCplKBdg{c}} 
}
\end{equation*}
\item If $\ExctCplEObjctBB{\infty}{\Vect{x}+\ExctCplJBdg{b}}$ is stable, under which condition(s) is it isomorphic to only one of $\ExctCplCoLimAbutFltrtnQtntBB{\Vect{x}}$ or $\ExctCplLimAbutFltrtnQtntBB{\Vect{x}+\ExctCplJBdg{b}+\ExctCplKBdg{c}}$? - In this situation Boardman \cite[5.2]{JMBoardman1999} speaks of \Defn{weak convergence} of the spectral sequence, with $\ExctCplCoLimAbut{n}$ respectively $\ExctCplLimAbut{n+\sigma}$ being the filtered target object. %
\index{weak convergence of a spectral sequence}\index{spectral sequence!weakly convergent}%
Many classical spectral sequences exhibit one of these two properties.
\end{enumerate}
{\em Stable $\ExctCplEPage{\infty}$}\quad  The answer to question (D.1) is given in the E-infinity extension theorem (\ref{thm:E-InfinityExtensionThm}). It tells us that 	$\SSObjctBB{\infty}{\Vect{x}+\ExctCplJBdg{b}}$ is stable exactly when $\Lim$, applied to this short exact sequence of $\ZCat$-diagrams, is exact:
\begin{equation*}
\xymatrix@R=5ex@C=4em{
\{ \Ker{\ExctCplIMapItrtdBB{r}{\Vect{x}+\ExctCplJBdg{b}+\ExctCplKBdg{c}-r\ExctCplIBdg{a}}}\} \ar@{{ |>}->}[r] &
	\{ \Ker{\ExctCplIMapItrtdBB{r+1}{\Vect{x}+\ExctCplJBdg{b}+\ExctCplKBdg{c}-r\ExctCplIBdg{a}}}\} \ar@{-{ >>}}[r] &
	\left\{ \tfrac{\Ker{\ExctCplIMapItrtdBB{r+1}{\Vect{x}+\ExctCplJBdg{b}+\ExctCplKBdg{c}-r\ExctCplIBdg{a}}}}{\Ker{\ExctCplIMapItrtdBB{r}{\Vect{x}+\ExctCplJBdg{b}+\ExctCplKBdg{c}-r\ExctCplIBdg{a}}}} \right\}
}
\end{equation*}
The associated sequence of limits is exact if and only if this map of $\LimOne$-objects is a monomorphism:
$$
\LimOneOf{\Ker{\ExctCplIMapItrtdBB{r}{\Vect{x}+\ExctCplJBdg{b}+\ExctCplKBdg{c}-r\ExctCplIBdg{a}}}}\longrightarrow \LimOneOf{\Ker{\ExctCplIMapItrtdBB{r+1}{\Vect{x}+\ExctCplJBdg{b}+\ExctCplKBdg{c}-r\ExctCplIBdg{a}}}}
$$
The question as to when this happens is classic and, in general, difficult to answer. The approach commonly found in the literature focuses on conditions under which the $\LimOne$-term of the leftmost $\ZCat$-diagram vanishes, e.g. via the Mittag-Leffler condition (\ref{thm:Mittag-Leffler-Lim1=Zero}). While sufficient, this condition leaves a lot of room for refinement, and this is exactly what we achieve in Section \ref{sec:Filtrations}, with conclusions presented in  Section \ref{sec:ConvergenceII}.

{\em Weak convergence}\quad As to question (D.2), we see immediately:
\begin{enumerate}[(WC.1)]
\setlength{\itemindent}{4ex}
 \item The monomorphism $\ExctCplCoLimAbutFltrtnQtntBB{\Vect{x}}\to \SSObjctBB{\infty}{\Vect{x}+\ExctCplJBdg{b}}$ is an isomorphism if and only if $\SSObjctBB{\infty}{\Vect{x}+\ExctCplJBdg{b}}$ is stable, and $\ExctCplLimAbutFltrtnQtntBB{\Vect{x}+\ExctCplJBdg{b}+\ExctCplKBdg{c}}=0$. If so, we say that $\SSObjctBB{\infty}{\Vect{x}+\ExctCplJBdg{b}}$ \Defn{matches the colimit abutment filtration}; see (\ref{def:E-InfinityRelations}). %
 \index{matches!colimit abutment filtration}%
 \item $\SSObjctBB{\infty}{\Vect{x}+\ExctCplJBdg{b}}$ maps isomorphically to $\ExctCplLimAbutFltrtnQtntBB{\Vect{x}+\ExctCplJBdg{b}+\ExctCplKBdg{c}}$ if and only if $\SSObjctBB{\infty}{\Vect{x}+\ExctCplJBdg{b}}$ is stable and  $\ExctCplCoLimAbutFltrtnQtntBB{\Vect{x}}=0$. If so, we say that $\SSObjctBB{\infty}{\Vect{x}+\ExctCplJBdg{b}}$ \Defn{matches the limit abutment filtration}. %
  \index{matches!limit abutment filtration}%
 \end{enumerate}
 We discuss both scenarios in greater detail in Section \ref{sec:ConvergenceII}. Let us now explain why the E-infinity extension theorem is true.

 {\em The E-infinity extension theorem explained}\quad Elementary computations show that each of the spectral sequence objects $\SSObjctBB{r+1}{\Vect{x}+\ExctCplJBdg{b}}$ fits into the short exact sequence $S(r+1)$; see (\ref{thm:E(r+1)Extension}), compare \cite{BEckmannPJHilton1966}.
\begin{equation*}
\xymatrix@C=4em{
\dfrac{ \Img{ \ExctCplIMapItrtdBB{r}{\Vect{x}} } }{ \Img{ \ExctCplIMapItrtdBB{r+1}{\Vect{x}-\ExctCplIBdg{a}} } } \cong \dfrac{\Ker{ \ExctCplKMapBB{}{\Vect{x} +\ExctCplJBdg{b}} } }{ \ExctCplBndrsBB{r}{\Vect{x} +\ExctCplJBdg{b}} } \ar@{{ |>}->}[r] &
	\dfrac{\ExctCplCyclesBB{r}{\Vect{x}+\ExctCplJBdg{b}}}{\ExctCplBndrsBB{r}{\Vect{x}+\ExctCplJBdg{b}}}=\SSObjctBB{r+1}{\Vect{x}+\ExctCplJBdg{b}} \ar@{-{ >>}}[r] &
	\dfrac{ \ExctCplCyclesBB{r}{\Vect{x} +\ExctCplJBdg{b}} }{ \Ker{ \ExctCplKMapBB{}{\Vect{x} +\ExctCplJBdg{b} } } } \cong \dfrac{ \Ker{ \ExctCplIMapItrtdBB{r+1}{\Vect{x}+\ExctCplJBdg{b} + \ExctCplKBdg{c} - r\ExctCplIBdg{a}} } }{ \Ker{ \ExctCplIMapItrtdBB{r}{\Vect{x}+\ExctCplJBdg{b} + \ExctCplKBdg{c} - r\ExctCplIBdg{a}} } }
}\tag*{$S(r+1)$:}
\end{equation*}
The burning question then is what happens to this short exact sequence as $r\to\infty$. On its kernel side, this requires taking colimits. As $\CoLimOver{\omega}$ is exact, we find immediately the short exact sequence.
\begin{equation*}
\xymatrix@C=3.5em{
\ExctCplCoLimAbutFltrtnQtntBB{\Vect{x}}\cong \dfrac{ \Img{ \ExctCplCoLimAbutMapBB{\Vect{x}} } }{ \Img{ \ExctCplCoLimAbutMapBB{\Vect{x}-\ExctCplIBdg{a} } } } \cong \dfrac{\Ker{ \ExctCplKMapBB{}{\Vect{x} +\ExctCplJBdg{b}} } }{ \ExctCplBndrsBB{\infty}{\Vect{x} +\ExctCplJBdg{b}} } \ar@{{ |>}->}[r] &
	\dfrac{\ExctCplCyclesBB{\infty}{\Vect{x}+\ExctCplJBdg{b}}}{\ExctCplBndrsBB{\infty}{\Vect{x}+\ExctCplJBdg{b}}}=\ExctCplEObjctBB{\infty}{\Vect{x}+\ExctCplJBdg{b} } \ar@{-{ >>}}[r] &
	\dfrac{ \ExctCplCyclesBB{\infty}{\Vect{x}+\ExctCplJBdg{b}} }{ \Ker{ \ExctCplKMapBB{}{\Vect{x} +\ExctCplJBdg{b} } } } \overset{?}{--} \dfrac{ \Ker{ \rho^{\Vect{x}+\ExctCplJBdg{b} + \ExctCplKBdg{c} +\ExctCplIBdg{a} } } }{ \Ker{ \rho^{\Vect{x}+\ExctCplJBdg{b} + \ExctCplKBdg{c} } } } = \ExctCplLimAbutFltrtnQtntBB{\Vect{x}+\ExctCplJBdg{b} + \ExctCplKBdg{c}}
}
\end{equation*}
This computation leaves mysterious the relationship between $\ExctCplLimAbutFltrtnQtntBB{\Vect{x}+\ExctCplJBdg{b} + \ExctCplKBdg{c}}$ and $ \ExctCplCyclesBB{\infty}{\Vect{x}+\ExctCplJBdg{b}} / \Ker{ \ExctCplKMapBB{}{\Vect{x} +\ExctCplJBdg{b} } }$.  We work with two views of this relationship. First, from the right hand side of $S(r+1)$, we have the short exact sequence of $\ZCat$-diagrams:
\begin{equation*}
\xymatrix@R=5ex@C=4em{
\Ker{ \ExctCplIMapItrtdBB{r}{\Vect{x}+\ExctCplJBdg{b} + \ExctCplKBdg{c} - r\ExctCplIBdg{a}} } \ar@{{ |>}->}[r] &
	\Ker{ \ExctCplIMapItrtdBB{r+1}{\Vect{x}+\ExctCplJBdg{b} + \ExctCplKBdg{c} - r\ExctCplIBdg{a}} } \ar@{-{ >>}}[r] &
	\tfrac{ \ExctCplCyclesBB{r}{\Vect{x} +\ExctCplJBdg{b}} }{ \Ker{ \ExctCplKMapBB{}{\Vect{x} +\ExctCplJBdg{b} } } }	
}
\end{equation*}
Passing to limits, in (\ref{thm:E-InfinityExtensionThm}) we identify $\ExctCplLimAbutFltrtnQtntBB{\Vect{x}+\ExctCplJBdg{b} + \ExctCplKBdg{c}}$ as the image factorization in the diagram below.
\begin{equation*}
\xymatrix@R=6ex@C=4em{
\LimOf{\Ker{ \ExctCplIMapItrtdBB{r}{\Vect{x}+\ExctCplJBdg{b} + \ExctCplKBdg{c} - r\ExctCplIBdg{a}} }} \ar@{{ |>}->}[r] \ar@{=}[d] &
	\LimOf{ \Ker{ \ExctCplIMapItrtdBB{r+1}{\Vect{x}+\ExctCplJBdg{b} + \ExctCplKBdg{c} - r\ExctCplIBdg{a}} } } \ar[r] \ar@{=}[d] &
	\LimOf{  \tfrac{ \ExctCplCyclesBB{r}{\Vect{x} +\ExctCplJBdg{b}} }{ \Ker{ \ExctCplKMapBB{}{\Vect{x} +\ExctCplJBdg{b} } } }	 } \cong \tfrac{ \ExctCplCyclesBB{\infty}{\Vect{x} +\ExctCplJBdg{b}} }{ \Ker{ \ExctCplKMapBB{}{\Vect{x} +\ExctCplJBdg{b} } } } \\
\ExctCplLimAbutFltrtnBB{\Vect{x}+\ExctCplJBdg{b} + \ExctCplKBdg{c}} \ar@{{ |>}->}[r] &
	\ExctCplLimAbutFltrtnBB{\Vect{x}+\ExctCplJBdg{b} + \ExctCplKBdg{c}+\ExctCplIBdg{a}} \ar@{-{ >>}}[r] &
	\ExctCplLimAbutFltrtnQtntBB{\Vect{x}+\ExctCplJBdg{b} + \ExctCplKBdg{c}} \ar@{{ |>}->}@<+2ex>[u]
}
\end{equation*}
For an alternate view of the relationship between $\ExctCplLimAbutFltrtnQtntBB{\Vect{x}+\ExctCplJBdg{b} + \ExctCplKBdg{c}}$ and $ \ExctCplCyclesBB{\infty}{\Vect{x}+\ExctCplJBdg{b}} / \Ker{ \ExctCplKMapBB{}{\Vect{x} +\ExctCplJBdg{b} } }$ we identify the objects involved as certain submodules of $\ExctCplDObjctBB{}{\Vect{x}+\ExctCplJBdg{b} + \ExctCplKBdg{c}}$:
\begin{equation*}
\xymatrix@R=6ex@C=4em{
\tfrac{ \ExctCplCyclesBB{\infty}{\Vect{x} +\ExctCplJBdg{b}} }{ \Ker{ \ExctCplKMapBB{}{\Vect{x} +\ExctCplJBdg{b} } } } \ar[r]^-{\cong} &
	\SetIntrsctn{ I^{\omega}_{\Vect{x}+\ExctCplJBdg{b} + \ExctCplKBdg{c}} }{ \Ker{\ExctCplIMapBB{}{\Vect{x}+\ExctCplJBdg{b} + \ExctCplKBdg{c} } } } \\
\ExctCplLimAbutFltrtnQtntBB{\Vect{x}+\ExctCplJBdg{b} + \ExctCplKBdg{c}} \ar@{{ |>}->}[u] \ar[r]_-{\cong} &
	\SetIntrsctn{ \bar{I}_{\Vect{x}+\ExctCplJBdg{b} + \ExctCplKBdg{c}} }{ \Ker{\ExctCplIMapBB{}{\Vect{x}+\ExctCplJBdg{b} + \ExctCplKBdg{c} } } } \ar@{{ |>}->}[u]
}
\end{equation*}
Here $I^{\omega}_{\Vect{x}+\ExctCplJBdg{b} + \ExctCplKBdg{c}} = \FamIntrsctn{r\geq 1}{\Img{\ExctCplIMapItrtdBB{r}{\Vect{x}+\ExctCplJBdg{b} + \ExctCplKBdg{c} - r\ExctCplIBdg{a}}}}$, which contains $\bar{I}_{\Vect{x}+\ExctCplJBdg{b} + \ExctCplKBdg{c}} = \Img{\ExctCplLimAbut{\Vect{x}+\sigma} \to \ExctCplDObjctBB{}{\Vect{x}+\ExctCplJBdg{b} + \ExctCplKBdg{c}} }$, the stable image subobject; see (\ref{def:ImageQuotientSubDiagrams}). Thus $\ExctCplEPage{\infty}$ is stable if and only if the vertical arrow on the right is an isomorphism. Sufficient for this to happen is $I^{\omega}_{\Vect{x}+\ExctCplJBdg{b} + \ExctCplKBdg{c}} = \bar{I}_{\Vect{x}+\ExctCplJBdg{b} + \ExctCplKBdg{c}}$, a significant weakening of the Mittag-Leffler condition.

Let us now turn to the details of the development outlined above.

\begin{definition}[Universal limit/colimit abutment of an exact couple]
\label{def:ExactCouple-UniversalLim/CoLimAbutment}
Consider a regular exact couple $\ExctCpl{C}$\MSComp, with $\sigma=\DtrmnntOfMtrx{[\ExctCplIBdg{a}\ \ (\ExctCplJBdg{b}+\ExctCplKBdg{c})]}\in \Set{\pm 1}$. As in (\ref{thm:ExactCouple-ZGradedZ-Diagrams}), for $n\in \ZNr$\MSComp, let $D(n)$ be the $\ZCat$-diagram: %
\begin{equation*}
\cdots \longrightarrow \ExctCplDObjctBB{}{\Vect{x} - \ExctCplIBdg{a}} \XRA{ i } \ExctCplDObjctBB{}{\Vect{x} } \XRA{ i } \ExctCplDObjctBB{}{\Vect{x} + \ExctCplIBdg{a}} \longrightarrow \cdots \qquad \text{with}\qquad \Vect{x}\DefEq \sigma n (\ExctCplJBdg{b}+\ExctCplKBdg{c}) 
\vspace{-\baselineskip}
\end{equation*}
\index[not]{$D(n)$ - $n$-th $\ZCat$-diagram of regular exact couple}%
\index{exact couple!colimit abutment}\index{exact couple!limit abutment}%
\vspace{-1.7ex}
\begin{enumerate}[(i)]
\item The \Defn{colimit abutment} of $\ExctCpl{C}$ in position $n$ is $\ExctCplCoLimAbut{n} \DefEq \CoLimOfOver{ \ExctCplDObjctBB{}{}(n) }{\ZCat}$. %
\index[not]{$\ExctCplCoLimAbut{n}$ - colimit abutment of exact couple}%
\item The \Defn{limit abutment} of $\ExctCpl{C}$ in position $n$ is  $\ExctCplLimAbut{n} \DefEq \LimOfOver{ \ExctCplDObjctBB{}{}(n) }{\ZCat}$. %
\index[not]{$\ExctCplLimAbut{n}$ - limit abutment of exact couple}%
\end{enumerate}
\end{definition}

\begin{lemma}[Adjacent images of $i$-iterations]
\label{thm:I-Iterations-AdjacentImages}%
For a regular exact couple $(\ExctCplEPage{},\ExctCplDPage{},\ExctCplIMap{},\ExctCplJMap{},\ExctCplKMap{})$ and $1\leq r<\infty$, there is a functorial isomorphism
\begin{equation*}
\xymatrix@C=5ex{
\dfrac{\Img{\ExctCplIMapItrtdBB{r}{\Vect{x}}} }{\Img{ \ExctCplIMapItrtdBB{r+1}{\Vect{x}-\ExctCplIBdg{a}}} }
          \ar[rr]^-{\alpha^{r}_{\Vect{x}} }_-{\cong} & &
     \dfrac{\Ker{ \ExctCplKMapBB{}{\Vect{x} +\ExctCplJBdg{b}}} }{ \ExctCplBndrsBB{r}{\Vect{x} +\ExctCplJBdg{b}}}
}
\end{equation*}
\end{lemma}
\begin{proof}
The two bottom rows of the diagram below are exact. So, we may apply the snake lemma to obtain the exact top row of kernels:
\begin{equation*}
\xymatrix@R=4ex@C=3em{
& \Ker{\ExctCplIMapItrtdBB{r+1}{\Vect{x} - \ExctCplIBdg{a}} } \ar[r] \ar@{{ |>}->}[d] &
	\Ker{\ExctCplIMapItrtdBB{r}{\Vect{x}} } \ar@{-{ >>}}[r] \ar@{{ |>}->}[d] &
	\ExctCplJMapBB{}{\Vect{x}}\left( \Ker{\ExctCplIMapItrtdBB{r}{\Vect{x}} } \right) \ar@{{ |>}->}[d] \\
\Ker{\ExctCplIMapBB{}{\Vect{x} - \ExctCplIBdg{a}} } \ar@{{ |>}->}[r] \ar[d] &
	\ExctCplDObjctBB{}{\Vect{x}- \ExctCplIBdg{a} } \ar@{-{ >>}}[d] \ar[r]^-{\ExctCplIMapBB{}{\Vect{x} - \ExctCplIBdg{a}} } &
	\ExctCplDObjctBB{}{\Vect{x}} \ar@{-{ >>}}[d] \ar@{-{ >>}}[r]^-{\ExctCplJMapBB{}{\Vect{x}} } &
	\Img{ \ExctCplJMapBB{}{\Vect{x}} } \ar@{.{ >>}}[d] \\
0 \ar[r] &
	\Img{\ExctCplIMapBB{r+1}{\Vect{x} - \ExctCplIBdg{a}} } \ar@{{ |>}->}[r] &
	\Img{\ExctCplIMapBB{r}{\Vect{x}} } \ar@{-{ >>}}[r] &
	\dfrac{\Img{\ExctCplIMapBB{r}{\Vect{x}}}}{ \Img{ \ExctCplIMapBB{r+1}{\Vect{x} - \ExctCplIBdg{a}}} }
}
\end{equation*}
Thus $\frac{\Img{\ExctCplIMapBB{r}{\Vect{x}}}}{ \Img{ \ExctCplIMapBB{r+1}{\Vect{x} - \ExctCplIBdg{a}}} } \cong \frac{ \Img{ \ExctCplJMapBB{}{\Vect{x}} } }{ \ExctCplJMapBB{}{\Vect{x}} \left( \Ker{\ExctCplIMapItrtdBB{r}{\Vect{x}} } \right) }$. Recalling (\ref{def:ExactCoupleItems}) that $\ExctCplBndrsBB{r}{\Vect{x} +\ExctCplJBdg{b}} = \ExctCplJMapBB{}{\Vect{x}}\left( \Ker{\ExctCplIMapItrtdBB{r}{\Vect{x}} }\right)$ and that $\Ker{ \ExctCplKMapBB{}{\Vect{x} +\ExctCplJBdg{b}} } = \Img{ \ExctCplJMapBB{}{\Vect{x}} }$ completes the proof.
\end{proof}

With similar reasoning we obtain:

\begin{lemma}[Adjacent kernels of $i$-iterations]
\label{thm:I-Iterations-AdjacentKernels}%
For an exact couple $(\ExctCplEPage{},\ExctCplDPage{},\ExctCplIMap{},\ExctCplJMap{},\ExctCplKMap{})$ and $1\leq r<\infty$, the following terms are isomorphic.
\begin{equation*}
\xymatrix@C=-2em@R=5ex{
\dfrac{ \Ker{\ExctCplIMapItrtdBB{r+1}{\Vect{x}+\ExctCplJBdg{b}+\ExctCplKBdg{c}-r\ExctCplIBdg{a} } } }{ \Ker{\ExctCplIMapItrtdBB{r}{\Vect{x}+\ExctCplJBdg{b}+\ExctCplKBdg{c}-r\ExctCplIBdg{a} } } } \ar[rr]^{\beta^{r}_{\Vect{x}+\ExctCplJBdg{b} } } \ar[rd]^(.6){\cong}_(.35){ \overline{\ExctCplIMapItrtdBB{r}{\Vect{x}+\ExctCplJBdg{b}+\ExctCplKBdg{c}-r\ExctCplIBdg{a} } } } &&
  \dfrac{ \ExctCplCyclesBB{r}{ \Vect{x} +\ExctCplJBdg{b} } }{ \Ker{ \ExctCplKMapBB{}{\Vect{x} +\ExctCplJBdg{b} } } } \ar[ld]_(.6){\cong}^(.4){ \overline{ \ExctCplKMapBB{}{\Vect{x}+\ExctCplJBdg{b} } } } \\
& \SetIntrsctn{ \Ker{ \ExctCplIMapBB{1}{\Vect{x}+\ExctCplJBdg{b}+\ExctCplKBdg{c} } } }{ \Img{ \ExctCplIMapBB{r}{\Vect{x}+\ExctCplJBdg{b}+\ExctCplKBdg{c}-r\ExctCplIBdg{a} } } }
} \tag*{$\lozenge$}
\end{equation*}
\end{lemma}
\begin{proof}
We begin with the top left corner and compute, using structural properties of the exact couple, to encounter the claimed isomorphisms in clockwise order.
\begin{equation*}
\begin{array}{rcl}
\dfrac{\Ker{ \ExctCplIMapItrtdBB{r+1}{\Vect{x}+ \ExctCplJBdg{b}+\ExctCplKBdg{c}- r\ExctCplIBdg{a} }} }{ \Ker{ \ExctCplIMapItrtdBB{r}{\Vect{x}+ \ExctCplJBdg{b}+\ExctCplKBdg{c}- r\ExctCplIBdg{a} }} } & = & \SetIntrsctn{ \Ker{ \ExctCplIMapBB{1}{\Vect{x}+\ExctCplJBdg{b}+\ExctCplKBdg{c}} } }{ \Img{ \ExctCplIMapItrtdBB{r}{\Vect{x}+\ExctCplJBdg{b}+\ExctCplKBdg{c} - r\ExctCplIBdg{a} } } } \\
	& = & \SetIntrsctn{ \Img{ \ExctCplKMapBB{}{\Vect{x} +\ExctCplJBdg{b} } } }{ \Img{ \ExctCplIMapItrtdBB{r}{ \Vect{x}+\ExctCplJBdg{b}+\ExctCplKBdg{c} - r\ExctCplIBdg{a} } } } \\
	& = & \dfrac{ \left( \ExctCplKMapBB{}{ \Vect{x}+\ExctCplJBdg{b} } \right)^{-1} \Img{ \ExctCplIMapItrtdBB{r}{ \Vect{x}+\ExctCplJBdg{b}+\ExctCplKBdg{c} - r\ExctCplIBdg{a} } } }{ \Ker{ \ExctCplKMapBB{}{ \Vect{x} +\ExctCplJBdg{b} } } } \\
	& \cong & \dfrac{ \ExctCplCyclesBB{r}{ \Vect{x}+ \ExctCplJBdg{b} } }{\Ker{ \ExctCplKMapBB{}{ \Vect{x} +\ExctCplJBdg{b} } } }
\end{array}
\end{equation*}
In the above computation, we have used the isomorphism in the diagram below. Consider a pair of composable maps $A \xrightarrow{\alpha} B \xrightarrow{\beta} C$:
\begin{equation*}
\xymatrix@C=4em@R=5ex{
\Ker{\alpha} \ar@{=}[d] \ar@{{ |>}->}[r] &
	\Ker{\beta\alpha} \ar@{-{ >>}}[r] \ar@{=}[d] &
	\dfrac{\Ker{\beta\alpha}}{\Ker{\alpha}} \ar[d]^{\cong} \\
\Ker{\alpha} \ar@{{ |>}->}[r] &
	\alpha^{-1}\Ker{\beta} \ar@{-{ >>}}[r]_-{\alpha} &
	\SetIntrsctn{\Img{\alpha} }{ \Ker{\beta} }
}
\end{equation*}
The claim follows.
\end{proof}

\begin{corollary}[$\ExctCplEPage{r+1}$-extension]
\label{thm:E(r+1)Extension}
For a regular exact couple $(\ExctCplEPage{},\ExctCplDPage{},\ExctCplIMap{},\ExctCplJMap{},\ExctCplKMap{})$ and $1\leq r<\infty$, the terms $\ExctCplEPage{r+1}$ fit into $1$-step module extensions as follows.
\begin{equation*}
\xymatrix@C=4em{
\dfrac{ \Img{ \ExctCplIMapItrtdBB{r}{\Vect{x}} } }{ \Img{ \ExctCplIMapItrtdBB{r+1}{\Vect{x}-\ExctCplIBdg{a}} } } \cong \dfrac{\Ker{ \ExctCplKMapBB{}{\Vect{x} +\ExctCplJBdg{b}} } }{ \ExctCplBndrsBB{r}{\Vect{x} +\ExctCplJBdg{b}} } \ar@{{ |>}->}[r] &
	\dfrac{\ExctCplCyclesBB{r}{\Vect{x}+\ExctCplJBdg{b}}}{\ExctCplBndrsBB{r}{\Vect{x}+\ExctCplJBdg{b}}}=\ExctCplEObjctBB{r+1}{\Vect{x}+\ExctCplJBdg{b} } \ar@{-{ >>}}[r] &
	\dfrac{ \ExctCplCyclesBB{r}{ \Vect{x} +\ExctCplJBdg{b} } }{ \Ker{ \ExctCplKMapBB{}{\Vect{x} +\ExctCplJBdg{b} } } } \cong \dfrac{ \Ker{ \ExctCplIMapItrtdBB{r+1}{\Vect{x}+\ExctCplJBdg{b} + \ExctCplKBdg{c} - r\ExctCplIBdg{a}} } }{ \Ker{ \ExctCplIMapItrtdBB{r}{\Vect{x}+\ExctCplJBdg{b} + \ExctCplKBdg{c} - r\ExctCplIBdg{a}} } }
}
\end{equation*}
\end{corollary}
\begin{proof}
The isomorphisms at the ends of the short exact sequence are from (\ref{thm:I-Iterations-AdjacentImages}) and  (\ref{thm:I-Iterations-AdjacentKernels}).
\end{proof}

We wish to describe how the module extension in (\ref{thm:E(r+1)Extension}) responds to the passage $r\to \infty$. This is accomplished in the E-infinity extension theorem (\ref{thm:E-InfinityExtensionThm}). First, we deal with adjacent filtration quotients of the colimit abutment, and the key to doing so is (\ref{thm:F_pA-CoLimit}). It asserts that the filtration stages of the colimit abutment satisfy $L_{\Vect{x}}\cong \CoLimOfOver{ \Img{\ExctCplIMapItrtdBB{r}{\Vect{x}} } }{r}$ and implies Proposition  \ref{thm:ImageFiltrationQuotients-Identification}.

\begin{proposition}[Identification of image filtration quotients]
\label{thm:ImageFiltrationQuotients-Identification}
Given an exact couple $(\ExctCplEPage{},\ExctCplDPage{},\ExctCplIMap{},\ExctCplJMap{},\ExctCplKMap{})$, the image filtration of its colimit abutment has the following adjacent filtration quotients
\begin{equation*}
\xymatrix@C=4em{
\ExctCplCoLimAbutFltrtnBB{\Vect{x}-\Vect{a}} \ar@{{ |>}->}[r] &
	\ExctCplCoLimAbutFltrtnBB{\Vect{x}} \ar@{-{ >>}}[r] &
	\ExctCplCoLimAbutFltrtnQtntBB{\Vect{x}} \cong \dfrac{ \Ker{ \ExctCplKMapBB{}{\Vect{x}+\ExctCplJBdg{b}} } }{ \ExctCplBndrsBB{\infty}{\Vect{x}+\ExctCplJBdg{b}} }
}
\end{equation*}
\end{proposition}
\begin{proof}
Via (\ref{thm:I-Iterations-AdjacentImages}) we obtain this short exact sequence of $\omega$\MSComp-diagrams:
\begin{equation*}
\xymatrix@R=5ex@C=4em{
\Img{\ExctCplIMapItrtdBB{r+1}{\Vect{x}-\Vect{a}}} \ar@{{ |>}->}[r] &
	\Img{\ExctCplIMapItrtdBB{r}{\Vect{x}} } \ar@{-{ >>}}[r] &
	\dfrac{ \Ker{ \ExctCplKMapBB{}{\Vect{x}+\ExctCplJBdg{b}} } }{ \ExctCplBndrsBB{r}{\Vect{x}+\ExctCplJBdg{b}} }
}
\end{equation*}
As $\CoLimOver{\omega}$ is exact (\ref{thm:CoLim^ZExactFunctor}), the upper sequence in the diagram below is short exact.
\begin{equation*}
\xymatrix@R=5ex@C=4em{
\CoLimOfOver{\Img{\ExctCplIMapItrtdBB{r+1}{\Vect{x}-\Vect{a}}}}{r} \ar@{{ |>}->}[r] \ar[d]_{\cong} &
	\CoLimOfOver{\Img{\ExctCplIMapItrtdBB{r}{\Vect{x}} }}{r} \ar@{-{ >>}}[r] \ar[d]_{\cong} &
	\CoLimOfOver{\dfrac{ \Ker{ \ExctCplKMapBB{}{\Vect{x}+\ExctCplJBdg{b}} } }{ \ExctCplBndrsBB{r}{\Vect{x}+\ExctCplJBdg{b}} }}{r} \ar[d]^{\cong} \\
\ExctCplCoLimAbutFltrtnBB{\Vect{x}-\Vect{a}} \ar@{{ |>}->}[r] &
	\ExctCplCoLimAbutFltrtnBB{\Vect{x}} \ar@{-{ >>}}[r] &
	\dfrac{ \Ker{ \ExctCplKMapBB{}{\Vect{x}+\ExctCplJBdg{b}} } }{ \ExctCplBndrsBB{\infty}{\Vect{x}+\ExctCplJBdg{b}} }
}
\end{equation*}
The left and center vertical arrows are isomorphisms by (\ref{thm:F_pA-CoLimit}). Further, via the construction of iterated boundary terms in a spectral sequence (\ref{thm:E-infinityPreparation}), $\ExctCplBndrsBB{\infty}{\Vect{x}+\Vect{b}}\cong \CoLimOfOver{\ExctCplBndrsBB{r}{\Vect{x}}}{r}$. So, the vertical arrow on the right is an isomorphism as well. This proves the claim.
\end{proof}

For every regular exact couple, we are now ready to describe the relationship between $\ExctCplCoLimAbutFltrtnQtntBB{\Vect{x}}$ and the transfinite objects $\ExctCplEObjctBB{\tau+1}{\Vect{x}+\ExctCplJBdg{b}}$.

\begin{lemma}[Transfinite E-extension lemma]
\label{thm:Transfinite-E-Extension}
Let $\ExctCpl{C}$ be a regular exact couple whose structure maps $\ExctCplIMap{},\ExctCplJMap{},\ExctCplKMap{}$ have bidegrees $\ExctCplIBdg{a},\ExctCplJBdg{b},\ExctCplKBdg{c}$, respectively. Then for any ordinal $\tau\geq \omega$, the objects $\ExctCplEObjctBB{\tau+1}{\Vect{x}+\ExctCplJBdg{b}}$ form module extensions.
\begin{equation*}
\xymatrix@R=5ex@C=4em{
\ExctCplCoLimAbutFltrtnQtntBB{\Vect{x}} \ar@{{ |>}->}[r] &
	\ExctCplEObjctBB{\tau+1}{\Vect{x}+\ExctCplJBdg{b}} \ar@{-{ >>}}[r] &
	\dfrac{\ExctCplCyclesBB{\tau}{\Vect{x}+\ExctCplJBdg{b}}}{\Ker{\ExctCplKMapBB{}{\Vect{x} + \ExctCplJBdg{b}}}} \cong \ZDiagImg{\tau}{\Vect{x}+\ExctCplJBdg{b}+\ExctCplKBdg{c}}{} \intrsctn \Ker{\ExctCplIMapBB{}{\Vect{x}+\ExctCplJBdg{b} + \ExctCplKBdg{c}}}
}
\end{equation*}
Moreover, these extensions are functorial with respect to morphisms of regular exact couples.
\end{lemma}
\begin{proof}
We begin by examining what we know at the least infinite ordinal $\omega$. Via (\ref{thm:ImageFiltrationQuotients-Identification}) we are able to splice two short exact sequences as in this commutative diagram.
\begin{equation*}
\xymatrix@R=5ex@C=3em{
{\color{OliveGreen} \ExctCplCoLimAbutFltrtnBB{\Vect{x}-\ExctCplIBdg{a}} } \ar@[OliveGreen]@{{ |>}->}[r] &
	{\color{OliveGreen} \ExctCplCoLimAbutFltrtnBB{\Vect{x}}} \ar[rr] \ar@[OliveGreen]@{-{ >>}}[d] &&
	{\color{deepmagenta} \SSObjctBB{\omega+1}{\Vect{x}+\ExctCplJBdg{b}} } \ar@[deepmagenta]@{-{ >>}}[r] &
	{\color{deepmagenta} \dfrac{ \ExctCplCyclesBB{\omega}{\Vect{x}+\ExctCplJBdg{b}}}{ \Ker{ \ExctCplKMapBB{}{\Vect{x} +\ExctCplJBdg{b}}} } \cong \dfrac{ \ExctCplCyclesBB{\omega}{\Vect{x}+\ExctCplJBdg{b}}}{ \Img{ \ExctCplJMapBB{}{\Vect{x}} } } } \\
& {\color{OliveGreen} \ExctCplCoLimAbutFltrtnQtntBB{\Vect{x}}{} } \ar[rr]_-{\cong}^-{\text{(\ref{thm:ImageFiltrationQuotients-Identification})} } &&
	{\color{deepmagenta} \dfrac{ \Ker{ \ExctCplKMapBB{}{\Vect{x}+ \ExctCplJBdg{b} } } }{ \ExctCplBndrsBB{\omega}{\Vect{x} + \ExctCplJBdg{b}} } }  \ar@[deepmagenta]@{{ |>}->}[u]
}
\end{equation*}
The sequence rendered in {\color{OliveGreen} green} on the left is exact by definition. The sequence rendered in {\color{deepmagenta} purple} on the right is seen to be exact upon recalling (\ref{thm:E-infinityAs(Co-)Limit}) that $\SSObjctBB{\omega+1}{\Vect{x}+\ExctCplJBdg{b}}=\SSObjctBB{\infty}{\Vect{x}+\ExctCplJBdg{b}} \cong \ExctCplCyclesBB{\infty}{\Vect{x}+\ExctCplJBdg{b}}/ \ExctCplBndrsBB{\infty}{\Vect{x}+\ExctCplJBdg{b}}$, while
\begin{equation*}
\ExctCplBndrsBB{\omega}{\Vect{x}+ \ExctCplJBdg{b}}=\ExctCplBndrsBB{\infty}{\Vect{x}+ \ExctCplJBdg{b}}\subseteq \Ker{ \ExctCplKMapBB{}{\Vect{x}+\ExctCplJBdg{b}} }\subseteq \ExctCplCyclesBB{\infty}{\Vect{x}+\ExctCplJBdg{b}}=\ExctCplCyclesBB{\omega}{\Vect{x}+\ExctCplJBdg{b}}
\end{equation*}
Now, consider an ordinal $\tau\geq \omega$. We have the following commutative diagram.
\begin{equation*}
\xymatrix@R=5ex@C=3em{
\dfrac{\Ker{\ExctCplKMapBB{}{\Vect{x}+\ExctCplJBdg{b}}}}{\ExctCplBndrsBB{\omega}{\Vect{x}+\ExctCplJBdg{b}}} = \dfrac{\Ker{\ExctCplKMapBB{}{\Vect{x}+\ExctCplJBdg{b}}}}{\ExctCplBndrsBB{\tau}{\Vect{x}+\ExctCplJBdg{b}}} \ar@{{ |>}->}[r] &
	\dfrac{\ExctCplCyclesBB{\tau}{\Vect{x}+\ExctCplJBdg{b}}}{\ExctCplBndrsBB{\tau}{\Vect{x}+\ExctCplJBdg{b}}}=\ExctCplEObjctBB{\tau+1}{\Vect{x}+\ExctCplJBdg{b}} \ar@{-{ >>}}[r] &
		\dfrac{\ExctCplCyclesBB{\tau}{\Vect{x}+\ExctCplJBdg{b}}}{\Ker{\ExctCplKMapBB{}{\Vect{x} + \ExctCplJBdg{b}}}} \ar[r]_-{\cong} &
		\ZDiagImg{\tau}{\Vect{x}+\ExctCplJBdg{b}+\ExctCplKBdg{c}}{} \intrsctn \Ker{\ExctCplIMapBB{}{\Vect{x}+\ExctCplJBdg{b} +\ExctCplKBdg{c}}}
}
\end{equation*}
Here, we used that (\ref{thm:E^tau,tau>omega}), for ordinals $\tau\geq \omega$, $\ExctCplBndrsBB{\tau}{\Vect{x}+\ExctCplJBdg{b}} = \ExctCplBndrsBB{\omega}{\Vect{x}+\ExctCplJBdg{b}}$.  Recalling then the construction of $\ExctCplCyclesBB{\tau}{\Vect{x}+\ExctCplJBdg{b}}$ from (\ref{def:ExactCoupleItems})  yields the claim.
\end{proof}

\begin{theorem}[Stable E-extension theorem]
\label{thm:Stable-E-Extension}
Let $\ExctCpl{C}$ be a regular exact couple whose structure maps $\ExctCplIMap{},\ExctCplJMap{},\ExctCplKMap{}$ have bidegrees $\ExctCplIBdg{a},\ExctCplJBdg{b},\ExctCplKBdg{c}$, respectively. Then the stable $E$\MSComp-objects of $\ExctCpl{C}$ form module extensions:
\begin{equation*}
\xymatrix@R=5ex@C=4em{
\ExctCplCoLimAbutFltrtnQtntBB{\Vect{x}} \ar@{{ |>}->}[r] &
	\ExctCplStblEObjctBB{\Vect{x}+\ExctCplJBdg{b}} \ar@{-{ >>}}[r] &
	\ExctCplLimAbutFltrtnQtntBB{\Vect{x}+\ExctCplJBdg{b}+\ExctCplKBdg{c}}
}
\end{equation*}
Moreover, these extensions are functorial with respect to morphisms of regular exact couples.
\end{theorem}
\begin{proof}
Let $\alpha$ be an infinite ordinal at which the image subobjects $\ZDiagImg{\alpha}{}{}$ of $\ExctCplDPage{}$ are stable in the sense of (\ref{sec:ImageFactorization-Z-Diagrams}). Further, let $n=\DtrmnntOf{[\ExctCplIBdg{a}\ \Vect{x}]}$ and $\sigma = \DtrmnntOf{[\ExctCplIBdg{a}\ (\ExctCplJBdg{b}+\ExctCplKBdg{c})]}$. Then we have this morphism of short exact sequences:
\begin{equation*}
\xymatrix@R=5ex@C=4em{
\ExctCplLimAbutFltrtnBB{\Vect{x}+\ExctCplJBdg{b}+\ExctCplKBdg{c}} \ar@{{ |>}->}[r] \ar@{{ |>}->}[d] &
	\ExctCplLimAbut{n+\sigma} \ar@{-{ >>}}[r] \ar@{=}[d] &
	\ZDiagImg{\alpha}{\Vect{x}+\ExctCplJBdg{b}+\ExctCplKBdg{c}}{} \ar@{-{ >>}}[d] \\
\ExctCplLimAbutFltrtnBB{\Vect{x}+\ExctCplJBdg{b}+\ExctCplKBdg{c} + \ExctCplIBdg{a}} \ar@{{ |>}->}[r] &
	\ExctCplLimAbut{n+\sigma} \ar@{-{ >>}}[r] &
	\ZDiagImg{\alpha}{\Vect{x}+\ExctCplJBdg{b}+\ExctCplKBdg{c} + \ExctCplIBdg{a}}{}
}
\end{equation*}
The snake lemma yields a functorial isomorphism
\begin{equation*}
\ExctCplLimAbutFltrtnQtntBB{\Vect{x}+\ExctCplJBdg{b}+\ExctCplKBdg{c}} = 	\CoKer{\ExctCplLimAbutFltrtnBB{\Vect{x}+\ExctCplJBdg{b}+\ExctCplKBdg{c}}\to \ExctCplLimAbutFltrtnBB{\Vect{x}+\ExctCplJBdg{b}+\ExctCplKBdg{c} + \ExctCplIBdg{a}}} \cong  \Ker{\ZDiagImg{\alpha}{\Vect{x}+\ExctCplJBdg{b}+\ExctCplKBdg{c}}{}\to \ZDiagImg{\alpha}{\Vect{x}+\ExctCplJBdg{b}+\ExctCplKBdg{c} + \ExctCplIBdg{a}}{}}  =  \ZDiagImg{\alpha}{\Vect{x}+\ExctCplJBdg{b}+\ExctCplKBdg{c}}{} \intrsctn \Ker{\ExctCplIMapBB{}{\Vect{x}+\ExctCplJBdg{b}+\ExctCplKBdg{c}}}
\end{equation*}
Now the transfinite E-extension lemma (\ref{thm:Transfinite-E-Extension}) yields this short exact sequence
\begin{equation*} 
\xymatrix@R=5ex@C=4em{
\ExctCplCoLimAbutFltrtnQtntBB{\Vect{x}} \ar@{{ |>}->}[r] &
	\ExctCplStblEObjctBB{\Vect{x}+\ExctCplJBdg{b}}=\ExctCplEObjctBB{\alpha+1}{\Vect{x}+\ExctCplJBdg{b}} \ar@{-{ >>}}[r] &
	\ZDiagImg{\alpha}{\Vect{x}+\ExctCplJBdg{b}+\ExctCplKBdg{c}}{} \intrsctn \Ker{\ExctCplIMapBB{}{\Vect{x}+\ExctCplJBdg{b} + \ExctCplKBdg{c}}} \cong \ExctCplLimAbutFltrtnQtntBB{\Vect{x}+\ExctCplJBdg{b}+\ExctCplKBdg{c}}
}
\end{equation*}
This is what was to be shown.
\end{proof}

\begin{theorem}[E-infinity extension theorem]
\label{thm:E-InfinityExtensionThm}%
For a regular exact couple $(\ExctCplEPage{},\ExctCplDPage{},\ExctCplIMap{},\ExctCplJMap{},\ExctCplKMap{})$, its stable $E$\MSComp-objects $\ExctCplStblEObjct$ and its spectral sequence based $E$\MSComp-infinity objects $\ExctCplEPage{\infty}$ are related via this morphism (pullback) of short exact sequences:
\begin{equation*}
\xymatrix@R=5ex@C=5em{
\ExctCplCoLimAbutFltrtnQtntBB{\Vect{x}} \ar@{{ |>}->}[r] \ar@{=}[d] &
	\ExctCplStblEObjctBB{\Vect{x}+\ExctCplJBdg{b}} \ar@{-{ >>}}[r] \ar@{{ |>}->}[d] \PullLU{rd} &
	\ExctCplLimAbutFltrtnQtntBB{\Vect{x}+\ExctCplJBdg{b}+\ExctCplKBdg{c}} \ar@{{ |>}->}[d]^{M} \\
\ExctCplCoLimAbutFltrtnQtntBB{\Vect{x}} \ar@{{ |>}->}[r] &
	\ExctCplEObjctBB{\infty}{\Vect{x}+\ExctCplJBdg{b}} \ar@{-{ >>}}[r] &
	\frac{ \ExctCplCyclesBB{\infty}{\Vect{x}+\ExctCplJBdg{b} } }{ \Img{ \ExctCplJMapBB{}{\Vect{x}} } }
}
\end{equation*}
The map $M$ is the kernel part of the image factorization of the map $t$ in the $5$-term exact sequence below.
\begin{equation*}
\xymatrix@R=5ex@C=2.5em{
\ExctCplLimAbutFltrtnBB{\Vect{x}+\ExctCplJBdg{b}+\ExctCplKBdg{c}} \ar@{{ |>}->}[r] &
	\ExctCplLimAbutFltrtnBB{\Vect{x}+\ExctCplIBdg{a}+\ExctCplJBdg{b}+\ExctCplKBdg{c}} \ar[r]^-{t} &
	\frac{ \ExctCplCyclesBB{\infty}{\Vect{x}+\ExctCplJBdg{b} } }{ \Ker{ \ExctCplKMapBB{}{\Vect{x}+\ExctCplJBdg{b}} } } \ar[r] &
	\LimOneOfOver{ \Ker{ \ExctCplIMapItrtdBB{r}{\Vect{x}+\ExctCplJBdg{b}+\ExctCplKBdg{c}-r\ExctCplIBdg{a} } } }{\ZCat} \ar[r] &
	\LimOneOfOver{ \Ker{ \ExctCplIMapItrtdBB{r+1}{\Vect{x}+\ExctCplJBdg{b}+\ExctCplKBdg{c} -r\ExctCplIBdg{a} } } }{\ZCat} 
}
\end{equation*}
Both of the above diagrams depend functorially on morphisms of the underlying exact couples.
\end{theorem}
\begin{proof}
The stable E-extension theorem (\ref{thm:Stable-E-Extension}) yields the short exact sequence at the top.  The transfinite E-extension lemma (\ref{thm:Transfinite-E-Extension}) yields the morphism of short exact sequences, keeping in mind the identities $\ExctCplCyclesBB{\omega}{}=\ExctCplCyclesBB{\infty}{}$ and $\ExctCplEPage{\omega+1}=\ExctCplEPage{\infty}$.

It remains to identify the map $M$. Via (\ref{thm:I-Iterations-AdjacentKernels}) we obtain the short exact sequence of $\omega$-diagrams  below.
\begin{equation*}
\xymatrix@R=5ex@C=3em{
\Ker{ \ExctCplIMapItrtdBB{r}{\Vect{x}+\ExctCplJBdg{b}+\ExctCplKBdg{c} -r\ExctCplIBdg{a} } } \ar@{{ |>}->}[r] &
	\Ker{ \ExctCplIMapItrtdBB{r+1}{\Vect{x}+\ExctCplJBdg{b}+\ExctCplKBdg{c} -r\ExctCplIBdg{a} } } \ar@{-{ >>}}[r] &
	\dfrac{ \ExctCplCyclesBB{r}{\Vect{x}+\ExctCplJBdg{b}}}{ \Ker{\ExctCplKMapBB{}{\Vect{x}+\ExctCplJBdg{b} } } } \cong \dfrac{ \ExctCplCyclesBB{r}{\Vect{x}+\ExctCplJBdg{b}}}{ \Img{\ExctCplJMapBB{}{\Vect{x}} } }
}
\end{equation*}
We will show that its associated $6$\MSComp-term exact $\Lim$-$\LimOne$-sequence contains the claimed  $5$-term exact sequence. Indeed, the first two terms in this 6-term sequence are identified by (\ref{thm:Kernel/ImageSequencesOfZ-Diagram}):
\begin{equation*}
\begin{array}{rcl}
\ExctCplLimAbutFltrtnBB{\Vect{x}+\ExctCplJBdg{b}+\ExctCplKBdg{c}}
	& \cong & \LimOfOver{ \Ker{ \ExctCplIMapItrtdBB{r}{\Vect{x}+\ExctCplJBdg{b}+\ExctCplKBdg{c} -r\ExctCplIBdg{a} } } }{r} \\
\ExctCplLimAbutFltrtnBB{\Vect{x}+\ExctCplIBdg{a}+\ExctCplJBdg{b}+\ExctCplKBdg{c}} 
	& \cong & \LimOfOver{ \Ker{ \ExctCplIMapItrtdBB{r+1}{\Vect{x}+\ExctCplJBdg{b}+\ExctCplKBdg{c} -r\ExctCplIBdg{a} } } }{r}
\end{array}
\end{equation*}
We compute the third term in this $6$-term sequence via the following claim:
\begin{equation*}
\LimOfOver{ \dfrac{\ExctCplCyclesBB{r}{\Vect{x}+\ExctCplJBdg{b}} }{ \Ker{\ExctCplKMapBB{}{\Vect{x}+\ExctCplJBdg{b} } } } }{r} \cong \dfrac{ \ExctCplCyclesBB{\infty}{\Vect{x}+\ExctCplJBdg{b}} }{ \Ker{ \ExctCplKMapBB{}{\Vect{x}+\ExctCplJBdg{b} } } }
\end{equation*}
This is, indeed, the case because we have the short exact sequence of $\omega$-indexed diagrams:
\begin{equation*}
\xymatrix@R=5ex@C=3em{
\Ker{\ExctCplKMapBB{}{\Vect{x}+\ExctCplJBdg{b}} } \ar@{{ |>}->}[r] &
	\ExctCplCyclesBB{r}{\Vect{x}+\ExctCplJBdg{b}} \ar@{-{ >>}}[r] &
	\dfrac{ \ExctCplCyclesBB{r}{\Vect{x}+\ExctCplJBdg{b}}}{ \Ker{\ExctCplKMapBB{}{\Vect{x}+\ExctCplJBdg{b} } } }
}
\end{equation*}
The limit functor, applied to this short exact sequence of $\omega$-diagrams is exact because its kernel objects are constant. So, the third term in the $6$-term sequence in question is as claimed. To see the relationship between this third term and $\ExctCplLimAbutFltrtnQtntBB{\Vect{x}+\ExctCplJBdg{b}+\ExctCplKBdg{c}}$, we recall its defining short exact sequence:
\begin{equation*}
\xymatrix@R=5ex@C=3em{
\ExctCplLimAbutFltrtnBB{\Vect{x}+\ExctCplJBdg{b}+\ExctCplKBdg{c}} \ar@{{ |>}->}[r] &
	\ExctCplLimAbutFltrtnBB{\Vect{x}+\ExctCplIBdg{a}+\ExctCplJBdg{b}+\ExctCplKBdg{c}} \ar@{-{ >>}}[r]^-{q} &
	\ExctCplLimAbutFltrtnQtntBB{\Vect{x}+\ExctCplJBdg{b}+\ExctCplKBdg{c}}
}
\end{equation*}
Accordingly, we  infer the existence of a monic map $M\from \ExctCplLimAbutFltrtnQtntBB{\Vect{x}+\ExctCplJBdg{b}+\ExctCplKBdg{c}}\to \tfrac{ \ExctCplCyclesBB{r}{\Vect{x}+\ExctCplJBdg{b}}}{ \Ker{\ExctCplKMapBB{}{\Vect{x}+\ExctCplJBdg{b} } } }$  so as to obtain this image factorization of $t$: $t=M\Comp q$.

It remains to verify that the diagrams in the E-infinity extension theorem depend functorially on morphisms between exact couples. This is straight forward, and the proof of (\ref{thm:E-InfinityExtensionThm}) is complete.
\end{proof}

Motivated by the relationship between E-infinity objects and adjacent filtration quotients of the limit and colimit abutment objects in (\ref{thm:E-InfinityExtensionThm}), we introduce the following terminology.

\begin{definition}[E-infinity relations]
\label{def:E-InfinityRelations}%
Given an exact couple $\ExctCpl{C}$, we distinguish the following extension types of $\ExctCplEObjctBB{\infty}{\Vect{x}+\ExctCplJBdg{b}}$:
\vspace{-1.7ex}
\begin{enumerate}[(i)]
\item $\ExctCplEObjctBB{\infty}{\Vect{x}+\ExctCplJBdg{b}}$ is \Defn{unstable} if  the inclusion $\ExctCplStblEObjctBB{\Vect{x}+\ExctCplJBdg{b}} \to \ExctCplEObjctBB{\infty}{\Vect{x}+\ExctCplJBdg{b}}$ fails to be an isomorphism. %
\index{unstable E-infinity object}%
\item $\ExctCplEObjctBB{\infty}{\Vect{x}+\ExctCplJBdg{b}}$ forms a \Defn{stable extension} of $\ExctCplCoLimAbutFltrtnQtntBB{\Vect{x}}$ and over $\ExctCplLimAbutFltrtnQtntBB{\Vect{x} + \ExctCplJBdg{b} + \ExctCplKBdg{c}}$ if  $\ExctCplEObjctBB{\infty}{\Vect{x}+\ExctCplJBdg{b}} \cong \ExctCplStblEObjctBB{\Vect{x}+\ExctCplJBdg{b}}$. %
\index{stable!E-infinity extension}%
\item $\ExctCplEObjctBB{\infty}{\Vect{x}+\ExctCplJBdg{b}}$ matches the \Defn{colimit abutment filtration} if the inclusion $\ExctCplCoLimAbutFltrtnQtntBB{\Vect{x}}\to \ExctCplEObjctBB{\infty}{\Vect{x}+\ExctCplJBdg{b}}$ is an isomorphism. %
\index{matches!colimit abutment filtration}%
\item $\ExctCplEObjctBB{\infty}{\Vect{x}+\ExctCplJBdg{b}}$ matches the \Defn{limit abutment filtration} if the maps below are isomorphisms %
\index{matches!limit abutment filtration}%
\begin{equation*}
\xymatrix@R=5ex@C=4em{
\ExctCplEObjctBB{\infty}{\Vect{x}+\ExctCplJBdg{b}}  &
	\ExctCplStblEObjctBB{\Vect{x}+\ExctCplJBdg{b}} \ar@{{ |>}->}[l]_-{\cong} \ar@{-{ >>}}[r]^-{\cong} &
	\ExctCplLimAbutFltrtnQtntBB{\Vect{x}+\ExctCplJBdg{b}+\ExctCplKBdg{c}}
}
\end{equation*}
\end{enumerate}
\end{definition}

Using the terminology defined in (\ref{def:E-InfinityRelations}) we formulate the following:

\begin{corollary}[Conditions for abutment matching]
\label{thm:AbutmentMatching-Recognition}%
In the setting of the E-infinity distribution theorem (\ref{thm:E-InfinityExtensionThm}) the following hold:
\vspace{-1.7ex}
\begin{enumerate}
\item $\ExctCplEObjctBB{\infty}{\Vect{x}+\ExctCplJBdg{b}}$ matches the colimit abutment filtration if and only if it is stable and $\ExctCplLimAbutFltrtnQtntBB{\Vect{x}+\ExctCplJBdg{b} + \ExctCplKBdg{c}}=0$.
\item $\ExctCplEObjctBB{\infty}{\Vect{x}+\ExctCplJBdg{b}}$ matches the limit abutment filtration if and only if it is stable and $\ExctCplCoLimAbutFltrtnQtntBB{\Vect{x}}=0$. \NoProof
\end{enumerate}
\end{corollary}

The E-infinity extension theorem prompts the question: under which condition(s) are the E-infinity objects of an exact couple stable? In raw form, the answer is given by combining the technical lemmas (\ref{thm:Eps-SubModule-D}) and (\ref{thm:Z_(x+b)-SubModule-D}). More refined conditions are given in (\ref{thm:StableE-InfinityExtension-Sufficient}) in Section \ref{sec:ConvergenceII}.

We adopt notation from (\ref{def:ImageQuotientSubDiagrams}):
\begin{equation*}
\begin{array}{rcl}
I^{\omega}_{\Vect{x} +\ExctCplJBdg{b} + \ExctCplKBdg{c}} & \DefEq & \FamIntrsctn{r\geq 1}{\Img{ \ExctCplIMapItrtdBB{r}{\Vect{x} +\ExctCplJBdg{b} + \ExctCplKBdg{c}- r\ExctCplIBdg{a} } \from \ExctCplDObjctBB{}{\Vect{x}  +\ExctCplJBdg{b} + \ExctCplKBdg{c} -r\ExctCplIBdg{a}}\longrightarrow  \ExctCplDObjctBB{}{\Vect{x}  +\ExctCplJBdg{b} + \ExctCplKBdg{c} } }  } \\
\bar{I}_{\Vect{x} +\ExctCplJBdg{b} + \ExctCplKBdg{c}} & \overset{\text{(\ref{thm:StableImageSubDiagram-Properties})} }{\cong} & \Img{ \ExctCplLimAbut{n+\sigma}\to \ExctCplDObjctBB{}{\Vect{x} +\ExctCplJBdg{b} + \ExctCplKBdg{c}} }
\end{array}
\end{equation*}

\begin{lemma}[$\ExctCplLimAbutFltrtnQtntBB{\Vect{x}+\ExctCplJBdg{b}+\ExctCplKBdg{c}}$ as a submodule of $\ExctCplDObjctBB{}{\Vect{x}+\ExctCplJBdg{b}+\ExctCplKBdg{c}}$]
\label{thm:Eps-SubModule-D}%
In the setting of the E-infinity extension theorem (\ref{thm:E-InfinityExtensionThm}), there is a functorial isomorphism
\begin{equation*}
\ExctCplLimAbutFltrtnQtntBB{\Vect{x}+\ExctCplJBdg{b}+\ExctCplKBdg{c} } \cong \SetIntrsctn{\bar{I}_{\Vect{x} +\ExctCplJBdg{b} + \ExctCplKBdg{c}}}{\Ker{\ExctCplIMapBB{}{\Vect{x}+\ExctCplJBdg{b}+\ExctCplKBdg{c}}}}.
\end{equation*}
\end{lemma}
\begin{proof}
Via (\ref{thm:StableImageSubDiagram-Properties}), we have this morphism of short exact sequences:
\begin{equation*}
\xymatrix@R=5ex@C=4em{
\ExctCplLimAbutFltrtnBB{\Vect{x}+\ExctCplJBdg{b}+\ExctCplKBdg{c}} \ar@{{ |>}->}[r] \ar@{{ |>}->}[d] &
	\ExctCplLimAbut{n+\sigma} \ar@{=}[d] \ar@{-{ >>}}[r] &
	\bar{I}_{\Vect{x}+\ExctCplJBdg{b}+\ExctCplKBdg{c}} \ar@{-{ >>}}[d] \\
\ExctCplLimAbutFltrtnBB{\Vect{x}+\ExctCplJBdg{b}+\ExctCplKBdg{c} + \ExctCplIBdg{a}} \ar@{{ |>}->}[r] &
	\ExctCplLimAbut{n+\sigma} \ar@{-{ >>}}[r] &
	\bar{I}_{\Vect{x}+\ExctCplJBdg{b}+\ExctCplKBdg{c} + \ExctCplIBdg{a}}
}
\end{equation*}
Then $\ExctCplLimAbutFltrtnQtntBB{\Vect{x}+\ExctCplJBdg{b}+\ExctCplKBdg{c}}$ is the cokernel of the vertical map on the left. By the snake lemma, it is functorially isomorphic to the kernel of the vertical map on the right, and so the claim follows.
\end{proof}

\begin{lemma}[$\tfrac{\ExctCplCyclesBB{\infty}{\Vect{x}+\ExctCplJBdg{b}} }{ \Ker{\ExctCplKMapBB{}{\Vect{x}+\ExctCplJBdg{b}}} }$ as a submodule of $\ExctCplDObjctBB{}{\Vect{x}+\ExctCplJBdg{b}+\ExctCplKBdg{c}}$]
\label{thm:Z_(x+b)-SubModule-D}%
In the setting of the E-infinity extension theorem (\ref{thm:E-InfinityExtensionThm}), there is a functorial isomorphism
\begin{equation*}
\dfrac{\ExctCplCyclesBB{\infty}{\Vect{x}+\ExctCplJBdg{b}} }{ \Ker{\ExctCplKMapBB{}{\Vect{x}+\ExctCplJBdg{b}}} } \cong \SetIntrsctn{I^{\omega}_{\Vect{x}+\ExctCplJBdg{b}+\ExctCplKBdg{c}}}{\Ker{\ExctCplIMapBB{}{\Vect{x}+\ExctCplJBdg{b}+\ExctCplKBdg{c}}}}.
\end{equation*}
\end{lemma}
\begin{proof}
By (\ref{def:ExactCoupleItems}) and (\ref{thm:SpecSequFromExactCouple}), we have $\ExctCplCyclesBB{r}{\Vect{x}+\ExctCplJBdg{b}} = (\ExctCplKMapBB{}{\Vect{x}+\ExctCplJBdg{b}})^{-1}\Img{\ExctCplIMapItrtdBB{r}{\Vect{x}+\ExctCplJBdg{b}+\ExctCplKBdg{c}-r\ExctCplIBdg{a}}}$ for every $r\geq 1$. Accordingly, we have the short exact sequence of $\OrdOmegaOp$-diagrams
\begin{equation*}
\xymatrix@R=5ex@C=4em{
\{\Ker{\ExctCplKMapBB{}{\Vect{x}+\ExctCplJBdg{b}}}\} \ar@{{ |>}->}[r] &
	\{ \ExctCplCyclesBB{r}{\Vect{x}+\ExctCplJBdg{b}} \} \ar@{-{ >>}}[r] &
	\left\{\SetIntrsctn{\Img{\ExctCplIMapItrtdBB{r}{\Vect{x}+\ExctCplJBdg{b}+\ExctCplKBdg{c}-r\ExctCplIBdg{a}}}}{\Img{\ExctCplKMapBB{}{\Vect{x}+\ExctCplJBdg{b}}}} \right\}
}
\end{equation*}
The diagram on the left is constant. So, taking limits yields the short exact sequence
\begin{equation*}
\xymatrix@R=5ex@C=4em{
\Ker{\ExctCplKMapBB{}{\Vect{x}+\ExctCplJBdg{b}}} \ar@{{ |>}->}[r] &
	\ExctCplCyclesBB{\infty}{\Vect{x}+\ExctCplJBdg{b}} \ar@{-{ >>}}[r] &
	\LimOf{\SetIntrsctn{\Img{\ExctCplIMapItrtdBB{r}{\Vect{x}+\ExctCplJBdg{b}+\ExctCplKBdg{c}-r\ExctCplIBdg{a}}}}{\Img{\ExctCplKMapBB{}{\Vect{x}+\ExctCplJBdg{b}}}} }
}
\end{equation*}
As limits commute, we have
\begin{equation*}
\LimOf{\SetIntrsctn{\Img{\ExctCplIMapItrtdBB{r}{\Vect{x}+\ExctCplJBdg{b}+\ExctCplKBdg{c}-r\ExctCplIBdg{a}}}}{\Img{\ExctCplKMapBB{}{\Vect{x}+\ExctCplJBdg{b}}}} } = \SetIntrsctn{ \left( \LimOf{\Img{\ExctCplIMapItrtdBB{r}{\Vect{x}+\ExctCplJBdg{b}+\ExctCplKBdg{c}-r\ExctCplIBdg{a}}} }\right) }{ \Img{\ExctCplKMapBB{}{\Vect{x}+\ExctCplJBdg{b}}}} = \SetIntrsctn{I^{\omega}_{\Vect{x}+\ExctCplJBdg{b}+\ExctCplKBdg{c}}}{ \Img{\ExctCplKMapBB{}{\Vect{x}+\ExctCplJBdg{b}}}}
\end{equation*}
As $\Img{\ExctCplKMapBB{}{\Vect{x}+\ExctCplJBdg{b}}} = \Ker{\ExctCplIMapBB{}{\Vect{x}+\ExctCplJBdg{b}+\ExctCplKBdg{c}}}$, the proof is complete.
\end{proof}

We close this section by stating the relationship between the two types of exact couple derivation (\ref{thm:DerivedExactCouples}), the  limit/colimit abutment objects and their filtrations.

\begin{lemma}[Derived exact couple and universal abutments - I]
\label{thm:ExactCoupleDerivation-UniversalAbutments-I}%
$Q$-derivation of an exact couple $\ExctCpl{C}=(E,D,i,j,k)$ relates universal abutment objects functorially: %
\index{derived exact couple!universal abutments - I}%
\begin{equation*}
\xymatrix@R=6ex@C=3em{
\ExctCplCoLimAbutFltrtnBB{\Vect{x}}(\ExctCpl{C}) \ar[r]^-{\cong} \ar[d] &
	\ExctCplCoLimAbutFltrtnBB{\Vect{x}}({Q\ExctCpl{C}}) \ar[d] &
	\ExctCplLimAbutFltrtnBB{\Vect{x}}(\ExctCpl{C}) \ar[r]^-{\cong} \ar[d] &
	\ExctCplLimAbutFltrtnBB{\Vect{x} - \ExctCplIBdg{a}}(Q\ExctCpl{C}) \ar[d] \\
\ExctCplCoLimAbut{n}(\ExctCpl{C}) \ar[r]_-{\psi_n}^-{\cong} &
	\ExctCplCoLimAbut{n}(Q\ExctCpl{C})  &
	\ExctCplLimAbut{n}(\ExctCpl{C}) \ar[r]_-{\psi^n}^-{\cong} &
	\ExctCplLimAbut{n}(Q\ExctCpl{C})
}
\end{equation*}
The above squares commute, and the horizontal maps  are isomorphisms. \NoProof
\end{lemma}

\begin{lemma}[Derived exact couple and universal abutments - II]
\label{thm:ExactCoupleDerivation-UniversalAbutments-II}%
$I$-derivation of an exact couple $\ExctCpl{C}=(E,D,i,j,k)$ relates universal abutment objects functorially: %
\index{derived exact couple!universal abutments - II}%
\begin{equation*}
\xymatrix@R=6ex@C=3em{
\ExctCplCoLimAbutFltrtnBB{\Vect{x} - \ExctCplIBdg{a}}(\ExctCpl{C}) \ar[r]^-{\cong} \ar[d] &
	\ExctCplCoLimAbutFltrtnBB{\Vect{x}}({I\ExctCpl{C}}) \ar[d] &
	\ExctCplLimAbutFltrtnBB{\Vect{x}}(\ExctCpl{C}) \ar[r]^-{\cong} \ar[d] &
	\ExctCplLimAbutFltrtnBB{\Vect{x}}(I\ExctCpl{C}) \ar[d] \\
\ExctCplCoLimAbut{n}(\ExctCpl{C}) \ar[r]_-{\phi_n}^-{\cong} &
	\ExctCplCoLimAbut{n}(I\ExctCpl{C})  &
	\ExctCplLimAbut{n}(\ExctCpl{C}) \ar[r]_-{\phi^n}^-{\cong} &
	\ExctCplLimAbut{n}(I\ExctCpl{C})
}
\end{equation*}
The above squares commute, and the horizontal maps  are isomorphisms. \NoProof
\end{lemma}

Thus either variant of exact couple derivation (\ref{thm:DerivedExactCouples}) leaves the universal abutment objects unchanged. However, it is necessary to keep track of index adjustments as in  (\ref{thm:ExactCoupleDerivation-UniversalAbutments-I}), respectively (\ref{thm:ExactCoupleDerivation-UniversalAbutments-II}).

Let us close this section by addressing the situation where an abutment for a spectral sequence exists which coincides with neither of the two universal abutments of the underlying exact couple. Expanding the flattened exact couple diagram (\ref{fig:ExactCouple-Flattened}) we then find this kind of a diagram:
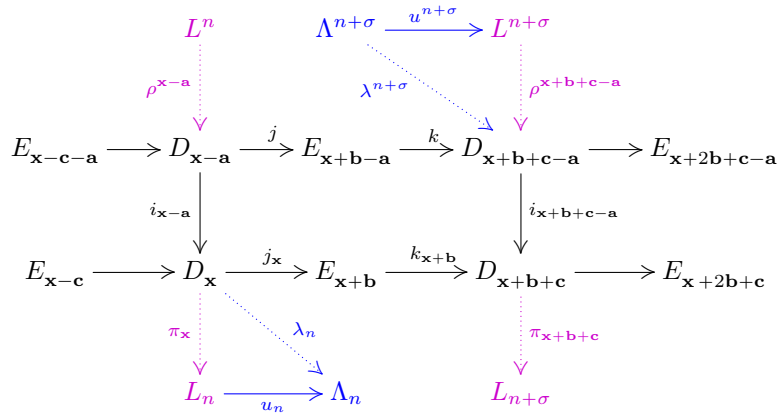
\begin{equation*}
\label{fig:ExactCouple-Abutted}%
%\resizebox{1\textwidth}{!}{$
\begin{xy}
\xymatrix@R=3em@C=2em{
& {\color{deepmagenta} \ExctCplLimAbut{n}} \ar@[deepmagenta]@{.>}[d]_-{\color{deepmagenta} \ExctCplLimAbutMapBB{\Vect{x} - \Vect{a}} } &
	{\color{blue} \Lambda^{n+\sigma}} \ar@[blue][r]^-{\color{blue} u^{n+\sigma}} \ar@[blue]@{.>}[rd]_(.35){\color{blue} \lambda^{n+\sigma}} &
  {\color{deepmagenta} \ExctCplLimAbut{n+\sigma}} \ar@[deepmagenta]@{.>}[d]^-{\color{deepmagenta} \ExctCplLimAbutMapBB{\Vect{x} +\Vect{b}+\Vect{c}-\Vect{a}}} &
     & \\
\ExctCplEObjctBB{}{\Vect{x} - \Vect{c} - \Vect{a} } \ar[r] &
     \ExctCplDObjctBB{}{\Vect{x} -\Vect{a}} \ar[d]_{ i_{\Vect{x} - \Vect{a}}} \ar[r]^{j} &
     \ExctCplEObjctBB{}{\Vect{x} + \Vect{b} - \Vect{a} } \ar[r]^{k} &
     \ExctCplDObjctBB{}{\Vect{x} + \Vect{b}+\Vect{c} - \Vect{a} } \ar[r] \ar[d]^{i_{\Vect{x} +\Vect{b}+\Vect{c} -\Vect{a} }} &
     \ExctCplEObjctBB{}{\Vect{x}+ 2\Vect{b}+\Vect{c} -\Vect{a} } \\
\ExctCplEObjctBB{}{\Vect{x} -\Vect{c}} \ar[r] &
	\ExctCplDObjctBB{}{\Vect{x}} \ar[r]^{j_{\Vect{x}}} \ar@[deepmagenta]@{.>}[d]_{\color{deepmagenta} \ExctCplCoLimAbutMapBB{\Vect{x}}} \ar@{.>}@[blue][rd]^(0.6){\color{blue} \lambda_n} &
     \ExctCplEObjctBB{}{\Vect{x} + \Vect{b}}  \ar[r]^{k_{\Vect{x}+\Vect{b}} } &
     \ExctCplDObjctBB{}{\Vect{x} + \Vect{b}+\Vect{c}} \ar@[deepmagenta]@{.>}[d]^-{\color{deepmagenta} \ExctCplCoLimAbutMapBB{\Vect{x}+ \Vect{b} + \Vect{c}}} \ar[r] &
     \ExctCplEObjctBB{}{\Vect{x}\,+2\Vect{b}+\Vect{c}}  \\
& {\color{deepmagenta} \ExctCplCoLimAbut{n}} \ar@[blue][r]_{\color{blue} u_n} &
	{\color{blue} \Lambda_n} &
  {\color{deepmagenta} \ExctCplCoLimAbut{n+\sigma}} &
}
\end{xy}
%$}
\end{equation*}
Typically, one has only one set of abutment objects:
\begin{enumerate}[$\bullet$]
\item either objects $\Lambda^{n+\sigma}$ which form the vertex of a cone into the exact couple,
\item or objects $\Lambda_n$ which form the vertex of a cocone out of the exact couple.
\end{enumerate}
Either way, there is a universal map $u_n\from L_n\to \Lambda_n$, respectively $u^{n+\sigma}\from \Lambda^{n+\sigma}\to L^{n+\sigma}$. Accordingly, the investigation of how the $E^{\infty}$-objects of the spectral sequence are related to adjacent filtration quotients of the abutment involves these two steps: First use the E-infinity extension theorem as a tool for determining how the E-infinity objects of the spectral sequence relate of the filtrations of the universal abutment(s); then carry this information through the maps $u_n$, respectively $u^n$. - The homotopy spectral sequence of a based tower of fibrations is a nice example.

\begin{example}[Homotopy spectral sequence of a tower of fibrations]
\label{exa:HomotopySS-Fibrations}%
Consider a $\ZCat$-diagram of based Hurewicz fibrations of topological spaces:
\begin{equation*}
X\DefEq \LimOf{X_p}\longrightarrow \cdots \longrightarrow X_{p+1} \XRA{f_{p+1}} X_{p} \XRA{f_p} X_{p-1} \XRA{f_{p-1}}\cdots \XRA{f_1} X_0 \XRA{f_0} \cdots 
\end{equation*}
Writing $Y_p$ for the fiber of $f_p$, each fibration $f_p$ yields a long exact sequence of homotopy groups
\begin{equation*}
\cdots\to \pi_{n+1}X_p \longrightarrow \pi_{n}Y_p \longrightarrow \pi_{n}X_{p-1}\longrightarrow \pi_{n}X_p\to \cdots 
\end{equation*}
If the homotopy groups the spaces $X_p$  are `complicated' while those of the fibers  $Y_p$ are `understood', we would like to port homotopy knowledge of the fibers to information about the homotopy groups $\pi_nX$.

First, we observe that the above long exact homotopy sequences form an exact couple.
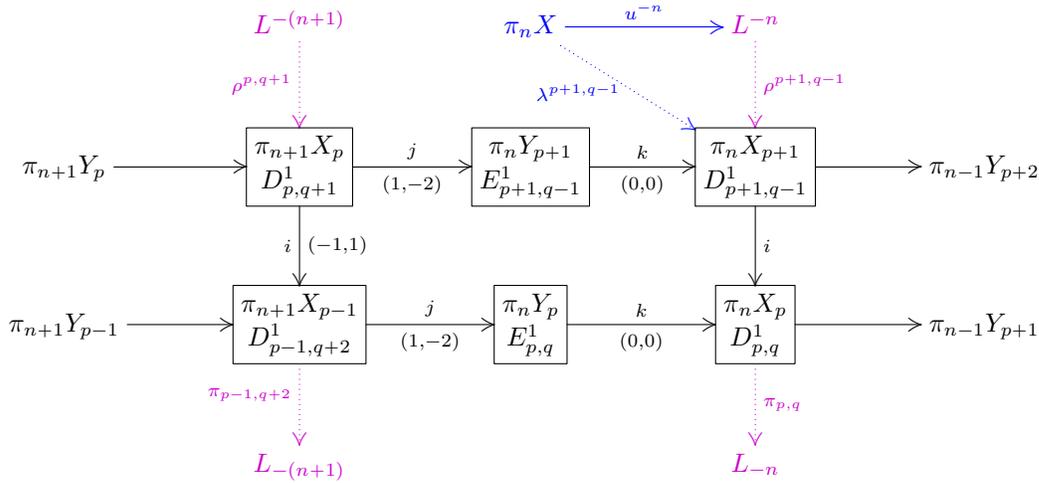
\begin{equation*}
\label{fig:HomotopySpecSeq-Abutted}%
%\resizebox{1\textwidth}{!}{$
\begin{xy}
\xymatrix@R=3em@C=4em{
& {\color{deepmagenta} \ExctCplLimAbut{-(n+1)}} \ar@[deepmagenta]@{.>}[d]_-{\color{deepmagenta} \ExctCplLimAbutMapBB{p,q+1} } &
	{\color{blue} \pi_nX} \ar@[blue][r]^-{\color{blue} u^{-n}} \ar@[blue]@{.>}[rd]_(.35){\color{blue} \lambda^{p+1,q-1}} &
  {\color{deepmagenta} \ExctCplLimAbut{-n}} \ar@[deepmagenta]@{.>}[d]^-{\color{deepmagenta} \ExctCplLimAbutMapBB{p+1,q-1}} &
     & \\
\pi_{n+1}Y_{p} \ar[r] &
     *+[F]\txt{$\pi_{n+1}X_p$\\$\ExctCplDObjct{1}{p}{q+1}$} \ar[d]_-{ i}^-{(-1,1)} \ar[r]^-{j}_-{(1,-2)} &
     *+[F]\txt{$\pi_{n}Y_{p+1}$\\$\ExctCplEObjct{1}{p+1}{q-1}$} \ar[r]^-{k}_-{(0,0)} &
     *+[F]\txt{$\pi_{n}X_{p+1}$\\$\ExctCplDObjct{1}{p+1}{q-1}$} \ar[r] \ar[d]^{i} &
     \pi_{n-1}Y_{p+2} \\
\pi_{n+1}Y_{p-1} \ar[r] &
	*+[F]\txt{$\pi_{n+1}X_{p-1}$\\$\ExctCplDObjct{1}{p-1}{q+2}$} \ar[r]^-{j}_-{(1,-2)} \ar@[deepmagenta]@{.>}[d]_{\color{deepmagenta} \ExctCplCoLimAbutMapBB{p-1,q+2}} &
     *+[F]\txt{$\pi_{n}Y_p$\\$\ExctCplEObjct{1}{p}{q}$}  \ar[r]^-{k}_-{(0,0)} &
     *+[F]\txt{$\pi_{n}X_p$\\$\ExctCplDObjct{1}{p}{q}$} \ar@[deepmagenta]@{.>}[d]^-{\color{deepmagenta} \ExctCplCoLimAbutMapBB{p,q}} \ar[r] &
     \pi_{n-1}Y_{p+1}  \\
& {\color{deepmagenta} \ExctCplCoLimAbut{-(n+1)}} & &
  {\color{deepmagenta} \ExctCplCoLimAbut{-n}} &
}
\end{xy}
%$}
\end{equation*}
According to Definition \ref{def:(ZxZ)-BigradedExactCouple}, this exact couple has $\sigma = +1$, hence is regular. The homotopy groups $\pi_nX\EqDef \Lambda^{-n}$ act as abutment with universal maps $u^{-n}\from \pi_nX\to L^{-n}$. These non-universal abutting groups are filtered by the kernels of the cone maps $\lambda^{p,q}$.

A significant complication here is that the comparison maps $u^{-n}$ {\em need not be isomorphisms}. So, we need to combine
\begin{enumerate}[(a)]
\item what the spectral sequence tells us about the limit abutment objects $\ExctCplLimAbut{n}$ with
\item properties of the maps $u^{-n}$.
\end{enumerate}
In the case of the homotopy spectral sequence of a tower of fibrations, we are lucky in the sense that information about the maps $u^{-n}$ is available via a {\em separate theorem}; see e.g. \cite{PSHirschhorn_PiFib2015} and the references found there. There is a short exact sequence
\begin{equation*}
\xymatrix@R=5ex@C=4em{
\LimOneOf{\pi_{n+1}X_p} \ar@{{ |>}->}[r] &
	\pi_n \LimOf{X_p} \ar@{-{ >>}}[r] &
	\LimOf{\pi_n X_p}
}
\end{equation*}
Given that $X$ is determined by any initial subdiagram of the diagram $X$, we can achieve $0$-objects for the colimit abutment, for example, by choosing $X_{-1}\DefEq \ast$. Consequently, $Y_0=X_0$ whose homotopy groups appear in the appropriate $E$-positions. If these homotopy groups are sufficiently well understood, then this simplification is viable. On the other hand, if for some reason, this simplification is not desirable, then we might have to deal with nonzero colimit abutment objects, their filtrations, and associated E-infinity extension issues.
\end{example}

\section[Convergence II]{Convergence II: More on Extensions of E-infinity Objects}
\label{sec:ConvergenceII}

Given a regular exact couple $\ExctCpl{C}$, we look for conditions under which one or more of the E-infinity relations of the limit/colimit abutment filtrations of $\ExctCpl{C}$ hold. According to the E-infinity extension theorem \ref{thm:E-InfinityExtensionThm}, we must first  ask whether the E-infinity objects are stable. Conditions under which this happens are presented in (\ref{thm:StableE-InfinityExtension-Lemma}) and its Corollary \ref{thm:StableE-InfinityExtension-Sufficient}.

Special among the cases where the E-infinity objects are stable are those situations where it either matches the colimit abutment, respectively, the limit abutment. Conditions for the former are given in (\ref{thm:CoLimitAbutmentMatching-Conditions}), while conditions for the latter are provided in (\ref{thm:LimitAbutmentMatching-Conditions}).

If a spectral sequence comes from an exact couple whose E-infinity objects are unstable, the situation is far more complicated. To get some minimal handle on it we assemble the 6-term exact sequences in the E-infinity extension theorem (\ref{thm:E-InfinityExtensionThm}) into exact couples, called the \Defn{lim-1 exact couples for non-stability}; see (\ref{thm:Lim1ExactCouple-NonStability}).  These exact couples are again regular, hence can be studied with the tools we just developed; see e.g. (\ref{thm:Lim1ExactCouple-StableE-Infinity}).

To round out the discussion of convergence of the spectral sequence associated with an exact couple, let us ask if there is a way of inferring its convergence properties from `nice' features of the spectral sequence alone. Examples show that this is only possible for the exact couple all of whose objects are $0$; see Example \ref{exa:SameSpecSeq-VariousExtensions}.

\begin{lemma}[Stable E-infinity extension lemma]
\label{thm:StableE-InfinityExtension-Lemma} %
For a regular exact couple $(\ExctCplEPage{},\ExctCplDPage{},\ExctCplIMap{},\ExctCplJMap{},\ExctCplKMap{})$ the following are equivalent: %
\index{stable!E-infinity extension}%
\vspace{-1.5ex}%
\begin{enumerate}[(i)]
\item {\em Stable extension}\quad The inclusion $\ExctCplStblEObjctBB{\Vect{x}+\ExctCplJBdg{b}} \to \ExctCplEObjctBB{\infty}{\Vect{x}+\ExctCplJBdg{b}}$ is an isomorphism.
\item The map $\LimOneOfOver{ \Ker{ \ExctCplIMapItrtdBB{r}{\Vect{x} +\ExctCplJBdg{b}+\ExctCplKBdg{c} -r\ExctCplIBdg{a} } } }{r}  \longrightarrow \LimOneOfOver{ \Ker{ \ExctCplIMapItrtdBB{r+1}{\Vect{x}+\ExctCplJBdg{b}+\ExctCplKBdg{c}-r\ExctCplIBdg{a} } } }{r}$ is a monomorphism.
\item The modules $\SetIntrsctn{\bar{I}_{\Vect{x} +\ExctCplJBdg{b} + \ExctCplKBdg{c}}}{\Ker{\ExctCplIMapBB{}{\Vect{x}+\ExctCplJBdg{b}+\ExctCplKBdg{c}}}}$ and $\SetIntrsctn{I^{\omega}_{\Vect{x}+\ExctCplJBdg{b}+\ExctCplKBdg{c}}}{ \Ker{\ExctCplIMapBB{}{\Vect{x}+\ExctCplJBdg{b}+\ExctCplKBdg{c}}}}$ are equal.
\end{enumerate}
\end{lemma}
\begin{proof}
The equivalence of (i) and (ii) follows from the E-infinity extension theorem (\ref{thm:E-InfinityExtensionThm}). The equivalence of (ii) and (iii) is established in (\ref{thm:Conditions-Lim1(K^pA)->Lim1(K^(p+1)A)Monomorphism}.iv).
\end{proof}

Via the stable E-infinity extension lemma \ref{thm:StableE-InfinityExtension-Lemma} we immediately obtain the following sufficient conditions for stable E-infinity extension. The formulation of these conditions involves the concepts of Mittag-Leffler condition (\ref{def:Mittag-Leffler/CoMittag-Leffler-Condition}) and $\omega$-Mittag-Leffler condition (\ref{def:Lambda-Mittag-LefflerCondition}).

\begin{corollary}[Stable E-infinity extension - sufficient conditions]
\label{thm:StableE-InfinityExtension-Sufficient}%
In Lemma \ref{thm:StableE-InfinityExtension-Lemma}, the object $\ExctCplEObjctBB{\infty}{\Vect{x}+\ExctCplJBdg{b}}$ is stable whenever at least one of the following conditions is satisfied. %
\index{stable!E-infinity extension, sufficient conditions}%
\vspace{-1.6ex}%
\begin{enumerate}[(i)]
\item \label{thm:StableE-InfinityExtension-Sufficient_OntoIOmega}The structure map $\ExctCplLimAbut{n+\sigma} \to I^{\omega}_{\Vect{x} +\ExctCplJBdg{b} + \ExctCplKBdg{c}}$ is surjective.
\item \label{thm:StableE-InfinityExtension-Sufficient-a_p|Mono} The restriction of $\ExctCplIMapBB{}{\Vect{x} +\ExctCplJBdg{b} + \ExctCplKBdg{c}}$ to $I^{\omega}_{\Vect{x} +\ExctCplJBdg{b} + \ExctCplKBdg{c}}$ is a monomorphism.
\item \label{thm:StableE-InfinityExtension-Sufficient-a_pMono} The map $\ExctCplIMapBB{}{\Vect{x} +\ExctCplJBdg{b} + \ExctCplKBdg{c}}$ is a monomorphism.
\item \label{thm:StableE-InfinityExtension-Sufficient-Omega-MittagLeffler} {\em $\omega$-Mittag-Leffler condition:} \  the $\ZCat$-diagram $\ExctCplDObjctBB{}{\Vect{x}+\ExctCplJBdg{b}+\ExctCplKBdg{c} + r\ExctCplIBdg{a}}$ satisfies the $\omega$-Mittag-Leffler condition. %
\index{$\omega$-Mittag-Leffler condition!stable E-infinity object}%
\item \label{thm:StableE-InfinityExtension-Sufficient-Mittag-Leffler} {\em Mittag-Leffler condition:} \ the $\ZCat$-diagram $\ExctCplDObjctBB{}{\Vect{x}+\ExctCplJBdg{b}+\ExctCplKBdg{c} + r\ExctCplIBdg{a}}$ satisfies the Mittag-Leffler condition. %
\index{Mittag-Leffler condition!stable E-infinity object}%
\end{enumerate}
\end{corollary}
\begin{proof}
If \ref{thm:StableE-InfinityExtension-Sufficient_OntoIOmega} holds, then $\bar{I}_{\Vect{x} +\ExctCplJBdg{b} + \ExctCplKBdg{c}}=I^{\omega}_{\Vect{x}+\ExctCplJBdg{b}+\ExctCplKBdg{c}}$, and so  $\ExctCplEObjctBB{\infty}{\Vect{x}+\ExctCplJBdg{b}}$ is stable by (\ref{thm:StableE-InfinityExtension-Lemma}.i).

If \ref{thm:StableE-InfinityExtension-Sufficient-a_p|Mono} holds, then $\Ker{ \ExctCplIMapItrtdBB{r}{\Vect{x} +\ExctCplJBdg{b}+\ExctCplKBdg{c} -r\ExctCplIBdg{a} } } \to \Ker{ \ExctCplIMapItrtdBB{r+1}{\Vect{x} +\ExctCplJBdg{b}+\ExctCplKBdg{c} -r\ExctCplIBdg{a} } }$ is an isomorphism for every $r\geq 1$. Therefore $\ExctCplEObjctBB{\infty}{\Vect{x}+\ExctCplJBdg{b}}$ is stable by (\ref{thm:StableE-InfinityExtension-Lemma}.ii). Condition \ref{thm:StableE-InfinityExtension-Sufficient-a_pMono} implies \ref{thm:StableE-InfinityExtension-Sufficient-a_p|Mono}, hence is sufficient.

We complete the proof by observing that the following implications hold:
\ref{thm:StableE-InfinityExtension-Sufficient-Mittag-Leffler} $\implies$ \ref{thm:StableE-InfinityExtension-Sufficient-Omega-MittagLeffler} $\implies$ \ref{thm:StableE-InfinityExtension-Sufficient_OntoIOmega}.
\end{proof}

\begin{proposition}[E-infinity matches colimit abutment - conditions]
\label{thm:CoLimitAbutmentMatching-Conditions}%
If $\ExctCpl{C}$ is a exact couple, then the object $\ExctCplEObjctBB{\infty}{\Vect{x}+\ExctCplJBdg{b}}$ of its associated spectral sequence satisfies %
\index{colimit abutment!matching}
\begin{equation*}
\ExctCplEObjctBB{\infty}{\Vect{x}+\ExctCplJBdg{b}} \cong \ExctCplCoLimAbutFltrtnQtntBB{\Vect{x}}
\end{equation*}
if and only if the restriction of $\ExctCplIMapBB{}{\Vect{x}+ \ExctCplJBdg{b}+\ExctCplKBdg{c} }$ to $I^{\omega}_{\Vect{x}+\Vect{b}+\Vect{c}}$ is a monomorphism. Further, this happens, whenever at least one of the following conditions is satisfied. %
\vspace{-1.6ex}%
\begin{enumerate}[(i)]
\item There exists $r\geq 1$ such that the restriction of $\ExctCplIMapBB{}{\Vect{x}+ \ExctCplJBdg{b}+\ExctCplKBdg{c} }$ to $\Img{ \ExctCplIMapItrtdBB{r}{\Vect{x} + \ExctCplJBdg{b}+\ExctCplKBdg{c} - r\ExctCplIBdg{a}} }$ is a monomorphism.
\item The map $\ExctCplIMapBB{}{\Vect{x}+ \ExctCplJBdg{b}+\ExctCplKBdg{c} }$ is a monomorphism.
\item $I^{\omega}_{\Vect{x}+\Vect{b}+\Vect{c}}=0$.
\item There exists $r\geq 1$ for which the isomorphisms below hold.
\begin{equation*}
\xymatrix@R=5ex@C=5em{
\dfrac{ \Img{ \ExctCplIMapItrtdBB{r}{\Vect{x}} } }{  \Img{ \ExctCplIMapItrtdBB{r+1}{\Vect{x}-\ExctCplIBdg{a} } } } \ar@{{ |>}->}[r]_-{\cong}^-{\text{(\ref{thm:E(r+1)Extension})}} &
	\ExctCplEObjctBB{r+1}{\Vect{x}+\ExctCplJBdg{b}} \cong \ExctCplEObjctBB{\infty}{\Vect{x}+\ExctCplJBdg{b}}
}
\end{equation*}
\item The $\ZCat$-diagram $\ExctCplDObjctBB{}{\Vect{x}+ \ExctCplJBdg{b}+\ExctCplKBdg{c} +r\ExctCplIBdg{a} }$, $r\in \ZNr$, is originally vanishing.
\end{enumerate}
\end{proposition}
\begin{proof}
From the E-infinity extension theorem \ref{thm:E-InfinityExtensionThm}, we see that the inclusion $\ExctCplCoLimAbutFltrtnQtntBB{\Vect{x}}$ in $\ExctCplEObjctBB{\infty}{\Vect{x}+\Vect{b}}$ is an isomorphism exactly when
\begin{equation*}
0 = \dfrac{\ExctCplCyclesBB{\infty}{\Vect{x}+\Vect{b}}}{\Img{\ExctCplJMapBB{}{\Vect{x}}}} = \dfrac{\ExctCplCyclesBB{\infty}{\Vect{x}+\Vect{b}}}{\Ker{\ExctCplKMapBB{}{\Vect{x} + \Vect{b} }}} \overset{\text{(\ref{thm:Z_(x+b)-SubModule-D})}}{\cong} I^{\omega}_{\Vect{x}+\ExctCplJBdg{b}+\ExctCplKBdg{c}}\intrsctn \Ker{\ExctCplIMapBB{}{\Vect{x}+\ExctCplJBdg{b}+\ExctCplKBdg{c}}}
\end{equation*}
So, this happens exactly when the restriction of $\ExctCplIMapBB{}{\Vect{x}+\ExctCplJBdg{b}+\ExctCplKBdg{c}}$ to $I^{\omega}_{\Vect{x}+\ExctCplJBdg{b}+\ExctCplKBdg{c}}$ is a monomorphism.

Condition (i) is sufficient for this to happen, as $I^{\omega}_{\Vect{x}+\ExctCplJBdg{b}+\ExctCplKBdg{c}}$ is contained in $\Img{\ExctCplIMapItrtdBB{r}{\Vect{x}+\ExctCplJBdg{b}+\ExctCplKBdg{c}-r\ExctCplIBdg{a}}}$ for every $r\geq 0$. It follows that each of conditions (ii) and (iii) is sufficient as either implies (i).

If (iv) is satisfied, we infer with (\ref{thm:E(r+1)Extension}) that, for large $r$,
\begin{equation*}
\Ker{\ExctCplIMapItrtdBB{r}{\Vect{x}+\ExctCplJBdg{b}+\ExctCplKBdg{c}-r\ExctCplIBdg{a} } } = \Ker{\ExctCplIMapItrtdBB{r+1}{\Vect{x}+\ExctCplJBdg{b}+\ExctCplKBdg{c}-r\ExctCplIBdg{a} } }
\end{equation*}
Recalling that $\ExctCplLimAbutFltrtnBB{\Vect{x}+\ExctCplJBdg{b}+\ExctCplKBdg{c}}\cong \LimOf{\Ker{\ExctCplIMapItrtdBB{r}{\Vect{x}+\ExctCplJBdg{b}+\ExctCplKBdg{c}-r\ExctCplIBdg{a} } }}$, we see that the maps below are isomorphisms:
\begin{equation*}
\ExctCplLimAbutFltrtnBB{\Vect{x}+\ExctCplJBdg{b}+\ExctCplKBdg{c}} \longrightarrow \ExctCplLimAbutFltrtnBB{\Vect{x}+\ExctCplIBdg{a}+\ExctCplJBdg{b}+\ExctCplKBdg{c}}  \qquad \text{and}\qquad 
\LimOneOfOver{\Ker{\ExctCplIMapItrtdBB{r}{\Vect{x}+\ExctCplJBdg{b}+\ExctCplKBdg{c}-r\ExctCplIBdg{a}}}}{\ZCat} \longrightarrow \LimOneOfOver{\Ker{\ExctCplIMapItrtdBB{r+1}{\Vect{x}+\ExctCplJBdg{b}+\ExctCplKBdg{c}-r\ExctCplIBdg{a}}}}{\ZCat}
\end{equation*}
Now the E-infinity extension theorem implies that $\ExctCplLimAbutFltrtnQtntBB{\Vect{x}+\ExctCplJBdg{b}+\ExctCplKBdg{c}}=0$, and that the inclusion $\ExctCplCoLimAbutFltrtnQtntBB{\Vect{x}} \to \ExctCplEObjctBB{\infty}{\Vect{x} + \ExctCplJBdg{b}}$ is an isomorphism.

Finally,  if (v) holds, then so does (iii), and the proof is complete.
\end{proof}

\begin{proposition}[E-infinity matches limit abutment - conditions]
\label{thm:LimitAbutmentMatching-Conditions}%
If $\ExctCpl{C}$ is a exact couple, then the object $\ExctCplEObjctBB{\infty}{\Vect{x}+\ExctCplJBdg{b}}$ of its associated spectral sequence satisfies %
\index{limit abutment!matching}%
\begin{equation*}
\ExctCplEObjctBB{\infty}{\Vect{x}+\ExctCplJBdg{b}} \cong \ExctCplLimAbutFltrtnQtntBB{\Vect{x}+\ExctCplJBdg{b}+\ExctCplKBdg{c}}
\end{equation*}
if and only if $\ExctCplEObjctBB{\infty}{\Vect{x}+\ExctCplJBdg{b}}$ is stable and $\ExctCplCoLimAbutFltrtnQtntBB{\Vect{x}} = 0$. The latter condition holds whenever at least one of the following conditions is satisfied.
\begin{enumerate}[(i)]
\item $\ExctCplDObjctBB{}{\Vect{x}}$ is generated by $\Img{\ExctCplIMapBB{}{\Vect{x}-\ExctCplIBdg{a}}}\union \Ker{\ExctCplDObjctBB{}{\Vect{x}}\to \ExctCplCoLimAbut{n}}$, where $n= \DtrmnntOfMtrx{[ \Vect{a}\ \ \Vect{x} ] }$.
\item There exists $r\geq 1$ such that $\ExctCplDObjctBB{}{\Vect{x}}$ is generated by $\Img{\ExctCplIMapBB{}{\Vect{x}-\ExctCplIBdg{a}}} \union \Ker{\ExctCplDObjctBB{}{\Vect{x}}\to \ExctCplDObjctBB{}{\Vect{x}+r\ExctCplIBdg{a}} }$.
\item $\ExctCplIMapBB{}{\Vect{x}-\ExctCplIBdg{a}}\from \ExctCplDObjctBB{}{\Vect{x}-\ExctCplIBdg{a}}\to \ExctCplDObjctBB{}{\Vect{x}}$ is surjective.
\item $\ExctCplCoLimAbut{n}=0$
\item The $\ZCat$-diagram $\ExctCplDObjctBB{}{\Vect{x}+r\ExctCplIBdg{a} }$ is eventually vanishing.
\item There exists $r\geq 1$ such that
\begin{equation*}
\ExctCplEObjctBB{r+1}{\Vect{x}+\ExctCplJBdg{b}} \cong \dfrac{\Ker{\ExctCplIMapItrtdBB{r+1}{\Vect{x}+\ExctCplJBdg{b}+\ExctCplKBdg{c}-r\ExctCplIBdg{a}}}}{\Ker{\ExctCplIMapItrtdBB{r}{\Vect{x}+\ExctCplJBdg{b}+\ExctCplKBdg{c}-r\ExctCplIBdg{a}}}}
\end{equation*}
\end{enumerate}
Among these properties the following implications hold: (iii) $\implies$ (ii) $\implies$ (i) $\Longleftrightarrow$ $(\ExctCplCoLimAbutFltrtnQtntBB{\Vect{x}} = 0)$; (iii) $\implies$ (iv) $\implies$ $(\ExctCplCoLimAbutFltrtnQtntBB{\Vect{x}} = 0)$; (vi) $\Longleftrightarrow$ (ii).
\end{proposition}
\begin{proof}
We already observed in (\ref{thm:AbutmentMatching-Recognition}) that $\ExctCplEObjctBB{\infty}{\Vect{x}+\ExctCplJBdg{b}}$ matches the limit abutment if and only if it is stable and $\ExctCplCoLimAbutFltrtnQtntBB{\Vect{x}} = 0$.

To see that $(\ExctCplCoLimAbutFltrtnQtntBB{\Vect{x}} = 0) \Longleftrightarrow$ (i), let us analyze the kernel-cokernel $6$-term exact sequence of the composite
\begin{equation*}
\ExctCplDObjctBB{}{\Vect{x}-\ExctCplIBdg{a}} \XRA{u\DefEq \ExctCplIMapBB{}{\Vect{x}-\ExctCplIBdg{a} } } \ExctCplDObjctBB{}{\Vect{x}} \XRA{\pi_{\Vect{x}}} \ExctCplCoLimAbutFltrtnBB{\Vect{x}}
\end{equation*}
Indeed, the $6$-term exact sequence contains the segment
\begin{equation*}
\Ker{\pi_{\Vect{x}} } \longrightarrow \CoKer{u} \longrightarrow \CoKer{\pi_{\Vect{x}}\Comp u } \longrightarrow \CoKer{\pi_{\Vect{x}}}=0
\end{equation*}
By definition,  $\CoKer{\pi_{\Vect{x}}\Comp u}=\ExctCplCoLimAbutFltrtnQtntBB{\Vect{x}}$. So, this object vanishes if and only if  $\Ker{\pi_{\Vect{x}}}\to \CoKer{u}$ is surjective. The latter condition is equivalent to (i).

(ii) $\implies$ (i) because $\Ker{\ExctCplDObjctBB{}{\Vect{x}}\to \ExctCplDObjctBB{}{\Vect{x}+r\ExctCplIBdg{a}} }\subseteq \Ker{\ExctCplDObjctBB{}{\Vect{x}}\to \ExctCplCoLimAbut{n}}$.

The implications (iii) $\implies$ (ii), (iv) $\implies$ $(\ExctCplCoLimAbutFltrtnQtntBB{\Vect{x}} = 0)$, and (v) $\implies$ (iv) are immediate.

To see that (vi) is equivalent to (ii), we recall the $E(r+1)$-extension result (\ref{thm:E(r+1)Extension}): we have the short exact sequence
\begin{equation*}
\xymatrix@C=4em{
\dfrac{ \Img{ \ExctCplIMapItrtdBB{r}{\Vect{x}} } }{ \Img{ \ExctCplIMapItrtdBB{r+1}{\Vect{x}-\ExctCplIBdg{a}} } } \ar@{{ |>}->}[r] &
	\ExctCplEObjctBB{r+1}{\Vect{x}+\ExctCplJBdg{b} } \ar@{-{ >>}}[r] &
	\dfrac{ \Ker{ \ExctCplIMapItrtdBB{r+1}{\Vect{x}+\ExctCplJBdg{b} + \ExctCplKBdg{c} - r\ExctCplIBdg{a}} } }{ \Ker{ \ExctCplIMapItrtdBB{r}{\Vect{x}+\ExctCplJBdg{b} + \ExctCplKBdg{c} - r\ExctCplIBdg{a}} } }
}
\end{equation*}
The right hand arrow is an isomorphism if and only if $\Img{\ExctCplIMapBB{r}{\Vect{x}}} = \Img{\ExctCplIMapBB{r+1}{\Vect{x}-\ExctCplIBdg{a}}}$, which is equivalent to (ii).
\end{proof}

Whenever the extension $\SSObjctBB{\infty}{\Vect{x}+\ExctCplJBdg{b}}$ fails to be stable, non-vanishing $\LimOne$-terms are necessarily involved. The lim-1 exact couple (\ref{thm:Lim1ExactCouple-NonStability}) associated with the non-stability of $\SSObjctBB{\infty}{}$-objects tells us that a complete understanding of this situation is much more complicated than one might initially suspect.

\begin{proposition}[Lim-1 exact couple from non-stability]
\label{thm:Lim1ExactCouple-NonStability}%
Let $\ExctCpl{C}$ be a regular exact couple, and fix $n\in \ZNr$. Then the kernel filtration of its limit abutment object $\ExctCplLimAbut{n+\sigma}$ fits into the lim-1 exact couple $(\ExctCplEPage{}(1),\ExctCplDPage{}(1))$ displayed below: %
\index{lim-1 exact couple}%
\begin{equation*}
\xymatrix@C=2em{
\vdots \ar@{{ |>}->}[d] & & \vdots \ar[d] \\
{\color{blue} \ExctCplLimAbutFltrtnBB{\Vect{x}+\ExctCplJBdg{b} + \ExctCplKBdg{c}}} \ar@[blue]@{{ |>}->}[d] \ar[r] &
	\dfrac{ \ExctCplCyclesBB{\infty}{\Vect{x}-\ExctCplIBdg{a}+\ExctCplJBdg{b} } }{ \Img{ \ExctCplJMapBB{}{\Vect{x}-\ExctCplIBdg{a} } } } \ar[r] &
	\LimOneOf{ \Ker{ \ExctCplIMapItrtdBB{r}{\Vect{x}- \ExctCplIBdg{a}+\ExctCplJBdg{b}+\ExctCplKBdg{c} - r\ExctCplIBdg{a} } } } \ar@{-{ >>}}[r] \ar[d] &
	\LimOneOf{ \dfrac{ \ExctCplCyclesBB{r}{\Vect{x}-2\ExctCplIBdg{a}+\ExctCplJBdg{b} } }{ \Img{ \ExctCplJMapBB{}{\Vect{x}-2\ExctCplIBdg{a} } } } } \\
{\color{blue} \ExctCplLimAbutFltrtnBB{\Vect{x}+\ExctCplIBdg{a}+\ExctCplJBdg{b} + \ExctCplKBdg{c} } } \ar@{{ |>}->}[d] \ar@[blue][r] &
	{\color{blue} \dfrac{ \ExctCplCyclesBB{\infty}{\Vect{x}+\ExctCplJBdg{b} } }{ \Img{ \ExctCplJMapBB{}{\Vect{x} } } } } \ar@[blue][r] &
	{\color{blue}\LimOneOf{ \Ker{ \ExctCplIMapItrtdBB{r}{\Vect{x}+\ExctCplJBdg{b}+\ExctCplKBdg{c} - r\ExctCplIBdg{a} } } } } \ar@{-{ >>}}[r] \ar@[blue][d] &
	\LimOneOf{ \dfrac{ \ExctCplCyclesBB{r}{\Vect{x}-\ExctCplIBdg{a}+\ExctCplJBdg{b} } }{ \Img{ \ExctCplJMapBB{}{\Vect{x}  -\ExctCplIBdg{a} } } } }  \\
\ExctCplLimAbutFltrtnBB{\Vect{x}+2\ExctCplIBdg{a}+\ExctCplJBdg{b} + \ExctCplKBdg{c}} \ar@{{ |>}->}[d] \ar[r] &
	\dfrac{ \ExctCplCyclesBB{\infty}{\Vect{x}+\ExctCplIBdg{a}+\ExctCplJBdg{b} } }{ \Img{ \ExctCplJMapBB{}{\Vect{x}+\ExctCplIBdg{a} } } } \ar[r] &
	{\color{blue} \LimOneOf{ \Ker{ \ExctCplIMapItrtdBB{r}{\Vect{x}+\ExctCplIBdg{a}+\ExctCplJBdg{b}+\ExctCplKBdg{c} - r\ExctCplIBdg{a} } } } } \ar@{-{ >>}}@[blue][r] \ar[d] &
	{\color{blue} \LimOneOf{ \dfrac{ \ExctCplCyclesBB{r}{\Vect{x}+\ExctCplJBdg{b} } }{ \Img{ \ExctCplJMapBB{}{\Vect{x} } } } } }  \\
\vdots & & \vdots 
}
\end{equation*}
Objects outside the displayed columns are $0$. The lim-1 exact couple is regular via the following settings:
\begin{equation*}
\begin{array}{rclcccrcl}
\ExctCplDObjctBB{}{\Vect{x}}(1) &
	\DefEq &
	\ExctCplLimAbutFltrtnBB{\Vect{x}+\ExctCplIBdg{a}+\ExctCplJBdg{b} + \ExctCplKBdg{c} } &
	\qquad  &
	\text{and} &
	\qquad &
	\ExctCplDObjctBB{}{\Vect{x} + \ExctCplJBdg{b}+\ExctCplKBdg{c}}(1) &
	\DefEq &
	\LimOneOf{ \Ker{ \ExctCplIMapItrtdBB{r}{\Vect{x}+\ExctCplJBdg{b}+\ExctCplKBdg{c} - r\ExctCplIBdg{a} } } } \\
\ExctCplEObjctBB{}{\Vect{x} + \ExctCplJBdg{b}}(1) &
	\DefEq &
	\dfrac{ \ExctCplCyclesBB{\infty}{\Vect{x}+\ExctCplJBdg{b} } }{ \Img{ \ExctCplJMapBB{}{\Vect{x} } } } &
	\qquad  &
	\text{and} &
	\qquad &
	\ExctCplEObjctBB{}{\Vect{x} + 2\ExctCplJBdg{b}+\ExctCplKBdg{c}}(1) &
	\DefEq &
	\LimOneOf{ \dfrac{ \ExctCplCyclesBB{r}{\Vect{x}-\ExctCplIBdg{a}+\ExctCplJBdg{b} } }{ \Img{ \ExctCplJMapBB{}{\Vect{x}  -\ExctCplIBdg{a} } } } }
\end{array}
\end{equation*}
\end{proposition}
\begin{proof}
This exact couple is formed from the 6-term exact sequences in the E-infinity extension theorem. To see that the $\ZCat$-diagram of lim-1-objects is valid, we just need to notice that
\begin{equation*}
\LimOneOf{ \Ker{ \ExctCplIMapItrtdBB{r}{\Vect{x}+\ExctCplIBdg{a}+\ExctCplJBdg{b}+\ExctCplKBdg{c} - r\ExctCplIBdg{a} } } } = \LimOneOf{ \Ker{ \ExctCplIMapItrtdBB{r+1}{\Vect{x} + \ExctCplJBdg{b}+\ExctCplKBdg{c} - r\ExctCplIBdg{a} } } }
\end{equation*}
Via the proposed settings, the bidegrees of its structure maps match those of the original exact couple. Thus the lim-1 exact couple is regular.
\end{proof}

\begin{proposition}[Lim-1 exact couple for stable $\ExctCplEPage{\infty}$]
\label{thm:Lim1ExactCouple-StableE-Infinity}
If the $\ExctCplEPage{\infty}$\MSComp-objects of a regular exact couple are stable, then its lim-1 exact couple has the following properties.
\begin{enumerate}[(i)]
\item The lim-1 exact couple matches its colimit abutment, and the associated spectral sequence collapses on page $1$.
\item The kernel filtration of the limit of the $\ZCat$-diagram of lim-1-objects is constant at $0$.
\end{enumerate}
\end{proposition}
\begin{proof}
As the objects $\ExctCplEObjctBB{\infty}{\Vect{x}+r\Vect{a}+\Vect{b}}$ are stable the E-infinity extension theorem (\ref{thm:E-InfinityExtensionThm}) tells us that the maps below are monomorphisms:
\begin{equation*}
\LimOneOf{ \Ker{ \ExctCplIMapItrtdBB{r}{\Vect{x}+\ExctCplJBdg{b}+\ExctCplKBdg{c} - r\ExctCplIBdg{a} } } } \longrightarrow \LimOneOf{ \Ker{ \ExctCplIMapItrtdBB{r+1}{\Vect{x} + \ExctCplJBdg{b}+\ExctCplKBdg{c} - r\ExctCplIBdg{a} } } }
\end{equation*}
By (\ref{thm:CoLimitAbutmentMatching-Conditions}) the lim-1 spectral sequence matches its colimit abutments. Further, by exactness, differentials with domain $\ExctCplCyclesBB{\infty}{\Vect{x}+\ExctCplJBdg{b}+ k\ExctCplIBdg{a}}/\Img{\ExctCplJMapBB{}{\Vect{x}+k\ExctCplIBdg{a}}}$ vanish, as they are a composite (\ref{thm:SpecSequFromExactCouple}) of additive relations involving the $0$-map
$$
\ExctCplCyclesBB{\infty}{\Vect{x}+\ExctCplJBdg{b}+k\ExctCplIBdg{a}}/\Img{\ExctCplJMapBB{}{\Vect{x}+k\ExctCplIBdg{a}}}\longrightarrow \LimOneOf{\Ker{ \ExctCplIMapItrtdBB{r}{\Vect{x}+k\ExctCplIBdg{a}+ \ExctCplJBdg{b}+\ExctCplKBdg{c} - r\ExctCplIBdg{a} } } }.
$$
All of the remaining differentials vanish as at least one of their domains or codomains is $0$. So, the lim-1 spectral sequence collapses on page $1$. Claim (ii) is true because the vertical maps of $\LimOne$-objects are monomorphisms.
\end{proof}

We close this section with examples which demonstrate that, whenever we wish to analyze the meaning of the E-infinity objects from the spectral sequence associated to a regular exact couple, we {\em must} consult the underlying exact couple. Boardman responds to the need for this consultation by making an advance selection of exact couples which he discusses, namely those which he deems relevant for `real world spectral sequences' \cite[p~ 15]{JMBoardman1999}, and the resulting selection is made explicit here \cite[p~19]{JMBoardman1999}.

\begin{example}[Same spectral sequence various extensions]
\label{exa:SameSpecSeq-VariousExtensions}%
Here are three distinct exact couples along with their relevant E-infinity extension diagrams. Each exact couple has exactly one nonzero $E$-object, and this $E$-object is common to all three exact couples. Thus, all three exact couples have the same associated spectral sequence: a first quadrant spectral sequence which collapses on page~$1$. So, all three exact couples yield the exact same $E^{\infty}$-page.

In each case this $E^{\infty}$-page is stable. However, depending upon the underlying exact couple, the one and only nonzero $E^{\infty}$-object matches the colimit abutment, respectively the limit abutment, respectively fits a proper extension involving the quotients of a certain pair of adjacent colimit abutment filtration stages and a certain pair of adjacent limit abutment filtration stages.

To be specific, the one and only nonzero $E$-term is to be $\ExctCplEObjctBB{1}{0,0}=\ZMod{6}$\MSComp, and the bidegrees of its structure maps are to be $\ExctCplIBdg{a}=(1,-1)$\MSComp, $\ExctCplJBdg{b}=(0,0)$\MSComp, and $\ExctCplKBdg{c}=(-1,0)$\MSComp.
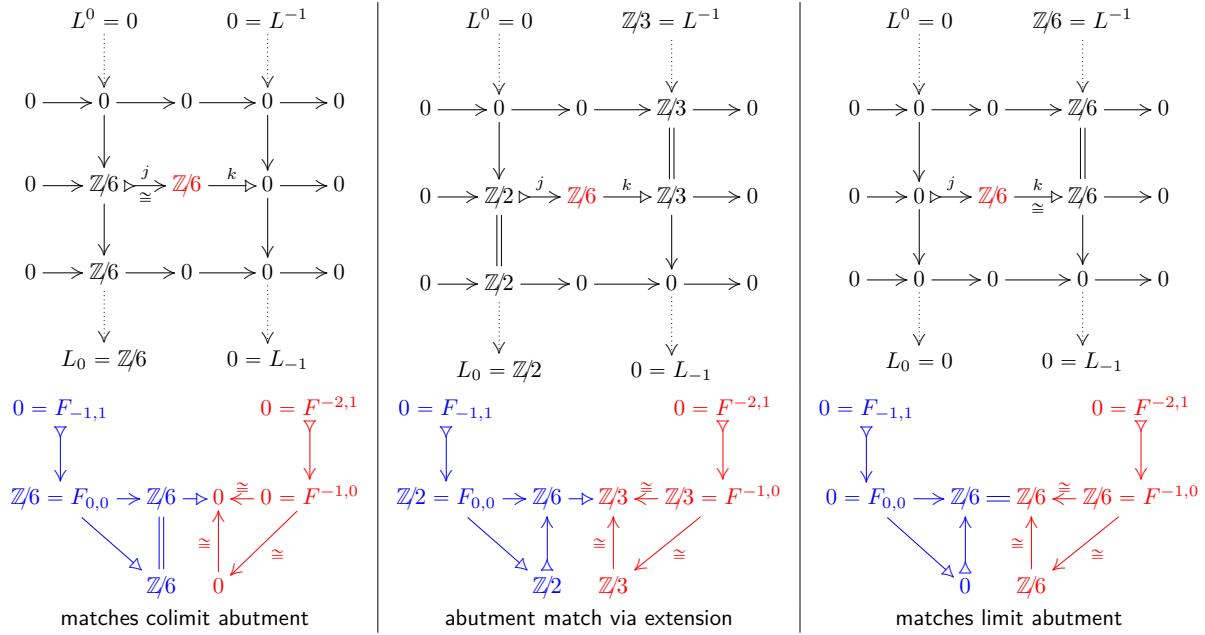
\begin{equation*}
\resizebox{.99\textwidth}{!}{$
\begin{array}{c|c|c}
\xymatrix@R=5ex@C=.5em{
& \ExctCplLimAbut{0}=0 \ar@{.>}[d] &&
	0=\ExctCplLimAbut{-1} \ar@{.>}[d] \\
0 \ar[r] &
0 \ar[d] \ar[r] &
0 \ar[r] &
0 \ar[r] \ar[d] &
0 \\
0 \ar[r] &
\ZMod{6} \ar[d] \ar@{{ |>}->}[r]^-{j}_-{\cong} &
{\color{red} \ZMod{6}} \ar@{-{ >>}}[r]^-{k} &
0 \ar[r] \ar[d] &
0 \\
0 \ar[r] &
\ZMod{6} \ar[r] \ar@{.>}[d] &
0 \ar[r] &
0 \ar[r] \ar@{.>}[d] &
0 \\
& \ExctCplCoLimAbut{0}=\ZMod{6} &&
	0=\ExctCplCoLimAbut{-1}
}
 &
\xymatrix@R=5ex@C=.5em{
& \ExctCplLimAbut{0}=0 \ar@{.>}[d] &&
	\ZMod{3}=\ExctCplLimAbut{-1} \ar@{.>}[d] \\
0 \ar[r] &
0 \ar[d] \ar[r] &
0 \ar[r] &
\ZMod{3} \ar[r] \ar@{=}[d] &
0 \\
0 \ar[r] &
	\ZMod{2} \ar@{=}[d] \ar@{{ |>}->}[r]^-{j} &
	{\color{red} \ZMod{6}} \ar@{-{ >>}}[r]^-{k} &
	\ZMod{3} \ar[r] \ar[d] &
	0 \\
0 \ar[r] &
	\ZMod{2} \ar[r] \ar@{.>}[d] &
0 \ar[r] &
0 \ar[r] \ar@{.>}[d] &
0 \\
& \ExctCplCoLimAbut{0}=\ZMod{2} &&
	0=\ExctCplCoLimAbut{-1}
} &
\xymatrix@R=5ex@C=.5em{
& \ExctCplLimAbut{0}=0 \ar@{.>}[d] &&
	\ZMod{6}=\ExctCplLimAbut{-1} \ar@{.>}[d] \\
0 \ar[r] &
0 \ar[d] \ar[r] &
0 \ar[r] &
\ZMod{6} \ar[r] \ar@{=}[d] &
0 \\
0 \ar[r] &
	0 \ar[d] \ar@{{ |>}->}[r]^-{j} &
	{\color{red} \ZMod{6}} \ar@{-{ >>}}[r]^-{k}_-{\cong} &
	\ZMod{6} \ar[r] \ar[d] &
	0 \\
0 \ar[r] &
	0 \ar[r] \ar@{.>}[d] &
	0 \ar[r] &
	0 \ar[r] \ar@{.>}[d] &
	0 \\
& \ExctCplCoLimAbut{0}=0 &&
	0=\ExctCplCoLimAbut{-1}
} \\
\xymatrix@R=5ex@C=1em{
{\color{blue} 0=\ExctCplCoLimAbutFltrtnBB{-1,1}} \ar@{{ |>}->}@[blue][d] &&& {\color{red} 0=\ExctCplLimAbutFltrtnBB{-2,1}} \ar@{{ |>}->}@[red][d] \\
% line 2
{\color{blue} \ZMod{6}=\ExctCplCoLimAbutFltrtnBB{0,0}} \ar@[blue][r] \ar@{-{ >>}}@[blue][rd] &
	{\color{blue} \ZMod{6} } \ar@{-{ >>}}@[blue][r] &
	{\color{red} 0 } &
	{\color{red} 0 = \ExctCplLimAbutFltrtnBB{-1,0} } \ar@[red][l]_-{\color{red} \cong} \ar@[red][ld]^{\color{red} \cong} \\
% line 3
& {\color{blue} \ZMod{6} } \ar@{=}@[blue][u] &
	{\color{red} 0 } \ar@[red][u]^{\color{red} \cong}
} & 
	\xymatrix@R=5ex@C=1em{
	{\color{blue} 0=\ExctCplCoLimAbutFltrtnBB{-1,1}} \ar@{{ |>}->}@[blue][d] &&& {\color{red} 0=\ExctCplLimAbutFltrtnBB{-2,1}} \ar@{{ |>}->}@[red][d] \\
	% line 2
	{\color{blue} \ZMod{2}=\ExctCplCoLimAbutFltrtnBB{0,0}} \ar@[blue][r] \ar@{-{ >>}}@[blue][rd] &
		{\color{blue} \ZMod{6} } \ar@{-{ >>}}@[blue][r] &
		{\color{red} \ZMod{3} } &
		{\color{red} \ZMod{3} = \ExctCplLimAbutFltrtnBB{-1,0} } \ar@[red][l]_-{\color{red} \cong} \ar@[red][ld]^{\color{red} \cong} \\
	% line 3
	& {\color{blue} \ZMod{2} } \ar@{{ |>}->}@[blue][u] &
		{\color{red} \ZMod{3} } \ar@[red][u]^{\color{red} \cong}
	} & 
	\xymatrix@R=5ex@C=1em{
	{\color{blue} 0=\ExctCplCoLimAbutFltrtnBB{-1,1}} \ar@{{ |>}->}@[blue][d] &&& {\color{red} 0=\ExctCplLimAbutFltrtnBB{-2,1}} \ar@{{ |>}->}@[red][d] \\
% line 2
	{\color{blue} 0=\ExctCplCoLimAbutFltrtnBB{0,0}} \ar@[blue][r] \ar@{-{ >>}}@[blue][rd] &
		{\color{blue} \ZMod{6} } \ar@{=}@[blue][r] &
		{\color{red} \ZMod{6} } &
		{\color{red} \ZMod{6} = \ExctCplLimAbutFltrtnBB{-1,0} } \ar@[red][l]_-{\color{red} \cong} \ar@[red][ld]^{\color{red} \cong} \\
	% line 3
	& {\color{blue} 0 } \ar@{{ |>}->}@[blue][u] &
		{\color{red} \ZMod{6} } \ar@[red][u]^{\color{red} \cong}
	} \\
	\text{matches colimit abutment} & \text{abutment match via extension} & \text{matches limit abutment}
\end{array}
$}
\end{equation*}%
\end{example}

\section[Comparison I]{Comparison I: E-infinity to Universal Abutments}
\label{sec:SpecSeq-Comparison-I}

Consider a morphism $f\from \ExctCpl{C}(1)\to \ExctCpl{C}(2)$ of exact couples. It induces a morphism $f_{\ast}$ of associated spectral sequences along with morphisms
\begin{equation*}
L_{\ast}f\from L_*(1)\longrightarrow L_*(2)  \qquad \text{and}\qquad L^{\ast}f\from L^*(1)\longrightarrow L^*(2)
\end{equation*}
of colimit/limit abutment objects. `Comparing spectral sequences' means to establish a relationship between properties of the maps $L_{\ast}f$ and/or $L^{\ast}f$ and properties of the maps $f^{\infty}\from \ExctCplEPage{\infty}(1)\to \ExctCplEPage{\infty}(2)$. - Here we address the question: Under which conditions is $L_*f$ and/or $L^*f$ a monomorphism, respectively an epimorphism, respectively an isomorphism?

We offer a sampling of conditions under which this happens. Many more such comparison results can easily be composed. They all share a common architecture: First, we use the E-infinity extension theorem (\ref{thm:E-InfinityExtensionThm}) to infer how $f$ acts on the objects $\ExctCplCoLimAbutFltrtnQtntBB{\Vect{x}}$ and/or $\ExctCplLimAbutFltrtnQtntBB{\Vect{x}+\ExctCplJBdg{b}+\ExctCplKBdg{c}}$ from the way it acts on objects $\ExctCplEObjctBB{\infty}{\Vect{x}+\ExctCplJBdg{b}}$. Then, via the background section \ref{sec:ZDiagramComparison}, we use this information to infer how $f$ acts on the universal abutting objects $\ExctCplCoLimAbut{n}$ and $\ExctCplLimAbut{n+\sigma}$.

For exact couple constituents we systematically use the notations introduced in Section \ref{sec:ExactCouples}.

\begin{lemma}[Monomorphism of colimit abutments - I]
\label{thm:Mono-UnivCoLimitAbutments-I}%
Let $f\from \ExctCpl{C}(1)\to \ExctCpl{C}(2)$ be a morphism of regular exact couples such that the following hold for a given $n\in \ZNr$: %
\vspace{-1.7ex}%
\begin{enumerate}[(i)]
\item For all $\Vect{x}$ with $\Dtrmnnt[\ExctCplIBdg{a}\ \ \Vect{x}]=n$, the map $f^{\infty}\from \SSObjctBB{\infty}{\Vect{x}+\ExctCplJBdg{b}}(1)\to \SSObjctBB{\infty}{\Vect{x}+\ExctCplJBdg{b}}(2)$ is a monomorphism.
\item $f$ induces a monomorphism \ $L_{n}f|\from \LimOfOver{\ExctCplCoLimAbutFltrtnBB{\Vect{x}(n)+r\ExctCplIBdg{a}}(1)}{r}\longrightarrow 
\LimOfOver{\ExctCplCoLimAbutFltrtnBB{\Vect{x}(n)+r\ExctCplIBdg{a}}(2)}{r}$.
\vspace{-1.7ex}%
\end{enumerate}
Then $f$ induces a monomorphism $L_nf\from \ExctCplCoLimAbut{n}(1)\to \ExctCplCoLimAbut{n}(2)$.
\end{lemma}
\begin{proof}
For any bidegree $\Vect{x}$ with $\Dtrmnnt[\ExctCplIBdg{a}\ \ \Vect{x}]=n$, we use hypothesis (i) in the E-infinity extension theorem (\ref{thm:E-InfinityExtensionThm}) to obtain this commutative diagram:
\begin{equation*}
\xymatrix@R=5ex@C=4em{
\ExctCplCoLimAbutFltrtnQtntBB{\Vect{x}}(1) \ar@{{ |>}->}[r] \ar[d] &
	\ExctCplEObjctBB{\infty}{\Vect{x}+\ExctCplJBdg{b}}(1) \ar@{{ |>}->}[d]^{f^{\infty}_{\Vect{x}+\ExctCplJBdg{b}}} \\
\ExctCplCoLimAbutFltrtnQtntBB{\Vect{x}}(2) \ar@{{ |>}->}[r] &
	\ExctCplEObjctBB{\infty}{\Vect{x}+\ExctCplJBdg{b}}(2)
}
\end{equation*}
The map $\ExctCplCoLimAbutFltrtnQtntBB{\Vect{x}}(1)\to \ExctCplCoLimAbutFltrtnQtntBB{\Vect{x}}(2)$ is a monomorphism by commutativity. So, the claim follows from (\ref{thm:ZMonoCoLim(A)FromMonoEps_.}).
\end{proof}

\begin{corollary}[Monomorphism of colimit abutments - II]
\label{thm:Mono-UnivCoLimitAbutments-II}%
Suppose the map $f\from \ExctCpl{C}(1)\to \ExctCpl{C}(2)$ of regular exact couples  induces a monomorphism $f^{\infty}\from \ExctCplEPage{\infty}(1)\to \ExctCplEPage{\infty}(2)$. Then $f$ induces a monomorphism $L_nf\from \ExctCplCoLimAbut{n}(1)\to \ExctCplCoLimAbut{n}(2)$ whenever at least one of the conditions below is satisfied: %
\vspace{-4.6ex}%
\begin{enumerate}[(i)]
\item $\LimOf{\ExctCplCoLimAbutFltrtnBB{\Vect{x}(n)+r\ExctCplIBdg{a}}}=0$.
\item The $\ZCat$-diagram $\ExctCplCoLimAbutFltrtnBB{\Vect{x}(n)+r\ExctCplIBdg{a}}$ is originally vanishing.
\end{enumerate}
\end{corollary}
\begin{proof}
This is so because of the implications (ii) $\implies$ (i) $\implies$ hypotheses of (\ref{thm:Mono-UnivCoLimitAbutments-I}).
\end{proof}

Notice that the comparison criteria (\ref{thm:Mono-UnivCoLimitAbutments-I}) and (\ref{thm:Mono-UnivCoLimitAbutments-II}) don't even require the stability of the E-infinity objects.

\begin{lemma}[Epimorphism/isomorphism of colimit abutments - I]
\label{thm:Epi-UnivCoLimitAbutments-I}
Let $f\from \ExctCpl{C}(1)\to \ExctCpl{C}(2)$ be a morphism of regular exact couples such that for a given $n\in \ZNr$ the following hold:
\vspace{-1.6ex}
\begin{enumerate}[(i)]
\item For all $\Vect{x}$ with $\Dtrmnnt[\ExctCplIBdg{a}\ \ \Vect{x}]=n$ the map $f$ induces an isomorphism $\ExctCplCoLimAbutFltrtnQtntBB{\Vect{x}}(1)\to \ExctCplCoLimAbutFltrtnQtntBB{\Vect{x}}(2)$.
\item $f$ induces an epimorphism $\LimOf{\ExctCplCoLimAbutFltrtnBB{\Vect{x}(n)+r\ExctCplIBdg{a}}(1) }\to  \LimOf{\ExctCplCoLimAbutFltrtnBB{\Vect{x}(n)+r\ExctCplIBdg{a}}(2) }$.
\item $f$ induces a monomorphism $\LimOneOf{\ExctCplCoLimAbutFltrtnBB{\Vect{x}(n)+r\ExctCplIBdg{a}}(1)}\longrightarrow \LimOneOf{\ExctCplCoLimAbutFltrtnBB{\Vect{x}(n)+r\ExctCplIBdg{a}}(2)}$.
\vspace{-1.6ex}
\end{enumerate}
Then $f$ induces an epimorphism $L_nf\from \ExctCplCoLimAbut{n}(1)\to \ExctCplCoLimAbut{n}(2)$. If the map in (ii) is an isomorphism, then so is $L_nf$.
\end{lemma}
\begin{proof}
This follows from (\ref{thm:ZEpiCoLim(A)-I}).
\end{proof}

\begin{corollary}[Epimorphism/isomorphism of colimit abutments - II]
\label{thm:Epi/Iso-UnivAbutments-II}%
Let $f\from \ExctCpl{C}(1)\to \ExctCpl{C}(2)$ be a morphism of regular exact couples such that for a given $n\in \ZNr$ the following hold: %
\index{epimorphism!colimit abutments}\index{isomorphism!colimit abutments}%
\vspace{-1.6ex}
\begin{enumerate}[(i)]
\item For all $\Vect{x}$ with $\Dtrmnnt[\ExctCplIBdg{a}\ \ \Vect{x}]=n$, $\ExctCplEObjctBB{\infty}{\Vect{x}+\Vect{b}}(1)$ and $\ExctCplEObjctBB{\infty}{\Vect{x}+\Vect{b}}(2)$ match colimit abutments.
\item For all $\Vect{x}$ with $\Dtrmnnt[\ExctCplIBdg{a}\ \ \Vect{x}]=n$, $f^{\infty}\from \ExctCplEObjctBB{\infty}{\Vect{x}+\Vect{b}}(1)\to \ExctCplEObjctBB{\infty}{\Vect{x}+\Vect{b}}(2)$ is an isomorphism.
\item $f$ induces an epimorphism $\LimOf{\ExctCplCoLimAbutFltrtnBB{\Vect{x}(n)+r\ExctCplIBdg{a}}(1) }\to  \LimOf{\ExctCplCoLimAbutFltrtnBB{\Vect{x}(n)+r\ExctCplIBdg{a}}(2) }$.
\item The $\ZCat$-diagram $\ExctCplDObjctBB{1}{\Vect{x}(n)+r\ExctCplIBdg{a}}(1)$ is originally stable. %
\vspace{-1.6ex}
\end{enumerate}
Then $f$ induces an epimorphism $L_nf\from \ExctCplCoLimAbut{n}(1)\to \ExctCplCoLimAbut{n}(2)$. If the map of limits in (iii) is an isomorphism, then $L_nf$ is an isomorphism as well.
\end{corollary}
\begin{proof}
Conditions (i) and (ii) imply condition (\ref{thm:Epi-UnivCoLimitAbutments-I}.i). Next, (iii) implies (\ref{thm:Epi-UnivCoLimitAbutments-I}.ii). Finally, if the $\ZCat$-diagram $\ExctCplDObjctBB{1}{\Vect{x}(n)+r\ExctCplIBdg{a}}(1)$ is originally stable, then so is the $\ZCat$\MSComp-diagram $\ExctCplCoLimAbutFltrtnBB{\Vect{x}(n)+r\ExctCplIBdg{a}}(1)$, implying that $\LimOneOf{\ExctCplCoLimAbutFltrtnBB{\Vect{x}(n)+r\ExctCplIBdg{a}}(1)} = 0$. But then  (\ref{thm:Epi-UnivCoLimitAbutments-I}.iii) is satisfied, and $L_nf$ is an epimorphism by (\ref{thm:Epi-UnivCoLimitAbutments-I}).

If the map in (iii) is an isomorphism, then $L_nf$ is a monomorphism as by (\ref{thm:Mono-UnivCoLimitAbutments-I}). So $L_nf$ is an isomorphism.
\end{proof}

\begin{lemma}[Monomorphism of limit abutments - I]
\label{thm:Mono-UnivLimitAbutments-I}
Let $f\from \ExctCpl{C}(1)\to \ExctCpl{C}(2)$ be a morphism of regular exact couples such that for a given $n\in \ZNr$ the following hold: %
\index{monomorphism!limit abutments}\index{limit abutment!monomorphism}%
\vspace{-1.6ex}%
\begin{enumerate}[(i)]
\item For$i=1,2$ the filtration diagram of $\ExctCplCoLimAbut{n}(i)$ is constant; i.e. $\ExctCplCoLimAbutFltrtnBB{\Vect{x}(n) + r\ExctCplIBdg{a}}(i)=\ExctCplCoLimAbut{n}(i)$ for all $r\in \ZNr$.
\item For all $\Vect{x}$ with $\Dtrmnnt[\ExctCplIBdg{a}\ \ \Vect{x}]=n$, $f^{\infty}\from \ExctCplEObjctBB{\infty}{\Vect{x}+\ExctCplJBdg{b}}(1)\to \ExctCplEObjctBB{\infty}{\Vect{x}+\ExctCplJBdg{b}}(2)$ is a monomorphism.
\item $L_{n+\sigma}f$ restricts to a monomorphism $\Img{\ExctCplLimAbut{n+\sigma}(1) \to \ExctCplCoLimAbut{n+\sigma}(1)}\longrightarrow \Img{\ExctCplLimAbut{n+\sigma}(1) \to \ExctCplCoLimAbut{n+\sigma}(1)}$. %
\vspace{-1.6ex}%
\end{enumerate}
Then $f$ induces a monomorphism $L^{n+\sigma}\from \ExctCplLimAbut{n+\sigma}(1) \to \ExctCplLimAbut{n+\sigma}(2)$.
\end{lemma}
\begin{proof}
By condition (i), $\ExctCplCoLimAbutFltrtnQtntBB{\Vect{x}}(1)=0=\ExctCplCoLimAbutFltrtnQtntBB{\Vect{x}}(2)$, for all $\Vect{x}$ with $\Dtrmnnt[\ExctCplIBdg{a}\ \ \Vect{x}]=n$. Combining this information with (ii), the E-infinity extension theorem (\ref{thm:E-InfinityExtensionThm}) yields the commutative square
\begin{equation*}
\xymatrix@R=5ex@C=4em{
\ExctCplEObjctBB{\infty}{\Vect{x}+\ExctCplJBdg{b}}(1)  \ar@{{ |>}->}[d]  &
	\ExctCplLimAbutFltrtnQtntBB{\Vect{x}+\ExctCplJBdg{b}+\ExctCplKBdg{c}}(1) \ar@{{ |>}->}[l] \ar[d] \\
\ExctCplEObjctBB{\infty}{\Vect{x}+\ExctCplJBdg{b}}(2) &
	\ExctCplLimAbutFltrtnQtntBB{\Vect{x}+\ExctCplJBdg{b}+\ExctCplKBdg{c}}(2) \ar@{{ |>}->}[l]
}
\end{equation*}
So the vertical map on the right is a monomorphism. Now the claim follows with (\ref{thm:ZMonoLim(A)}.i)
\end{proof}

\begin{corollary}[Monomorphism of limit abutments - II]
\label{thm:Mono-UnivLimitAbutments-II}
Let $f\from \ExctCpl{C}(1)\to \ExctCpl{C}(2)$ be a morphism of regular exact couples such that for a given $n\in \ZNr$ the following hold: %
\vspace{-1.6ex}%
\begin{enumerate}[(i)]
\item For all $\Vect{x}$ with $\Dtrmnnt[\ExctCplIBdg{a}\ \ \Vect{x}]=n$, $f^{\infty}\from \ExctCplEObjctBB{\infty}{\Vect{x}+\ExctCplJBdg{b}}(1)\to \ExctCplEObjctBB{\infty}{\Vect{x}+\ExctCplJBdg{b}}(2)$ is a monomorphism.
\item $\ExctCplCoLimAbut{n}(1)=0= \ExctCplCoLimAbut{n}(2)$ and $\ExctCplCoLimAbut{n+\sigma}(1) = 0$ %
\vspace{-1.6ex}%
\end{enumerate}
Then $f$ induces a monomorphism $L^{n+\sigma}\from \ExctCplLimAbut{n+\sigma}(1) \to \ExctCplLimAbut{n+\sigma}(2)$.
\end{corollary}
\begin{proof}
The claim follows from (\ref{thm:Mono-UnivLimitAbutments-I}) for the following reasons: Condition (i) implies (\ref{thm:Mono-UnivLimitAbutments-I}.ii), and condition (ii) implies both of (\ref{thm:Mono-UnivLimitAbutments-I}.i) and (\ref{thm:Mono-UnivLimitAbutments-I}.iii).
\end{proof}

\begin{lemma}[Isomorphism of universal abutments]
\label{thm:Iso-UnivAbutments-I}%
Let $f\from \ExctCpl{C}(1)\to \ExctCpl{C}(2)$ be a morphism of regular exact couples such that for a given $n\in \ZNr$ the following hold: %
\index{isomorphism!universal abutments}%
\vspace{-1.6ex}%
\begin{enumerate}[(i)]
\item For all $\Vect{x}$ with $\Dtrmnnt[\ExctCplIBdg{a}\ \ \Vect{x}]=n$ and $i=1,2$, $\ExctCplEObjctBB{\infty}{\Vect{x}+\ExctCplJBdg{b}}(i)$ is stable.
\item For all $\Vect{x}$ with $\Dtrmnnt[\ExctCplIBdg{a}\ \ \Vect{x}]=n$, $f^{\infty}\from \ExctCplEObjctBB{\infty}{\Vect{x}+\ExctCplJBdg{b}}(1)\to \ExctCplEObjctBB{\infty}{\Vect{x}+\ExctCplJBdg{b}}(2)$ is an isomorphism.
\item For all $r\in \ZNr$, $\ExctCplCoLimAbutFltrtnQtntBB{\Vect{x}(n)+r\ExctCplIBdg{a}}(1) \to \ExctCplCoLimAbutFltrtnQtntBB{\Vect{x}(n)+r\ExctCplIBdg{a}}(2)$ is an epic or $\ExctCplLimAbutFltrtnQtntBB{\Vect{x}(n+\sigma)+r\ExctCplIBdg{a}}(1) \to \ExctCplLimAbutFltrtnQtntBB{\Vect{x}(n+\sigma)+r\ExctCplIBdg{a}}(2)$ is monic.
\item $\ExctCplCoLimAbut{n}f|\from \LimOf{\ExctCplCoLimAbutFltrtnBB{\Vect{x}(n)+r\ExctCplIBdg{a}}(1)} \to \LimOf{\ExctCplCoLimAbutFltrtnBB{\Vect{x}(n)+r\ExctCplIBdg{a}}(2)}$ is an isomorphism.
\item $\ExctCplCoLimAbut{n}f$ induces a monomorphism $\LimOneOf{\ExctCplCoLimAbutFltrtnBB{\Vect{x}(n)+r\ExctCplIBdg{a}}(1)} \to \LimOneOf{\ExctCplCoLimAbutFltrtnBB{\Vect{x}(n)+r\ExctCplIBdg{a}}(2)}$.
\item $\ExctCplCoLimAbut{n+\sigma}f$ induces an isomorphism $\Img{\ExctCplLimAbut{n+\sigma}(1)\to \ExctCplCoLimAbut{n+\sigma}(1)} \to \Img{\ExctCplLimAbut{n+\sigma}(2)\to \ExctCplCoLimAbut{n+\sigma}(2)}$. %
\vspace{-1.6ex}%
\end{enumerate}
Then $f$ induces isomorphisms $\ExctCplCoLimAbut{n}f\from \ExctCplCoLimAbut{n}(1)\to \ExctCplCoLimAbut{n}(2)$ and $\ExctCplLimAbut{n+\sigma}\from \ExctCplLimAbut{n+\sigma}(1)\to \ExctCplLimAbut{n+\sigma}(2)$.
\end{lemma}
\begin{proof}
We begin by observing that, under assumption (i), conditions (ii), and (iii) combined are equivalent to having isomorphisms
\begin{equation*}
\ExctCplCoLimAbutFltrtnQtntBB{\Vect{x}(n)+r\ExctCplIBdg{a}}(1) \XRA{\cong} \ExctCplCoLimAbutFltrtnQtntBB{\Vect{x}(n)+r\ExctCplIBdg{a}}(2)  \qquad \text{and}\qquad \ExctCplLimAbutFltrtnQtntBB{\Vect{x}(n+\sigma)+r\ExctCplIBdg{a}}(1) \XRA{\cong} \ExctCplLimAbutFltrtnQtntBB{\Vect{x}(n+\sigma)+r\ExctCplIBdg{a}}(2)\qquad \text{for all}\ \ r\in \ZNr.
\end{equation*}
For each $r\in \ZNr$, this follows via the morphism of short exact sequences below.
\begin{equation*}
\xymatrix@R=5ex@C=4em{
\ExctCplCoLimAbutFltrtnQtntBB{\Vect{x}(n) + r\ExctCplIBdg{a}}(1) \ar@{{ |>}->}[r] \ar@{{ |>}->}[d] &
	\ExctCplEObjctBB{\infty}{\Vect{x}(n)  + \ExctCplJBdg{b}+ r\ExctCplIBdg{a}}(1) \ar@{-{ >>}}[r] \ar[d]^{\cong}&
	\ExctCplLimAbutFltrtnQtntBB{\Vect{x}(n)  + \ExctCplJBdg{b} + \ExctCplKBdg{c} + r\ExctCplIBdg{a}}(1) \ar@{-{ >>}}[d] \\
\ExctCplCoLimAbutFltrtnQtntBB{\Vect{x}(n) + r\ExctCplIBdg{a}}(2) \ar@{{ |>}->}[r] &
	\ExctCplEObjctBB{\infty}{\Vect{x}(n)  + \ExctCplJBdg{b}+ r\ExctCplIBdg{a}}(2) \ar@{-{ >>}}[r] &
	\ExctCplLimAbutFltrtnQtntBB{\Vect{x}(n)  + \ExctCplJBdg{b} + \ExctCplKBdg{c} + r\ExctCplIBdg{a}}(2)
}
\end{equation*}
Combined with (iv) and (v), we see that the hypotheses of (\ref{thm:ZEpiCoLim(A)-I}) are satisfied. So $\ExctCplCoLimAbut{n}f$ is an isomorphism. Combining the isomorphism of the vertical arrow on the right with (vi), the hypotheses of (\ref{thm:ZIsoLim(A)-I}) are satisfied. So $\ExctCplLimAbut{n+\sigma}$ is an isomorphism as well.
\end{proof}

\begin{corollary}[Isomorphism of limit abutment - I]
\label{thm:Iso-LimitAbutments-I}
Let $f\from \ExctCpl{C}(1)\to \ExctCpl{C}(2)$ be a morphism of regular exact couples such that for a given $n\in \ZNr$ the following hold: %
\index{isomorphism!limit abutment}%
\vspace{-1.6ex}%
\begin{enumerate}[(i)]
\item For all $\Vect{x}$ with $\Dtrmnnt[\ExctCplIBdg{a}\ \ \Vect{x}]=n$ and $i=1,2$, $\ExctCplEObjctBB{\infty}{\Vect{x}+\ExctCplJBdg{b}}(i)$ matches the limit abutment.
\item For all $\Vect{x}$ with $\Dtrmnnt[\ExctCplIBdg{a}\ \ \Vect{x}]=n$, $f^{\infty}\from \ExctCplEObjctBB{\infty}{\Vect{x}+\ExctCplJBdg{b}}(1)\to \ExctCplEObjctBB{\infty}{\Vect{x}+\ExctCplJBdg{b}}(2)$ is an isomorphism.
\item $\ExctCplCoLimAbut{n+\sigma}f$ induces an isomorphism $\Img{\ExctCplLimAbut{n+\sigma}(1)\to \ExctCplCoLimAbut{n+\sigma}(1)} \to \Img{\ExctCplLimAbut{n+\sigma}(2)\to \ExctCplCoLimAbut{n+\sigma}(2)}$. %
\vspace{-1.6ex}%
\end{enumerate}
Then $f$ induces an isomorphism $\ExctCplLimAbut{n+\sigma}\from \ExctCplLimAbut{n+\sigma}(1)\to \ExctCplLimAbut{n+\sigma}(2)$. \NoProof
\end{corollary}

\begin{corollary}[Isomorphism of limit abutment - II]
\label{thm:Iso-LimitAbutments-II}
Let $f\from \ExctCpl{C}(1)\to \ExctCpl{C}(2)$ be a morphism of regular exact couples such that for a given $n\in \ZNr$ conditions (i) and (ii) of (\ref{thm:Iso-LimitAbutments-I}) are satisfied. Then $f$ induces an isomorphism $\ExctCplLimAbut{n+\sigma}\from \ExctCplLimAbut{n+\sigma}(1)\to \ExctCplLimAbut{n+\sigma}(2)$ whenever at least one of the conditions below hold. %
\index{isomorphism!limit abutment}%
\vspace{-1.6ex}%
\begin{enumerate}[(i)]
\item $\Img{\ExctCplLimAbut{n+\sigma}(1)\to \ExctCplCoLimAbut{n+\sigma}(1)} = 0 = \Img{\ExctCplLimAbut{n+\sigma}(2)\to \ExctCplCoLimAbut{n+\sigma}(2)}$.
\item $\LimOf{\ExctCplCoLimAbutFltrtnBB{\Vect{x}(n)+\sigma+r\ExctCplIBdg{a}}(1) } = 0 = \LimOf{\ExctCplCoLimAbutFltrtnBB{\Vect{x}(n)+\sigma+r\ExctCplIBdg{a}}(2) }$.
\item $\ExctCplCoLimAbut{\Vect{x}(n)+\sigma}(1) = 0 = \ExctCplCoLimAbut{\Vect{x}(n)+\sigma}(2)$.
\item The $\ZCat$\MSComp-diagrams $\ExctCplDObjctBB{}{\Vect{x}(n) + \ExctCplJBdg{b}+\ExctCplKBdg{c} + r\ExctCplIBdg{a}}(1)$ and $\ExctCplDObjctBB{}{\Vect{x}(n) + \ExctCplJBdg{b}+\ExctCplKBdg{c} + r\ExctCplIBdg{a}}(2)$ vanish eventually.
\end{enumerate}
\end{corollary}
\begin{proof}
We have the implications (iv) $\Rightarrow$ (iii) $\Rightarrow$ (ii) $\Rightarrow$ (i) $\Rightarrow$ (\ref{thm:Iso-LimitAbutments-I}.iii).
\end{proof}

\begin{proposition}[Epimorphism of limit abutment]
\label{thm:Epi-LimitAbutment}
Let $f\from \ExctCpl{C}(1)\to \ExctCpl{C}(2)$ be a morphism of regular exact couples such that for a given $n\in \ZNr$ the objects $\ExctCplEObjctBB{\infty}{\Vect{x}(n)+\ExctCplJBdg{b} + r\ExctCplIBdg{a}}(i)$ are stable. Suppose the following hold: %
\index{epimorphism!limit abutment}%
\vspace{-1.6ex}%
\begin{enumerate}[(i)]
\item $f$ induces epimorphisms $\ExctCplEObjctBB{\infty}{\Vect{x}(n)+\ExctCplJBdg{b} + r\ExctCplIBdg{a}}(1)\to \ExctCplEObjctBB{\infty}{\Vect{x}(n)+\ExctCplJBdg{b} + r\ExctCplIBdg{a}}(2)$.
\item $L_{\Vect{x}(n+\sigma)}(1)=0=L_{\Vect{x}(n+\sigma)}(2)$.
\item The structure maps of $\ExctCplDObjctBB{}{\Vect{x}(n)+\ExctCplJBdg{b}+\ExctCplKBdg{c}+r\ExctCplIBdg{a}}(1)$ have kernels satisfying the descending chain condition. %
\vspace{-1.6ex}%
\end{enumerate}
Then $f$ induces an epimorphism $L^{n+\sigma}(1)\to L^{n+\sigma}(2)$. %
\end{proposition}
\begin{proof}
Condition (i) combined with the E-infinity extension theorem imply the $f$ induces epimorphisms
\begin{equation*}
\ExctCplLimAbutFltrtnQtntBB{\Vect{x}(n)+\ExctCplJBdg{b}+\ExctCplKBdg{c}+r\ExctCplIBdg{a}}(1)\longrightarrow \ExctCplLimAbutFltrtnQtntBB{\Vect{x}(n)+\ExctCplJBdg{b}+\ExctCplKBdg{c}+r\ExctCplIBdg{a}}(2)
\end{equation*}
Via (\ref{thm:KernelFiltrationLim(A)-Props}), condition (ii) implies that the kernel filtrations of $\ExctCplLimAbut{n+\sigma}(1)$ and $\ExctCplLimAbut{n+\sigma}(2)$ are exhaustive. The claim follows with (\ref{thm:ZEpiLim(A)-Sufficient}).
\end{proof}

\section[Comparison II]{Comparison II: Inducing an Isomorphism of Spectral Sequences}
\label{sec:SpecSec-Comparison-II}

Consider a morphism $f\from \mathcal{C}(1)\to \mathcal{C}(2)$ of exact couples. In the previous section we provided conditions under which a monomorphism/epimorphism/isomorphism of E-infinity objects of their spectral sequences yields a monomorphism/epimorphism/isomorphism of universal abutting objects.

Here we reverse the direction of this investigation, both for spectral sequences of homological type as well as for spectral sequences of cohomological type. For spectral sequences of homological type, we state conditions under which an isomorphism of colimit abutting objects can only arise from an isomorphism of spectral sequences; see (\ref{thm:CoLimitAbutmentIso->E^2Iso,I}) and (\ref{thm:CoLimitAbutmentIso->E^2Iso,II}). This type of result overlaps with Zeeman's comparison theorems \cite{ECZeeman1957}, and partially generalizes them. Corresponding results for spectral sequences of cohomological type also hold; see (\ref{thm:LimitAbutmentIso->E^2Iso,I}) and (\ref{thm:LimitAbutmentIso->E^2Iso,II}). - We adopt the following setup:

\begin{convention}[Setup - I]
\label{Setup:ReverseComparison-I}%
Let $(f,g)\from \ExctCpl{C}(1) \longrightarrow \ExctCpl{C}(2)$ be a morphism of exact couples of $R$-modules with the following properties: %
\vspace{-1.6ex}
\begin{enumerate}[(1)]
\item The associated spectral sequences satisfy: $\SSObjct{r}{p}{q}(i)=0$ for $r\geq 2$, and $p<0$ or $q<0$.
\item For $r\geq 2$ and $i\in \Set{1,2}$, the differential  $d^r(i)\from \ExctCplEObjctBB{r}{}(i) \to \ExctCplEObjctBB{r}{}(i)$ has bidegree $\Vect{b}_r=(-r,r-1)$.
\item The E-infinity objects match the colimit abutment in the following manner: $\ExctCplCoLimAbut{n}(i)=0$ for $n<0$ and, for each $n\geq 0$\MSComp, $\ExctCplCoLimAbut{n}(i)$ is filtered by we have a filtration
\begin{equation*}
0 = \ExctCplCoLimAbutFltrtn{-1}{n+1}(i)\subseteq \ExctCplCoLimAbutFltrtn{0}{n}(i)\subseteq \cdots \subseteq\ExctCplCoLimAbutFltrtn{n}{0}(i) = \ExctCplCoLimAbut{n}(i)
\end{equation*}
Adjacent filtration quotients match $E^{\infty}$-objects via these short exact sequences:
\begin{equation*}
\xymatrix@R=5ex@C=4em{
0 \longrightarrow \ExctCplCoLimAbutFltrtn{p-1}{q+1}(i) \ar@{{ |>}->}[r] &
	\ExctCplCoLimAbutFltrtn{p}{q}(i) \ar@{-{ >>}}[r] &
	\ExctCplEObjct{\infty}{p}{q}(i) \longrightarrow 0
}
\end{equation*}
\end{enumerate}
\end{convention}

\begin{theorem}[Colimit abutment isomorphism yields spectral sequence isomorphism, I]
\label{thm:CoLimitAbutmentIso->E^2Iso,I}
Using the setup in (\ref{Setup:ReverseComparison-I}), assume the following: %
\vspace{-1.6ex}%
\begin{enumerate}[(i)]
\item %
\label{thm:AbutmentIso->E^2Iso,I-Aug}%
The map of colimit abutment objects $L_n(f)\from \ExctCplCoLimAbut{n}(1) \to \ExctCplCoLimAbut{n}(2)$ is an isomorphism for all $n\in\ZNr$.
\item %
\label{thm:AbutmentIso->E^2Iso,-0,0}%
$\SSMapAt{f}{2}{0}{0}\from \SSObjct{2}{0}{0}(1)\to \SSObjct{2}{0}{0}(2)$ is an isomorphism.
\item %
\label{thm:AbutmentIso->E^2Iso,-q>0}%
If $\SSMapAt{f}{2}{p}{0}\from E^{2}_{p,0}(1)\to E^{2}_{p,0}(2)$ is an isomorphism for $0\leq p\leq n$, then $\SSMapAt{f}{2}{p}{q}$ is an isomorphism for all $(p,q)$ with $p\leq n$ and all $q\in \ZNr$. %
\vspace{-1.6ex}%
\end{enumerate}
Then the induced map $(f^r\, |\, r\geq 2)$ of spectral sequences is an isomorphism.
\end{theorem}
\begin{proof}
We begin by observing that the isomorphisms $L_n(f)$ restrict to monomorphisms
\begin{equation*}\label{eq:AbutmentIso->E^2Iso,I-Monos}
\xymatrix@R=5ex@C=5em{
\ExctCplCoLimAbutFltrtn{p}{q}(1) \ar@{{ |>}->}[r] &
	\ExctCplCoLimAbutFltrtn{p}{q}(2)
}\qquad \text{for all\ }p,q\in \ZNr
\tag{M}
\end{equation*}
Now we prove the theorem by induction: For $n\geq 0$, let $P_n$ assert that the following statements are true.
\begin{enumerate}
\item[$(A_n)$] $f^{2}_{p,q}\from \ExctCplEObjct{2}{p}{q}(1) \XRA{\cong} \ExctCplEObjct{2}{p}{q}(2)$ is an isomorphism for $0\leq p\leq n$ and each $q\in\ZNr$.
\item[$(B_n)$] $f^{r}_{p,q}\from \ExctCplEObjct{r}{p}{q}(1) \XRA{\cong} \ExctCplEObjct{r}{p}{q}(2)$ is an isomorphism for all $p\leq n-1$, $2\leq r\leq n+1-p$, and $q\in\ZNr$.
\item[$(C_n)$] $f^{r}_{p,q}\from \xymatrix@R=5ex@C=2em{ \ExctCplEObjct{r}{p}{q}(1) \ar@{-{ >>}}[r] & \ExctCplEObjct{r}{p}{q}(2) }$ is an epimorphism for all $r\geq 2$, and $p\leq n$ and all $q\in\ZNr$.
\item[$(D_n)$] $f^{\infty}_{p,q}\from \ExctCplEObjct{\infty}{p}{q}(1) \XRA{\cong} \ExctCplEObjct{\infty}{p}{q}(2)$ is an isomorphism for $p\leq n$ and all $q\in\ZNr$.
\item[$(E_n)$] $\ExctCplCoLimAbutFltrtn{p}{q}(1) \XRA{\cong} \ExctCplCoLimAbutFltrtn{p}{q}(2)$ is an isomorphism for $0\leq p\leq n$ and all $q\in\ZNr$.
\item[$(F_n)$] $f^{r}_{p,q}\from \ExctCplEObjct{r}{p}{q}(1) \XRA{\cong} \ExctCplEObjct{r}{p}{q}(2)$ is an isomorphism  for $r\geq 2$ and $0\leq p+q\leq n$.
\end{enumerate}
The graphic below serves to visualize property $P_n$: $(A_n),(C_n),(D_n),(E_n)$ hold in the blue rectangular region; $(B_n)$ applies to the vertical strips; $(F_n)$ holds in the triangular region.
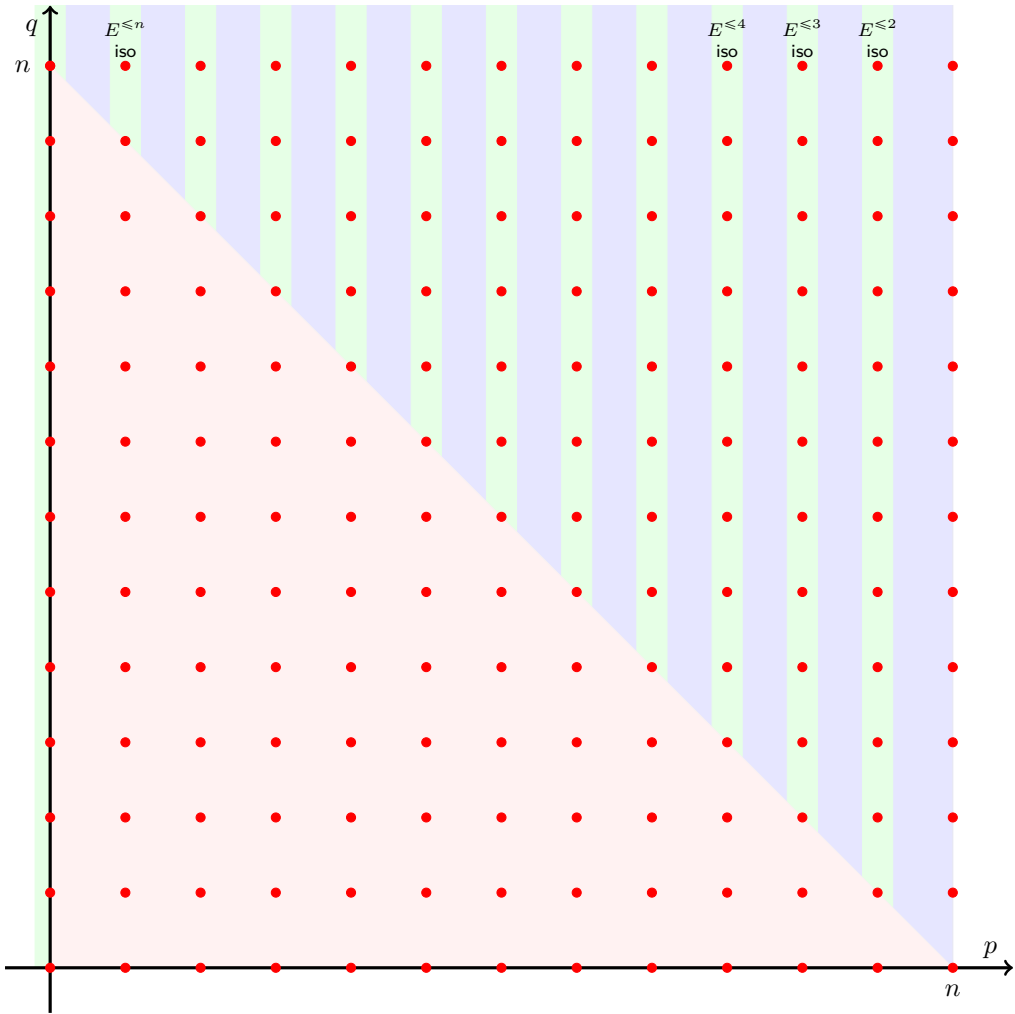
\begin{center}
\begin{tikzpicture}
\clip (-0.6,-0.6) rectangle (12.8,12.8);
\filldraw[fill=blue!10,draw=black!50!blue!10] (0,0) -- (12,0) -- (12,13) -- (0,13) -- (0,0);
\filldraw[fill=green!10,draw=green!50!green!10] (10.8,0) -- (11.2,0) -- (11.2,13) -- (10.8,13) -- (10.8,0);
\node[black,above] at (11,12.0){$\genfrac{}{}{0pt}{}{E^{\leq 2}}{\text{iso}}$};
\filldraw[fill=green!10,draw=green!50!green!10] (9.8,0) -- (10.2,0) -- (10.2,13) -- (9.8,13) -- (9.8,0);
\node[black,above] at (10,12.0){$\genfrac{}{}{0pt}{}{E^{\leq 3}}{\text{iso}}$};
\filldraw[fill=green!10,draw=green!50!green!10] (8.8,0) -- (9.2,0) -- (9.2,13) -- (8.8,13) -- (8.8,0);
\node[black,above] at (9,12.0){$\genfrac{}{}{0pt}{}{E^{\leq 4}}{\text{iso}}$};
\filldraw[fill=green!10,draw=green!50!green!10] (7.8,0) -- (8.2,0) -- (8.2,13) -- (7.8,13) -- (7.8,0);
\filldraw[fill=green!10,draw=green!50!green!10] (6.8,0) -- (7.2,0) -- (7.2,13) -- (6.8,13) -- (6.8,0);
\filldraw[fill=green!10,draw=green!50!green!10] (5.8,0) -- (6.2,0) -- (6.2,13) -- (5.8,13) -- (5.8,0);
\filldraw[fill=green!10,draw=green!50!green!10] (4.8,0) -- (5.2,0) -- (5.2,13) -- (4.8,13) -- (4.8,0);
\filldraw[fill=green!10,draw=green!50!green!10] (3.8,0) -- (4.2,0) -- (4.2,13) -- (3.8,13) -- (3.8,0);
\filldraw[fill=green!10,draw=green!50!green!10] (2.8,0) -- (3.2,0) -- (3.2,13) -- (2.8,13) -- (2.8,0);
\filldraw[fill=green!10,draw=green!50!green!10] (1.8,0) -- (2.2,0) -- (2.2,13) -- (1.8,13) -- (1.8,0);
\filldraw[fill=green!10,draw=green!50!green!10] (0.8,0) -- (1.2,0) -- (1.2,13) -- (0.8,13) -- (0.8,0);
\node[black,above] at (1,12.0){$\genfrac{}{}{0pt}{}{E^{\leq n}}{\text{iso}}$};
\filldraw[fill=green!10,draw=green!50!green!10] (-0.2,0) -- (0.2,0) -- (0.2,13) -- (-0.2,13) -- (-0.2,0);
\filldraw[fill=red!5,draw=black!50!red!5] (0,0) -- (12,0) -- (0,12)  -- (0,0);%  Triangular shape
%\draw[step=1cm,gray,very thin] (0,0) grid (12.8,12.0);%  Coordinate grid
\draw[black,->,very thick] (-0.6,0) -- (12.8,0);%   p-axis
\node[black,above] at (12.5,0){$p$};
\draw[black,->,very thick] (0,-0.6) -- (0,12.8);%   q-axis
\node[black,left] at (0,12.5){$q$};
\filldraw[fill=red!100,draw=red!10!red] (0,0) circle (0.6mm);
\filldraw[fill=red!100,draw=red!10!red] (1,0) circle (0.6mm);
\filldraw[fill=red!100,draw=red!10!red] (2,0) circle (0.6mm);
\filldraw[fill=red!100,draw=red!10!red] (3,0) circle (0.6mm);
\filldraw[fill=red!100,draw=red!10!red] (4,0) circle (0.6mm);
\filldraw[fill=red!100,draw=red!10!red] (5,0) circle (0.6mm);
\filldraw[fill=red!100,draw=red!10!red] (6,0) circle (0.6mm);
\filldraw[fill=red!100,draw=red!10!red] (7,0) circle (0.6mm);
\filldraw[fill=red!100,draw=red!10!red] (8,0) circle (0.6mm);
\filldraw[fill=red!100,draw=red!10!red] (9,0) circle (0.6mm);
\filldraw[fill=red!100,draw=red!10!red] (10,0) circle (0.6mm);
\filldraw[fill=red!100,draw=red!10!red] (11,0) circle (0.6mm);
\filldraw[fill=red!100,draw=red!10!red] (12,0) circle (0.6mm);
\filldraw[fill=red!100,draw=red!10!red] (0,1) circle (0.6mm);
\filldraw[fill=red!100,draw=red!10!red] (1,1) circle (0.6mm);
\filldraw[fill=red!100,draw=red!10!red] (2,1) circle (0.6mm);
\filldraw[fill=red!100,draw=red!10!red] (3,1) circle (0.6mm);
\filldraw[fill=red!100,draw=red!10!red] (4,1) circle (0.6mm);
\filldraw[fill=red!100,draw=red!10!red] (5,1) circle (0.6mm);
\filldraw[fill=red!100,draw=red!10!red] (6,1) circle (0.6mm);
\filldraw[fill=red!100,draw=red!10!red] (7,1) circle (0.6mm);
\filldraw[fill=red!100,draw=red!10!red] (8,1) circle (0.6mm);
\filldraw[fill=red!100,draw=red!10!red] (9,1) circle (0.6mm);
\filldraw[fill=red!100,draw=red!10!red] (10,1) circle (0.6mm);
\filldraw[fill=red!100,draw=red!10!red] (11,1) circle (0.6mm);
\filldraw[fill=red!100,draw=red!10!red] (12,1) circle (0.6mm);
\filldraw[fill=red!100,draw=red!10!red] (0,2) circle (0.6mm);
\filldraw[fill=red!100,draw=red!10!red] (1,2) circle (0.6mm);
\filldraw[fill=red!100,draw=red!10!red] (2,2) circle (0.6mm);
\filldraw[fill=red!100,draw=red!10!red] (3,2) circle (0.6mm);
\filldraw[fill=red!100,draw=red!10!red] (4,2) circle (0.6mm);
\filldraw[fill=red!100,draw=red!10!red] (5,2) circle (0.6mm);
\filldraw[fill=red!100,draw=red!10!red] (6,2) circle (0.6mm);
\filldraw[fill=red!100,draw=red!10!red] (7,2) circle (0.6mm);
\filldraw[fill=red!100,draw=red!10!red] (8,2) circle (0.6mm);
\filldraw[fill=red!100,draw=red!10!red] (9,2) circle (0.6mm);
\filldraw[fill=red!100,draw=red!10!red] (10,2) circle (0.6mm);
\filldraw[fill=red!100,draw=red!10!red] (11,2) circle (0.6mm);
\filldraw[fill=red!100,draw=red!10!red] (12,2) circle (0.6mm);
\filldraw[fill=red!100,draw=red!10!red] (0,3) circle (0.6mm);
\filldraw[fill=red!100,draw=red!10!red] (1,3) circle (0.6mm);
\filldraw[fill=red!100,draw=red!10!red] (2,3) circle (0.6mm);
\filldraw[fill=red!100,draw=red!10!red] (3,3) circle (0.6mm);
\filldraw[fill=red!100,draw=red!10!red] (4,3) circle (0.6mm);
\filldraw[fill=red!100,draw=red!10!red] (5,3) circle (0.6mm);
\filldraw[fill=red!100,draw=red!10!red] (6,3) circle (0.6mm);
\filldraw[fill=red!100,draw=red!10!red] (7,3) circle (0.6mm);
\filldraw[fill=red!100,draw=red!10!red] (8,3) circle (0.6mm);
\filldraw[fill=red!100,draw=red!10!red] (9,3) circle (0.6mm);
\filldraw[fill=red!100,draw=red!10!red] (10,3) circle (0.6mm);
\filldraw[fill=red!100,draw=red!10!red] (11,3) circle (0.6mm);
\filldraw[fill=red!100,draw=red!10!red] (12,3) circle (0.6mm);
\filldraw[fill=red!100,draw=red!10!red] (0,4) circle (0.6mm);
\filldraw[fill=red!100,draw=red!10!red] (1,4) circle (0.6mm);
\filldraw[fill=red!100,draw=red!10!red] (2,4) circle (0.6mm);
\filldraw[fill=red!100,draw=red!10!red] (3,4) circle (0.6mm);
\filldraw[fill=red!100,draw=red!10!red] (4,4) circle (0.6mm);
\filldraw[fill=red!100,draw=red!10!red] (5,4) circle (0.6mm);
\filldraw[fill=red!100,draw=red!10!red] (6,4) circle (0.6mm);
\filldraw[fill=red!100,draw=red!10!red] (7,4) circle (0.6mm);
\filldraw[fill=red!100,draw=red!10!red] (8,4) circle (0.6mm);
\filldraw[fill=red!100,draw=red!10!red] (9,4) circle (0.6mm);
\filldraw[fill=red!100,draw=red!10!red] (10,4) circle (0.6mm);
\filldraw[fill=red!100,draw=red!10!red] (11,4) circle (0.6mm);
\filldraw[fill=red!100,draw=red!10!red] (12,4) circle (0.6mm);
\filldraw[fill=red!100,draw=red!10!red] (0,5) circle (0.6mm);
\filldraw[fill=red!100,draw=red!10!red] (1,5) circle (0.6mm);
\filldraw[fill=red!100,draw=red!10!red] (2,5) circle (0.6mm);
\filldraw[fill=red!100,draw=red!10!red] (3,5) circle (0.6mm);
\filldraw[fill=red!100,draw=red!10!red] (4,5) circle (0.6mm);
\filldraw[fill=red!100,draw=red!10!red] (5,5) circle (0.6mm);
\filldraw[fill=red!100,draw=red!10!red] (6,5) circle (0.6mm);
\filldraw[fill=red!100,draw=red!10!red] (7,5) circle (0.6mm);
\filldraw[fill=red!100,draw=red!10!red] (8,5) circle (0.6mm);
\filldraw[fill=red!100,draw=red!10!red] (9,5) circle (0.6mm);
\filldraw[fill=red!100,draw=red!10!red] (10,5) circle (0.6mm);
\filldraw[fill=red!100,draw=red!10!red] (11,5) circle (0.6mm);
\filldraw[fill=red!100,draw=red!10!red] (12,5) circle (0.6mm);
\filldraw[fill=red!100,draw=red!10!red] (0,6) circle (0.6mm);
\filldraw[fill=red!100,draw=red!10!red] (1,6) circle (0.6mm);
\filldraw[fill=red!100,draw=red!10!red] (2,6) circle (0.6mm);
\filldraw[fill=red!100,draw=red!10!red] (3,6) circle (0.6mm);
\filldraw[fill=red!100,draw=red!10!red] (4,6) circle (0.6mm);
\filldraw[fill=red!100,draw=red!10!red] (5,6) circle (0.6mm);
\filldraw[fill=red!100,draw=red!10!red] (6,6) circle (0.6mm);
\filldraw[fill=red!100,draw=red!10!red] (7,6) circle (0.6mm);
\filldraw[fill=red!100,draw=red!10!red] (8,6) circle (0.6mm);
\filldraw[fill=red!100,draw=red!10!red] (9,6) circle (0.6mm);
\filldraw[fill=red!100,draw=red!10!red] (10,6) circle (0.6mm);
\filldraw[fill=red!100,draw=red!10!red] (11,6) circle (0.6mm);
\filldraw[fill=red!100,draw=red!10!red] (12,6) circle (0.6mm);
\filldraw[fill=red!100,draw=red!10!red] (0,7) circle (0.6mm);
\filldraw[fill=red!100,draw=red!10!red] (1,7) circle (0.6mm);
\filldraw[fill=red!100,draw=red!10!red] (2,7) circle (0.6mm);
\filldraw[fill=red!100,draw=red!10!red] (3,7) circle (0.6mm);
\filldraw[fill=red!100,draw=red!10!red] (4,7) circle (0.6mm);
\filldraw[fill=red!100,draw=red!10!red] (5,7) circle (0.6mm);
\filldraw[fill=red!100,draw=red!10!red] (6,7) circle (0.6mm);
\filldraw[fill=red!100,draw=red!10!red] (7,7) circle (0.6mm);
\filldraw[fill=red!100,draw=red!10!red] (8,7) circle (0.6mm);
\filldraw[fill=red!100,draw=red!10!red] (9,7) circle (0.6mm);
\filldraw[fill=red!100,draw=red!10!red] (10,7) circle (0.6mm);
\filldraw[fill=red!100,draw=red!10!red] (11,7) circle (0.6mm);
\filldraw[fill=red!100,draw=red!10!red] (12,7) circle (0.6mm);
\filldraw[fill=red!100,draw=red!10!red] (0,8) circle (0.6mm);
\filldraw[fill=red!100,draw=red!10!red] (1,8) circle (0.6mm);
\filldraw[fill=red!100,draw=red!10!red] (2,8) circle (0.6mm);
\filldraw[fill=red!100,draw=red!10!red] (3,8) circle (0.6mm);
\filldraw[fill=red!100,draw=red!10!red] (4,8) circle (0.6mm);
\filldraw[fill=red!100,draw=red!10!red] (5,8) circle (0.6mm);
\filldraw[fill=red!100,draw=red!10!red] (6,8) circle (0.6mm);
\filldraw[fill=red!100,draw=red!10!red] (7,8) circle (0.6mm);
\filldraw[fill=red!100,draw=red!10!red] (8,8) circle (0.6mm);
\filldraw[fill=red!100,draw=red!10!red] (9,8) circle (0.6mm);
\filldraw[fill=red!100,draw=red!10!red] (10,8) circle (0.6mm);
\filldraw[fill=red!100,draw=red!10!red] (11,8) circle (0.6mm);
\filldraw[fill=red!100,draw=red!10!red] (12,8) circle (0.6mm);
\filldraw[fill=red!100,draw=red!10!red] (0,9) circle (0.6mm);
\filldraw[fill=red!100,draw=red!10!red] (1,9) circle (0.6mm);
\filldraw[fill=red!100,draw=red!10!red] (2,9) circle (0.6mm);
\filldraw[fill=red!100,draw=red!10!red] (3,9) circle (0.6mm);
\filldraw[fill=red!100,draw=red!10!red] (4,9) circle (0.6mm);
\filldraw[fill=red!100,draw=red!10!red] (5,9) circle (0.6mm);
\filldraw[fill=red!100,draw=red!10!red] (6,9) circle (0.6mm);
\filldraw[fill=red!100,draw=red!10!red] (7,9) circle (0.6mm);
\filldraw[fill=red!100,draw=red!10!red] (8,9) circle (0.6mm);
\filldraw[fill=red!100,draw=red!10!red] (9,9) circle (0.6mm);
\filldraw[fill=red!100,draw=red!10!red] (10,9) circle (0.6mm);
\filldraw[fill=red!100,draw=red!10!red] (11,9) circle (0.6mm);
\filldraw[fill=red!100,draw=red!10!red] (12,9) circle (0.6mm);
\filldraw[fill=red!100,draw=red!10!red] (0,10) circle (0.6mm);
\filldraw[fill=red!100,draw=red!10!red] (1,10) circle (0.6mm);
\filldraw[fill=red!100,draw=red!10!red] (2,10) circle (0.6mm);
\filldraw[fill=red!100,draw=red!10!red] (3,10) circle (0.6mm);
\filldraw[fill=red!100,draw=red!10!red] (4,10) circle (0.6mm);
\filldraw[fill=red!100,draw=red!10!red] (5,10) circle (0.6mm);
\filldraw[fill=red!100,draw=red!10!red] (6,10) circle (0.6mm);
\filldraw[fill=red!100,draw=red!10!red] (7,10) circle (0.6mm);
\filldraw[fill=red!100,draw=red!10!red] (8,10) circle (0.6mm);
\filldraw[fill=red!100,draw=red!10!red] (9,10) circle (0.6mm);
\filldraw[fill=red!100,draw=red!10!red] (10,10) circle (0.6mm);
\filldraw[fill=red!100,draw=red!10!red] (11,10) circle (0.6mm);
\filldraw[fill=red!100,draw=red!10!red] (12,10) circle (0.6mm);
\filldraw[fill=red!100,draw=red!10!red] (0,11) circle (0.6mm);
\filldraw[fill=red!100,draw=red!10!red] (1,11) circle (0.6mm);
\filldraw[fill=red!100,draw=red!10!red] (2,11) circle (0.6mm);
\filldraw[fill=red!100,draw=red!10!red] (3,11) circle (0.6mm);
\filldraw[fill=red!100,draw=red!10!red] (4,11) circle (0.6mm);
\filldraw[fill=red!100,draw=red!10!red] (5,11) circle (0.6mm);
\filldraw[fill=red!100,draw=red!10!red] (6,11) circle (0.6mm);
\filldraw[fill=red!100,draw=red!10!red] (7,11) circle (0.6mm);
\filldraw[fill=red!100,draw=red!10!red] (8,11) circle (0.6mm);
\filldraw[fill=red!100,draw=red!10!red] (9,11) circle (0.6mm);
\filldraw[fill=red!100,draw=red!10!red] (10,11) circle (0.6mm);
\filldraw[fill=red!100,draw=red!10!red] (11,11) circle (0.6mm);
\filldraw[fill=red!100,draw=red!10!red] (12,11) circle (0.6mm);
\filldraw[fill=red!100,draw=red!10!red] (0,12) circle (0.6mm);
\filldraw[fill=red!100,draw=red!10!red] (1,12) circle (0.6mm);
\filldraw[fill=red!100,draw=red!10!red] (2,12) circle (0.6mm);
\filldraw[fill=red!100,draw=red!10!red] (3,12) circle (0.6mm);
\filldraw[fill=red!100,draw=red!10!red] (4,12) circle (0.6mm);
\filldraw[fill=red!100,draw=red!10!red] (5,12) circle (0.6mm);
\filldraw[fill=red!100,draw=red!10!red] (6,12) circle (0.6mm);
\filldraw[fill=red!100,draw=red!10!red] (7,12) circle (0.6mm);
\filldraw[fill=red!100,draw=red!10!red] (8,12) circle (0.6mm);
\filldraw[fill=red!100,draw=red!10!red] (9,12) circle (0.6mm);
\filldraw[fill=red!100,draw=red!10!red] (10,12) circle (0.6mm);
\filldraw[fill=red!100,draw=red!10!red] (11,12) circle (0.6mm);
\filldraw[fill=red!100,draw=red!10!red] (12,12) circle (0.6mm);
\node[black,below]  at (12,-0.1) {$n$};
\filldraw[fill=red!100,draw=red!10!red] (0,12) circle (0.6mm);
\node[black,left]  at (-0.1,12) {$n$};
\filldraw[fill=red!10,draw=green!50!black] (-1,0) -- (-4,0)  arc (180:270:4) -- (0,-1) arc (270:180:1);
\end{tikzpicture}
\end{center}
{\em Claim: $P_0$ is true}\quad $(A_0)$ holds by combining hypotheses \ref{thm:AbutmentIso->E^2Iso,-0,0} and \ref{thm:AbutmentIso->E^2Iso,-q>0}. Condition $(B_0)$ is vacuously satisfied. - Property $(C_0)$ requires us to show that, for each $q\in\ZNr$, we have an epimorphism
\begin{equation*}
\xymatrix@R=5ex@C=4em{
\ExctCplEObjct{r}{0}{q}(1) \ar@{-{ >>}}[r] &
	\ExctCplEObjct{r}{0}{q}(2)
}
\end{equation*}
As exiting differentials from these objects vanish, we have the following commutative diagram for each $r\geq 2$:
\begin{equation*}
\xymatrix@R=5ex@C=3em{
\ExctCplEObjct{2}{0}{q}(1) \ar[d]_{\text{\ref{thm:AbutmentIso->E^2Iso,-q>0}}}^{\cong} \ar@{-{ >>}}[r] &
	\ExctCplEObjct{3}{0}{q}(1) \ar[d] \ar@{-{ >>}}[r] &
	\cdots \ar@{-{ >>}}[r] &
	\ExctCplEObjct{r}{0}{q}(1) \ar[d] \\
\ExctCplEObjct{2}{0}{q}(2) \ar@{-{ >>}}[r] &
	\ExctCplEObjct{3}{0}{q}(1) \ar@{-{ >>}}[r] &
	\cdots \ar@{-{ >>}}[r] &
	\ExctCplEObjct{r}{0}{q}(1)
}
\end{equation*}
So each vertical map is an epimorphism. - Properties $(D_0)$ and $(E_0)$ follow together via this observation:
\begin{equation*}
\xymatrix@R=6ex@C=2em{
\ExctCplEObjct{\infty}{0}{q}(1) \ar[d]_{\cong} \ar[r]^{\cong} &
	\ExctCplEObjct{q+2}{0}{q}(1) \ar@{-{ >>}}[rr]^{(C_0)} &&
	\ExctCplEObjct{q+2}{0}{q}(2) \ar[r]^-{\cong} &
	\ExctCplEObjct{\infty}{0}{q}(2) \ar[d]^{\cong} \\
\ExctCplCoLimAbutFltrtn{0}{q}(1) \ar@{{ |>}->}[rrrr]_{\text{(\ref{eq:AbutmentIso->E^2Iso,I-Monos})}} &&&&
	\ExctCplCoLimAbutFltrtn{0}{q}(2)
}
\end{equation*}
By commutativity, the horizontal maps are simultaneously monic and epic. So, all maps in this diagram are isomorphisms. - Finally, property $(F_0)$ holds because
\begin{equation*}
\xymatrix@R=5ex@C=4em{
\ExctCplEObjct{\infty}{0}{0}(1) \cong \ExctCplEObjct{2}{0}{0}(1) \ar[r]^-{\cong}_{\text{\ref{thm:AbutmentIso->E^2Iso,-q>0}}} &
	\ExctCplEObjct{2}{0}{0}(2)\cong \ExctCplEObjct{\infty}{0}{0}(2)
}
\end{equation*}
This completes the verification of property $P_0$.

Assume now that $P_n$ holds for some $n\geq 0$. We must verify the truth of $P_{n+1}$. We begin with:

{\bfseries Property $(D_{n+1})$ holds in position $(n+1,0)$}\quad To see this, consider the morphism of short exact sequences:
\begin{equation*}
\xymatrix@R=5ex@C=4em{
\ExctCplCoLimAbutFltrtn{n}{1}(1) \ar[d]_{\cong}^{(E_n)} \ar@{{ |>}->}[r] &
	\ExctCplCoLimAbutFltrtn{n+1}{0}(1)=\ExctCplCoLimAbut{n+1}(1) \ar[d]_{\cong}^{\text{ \ref{thm:AbutmentIso->E^2Iso,I-Aug} }} \ar@{-{ >>}}[r] &
	\ExctCplEObjct{\infty}{n+1}{0}(1) \ar[d] \\
\ExctCplCoLimAbutFltrtn{n}{1}(2) \ar@{{ |>}->}[r] &
	\ExctCplCoLimAbutFltrtn{n+1}{0}(2)=\ExctCplCoLimAbut{n+1}(2)  \ar@{-{ >>}}[r] &
	\ExctCplEObjct{\infty}{n+1}{0}(2)
}
\end{equation*}
As the vertical maps on the left and in the middle are isomorphisms, so is the vertical map on the right. With this information we show:

{\bfseries Property $(F_{n+1})$ holds in position $(n+1,0)$}\quad We already know that
\begin{equation*}
\ExctCplEObjct{n+2}{n+1}{0}(1) \cong \ExctCplEObjct{\infty}{n+1}{0}(1) \XRA{\cong} \ExctCplEObjct{\cong}{n+1}{0}(2)\cong \ExctCplEObjct{n+2}{n+1}{0}(2).
\end{equation*}
So, assume inductively that, for $0\leq l<n$ we have an isomorphism
\begin{equation*}
f^{n+2-l}_{n+1,0}\from \ExctCplEObjct{n+2-l}{n+1}{0}(1) \XRA{\cong} \ExctCplEObjct{n+2-l}{n+1}{0}(2)
\end{equation*}
We need to show that $f^{n+2-(l+1)}_{n+1,0}$ is an isomorphism. As incoming differentials into position $(n+1,0)$ vanish, we have the morphism of exact sequences below.
\begin{equation*}
\xymatrix@R=5ex@C=2.5em{
\ExctCplEObjct{n+2-l}{n+1}{0}(1) \ar@{{ |>}->}[r] \ar[d]_{\cong} &
	\ExctCplEObjct{n+2-(l+1)}{n+1}{0}(1) \ar[rr]^-{d^{n+2-(l+1)}(1)} \ar[d] &&
	\ExctCplCycles{n+2-(l+1)}{l}{n-l}(1) \ar@{-{ >>}}[r] \ar[d]_{\cong}^{(F_n)} &
	\ExctCplEObjct{n+2-l}{l}{n-l}(1) \ar[d]_{\cong}^{(F_n)} \\
\ExctCplEObjct{n+2-l}{n+1}{0}(2) \ar@{{ |>}->}[r] &
	\ExctCplEObjct{n+2-(l+1)}{n+1}{0}(2) \ar[rr]_-{d^{n+2-(l+1)}(2)} &&
	\ExctCplCycles{n+2-(l+1)}{l}{n-l}(2) \ar@{-{ >>}}[r] &
	\ExctCplEObjct{n+2-l}{l}{n-l}(2)
}
\end{equation*}
The $5$-lemma yields the required isomorphism $\ExctCplEObjct{n+2-(l+1)}{n+1}{0}(1)\XRA{\cong} \ExctCplEObjct{n+2-(l+1)}{n+1}{0}(2)$. The induction is complete, and we have isomorphisms
\begin{equation*}
f^{r}_{n+1,0}\from \ExctCplEObjct{r}{n+1}{0}(1)\XRA{\cong} \ExctCplEObjct{r}{n+1}{0}(2)\qquad r\geq 2
\end{equation*}
{\bfseries Property $(A_{n+1})$}\quad follow from induction hypothesis ($A_n$), combined with assumption \ref{thm:AbutmentIso->E^2Iso,-q>0}.

{\bfseries Property $(C_{n+1})$}\quad While inferring $(C_{n+1})$, we also collect information helpful toward proving $(B_{n+1})$\MSComp. As $(C_n)$ holds, it remains to show that the map below is an epimorphism for each $r\geq 2$ and $q\in \ZNr$:
\begin{equation*}
\xymatrix@R=5ex@C=4em{
\ExctCplEObjct{r}{n+1}{q}(1) \ar@{-{ >>}}[r]^-{f^{r}_{n+1,0}} &
	\ExctCplEObjct{r}{n+1}{q}(2)
}
\end{equation*}
For $r=2$, this is true because we already established property $(A_{n+1})$. So, assume that we have such an epimorphism for $r\geq 2$, and consider this morphism of exact sequences:
\begin{equation*}
\xymatrix@R=6ex@C=3em{
\ExctCplCycles{r}{n+1}{q}(1) \ar@{{ |>}->}[r] \ar[d]_{u} &
	\ExctCplEObjct{r}{n+1}{q}(1) \ar[r]^-{d^{r}(1)} \ar@{-{ >>}}[d]_{\text{(IH)}}&
	\ExctCplCycles{r}{n+1-r}{q+r-1}(1)	\ar@{-{ >>}}[r] \ar[d]_{\overset{!}{\cong}}^{w} &
	\ExctCplEObjct{r+1}{n+1-r}{q+r-1}(1) \ar[d]^{v} \\
\ExctCplCycles{r}{n+1}{q}(2) \ar@{{ |>}->}[r] &
	\ExctCplEObjct{r}{n+1}{q}(2) \ar[r]_-{d^{r}(2)} &
	\ExctCplCycles{r}{n+1-r}{q+r-1}(2)	\ar@{-{ >>}}[r] &
	\ExctCplEObjct{r+1}{n+1-r}{q+r-1}(2)
}
\end{equation*}
First, we use the morphism of exact sequences below, along with property $(B_n)$, to see that the second vertical arrow from the right is an isomorphism. 
\begin{equation*}
\xymatrix@R=6ex@C=4em{
\ExctCplCycles{r}{n+1-r}{q+r-1}(1) \ar@{{ |>}->}[r]  \ar[d]_{w} &
	\ExctCplEObjct{r}{n+1-r}{q+r-1}(1) \ar[r]^-{d^{n+1-r}(1)} \ar[d]_{\cong}^{(B_n)} &
	\ExctCplEObjct{r}{n+1-2r}{q+2r-2}(1) \ar[d]_{\cong}^{(B_n)} \\
\ExctCplCycles{r}{n+1-r}{q+r-1}(2) \ar@{{ |>}->}[r] &
	\ExctCplEObjct{r}{n+1-r}{q+r-1}(2) \ar[r]_-{d^{n+1-r}(2)} &
	\ExctCplEObjct{r}{n+1-2r}{q+2r-2}(2)
}
\end{equation*}
With $w$ an isomorphism, a diagram chase shows that the map $u$ is an epimorphism, and the map $v$ is an isomorphism. We complete the verification of $(C_{n+1})$ by observing that the epimorphism $u$ induces an epimorphism
\begin{equation*}
\xymatrix@R=5ex@C=4em{
\ExctCplEObjct{r+1}{n+1}{q}(1) \ar@{-{ >>}}[r]^{f^{r+1}_{n+1,q}} &
	\ExctCplEObjct{r+1}{n+1}{q}(2)
}
\end{equation*}
{\bfseries Property $(B_{n+1})$}\quad We only need to collect all available information. We know:
\vspace{-1.6ex}%
\begin{enumerate}[(1)]
\item Have isomorphisms $\ExctCplEObjct{2}{n}{q}(1) \XRA{\cong} \ExctCplEObjct{2}{n}{q}(2)$ by property $(A_n)$.
\item For $p\leq n-1$ property $(B_n)$ gives isomorphisms $\ExctCplEObjct{r}{p}{q}(1) \XRA{\cong} \ExctCplEObjct{r}{p}{q}(2)$ for $2\leq r\leq n+1-p$ and $q\in\ZNr$. %
\vspace{-1.6ex}%
\end{enumerate}
So, it remains to establish such an isomorphism for $r=n+2-p$ and all $p\leq n-1$, $q\in\ZNr$. Now, for $r=n+1-p$, we just found that the map $v=f^{r+1}_{n+1-r,q'+r-1}$ above is an isomorphism for all $r\geq 2$, and all $q'\in \ZNr$. Choosing $q'=q-(n-p)$ yields isomorphisms
\begin{equation*}
f^{r+1}_{n+1-r,q'+r-1} = f^{n+1-p+1}_{n+1-(n+1-p),q-(n-p)+n+1-p-1}=f^{n+2-p}_{p,q},\qquad p\leq n-1,\quad q\in \ZNr
\end{equation*}
So $(B_{n+1})$ holds.

{\bfseries Properties $(D_{n+1})$ and $(E_{n+1})$}\quad As $(D_n)$ and $(E_n)$ hold, it remains to consider positions $(n+1,q)$. Let us consider this morphism of short exact sequences:
\begin{equation*}
\xymatrix@R=5ex@C=4em{
\ExctCplCoLimAbutFltrtn{n}{q+1}(1) \ar@{{ |>}->}[r] \ar[d]_{\cong}^{(E_n)} &
	\ExctCplCoLimAbutFltrtn{n+1}{q}(1) \ar@{-{ >>}}[r] \ar@{{ |>}->}[d]^{\text{(\ref{eq:AbutmentIso->E^2Iso,I-Monos})}} &
	\ExctCplEObjct{\infty}{n+1}{q}(1) \ar@{-{ >>}}[d]^{(C_{n+1})} \\
\ExctCplCoLimAbutFltrtn{n}{q+1}(2) \ar@{{ |>}->}[r] &
	\ExctCplCoLimAbutFltrtn{n+1}{q}(2) \ar@{-{ >>}}[r] &
	\ExctCplEObjct{\infty}{n+1}{q}(2)
}
\end{equation*}
We know that the vertical map on the left is an isomorphism by property $(E_n)$. The vertical in the middle is a monomorphism by general property (M). The vertical map on the right is an epimorphism by property $(C_{n+1})$. In this situation the snake lemma implies that all vertical maps are isomorphisms.

{\bfseries Property $(F_{n+1})$ for $p\leq n$}\quad Based on the validity of $(F_n)$ and what we proved earlier in position $(n+1,0)$, it only remains to establish an isomorphism
\begin{equation*}
f^{r}_{p,n+1-p}\from \ExctCplEObjct{r}{p}{n+1-p}(1) \XRA{\cong} \ExctCplEObjct{r}{p}{n+1-p}(2)\quad \text{for}\ 0\leq p\leq n,\quad r\geq 2
\end{equation*}
By property $(B_{n+1})$ we have such an isomorphism for $2\leq r\leq n+2-p$. So, we are left to deal with the cases $r\geq n+3-p$. Via property $(D_{n+1})$ we see that, for $0\leq p\leq n$,
\begin{equation*}
\ExctCplEObjct{n+3}{p}{n+1-p}(1)\cong \ExctCplEObjct{n+4}{p}{n+1-p}(1)\cong \cdots \cong \ExctCplEObjct{\infty}{p}{n+1-p}(1)\overset{(D_{n+1})}{\cong}\ExctCplEObjct{\infty}{p}{n+1-p}(2) \cong \cdots \cong \ExctCplEObjct{n+4}{p}{n+1-p}(2) \cong \ExctCplEObjct{n+3}{p}{n+1-p}(2)
\end{equation*}
This is so because differentials $d^{\geq n+3}$ exiting from, respectively, arriving at any of these positions are all $0$. So let $0\leq \mu< n+3-(n+3-p)=p$, and assume inductively that we have an isomorphism
\begin{equation*}
f^{n+3-\mu}_{p,n+1-p}\from \ExctCplEObjct{n+3-\mu}{p}{n+1-p}(1) \XRA{\cong} \ExctCplEObjct{n+3-\mu}{p}{n+1-p}(2)
\end{equation*}
The differential $d^{n+3-(\mu+1)}$ arriving at position $(p,n+1-p)$ originates from position
\begin{equation*}
\left( p+n+3-(\mu+1)\ ,\ n+1-p-(n+3-(\mu+1)-1))\right) = (\dots , \mu-p).
\end{equation*}
As $\mu<p$, the second coordinate is negative, and so the incoming differential is the $0$-map. Consequently, we have this  morphism of exact sequences
\begin{equation*}
\xymatrix@R=6ex@C=4em{
\ExctCplEObjct{n+3-\mu}{p}{n+1-p}(1) \ar@{{ |>}->}[r] \ar[d]_{\cong} &
	\ExctCplEObjct{n+3-(\mu+1)}{p}{n+1-p}(1) \ar[r]^-{d^{n+3-(\mu+1)}} \ar[d] &
	\ExctCplCycles{n+3-(\mu+1)}{\dots}{\dots}(1) \ar@{-{ >>}}[r] \ar[d]_{\cong} &
	\ExctCplEObjct{n+3-\mu}{\dots}{\dots}(1) \ar[d]_{\cong} \\
\ExctCplEObjct{n+3-\mu}{p}{n+1-p}(2) \ar@{{ |>}->}[r] &
	\ExctCplEObjct{n+3-(\mu+1)}{p}{n+1-p}(2) \ar[r]_-{d^{n+3-(\mu+1)}} &
	\ExctCplCycles{n+3-(\mu+1)}{\dots}{\dots}(2) \ar@{-{ >>}}[r] &
	\ExctCplEObjct{n+3-\mu}{\dots}{\dots}(2)
}
\end{equation*}
The vertical map on the left is an isomorphism by induction hypothesis. Each of the two vertical maps on the right is an isomorphism by property $(F_n)$. Thus the second vertical map from the left is an isomorphism as well, and the induction is complete. Consequently, property $(F_{n+1})$ holds, and the induction step from $P_n$ to $P_{n+1}$ is established. - This completes the proof of the theorem.
\end{proof}

In Corollary \ref{thm:CoLimitAbutmentIso->E^2Iso-FunctorialEdges,I} below, we use this notation. For $n\geq 0$, regard the set
\begin{equation*}
J_{n}\DefEq (\Prdct{\Set{0,\dots ,n}}{0})\union (\Prdct{0}{\ZNr_{\geq 0}})
\end{equation*}
as a discrete category. Let $\mathcal{F}_n$ denote the functor category $\FuncCat{J_n}{\ModulesOver{R}}$. Then  a page of a spectral sequence such as $E^{2}_{\ast,\ast}$ determines the object $T(E^2)$ in $\mathcal{F}_n$ by sending the pair $(u,v)$ of $J_n$ to $\ExctCplEObjct{2}{u}{v}$.

\begin{corollary}[Colimit abutment isomorphism with $E^2$ functorially determined by edges, I]
\label{thm:CoLimitAbutmentIso->E^2Iso-FunctorialEdges,I}
Using the setting described in (\ref{Setup:ReverseComparison-I}), assume the following: %
\vspace{-1.6ex}%
\begin{enumerate}[(i)]
\item The map of colimit abutment objects $L_n(f)\from \ExctCplCoLimAbut{n}(1) \to \ExctCplCoLimAbut{n}(2)$ is an isomorphism for all $n\in\ZNr$.
\item $\SSMapAt{f}{2}{0}{0}\from E^{2}_{0,0}(1)\to E^{2}_{0,0}(2)$ is an isomorphism.
\item  For each $n\geq 1$ and $q\geq 1$ there is a functor
\begin{equation*}
\Phi_{n,q}\from \mathcal{F}_n \longrightarrow \ModulesOver{R}
\end{equation*}
such that $\ExctCplEObjct{2}{n}{q}(i) =\Phi_{n,q}\left( T(E^2(i)) \right)$.
\vspace{-1.6ex}%
\end{enumerate}
Then the induced map $(f^r\, |\, r\geq 2)$ of spectral sequences is an isomorphism.
\end{corollary}
\begin{proof}
The hypotheses of (\ref{thm:CoLimitAbutmentIso->E^2Iso,I}) are satisfied.
\end{proof}

As a corollary, we infer part of Zeeman's comparison theorem \cite{ECZeeman1957}:

\begin{corollary}[Zeeman's comparison theorem, I]
\label{thm:ZeemanComparison,I}
Using the setup (\ref{Setup:ReverseComparison-I}) interpreted in the category of abelian groups, assume the following: %
\vspace{-1.6ex}%
\begin{enumerate}[(i)]
\item The map of colimit abutment objects $L_n(f)\from \ExctCplCoLimAbut{n}(1) \to \ExctCplCoLimAbut{n}(2)$ is an isomorphism for all $n\in\ZNr$.
\item $\SSMapAt{f}{2}{0}{q}\from E^{2}_{0,q}(1)\to E^{2}_{0,q}(2)$ is an isomorphism for all $q\geq 0$.
\item For each $p\geq 1$ and $q\geq 0$, there is a functorial short exact sequence
\begin{equation*}
\xymatrix@R=5ex@C=4em{
\ExctCplEObjct{2}{p}{0}\otimes \ExctCplEObjct{2}{0}{q}  \ar@{{ |>}->}[r] &
	\ExctCplEObjct{2}{p}{q} \ar@{-{ >>}}[r] &
	\mathit{Tor}_1(\ExctCplEObjct{2}{p-1}{0},\ExctCplEObjct{2}{0}{q})
}
\end{equation*}
\end{enumerate}
Then the induced map $(f^r|r\geq 2)$ of spectral sequences is an isomorphism. \NoProof
\end{corollary}

Theorem \ref{thm:CoLimitAbutmentIso->E^2Iso,I} has the following twin sibling.

\begin{theorem}[Colimit abutment isomorphism yields spectral sequence isomorphism, II]
\label{thm:CoLimitAbutmentIso->E^2Iso,II}
Using the setup described in (\ref{Setup:ReverseComparison-I}), assume the following: %
\vspace{-1.6ex}%
\begin{enumerate}[(i)]
\item %
\label{thm:AbutmentIso->E^2Iso,II-Ln}
The map of colimit abutment objects $L_n(f)\from \ExctCplCoLimAbut{n}(1) \to \ExctCplCoLimAbut{n}(2)$ is an isomorphism for all $n\in\ZNr$.
\item $\SSMapAt{f}{2}{0}{0}\from \ExctCplEObjct{2}{0}{0}(1)\to \ExctCplEObjct{2}{0}{0}(2)$ is an isomorphism.
\item %
\label{thm:AbutmentIso->E^2Iso,II-2-Iso-q}
If $\SSMapAt{f}{2}{0}{q}\from \ExctCplEObjct{2}{0}{q}(1)\to \ExctCplEObjct{2}{0}{q}(2)$ is an isomorphism for $0\leq  q\leq n$, then $\SSMapAt{f}{2}{p}{q}$ is an isomorphism for all $(p,q)$ with $q\leq n$ and all $p\in \ZNr$. %
\vspace{-1.6ex}%
\end{enumerate}
Then the induced map $(f^r\, |\, r\geq 2)$ of spectral sequences is an isomorphism.
\end{theorem}
\begin{proof}
As in the proof of (\ref{thm:CoLimitAbutmentIso->E^2Iso,I}), we make use of the fact that the isomorphisms $L_n(f)$ restrict to monomorphisms
\begin{equation*}
\label{eq:AbutmentIso->E^2Iso,II}
\xymatrix@R=5ex@C=5em{
\ExctCplCoLimAbutFltrtn{p}{q}(1) \ar@{{ |>}->}[r] &
	\ExctCplCoLimAbutFltrtn{p}{q}(2)
}\qquad \text{for all\ }p,q\geq 0
\tag{M}
\end{equation*}
We prove the theorem by induction. For $n\geq 0$, let $P_n$ assert that the following statements are true. %
\vspace{-1.6ex}%
\begin{enumerate}
\item[$(A_n)$] $\xymatrix@R=5ex@C=2em{ f^{r}_{p,q}\from \ExctCplEObjct{r}{p}{q}(1) \ar@{{ |>}->}[r] & \ExctCplEObjct{r}{p}{q}(2)}$ is a monomorphism for $p\in \ZNr$, $q\leq n$, and $r\geq 2$.
\item[$(B_n)$] $f^{r}_{p,q}\from \ExctCplEObjct{r}{p}{q}(1) \XRA{\cong} \ExctCplEObjct{r}{p}{q}(2)$ is an isomorphism for $p\in \ZNr$, $0\leq q\leq n$, $2\leq r\leq n+2-q$.
\item[$(C_n)$] $\ExctCplCoLimAbutFltrtn{p}{q}(1) \XRA{\cong} \ExctCplCoLimAbutFltrtn{p}{q}(2)$ is an isomorphism for $p\in \ZNr$ and $0\leq q\leq n+1$.
\item[$(D_n)$] $f^{\infty}_{p,q}\from \ExctCplEObjct{\infty}{p}{q}(1) \XRA{\cong} \ExctCplEObjct{\infty}{p}{q}(2)$ is an isomorphism for $p\in\ZNr$ and $0\leq q\leq n$.
\item[$(E_n)$] $f^{r}_{p,q}\from \ExctCplEObjct{r}{p}{q}(1) \XRA{\cong} \ExctCplEObjct{r}{p}{q}(2)$ is an isomorphism for $q\leq n$, $p+q\leq n+2$, and all $r\geq 2$.
\end{enumerate}
{\bfseries Claim: $P_0$ is true}\quad $(B_0)$ is true by hypothesis \ref{thm:AbutmentIso->E^2Iso,II-2-Iso-q}. Property $(A_0)$ in position $(p,0)$ requires all vertical maps in this commuting diagram to be monomorphisms.
\begin{equation*}
\xymatrix@R=5ex@C=4em{
\ExctCplEObjct{r}{p}{0}(1) \ar@{{ |>}->}[r] \ar[d] &
	\cdots \ar@{{ |>}->}[r] &
	\ExctCplEObjct{2}{p}{0}(1) \ar[d]_{\cong}^{\text{\ref{thm:AbutmentIso->E^2Iso,II-2-Iso-q}}} \\ 
\ExctCplEObjct{r}{p}{0}(2) \ar@{{ |>}->}[r] &
	\cdots \ar@{{ |>}->}[r] &
	\ExctCplEObjct{2}{p}{0}(2)
}
\end{equation*}
This follows from commutativity, using that the vertical map on the right is a monomorphism by hypothesis \ref{thm:AbutmentIso->E^2Iso,II-2-Iso-q}.

Properties $(C_0)$ and $(D_0)$ follow together via this morphism of short exact sequences:
\begin{equation*}
\xymatrix@R=5ex@C=4em{
\ExctCplCoLimAbutFltrtn{p-1}{1}(1) \ar@{{ |>}->}[r] \ar[d] &
	\ExctCplCoLimAbutFltrtn{p}{0}(1) \ar@{-{ >>}}[r] \ar[d]_{\cong}^{\text{\ref{thm:AbutmentIso->E^2Iso,II-Ln}}} &
	\ExctCplEObjct{\infty}{p}{0}(1) \ar[r]^-{\cong} \ar[d]^{\SSMapAt{f}{\infty}{p}{0}} &
	\ExctCplEObjct{p+1}{p}{0}(1) \ar@{{ |>}->}[d]^{(A_0)}\\
\ExctCplCoLimAbutFltrtn{p-1}{1}(2) \ar@{{ |>}->}[r] &
	\ExctCplCoLimAbutFltrtn{p}{0}(2) \ar@{-{ >>}}[r] &
	\ExctCplEObjct{\infty}{p}{0}(2) \ar[r]^-{\cong} &
	\ExctCplEObjct{p+1}{p}{0}(2)
}
\end{equation*}
Then $\SSMapAt{f}{\infty}{p}{0}$ is an epimorphism by commutativity and a monomorphism by $(A_0)$. Consequently, the vertical map on the left is an isomorphism as well.

Property $(E_0)$: For $p=0,1$ we have, for $r\geq 2$,
\begin{equation*}
\ExctCplEObjct{r}{p}{0}(1) \cong \ExctCplEObjct{2}{p}{0}(1) \xrightarrow[\text{\ref{thm:AbutmentIso->E^2Iso,II-2-Iso-q}}]{\cong} \ExctCplEObjct{2}{p}{0}(2) \cong \ExctCplEObjct{r}{p}{0}(2).
\end{equation*}
Further hypothesis \ref{thm:AbutmentIso->E^2Iso,II-2-Iso-q}  gives $\ExctCplEObjct{2}{2}{0}(1) \cong \ExctCplEObjct{2}{2}{0}(2)$ and, for $r\geq 3$, we have:
\begin{equation*}
\ExctCplEObjct{r}{2}{0}(1) \cong \ExctCplEObjct{3}{2}{0}(1) \cong \ExctCplEObjct{\infty}{2}{0}(1) \xrightarrow[(D_0)]{\cong} \ExctCplEObjct{\infty}{2}{0}(2)\cong \ExctCplEObjct{3}{2}{0}(2)\cong \ExctCplEObjct{r}{2}{0}(2).
\end{equation*}
This completes the verification of $P_0$.

Now let $n\geq 0$, and assume that $P_n$ is true. We must show that $P_{n+1}$ is true.

{\bfseries $(E_{n+1})$ is true in position $(0,n+1)$}\quad that is $\ExctCplEObjct{r}{0}{n+1}(1)\to \ExctCplEObjct{r}{p}{0}(2)$ is an isomorphism for $r\geq 2$. For $r=\infty$, this is true by property $C_n$, via the isomorphism $\ExctCplCoLimAbutFltrtn{0}{n+1}(i)\cong \ExctCplEObjct{\infty}{0}{n+1}(i)$. So, we have the required isomorphism for $r\geq n+3$. Now, let $0\leq l<n+1$, and assume inductively that we have an isomorphism
\begin{equation*}
\ExctCplEObjct{n+3-l}{0}{n+1}(1) \XRA{\cong} \ExctCplEObjct{n+3-l}{0}{n+1}(2).
\end{equation*}
To establish an isomorphism for $r=n+3-(l+1)$, consider this morphism of exact sequences:
\begin{equation*}
\xymatrix@R=5ex@C=3em{
\ExctCplCycles{n+3-(l+1)}{n+3-(l+1)}{l}(1) \ar@{{ |>}->}[r] \ar[d] &
	\ExctCplEObjct{n+3-(l+1)}{n+3-(l+1)}{l}(1) \ar[rr]^-{d^{n+3-(l+1)}(1)} \ar[d]_{(B_n)}^{\cong} &&
	\ExctCplEObjct{n+3-(l+1)}{0}{n+1}(1) \ar@{-{ >>}}[r] \ar[d] &
	\ExctCplEObjct{n+3-l}{0}{n+1}(1) \ar[d]_{\cong}^{\text{(IH)}} \\
\ExctCplCycles{n+3-(l+1)}{n+3-(l+1)}{l}(2) \ar@{{ |>}->}[r] &
	\ExctCplEObjct{n+3-(l+1)}{n+3-(l+1)}{l}(2) \ar[rr]_-{d^{n+3-(l+1)}(2)} &&
	\ExctCplEObjct{n+3-(l+1)}{0}{n+1}(2) \ar@{-{ >>}}[r] &
	\ExctCplEObjct{n+3-l}{0}{n+1}(2)
}
\end{equation*}
The vertical map on the right is an isomorphism by induction hypothesis on $l$. The second vertical map from the left is an isomorphism by $(B_n)$. Thus the second vertical map from the right is an isomorphism if and only if the vertical map on the left is an isomorphism. To see that this is the case, consider this morphism of short exact sequences:
\begin{equation*}
\xymatrix@R=5ex@C=4em{
\ExctCplBndrs{n+3-(l+1)}{n+3-(l+1)}{l}(1) \ar@{{ |>}->}[r] \ar[d]_{(B_n)}^{\cong} &
	\ExctCplCycles{n+3-(l+1)}{n+3-(l+1)}{l}(1) \ar@{-{ >>}}[r] \ar[d] &
	\ExctCplEObjct{n+3-l}{n+3-(l+1)}{l}(1) \ar[d]_{\cong}^{(E_n)} \\
\ExctCplBndrs{n+3-(l+1)}{n+3-(l+1)}{l}(2) \ar@{{ |>}->}[r] &
	\ExctCplCycles{n+3-(l+1)}{n+3-(l+1)}{l}(2) \ar@{-{ >>}}[r] &
	\ExctCplEObjct{n+3-l}{n+3-(l+1)}{l}(2)
}
\end{equation*}
The left and right vertical arrows are isomorphisms by properties $(B_n)$ and $(E_n)$, respectively. So, the vertical map in the middle is an isomorphism. - This completes the induction on $l$, and we have isomorphisms $f^{r}_{0,n+1}$ for all $r\geq 2$.

{\bfseries Property $(B_{n+1})$ in positions $(p,n+1)$}\quad This property holds in position $(0,n+1)$ because we just showed that $(E_{n+1})$ holds in position $(0,n+1)$. So, $(B_{n+1})$ holds in positions $(p,n+1)$ via hypothesis \ref{thm:AbutmentIso->E^2Iso,II-2-Iso-q}. To complete the verification of $(B_{n+1})$, we first need:

{\bfseries Property $(A_{n+1})$}\quad As $(A_n)$ holds, it only remains to establish monomorphisms $\ExctCplEObjct{r}{p}{n+1}(1)\to \ExctCplEObjct{r}{p}{n+1}(2)$ for $p\in\ZNr$. For $r=2$, this follows from the part of property $(B_{n+1})$ we just found true. Assume inductively that we have a monomorphism for some $r\geq 2$. Then we have this commutative square:
\begin{equation*}
\xymatrix@R=5ex@C=4em{
\ExctCplEObjct{r}{p+r}{n+2-r}(1) \ar[r]^{d} \ar[d]^{\cong}_{(B_n)} &
	\ExctCplEObjct{r}{p}{n+1}(1) \ar@{{ |>}->}[d]^{\text{(IH)}} \\
\ExctCplEObjct{r}{p+r}{n+2-r}(2) \ar[r]^{d} &
	\ExctCplEObjct{r}{p}{n+1}(2)
}
\end{equation*}
The vertical map on the left is an isomorphism by property $(B_n)$, and the vertical map on the right is a monomorphism by induction hypothesis. Via the monomorphism propagation lemma (\ref{thm:E^rEpiMonoIsoPropagation}.i), we see that $\SSMapAt{f}{r+1}{p}{n+1}$ is a monomorphism. This completes the inductive verification of property $A_{n+1}$.

{\bfseries Property $(B_{n+1})$ in positions $(p,\leq n)$}\quad For $q\leq n$ condition $(B_n)$ tells us that the map $f^{r}_{p,q}$ is an isomorphism for $2\leq r\leq n+2-q$. So, we must establish an isomorphism for $r=n+3-q$. This follows by applying the isomorphism propagation lemma (\ref{thm:E^rEpiMonoIsoPropagation}.iii) to the commutative diagram below.
\begin{equation*}
\xymatrix@R=5ex@C=4em{
\ExctCplEObjct{n+2-q}{p+n+2-q}{q-(n+1-q)}(1) \ar[r]^-{d^{n+2-q}} \ar[d]_{\cong}^{(B_n)} &
	\ExctCplEObjct{n+2-q}{p}{q}(1) \ar[r]^-{d^{n+2-q}} \ar[d]_{\cong}^{(B_n)} &
	\ExctCplEObjct{n+2-q}{\dots }{n+1}(1) \ar@{{ |>}->}[d]^{(A_{n+1})} \\
\ExctCplEObjct{n+2-q}{p+n+2-q}{q-(n+1-q)}(2) \ar[r]_-{d^{n+2-q}} &
	\ExctCplEObjct{n+2-q}{p}{q}(2) \ar[r]_-{d^{n+2-q}} &
	\ExctCplEObjct{n+2-q}{\dots }{n+1}(2)
}
\end{equation*}
The vertical map on the left is an isomorphism because at least one of the following applies: property $(B_n)$, or domain and codomain vanish. The vertical map in the middle is an isomorphism by $(B_n)$, and the vertical map on the right is a monomorphism by property $(A_{n+1})$.

{\bfseries Properties $(C_{n+1})$ and $(D_{n+1})$}\quad As $(C_n)$ and $(D_n)$ hold, it only remains to consider the cases involved in this morphism of short exact sequences:
\begin{equation*}
\xymatrix@R=5ex@C=4em{
\ExctCplCoLimAbutFltrtn{p-1}{n+2}(1) \ar@{{ |>}->}[r] \ar[d]_{u} &
	\ExctCplCoLimAbutFltrtn{p}{n+1}(1) \ar@{-{ >>}}[r]  \ar[d]_{(C_n)}^{\cong} &
	\ExctCplEObjct{\infty}{p}{n+1}(1) \ar@{{ |>}->}[d]^{(A_{n+1})} \\
\ExctCplCoLimAbutFltrtn{p-1}{n+2}(2) \ar@{{ |>}->}[r] &
	\ExctCplCoLimAbutFltrtn{p}{n+1}(2) \ar@{-{ >>}}[r] &
	\ExctCplEObjct{\infty}{p}{n+1}(2)
}
\end{equation*}
The vertical arrow in the middle is an isomorphism by $(C_{n})$. The one on the right is a monomorphism by $(A_{n+1})$, and an epimorphism by commutativity. Therefore the map $u$ is an isomorphism as well.

{\bfseries Verification of $(E_{n+1})$}\quad The required properties in position $(0,n+1)$ were established earlier. It remains to consider positions $(1,n+1)$, and $(p,q)$ with $p+q=n+3$ and $0\leq q\leq n+1$. In position $(1,n+1)$, we have the required isomorphism for $r=2$ from $(B_{n+1})$. For $r>2$, we obtain the isomorphism inductively via the isomorphism propagation lemma (\ref{thm:E^rEpiMonoIsoPropagation}).

Now consider position $(p,q)$ with $p,q\geq 0$, $p+q=n+3$, and $0\leq q\leq n+1$. From $(B_{n+1})$ we already know that the map $\ExctCplEObjct{r}{p}{q}(1)\to \ExctCplEObjct{r}{p}{q}(2)$ is an isomorphism for $2\leq r\leq n+3-q$. So, we need to establish isomorphisms for $r\geq n+4-q$. We want to use the isomorphism
\begin{equation*}
\ExctCplEObjct{\infty}{p}{q}(1) \XRA{(D_{n+1})} \ExctCplEObjct{\infty}{p}{q}(2).
\end{equation*}
To do so, we keep in mind that $\ExctCplEObjct{r}{p}{q}(i)\cong \ExctCplEObjct{\infty}{p}{q}(i)$ whenever
\begin{equation*}
r\geq \max\Set{p+1,q+2} = \max\Set{n+4-q,q+2}.
\end{equation*}
In the case where $n+4-q\geq q+2$, we have for $r\geq n+4-q$:
\begin{equation*}
\ExctCplEObjct{r}{p}{q}(1) \cong \ExctCplEObjct{\infty}{p}{q}(1) \XRA{\cong} \ExctCplEObjct{\infty}{p}{q}(2)\cong \ExctCplEObjct{r}{p}{q}(2),
\end{equation*}
which is exactly what is needed. In the case where $q+2>n+4-q$, i.e. $q>n+2-q$, we have isomorphisms 
\begin{equation*}
\ExctCplEObjct{r}{p}{q}(1) \cong \ExctCplEObjct{\infty}{p}{q}(1) \XRA{\cong} \ExctCplEObjct{\infty}{p}{q}(2)\cong \ExctCplEObjct{r}{p}{q}(2),\qquad r\geq q+2,
\end{equation*}
and we require such isomorphisms for $n+4-q\leq r\leq q+2$. In this case we have epimorphisms
\begin{equation*}
\xymatrix@R=5ex@C=2em{
\ExctCplEObjct{n+4-q}{p}{q}(i)  \ar@{-{ >>}}[r] &
	\cdots \ar@{-{ >>}}[r] &
	\ExctCplEObjct{q+2-l}{p}{q}(i) \ar@{-{ >>}}[r] &
	\cdots \ar@{-{ >>}}[r] &
	\ExctCplEObjct{q+2}{p}{q}(i)
}
\end{equation*}
because the differential exiting from $\ExctCplEObjct{r}{p}{q}(i)$ has $0$ target: $p-r\leq p-(n+4-q)=p-(p+1)=-1$. For $l\geq 0$ with $n+4-q<q+2-l$ assume inductively that we have an isomorphism
\begin{equation*}
\ExctCplEObjct{q+2-l}{p}{q}(1) \XRA{\cong} \ExctCplEObjct{q+2-l}{p}{q}(2).
\end{equation*}
Now consider this morphism of exact sequences:
\begin{equation*}
\xymatrix@R=5ex@C=5em{
\ExctCplEObjct{q+2-(l+1)}{\dots}{q-(q+2-(l+1)-1)}(1) \ar[r]^-{d^{q+2-(l+1)}} \ar[d]_{u} &
	\ExctCplEObjct{q+2-(l+1)}{p}{q}(1) \ar@{-{ >>}}[r] \ar@{{ |>}->}[d]_{(A_{n+1})} &
	\ExctCplEObjct{p+2-l}{p}{q}(1) \ar[d]_{\cong}^{\text{(IH)}} \\
\ExctCplEObjct{q+2-(l+1)}{\dots}{q-(q+2-(l+1)-1)}(1) \ar[r]_-{d^{q+2-(l+1)}} &
	\ExctCplEObjct{q+2-(l+1)}{p}{q}(1) \ar@{-{ >>}}[r] &
	\ExctCplEObjct{p+2-l}{p}{q}(1)
}
\end{equation*}
The vertical map on the right is an isomorphism by induction hypothesis. The vertical map in the middle is a monomorphism by $(A_{n+1})$, and is seen to be an epimorphism if the map $u$ is one such. In fact, property $(B_{n+1})$ gives that $u$ is an isomorphism provided
\begin{equation*}
\begin{array}{rcl}
q+2-(l+1) & \leq & n+3 - (q-(q+2-(l+1)-1)) \\
q+1-l & \leq & n+3 - l \\
q & \leq & n+2
\end{array}
\end{equation*}
This condition is met, because only $q\leq n+1$ is relevant to property $(E_{n+1})$. Thus $u$ is an isomorphism, and the induction on $l$ is complete. So $(E_{n+1})$ holds.

Thus property $P_n$ implies $P_{n+1}$, and so the proof of the theorem is complete.
\end{proof}

Corollary \ref{thm:CoLimitAbutmentIso->E^2Iso-FunctorialEdges,II} below is (\ref{thm:CoLimitAbutmentIso->E^2Iso-FunctorialEdges,I}), adapted to the hypotheses of Theorem \ref{thm:CoLimitAbutmentIso->E^2Iso,II}. - For $n\geq 0$, regard the set
\begin{equation*}
K_{n}\DefEq (\Prdct{0}{\Set{0,\dots ,n}})\union (\Prdct{\ZNr_{\geq 0}}{0})
\end{equation*}
as a discrete category. Let $\mathcal{G}_n$ denote the functor category $\FuncCat{K_n}{\ModulesOver{R}}$. Then  a page of a spectral sequence such as $E^{2}_{\ast,\ast}$ determines the object $T(E^2)$ in $\mathcal{G}_n$ by sending the pair $(u,v)$ of $K_n$ to $\ExctCplEObjct{2}{u}{v}$.

\begin{corollary}[Colimit abutment isomorphism with $E^2$ functorially determined by edges, II]
\label{thm:CoLimitAbutmentIso->E^2Iso-FunctorialEdges,II}
Using the setting described in (\ref{Setup:ReverseComparison-I}), assume the following: %
\vspace{-1.6ex}%
\begin{enumerate}[(i)]
\item The map of colimit abutment objects $L_n(f)\from \ExctCplCoLimAbut{n}(1) \to \ExctCplCoLimAbut{n}(2)$ is an isomorphism for all $n\in\ZNr$.
\item $\SSMapAt{f}{2}{0}{0}\from E^{2}_{0,0}(1)\to E^{2}_{0,0}(2)$ is an isomorphism.
\item  For each $p\geq 1$ and $n\geq 1$ there is a functor
\begin{equation*}
\Psi_{p,n}\from \mathcal{F}_n \longrightarrow \ModulesOver{R}
\end{equation*}
such that $\ExctCplEObjct{2}{n}{q}(i) =\Psi_{p,n}\left( T(E^2(i)) \right)$.
\vspace{-1.6ex}%
\end{enumerate}
Then the induced map $(f^r\, |\, r\geq 2)$ of spectral sequences is an isomorphism. \NoProof
\end{corollary}

Reverse comparison is also valid for spectral sequences of cohomological type. We deviate from our existing conventions, and follow standard practice by writing $E_{r}^{p,q}$ in place of the generic $E^{r}_{p,q}$. We adopt the following setup.

\begin{convention}[Setup - II]
\label{Setup:ReverseComparison-II}%
Let $(f,g)\from \ExctCpl{C}(1) \longrightarrow \ExctCpl{C}(2)$ be a morphism of exact couples of $R$-modules with the following properties: %
\vspace{-1.6ex}
\begin{enumerate}[(1)]
\item The associated spectral sequences satisfy: $E_{r}^{p,q}(i)=0$ for $p<0$ or $q<0$, and all $r\geq 2$.
\item For $r\geq 2$ and $i\in \Set{1,2}$, the differential  $d_r(i)\from E_{r}(i)\to E_{r}(i)$ has bidegree $\Vect{b}_r=(r,-r+1)$.
\item The E-infinity objects match the limit abutment in the following manner: $\ExctCplLimAbut{n}(i)=0$ for $n<0$ and, for each $n\geq 0$\MSComp, $\ExctCplLimAbutFltrtn{n}{0}(i)$ is filtered by 
\begin{equation*}
0 = \ExctCplLimAbutFltrtn{n}{0}(i)\subseteq \ExctCplLimAbutFltrtn{n-1}{1}(i)\subseteq \cdots \subseteq\ExctCplLimAbutFltrtn{0}{n}(i) \subseteq\ExctCplLimAbutFltrtn{-1}{n+1}(i) = \ExctCplLimAbut{n}(i)
\end{equation*}
Adjacent filtration quotients match $E_{\infty}$ objects via these short exact sequences:
\begin{equation*}
\xymatrix@R=5ex@C=4em{
0 \longrightarrow \ExctCplLimAbutFltrtn{p}{q}(i) \ar@{{ |>}->}[r] &
	\ExctCplLimAbutFltrtn{p-1}{q+1}(i) \ar@{-{ >>}}[r] &
	E_{\infty}^{p,q}(i) \longrightarrow 0
}
\end{equation*}
\end{enumerate}
\end{convention}

\begin{theorem}[Limit abutment isomorphism yields spectral sequence isomorphism, I]
\label{thm:LimitAbutmentIso->E^2Iso,I}
Using the setup in (\ref{Setup:ReverseComparison-II}), assume the following: %
\vspace{-1.6ex}%
\begin{enumerate}[(i)]
\item %
\label{thm:LimitAbutmentIso->E^2Iso,I-Aug}%
The map of limit abutment objects $L^n(f)\from \ExctCplLimAbut{n}(1) \to \ExctCplLimAbut{n}(2)$ is an isomorphism for all $n\in\ZNr$.
\item %
\label{thm:LimitAbutmentIso->E^2Iso,-0,0}%
$f_{2}^{0,0}\from E_{2}^{0,0}(1)\to E_{2}^{0,0}(2)$ is an isomorphism.
\item %
\label{thm:LimitAbutmentIso->E^2Iso,-q>0}%
If $f_{2}^{p,0}\from E_{2}^{p,0}(1)\to E_{2}^{p,0}(2)$ is an isomorphism for $0\leq p\leq n$, then $f_{2}^{p,q}$ is an isomorphism for all $(p,q)$ with $p\leq n$, and all $q\in \ZNr$. %
\vspace{-1.6ex}%
\end{enumerate}
Then the induced map $(f_r\, |\, r\geq 2)$ of spectral sequences is an isomorphism.
\end{theorem}
\begin{proof}
Theorem (\ref{thm:LimitAbutmentIso->E^2Iso,I}) may be proved by induction on $n\geq 0$, applied to property $P_n$. It asserts that the following statements are true.
\begin{enumerate}
\item[$(A_n)$] $f_{2}^{p,q}\from E_{2}^{p,q}(1) \XRA{\cong} E_{2}^{p,q}(2)$ is an isomorphism for $0\leq p\leq n$ and each $q\in\ZNr$.
\item[$(B_n)$] $f_{r}^{p,q}\from E_{r}^{p,q}(1) \to E_{r}^{p,q}(2)$ is a monomorphism for all $p\leq n$, and all $r\geq 2$.
\item[$(C_n)$] $f_{r}^{p,q}\from E_{r}^{p,q}(1)\to E_{r}^{p,q}(2)$ is an isomorphism for all $p\leq n-1$, $2\leq r\leq n+1-p$, and all $q\in \ZNr$.
\item[$(D_n)$] $F^{p,q}(1)\XRA{\cong} F^{p,q}(2)$ is an isomorphism for $p\leq n$ and all $q\in\ZNr$.
\item[$(E_n)$] $f_{\infty}^{p,q}\from E_{\infty}^{p,q}(1)\XRA{\cong} E_{\infty}^{p,q}(2)$ is an isomorphism for $p\leq n$ and all $q\in\ZNr$.
\end{enumerate}
The argument makes use of the monomorphism/isomorphism propagation lemma (\ref{thm:E^rEpiMonoIsoPropagation}), combined with the methods used in the proofs of (\ref{thm:CoLimitAbutmentIso->E^2Iso,I}) and (\ref{thm:CoLimitAbutmentIso->E^2Iso,II}).
\end{proof}

\begin{theorem}[Limit abutment isomorphism yields spectral sequence isomorphism, II]
\label{thm:LimitAbutmentIso->E^2Iso,II}
Using the setup in (\ref{Setup:ReverseComparison-II}), assume the following: %
\vspace{-1.6ex}%
\begin{enumerate}[(i)]
\item %
\label{thm:LimitAbutmentIso->E^2Iso,II-Aug}%
The map of limit abutment objects $L^n(f)\from \ExctCplLimAbut{n}(1) \to \ExctCplLimAbut{n}(2)$ is an isomorphism for all $n\in\ZNr$.
\item %
\label{thm:LimitAbutmentIso->E^2-II-Iso,-0,0}%
$f_{2}^{0,0}\from E_{2}^{0,0}(1)\to E_{2}^{0,0}(2)$ is an isomorphism.
\item %
\label{thm:LimitAbutmentIso->E^2Iso,-p>0}%
If $f_{2}^{0,q}\from E_{2}^{0,q}(1)\to E_{2}^{0,q}(2)$ is an isomorphism for $0\leq q\leq n$, then $f_{2}^{p,q}$ is an isomorphism for all $(p,q)$ with $q\leq n$ and all $p\in \ZNr$. %
\vspace{-1.6ex}%
\end{enumerate}
Then the induced map $(f_r\, |\, r\geq 2)$ of spectral sequences is an isomorphism.
\end{theorem}
\begin{proof}
We begin by observing that the isomorphisms $L^nf$ restrict to monomorphisms $F^{p,q}(1)\to F^{p,q}(2)$ for all $p,q\in \ZNr$. Accordingly, $f_{\infty}^{p,0}\from E_{\infty}^{p,0}(1)\to E_{\infty}^{p,0}(2)$ is a monomorphism for every $p\in \ZNr$.

Theorem (\ref{thm:LimitAbutmentIso->E^2Iso,II}) may be proved by induction on $n\geq 0$, applied to property $P_n$. It asserts that the following statements are true.
\begin{enumerate}
\item[$(A_n)$] $f_{2}^{p,q}\from E_{2}^{p,q}(1) \XRA{\cong} E_{2}^{p,q}(2)$ is an isomorphism for all $p\in \ZNr$, and $0\leq q\leq n$.
\item[$(B_n)$] $f_{r}^{p,q}\from E_{r}^{p,q}(1) \to E_{r}^{p,q}(2)$ is an epimorphism for all $p\in \ZNr$, $q\leq n$, and all $r\geq 2$.
\item[$(C_n)$] $f_{r}^{p,q}\from E_{r}^{p,q}(1)\to E_{r}^{p,q}(2)$ is an isomorphism for all $q\leq n$, and all $q\leq r\leq n+2-q$.
\item[$(D_n)$] $f_{\infty}^{p,q}$ is an isomorphism for $p\in \ZNr$ and $q\leq n$; further $F^{p,q}(1) \XRA{\cong} F^{p,q}(2)$ is an isomorphism for all $p\in \ZNr$ and all $q\leq n+1$.
\item[$(E_n)$] $f_{\infty}^{p,n+1}$ is a monomorphism for all $p\leq \ZNr$.
\item[$(F_n)$] Whenever $q\leq n$, then $f^{p,q}_{r}$ is an isomorphism for all $p+q\leq n+2$ and all $r\geq 2$.
\end{enumerate}
The argument makes use of the monomorphism/isomorphism propagation lemma (\ref{thm:E^rEpiMonoIsoPropagation}), combined with the methods used in the proofs of (\ref{thm:CoLimitAbutmentIso->E^2Iso,I}) and (\ref{thm:CoLimitAbutmentIso->E^2Iso,II}).
\end{proof}

\part{Appendix}
\label{part:Appendix}

\chapter{Certain (Co-)Limits in Module Categories}
\label{chap:CoLimsInModuleCategories}

Given an exact couple $\ExctCpl{C}$, we are interested in the relationship between the filtration stages of its universal abutment objects and the $\SSPage{\infty}$ of its spectral sequence. This discussion involves heavy use of  certain types limits and colimits, as well as associated derived limits.

Assuming basic knowledge of these concepts and their properties from an introductory text to category theory, e.g. \cite{SMacLane1998-Cats}, we fix notation needed for the kinds of (co-)limits we encounter, and we collect relevant background. Much of this material is standard and well documented; see for example Boardman's excellent summary \cite{JMBoardman1999}, as well as the expositions in \cite{CAWeibel1994}, \cite{JPMayKPonto2011}, \cite{SEilenbergJCMoore1961}, \cite{JMilnor1962}.

However, one very technical result is useful in analyzing the meaning of the E-infinity objects of the spectral sequence of an exact couple. Here, it is in some more detail: We will be working with limits and colimits of diagrams $(A,a)$ of $R$-modules which are of the form
\begin{equation*}
\cdots\to  A_{p-1} \XRA{a_{p-1}} A_{p} \XRA{a_p} A_{p+1}\to \cdots
\end{equation*}
Applied to a short exact sequence $A\to B\to C$ of such diagrams, colimit returns a short exact sequence while limit returns a $6$-term exact sequence
\begin{equation*}
\xymatrix@R=5ex@C=3em{
\LimOf{A} \ar@{{ |>}->}[r] &
	\LimOf{B} \ar[r] &
	\LimOf{C} \ar[r] &
	\LimOneOf{A} \ar[r] &
	\LimOneOf{B} \ar@{-{ >>}}[r] &
	\LimOneOf{C}
}
\end{equation*}
Thus the segment $\LimOf{A}\to \LimOf{B}\to \LimOf{C}$ is short exact if and only if the map $\LimOneOf{A}\to \LimOneOf{B}$ is a monomorphism. In general, this $\LimOne$-property is difficult to investigate. Consequently, the sufficient condition $\LimOneOf{A}=0$ has gained significance. More easily verifiable criteria, such as the Mittag-Leffler property.

In the context of the E-infinity extension theorem \ref{thm:E-InfinityExtensionThm}, we are led to analyzing the filtration (\ref{thm:KernelFiltrationLim(A)-Props}) of $\LimOf{A}$ by the kernels of the limit cone maps $\rho^p\from \LimOf{A}\to A_p$, and there we encounter short exact sequences of $\ZCat$-diagrams $K^pA\to K^{p+1}A\to K^{p+1}A/K^pA$ for which  it is surprisingly easy to investigate whether the induced map $\kappa\from \LimOneOf{K^pA}\to \LimOneOf{K^{p+1}A}$ is a monomorphism; see the development following (\ref{thm:Kernel/ImageSequencesOfZ-Diagram}).

Given an ordinal $\lambda$, this investigation motivates the concept of $\lambda$-Mittag-Leffler condition, and results in the insight that $\kappa$ is a monomorphism whenever, $A$ satisfies the $\omega$-Mittag-Leffler condition (\ref{thm:Omega-Mittag-Leffler->Lim1(K^pA)->Lim1(K^(p+1)A)Monomorphism-All-p}).

\section{Setup and conventions}
\label{sec:SetupConventions}

{\bfseries Graded objects}\quad When discussing objects in a category $\mathcal{C}$ which are graded over the integers, we repurpose the symbol $\ZNr$ to denote the discrete category whose objects are the numbers $\dots , -2,-1,0,1,2,\dots$. Now we define a \Defn{$\ZNr$\!-graded object} in $\mathcal{C}$ as a functor $X\from \ZNr \to \mathcal{C}$. Similarly, a \Defn{$(\ZNr\prdct \ZNr)$\MSComp-bigraded object} in $\mathcal{C}$ is a functor $(\ZNr\prdct \ZNr)\to \mathcal{C}$. %
\index{$\ZNr$-graded object}\index{$(\ZNr\prdct\ZNr)$-graded object}\index[not]{$\ZNr$ - integers as a discrete category}

A \Defn{morphism $f\from X\to Y$ of $(\ZNr\prdct \ZNr)$-graded objects} in $\mathcal{C}$ is a natural transformation of functors. For $\Vect{t}\in (\ZNr\prdct \ZNr)$ define the translation functor
\begin{equation*}
\hat{\Vect{t}}\from \ZNr\prdct \ZNr \longrightarrow \ZNr\prdct \ZNr,\qquad \hat{\Vect{t}}(\Vect{a}) \DefEq \Vect{a} + \Vect{t}.
\end{equation*}
A \Defn{morphism $g\from X\to Y$ of bidegree $\Vect{t}$} is given by a natural transformation $g\from X\to Y\Comp \hat{\Vect{t}}$. Thus $g$ consists of a family of morphisms $g_{\Vect{a}}\from X_{\Vect{a}} \to Y_{\Vect{a}+\Vect{t}}$, one such for every $\Vect{a}$ in $(\ZNr\prdct \ZNr)$. %
\index{morphism!of bigraded objects}

{\bfseries $\ZCat$-diagrams}\quad We write $\ZCat$ for the category determined by the ordering structure of the integers: %
\index[not]{$\ZCat$ - category modeled on ordering structure of $\ZNr$}
\begin{equation*}
\ZCat\qquad\qquad
\xymatrix@R=5ex@C=2em{
\cdots \ar[r] &
\ \   -2 \ar[r] &
\ \  -1 \ar[r] &
 0 \ar[r] &
 1 \ar[r] &
 2 \ar[r] &
 \cdots
}
\end{equation*}
The category $\ZCat$ contains the final subcategory $\OrdOmega$, representing the ordering structure of the natural numbers: %
\index[not]{$\OrdOmega$ - category $0\to 1\to 2\to \cdots$}
\begin{equation*}
\OrdOmega\qquad\qquad
\xymatrix@R=5ex@C=3em{
0 \ar[r] &
	1 \ar[r] &
	2 \ar[r] &
	\cdots
}
\end{equation*}
The opposite of $\OrdOmega$ is canonically isomorphic to this initial subcategory of $\ZCat$: %
\index[not]{$\OrdOmegaOp$ - category $\cdots \to -2\to -1\to 0$}
\begin{equation*}
\OrdOmegaOp\quad\quad
\xymatrix@R=5ex@C=3em{
\cdots \ar[r] &
	-2 \ar[r] &
	-1 \ar[r] &
	0
}
\end{equation*}

A diagram in $\mathcal{C}$ modeled on $\ZCat$, a $\ZCat$-diagram for short, is given by a functor $X\from \ZCat\to \mathcal{C}$:
\begin{equation*}
\xymatrix@R=5ex@C=4em{
\cdots \ar[r] &
	X_{-2} \ar[r]^-{x_{-2}} &
	X_{-1} \ar[r]^-{x_{-1}} &
	X_0 \ar[r]^-{x_{0}} &
	X_1 \ar[r]^-{x_{1}} &
	X_2 \ar[r] &
	\cdots 
}
\end{equation*}
A morphism $\xi\from X\to Y$ of $\ZCat$-diagrams is given by a natural transformation of functors.

{\bfseries $\LimOver{\ZCat}$}\quad If $\mathcal{C}$ is complete, then any $\ZCat$-diagram $X$ has a limit, with corresponding universal cone %
\index[not]{$X_{-\infty} \DefEq \LimOfOver{X}{\ZCat}$}%
\begin{equation*}
\lambda\from X_{-\infty} \DefEq \LimOfOver{X}{\ZCat} \longrightarrow X.
\end{equation*}
As the inclusion $\OrdOmegaOp\to \ZCat$ is initial, we will frequently use the canonical isomorphism
\begin{equation*}
\LimOfOver{X}{\ZCat} \longrightarrow \LimOfOver{X|}{\OrdOmegaOp},
\end{equation*}
where $X|$ denotes the restriction of $X$ to $\OrdOmegaOp$. 

{\bfseries $\CoLimOver{\ZCat}$}\quad If $\mathcal{C}$ is cocomplete, then any $\ZCat$-diagram $X\from \ZCat\to \mathcal{C}$ has a colimit, with corresponding universal cocone
\begin{equation*}
\tau\from X \longrightarrow \CoLimOfOver{X}{\ZCat} \EqDef X_{\infty}.
\end{equation*}
As the inclusion $\OrdOmega\to \ZCat$ is final, we will frequently use the canonical isomorphism
\begin{equation*}
\CoLimOfOver{X|}{\OrdOmega} \longrightarrow \CoLimOfOver{X}{\ZCat}.
\end{equation*}

\begin{definition}[Types of $\ZCat$-diagrams]
\label{def:Z-DiagramTypes}%
In a pointed category $\mathcal{C}$, i.e. $\mathcal{C}$ has a $0$-object, a $\ZCat$-diagram $X$ is called
\begin{enumerate}
\item \Defn{originally vanishing} if there exists $n_0\in \ZNr$ such that $X_n=0$ for $n\leq n_0$; %
\index{originally!vanishing $\ZCat$-diagram}\index{$\ZCat$-diagram!originally vanishing}
\item \Defn{originally stable} if there exists $n_0\in\ZNr$ such that $X_n=X_{n_0}$ for $n\leq n_0$; %
\index{originally!stable $\ZCat$-diagram}\index{$\ZCat$-diagram!originally stable}
\item \Defn{eventually stable} if there exists $n_0\in\ZNr$ such that $X_n=X_{n_0}$ for $n\geq n_0$;
\index{eventually!stable $\ZCat$-diagram}\index{$\ZCat$-diagram!eventually stable}
\item \Defn{eventually vanishing} if there exists $n_0\in \ZNr$ such that $X_n=0$ for $n\geq n_0$. %
\index{eventually!vanishing $\ZCat$-diagram}\index{$\ZCat$-diagram!eventually vanishing}
\end{enumerate}
\end{definition}

\section[Colimits of $\ZCat$-diagrams of $R$-modules]{Colimits of $\ZCat$-diagrams of $R$-modules}
\label{sec:ZCoLim}

Let us now specialize to colimits of $\ZCat$-diagrams of left modules over some unital ring $R$. For the following basic properties of such diagrams, the reader may consult \cite{JJRotman2009,CAWeibel1994}.

\begin{proposition}[Construction of colimits of $\ZCat$-diagrams in $\LModules{R}$]
\label{thm:CoLim^Z(A)}
The colimit of a $\ZCat$-diagram  $A\from \ZCat\to \LModules{R}$ may be constructed as the cokernel in this short exact sequence:
\begin{equation*}
\xymatrix@R=5ex@C=2em{
0 \ar[r] &
	\FamSum{q\in\ZNr}{A_q} \ar@{{ |>}->}[rr]^-{\partial} &&
	\FamSum{n\in\ZNr}{A_n} \ar@{-{ >>}}[rr]^-{\pi} &&
	A_{\infty} \ar[r] &
	0
}
\end{equation*}
If $a_p\from A_p\to A_{p+1}$, $p\in \ZNr$, are the structure maps of $A$, then $\partial$ is the map whose restriction to $A_q$ is $A_q \XRA{\IdMapOn{A_q} - a_q} \Sum{A_q}{A_{q+1}}$.   The structure maps $\pi_p\from A_p\to A_{\infty}$ of a colimit cocone are given by the composites $A_p \XRA{\InclsnOf{p}} \bigoplus A_n \XRA{\pi} A_{\infty}$. \NoProof
\end{proposition}

\begin{corollary}[Basic properties of colimits of $\ZCat$-diagrams]
\label{thm:CoLim(ZCatDiagram)-Properties}%
In the situation of (\ref{thm:CoLim^Z(A)}),  every $y\in A_{\infty}$ is in the image of some $\pi_p\from A_p\to A_{\infty}$. If $x_p\in A_p$, then $\pi_p(x_p)=0$ if and only if there exists $n\geq p$ such that $x_p$ belongs to the kernel of the structure map $A_p\to A_{n}$.  \NoProof
\end{corollary}

\begin{definition}[Exact sequence of $\ZCat$-diagrams of $R$-modules]
\label{def:ExactSequenceZ-Diagrams}
A diagram of $A\XRA{a} B \XRA{a}C$ is called exact in position $B$ if, for each $p\in \ZNr$ the sequence of $R$-modules below is exact in position $B_p$:
\begin{equation*}
\xymatrix@R=5ex@C=4em{
A_p \ar[r]^-{a_p} &
	B_p \ar[r]^-{b_p} &
	C_p
}
\end{equation*}
The given diagram is short exact if the diagram $0\to A \XRA{a} B \XRA{b} C \to 0$ is exact in positions $A,B,C$.
\end{definition}

\begin{theorem}[$\CoLimOver{\ZCat}$ is exact]
\label{thm:CoLim^ZExactFunctor}
If $A\XRA{\alpha} B \XRA{\beta} C$ is a short exact sequence of $\ZCat$-diagrams, then the associated sequence of colimits is short exact:
\begin{equation*}
\xymatrix@R=5ex@C=4em{
A_{\infty} \ar[r]^-{\alpha_{\infty} } &
	B_{\infty} \ar[r]^-{\beta_{\infty} } &
	C_{\infty}
}
\end{equation*}
\end{theorem}

\begin{corollary}[Kernels / images in $\ModulesOver{R}$ commute with $\CoLimOver{\ZCat}$]
\label{thm:CoLim^ZCat-InModules-CommutesKer/Img}
A morphism $\alpha\from A\to B$ of $\ZCat$-diagrams of $R$-modules, yields this diagram of short exact colimit terms:
\begin{equation*}
\xymatrix@R=5ex@C=3em{
(\Ker{\alpha})_{\infty} \ar@{{ |>}->}[r] &
	A_{\infty} \ar@{-{ >>}}[r] \ar@/_1.5em/[rr]_-{\alpha_{\infty}} &
\Img{\alpha}_{\infty} \ar@{{ |>}->}[r] &
	B_{\infty} \ar@{-{ >>}}[r] &
	(\CoKer{\alpha})_{\infty}
}
\end{equation*}
So, the following functorial isomorphisms result: $(\Ker{\alpha})_{\infty}\cong \Ker{\alpha_{\infty}}$ and $(\Img{\alpha})_{\infty}\cong \Img{\alpha_{\infty}}$. \NoProof
\end{corollary}

\section[Limits of $\ZCat$-diagrams of $R$-modules]{Limits of $\ZCat$-diagrams of $R$-modules}
\label{sec:ZLim}\label{sec:LimOne}

We turn to limits of $\ZCat$-diagrams of modules over a unital ring $R$. On $\ZCat$-diagrams in $\ModulesOver{R}$ the limit functor commutes with kernels (as both are limits), but it fails to be right exact. This fact complicates considerably the discussion of the relationship between the limit abutting objects of an exact couple and its associated spectral sequence. Some aspects of this relationship require a detailed analysis of properties of the derived functors of the $\LimOver{\ZCat}$.  Fortunately, only $\LimOneOver{\ZCat}$ enters explicitly into computations, as the higher derived functors $\LimDrvdOfOver{\geq 2}{}{\ZCat}$ vanish.

In (\ref{thm:Lim^Z(A)Construction}), we rely on Steenrod's simultaneous construction of both, $\LimOver{\ZCat}$ and $\LimOneOver{\ZCat}$ and develop their properties by explicit computation; \cite[p.~87]{JWMilnor2009}, and \cite[Sec.~2]{SEilenbergJCMoore1961} in abelian categories. Natural uniqueness of these constructions follows as in \cite[3.5]{CAWeibel1994}, which we also recommend for background on $\Lim$ and $\LimOne$.

\begin{proposition}[Construction of $\LimOver{\ZCat}$ and $\LimOneOver{\ZCat}$]
\label{thm:Lim^Z(A)Construction}
The limit of a $\ZCat$-diagram $A\from \ZCat\to \ModulesOver{R}$ may be constructed functorially as the kernel in this exact sequence:
\begin{equation*}
\xymatrix@R=5ex@C=1.5em{
0 \ar[r] &
	A_{-\infty} \ar@{{ |>}->}[rr] &&
	\FamPrdct{n\in\ZNr}{A_n} \ar[rr]^-{d} &&
	\FamPrdct{q\in\ZNr}{A_q} \ar@{-{ >>}}[rr]^-{\pi} &&
	\LimOneOf{A} \ar[r] &
	0
}
\end{equation*}
The $q$-th coordinate map of $d$ is $\PrjctnOnto{q}- (a_{q-1}\Comp \PrjctnOnto{q-1})\from \FamPrdct{n\in\ZNr}{A_n} \to A_q$, with  $a_{q-1}\from A_{q-1}\to A_{q}$, $q\in \ZNr$, are the structure maps of $A$, then . The universal maps of a limit cone are maps $\rho_p\from  A_{-\infty}\to A_p$ obtained by restricting the coordinate projections $\FamPrdct{n\in\ZNr}{A_n}\to A_p$ to $A_{-\infty}$. \NoProof
\end{proposition}

\begin{proposition}[$6$-term exact $\Lim$ - $\LimOne$-sequence]
\label{thm:Lim-LimOne6TermExact Sequence}%
A short exact sequence (\ref{def:ExactSequenceZ-Diagrams}) $A\XRA{\alpha} B\XRA{\beta} C$ of $\ZCat$-diagrams yields the $6$-term exact sequence below.
\begin{equation*}
0\to A_{-\infty} \XRA{\alpha_{-\infty}} B_{-\infty} \XRA{\beta_{-\infty}} C_{-\infty} \XRA{\delta} \LimOneOf{A} \XRA{\alpha^{1}} \LimOneOf{B} \XRA{\beta^{1}} \LimOneOf{C} \to 0
\end{equation*}
This $6$-term sequence is functorial with respect to morphisms of short exact sequences of $\ZCat$-diagrams. \NoProof
\end{proposition}

Via (\ref{thm:Lim-LimOne6TermExact Sequence}) we see that the functor $\LimOver{\ZCat}$ is left exact but, in general, not exact. So, we look for conditions under which the lim-sequence $A_{-\infty}\to B_{-\infty}\to C_{-\infty}$ is short exact. Sufficient is: $\LimOneOf{A}=0$.

\begin{proposition}[Epimorphic structure maps yield $\LimOneOf{A}=0$]
\label{thm:Lim^1=0IfStructureMapsAreEpimorphic}
Suppose a $\ZCat$-diagram $A$ in $\ModulesOver{R}$ as epimorphic structure maps $a_p\from A_p\to A_{p+1}$. Then $\LimOneOf{A}=0$. 
\end{proposition}
\begin{proof}
We must show that $d\from \FamPrdct{n\in\ZNr}{A_n}\to \FamPrdct{q\in\ZNr}{A_q}$ is surjective. So, let $y=(\dots,y_{-1},y_0,y_1,\dots )\in \Pi A_n$. Define $x=(\dots ,x_{-1},x_0,x_1,\dots )$ by setting
\begin{equation*}
x_{-1}\DefEq -y_0,\qquad x_0\DefEq 0,\qquad x_1\DefEq y_1,
\end{equation*}
followed by these two recursive processes: For $k\geq 1$, suppose $(x_{-k},\dots ,0,\dots ,x_k)$ has been defined with
\begin{equation*}
d|(x_{-k},\dots ,0,\dots ,x_k) = (y_{-k+1},\dots ,y_k).
\end{equation*}
Then choose $x_{-k-1}\in a_{-k-1}^{-1}(x_{-k}-y_{-k})$, and put $x_{k+1}\DefEq y_{k+1}+a_k(x_k)$. The resulting $x\in \Pi A_n$ satisfies $d(x)=y$, and the claim follows.
\end{proof}

After more preparation in the next section, we will come back to conditions under which $\LimOver{\ZCat}$ is exact.

\section{Internal Image Factorization of $\ZCat$-Diagrams}
\label{sec:ImageFactorization-Z-Diagrams}

Given a $\ZCat$-diagram $(A,a)$, we analyze properties of its image and quotient subdiagrams. This development serves two purposes: (a) to gain deeper understanding of $\LimOne$-terms and, as a consequence: (b) to analyze the transfinite stabilization of the spectral sequence associated to a $(\Prdct{\ZNr}{\ZNr})$-bigraded exact couple in Sections \ref{sec:ExactCouples}, \ref{sec:Convergence-I}, \ref{sec:ConvergenceII}. We build upon \cite[Sec.~2]{SEilenbergJCMoore1961} and \cite[Sec.~3]{JMBoardman1999}.

A key role in what follows is played by the tautological morphism $\hat{a}\from A_{\ast}\to A_{\ast+1}$. In the diagram below, the kernel and cokernel of $\hat{a}$ are denoted $KA$ and $RA$, respectively.
\begin{equation*}
\vcenter{\tag{IQ}\label{eq:ImageQuotient-Diagrams}%
\xymatrix@R=7ex@C=4em{
\ZDiagKer{}{p-1}{A} \ar@{{ |>}->}[r] \ar[d]_{0} &
	A_{p-1} \ar@{-{ >>}}[r] \ar@/^3ex/[rr]^-{a_{p-1}} \ar[d]_{a_{p-1}} &
	\ZDiagQuo{}{p-1}{A} = \ZDiagImg{}{p}{A} \ar@{{ |>}->}[r] \ar@<-5ex>[d]_{\ZDiagQuo{}{p-1}{a}} \ar@<+5ex>[d]^{\ZDiagImg{}{p}{a}} &
	A_p \ar@{-{ >>}}[r] \ar[d]^{a_p} &
	\ZDiagCoKer{}{p}{A} \ar[d]^{0} \\
\ZDiagKer{}{p}{A} \ar@{{ |>}->}[r] &
	A_{p} \ar@{-{ >>}}[r] \ar@/_3ex/[rr]_-{a_{p}} &
	\ZDiagQuo{}{p}{A} = \ZDiagImg{}{p+1}{A} \ar@{{ |>}->}[r] &
	A_{p+1} \ar@{-{ >>}}[r] &
	\ZDiagCoKer{}{p+1}{A}
} }
\end{equation*}
The image factorizations of the maps $a_p$ yield a $\ZCat$-diagram which has two distinct relationships to $A$ as indicated:

\begin{definition}[Image subdiagram / image quotient diagram of a $\ZCat$-diagram]
\label{def:Image/QuotientSubdiagrams}
In the setting described above:
\vspace{-2ex}
\begin{enumerate}[$\bullet$]
\item the epic map of $\ZCat$-diagrams $A\to \ZDiagQuo{}{}{A}$ is called the \Defn{image quotient diagram} of $A$ , and %
\index{image!quotient diagram}\index[not]{$\ZDiagQuo{}{}{A}$ - image quotient diagram of $A$}%
\item the kernel map $\ZDiagImg{}{}{A}\to A$ is called the \Defn{image subdiagram} of $A$. %
\index{image!subdiagram}\index[not]{$\ZDiagImg{}{}{A}$ - image subdiagram of $A$}%
\end{enumerate} 
\end{definition}

\begin{remark}[Image subdiagram diagram vs. image quotient diagram]
\label{rem:ImageSub/Quo-Diagrams}%
Starting from a $\ZCat$-diagram $A$, the two $\ZCat$-diagrams $\ZDiagQuo{}{}{A}$ and $\ZDiagImg{}{}{A}$ are the same up to a shift in positions. So, it may seem redundant to introduce distinguishing notation. However, the constructions $Q$ and $I$ have different outcomes under iteration. Further, $\ZDiagQuo{}{}{A}$ is related to $A$ via a cokernel, while $\ZDiagImg{}{}{A}$ is related to $A$ via a kernel. Thus, when taking taking limits, respectively colimits, of $\ZCat$-diagrams, the image quotient diagrams and the image subdiagrams respond differently. This distinction will prove to be essential.
\end{remark}

\begin{proposition}[Image quotient operation: properties]
\label{thm:ImageQuotient-Props}%
\label{thm:Q^(omega)_pA-All-tau}% deprecated
Given a $\ZCat$-diagram $A$ and $r\geq 1$, let $\ZDiagQuo{r+1}{}{A}\DefEq \ZDiagQuo{}{}{\left( \ZDiagQuo{r}{}{A} \right)}$. Then the following hold:
\vspace{-2ex}
\begin{enumerate}[(i)]
\item If the structure maps of $A$ are monomorphisms, then $A\to \ZDiagQuo{}{}{A}$ is an isomorphism.
\label{exa:Q^(tau)A-MonomorphicStructureMaps}% deprecated
\item For $r\geq 1$ and $p\in \ZNr$
\begin{equation*}
A_p \longrightarrow \ZDiagQuo{r}{p}{A}\quad \text{is given by}\quad A_p \longrightarrow \Img{A_p\to A_{p+1}\to \cdots \to A_{p+r}}
\end{equation*}
\item If $\ZDiagQuo{\omega}{}{A}\DefEq \CoLimOf{A\to \ZDiagQuo{}{}{A}\to \cdots \ZDiagQuo{r}{}{A}\to \cdots}$, then $\ZDiagQuo{\omega}{p}{A}\cong \Img{A_p\to \CoLimOf{A}}$ for every $p\in \ZNr$.
\item $\ZDiagQuo{\omega}{}{A}$ has monomorphic structure maps $\ZDiagQuo{\omega}{p}{a}\from \ZDiagQuo{\omega}{p}{A}\to \ZDiagQuo{\omega}{p+1}{A}$.
\end{enumerate}
\end{proposition}
\begin{proof}
(i) is checks directly. (ii) follows by induction. For $r=1$, the claim holds by definition. If the claim holds for $r\geq 1$, then we face this map of $\ZCat$-diagrams:
\begin{equation*}
\xymatrix@R=5ex@C=1em{
\vdots \ar[d] && \vdots \ar[d] \\
A_p \ar@{-{ >>}}[rr]^-{\xi_{p}} \ar[d]_{a_p} &&
	\ZDiagQuo{r}{p}{A} \ar@{=}[r] \ar[d]^{\ZDiagQuo{r}{p}{a}} &
	\Img{a_{p+r-1}\Comp \cdots\Comp a_p} \\
A_{p+1} \ar@{-{ >>}}[rr]_-{\xi_{p+1}} \ar[d] &&
	\ZDiagQuo{r}{p+1}{A} \ar@{=}[r] \ar[d] &
	\Img{a_{p+r}\Comp \cdots\Comp a_{p+1}} \\
\vdots && \vdots 
}
\end{equation*}
Then $\ZDiagQuo{r}{p}{a}$ is the factorization of $\xi_{p+1}a_p$ through  $\xi_p$. Since $\xi_p$ is surjective, we have
\begin{equation*}
Q_p(\ZDiagQuo{r}{}{A}) = \Img{\ZDiagQuo{r}{p}{a}} = \Img{\ZDiagQuo{r}{p}{a}\Comp \xi_p} = \Img{\xi_{p+1}\Comp a_p} = \Img{a_{p+r}\Comp \cdots \Comp a_{p+1}\Comp a_p} = \ZDiagQuo{r+1}{p}{A}.
\end{equation*}
To see (iii), consider the $\omega$-diagram of image factorizations:
\begin{equation*}
\xymatrix@R=4ex@C=4em{
A_p \ar@{-{ >>}}[r] \ar@{=}[d] &
	\ZDiagQuo{r}{p}{A} \ar@{{ |>}->}[r] \ar@{-{ >>}}[d] &
	A_{p+r} \ar[d]^{a_{p+r}} \\
A_p \ar@{-{ >>}}[r] \ar@{=}[d] &
	\ZDiagQuo{r+1}{p}{A} \ar@{{ |>}->}[r] \ar@{-{ >>}}[d] &
	A_{p+r+1} \ar[d] \\
\vdots & \vdots & \vdots 
}
\end{equation*}
As $\CoLimOver{\omega}$ preserves image factorizations (\ref{thm:CoLim^ZCat-InModules-CommutesKer/Img}), we see that the cocone map $A_p\to \CoLimOf{A}$ has the image factorization
\begin{equation*}
\xymatrix@R=5ex@C=4em{
A_p \ar@{-{ >>}}[r] &
	\ZDiagQuo{\omega}{p}{A} \ar@{{ |>}->}[r] &
	\CoLimOf{A}
}
\end{equation*}
(iv) follows by commutativity applied to the diagrams
\begin{equation*}
\xymatrix@R=5ex@C=4em{
A_p \ar@{-{ >>}}[r] \ar[d]_{a_p} &
	\ZDiagQuo{\omega}{p}{A} \ar@{{ |>}->}[r] \ar[d]^{\ZDiagQuo{\omega}{p}{a}} &
	\CoLimOf{A} \ar@{=}[d] \\
A_{p+1} \ar@{-{ >>}}[r] &
	\ZDiagQuo{\omega}{p+1}{A} \ar@{{ |>}->}[r] &
	\CoLimOf{A}
}
\end{equation*}
This was to be shown.
\end{proof}

\begin{proposition}[Image subdiagram operation: properties]
\label{thm:ImageSubDiagram-Props}
Given a $\ZCat$-diagram $A$ and $r\geq 1$, let $\ZDiagImg{r+1}{}{A}\DefEq \ZDiagImg{}{}{\left( \ZDiagImg{r}{}{A} \right)}$. Then the following hold:
\begin{enumerate}[(i)]
\item If the structure maps of $A$ are surjective, then $\ZDiagImg{}{}{A}\to A$ is an isomorphism.
\item For $r\geq 1$ and $p\in \ZNr$
\begin{equation*}
\ZDiagImg{r}{p}{A}\longrightarrow A_p\quad \text{is given by}\quad \Img{A_{p-r}\to A_{p-r+1}\to \cdots \to A_{p}} \longrightarrow A_p
\end{equation*}
\end{enumerate}
\end{proposition}
(i) Checks directly. We prove (ii) by induction. For $r=1$, the claim holds by definition. If the claim holds for $r\geq 1$, then we face this map of $\ZCat$-diagrams:
\begin{equation*}
\xymatrix@R=5ex@C=2em{
\Img{a_{p-2}\Comp \cdots \Comp a_{p-r-1}} \ar@{=}[r] &
	\ZDiagImg{r}{p-1}{A} \ar@{{ |>}->}[rr]^-{\eta_{p-1}} \ar[d]_{\ZDiagImg{r}{p-1}{a}} &&
	A_{p-1} \ar[d]^{a_{p-1}} \\
\Img{a_{p-1}\Comp \cdots \Comp a_{p-r}} \ar@{=}[r] &
	\ZDiagImg{r}{p}{A} \ar@{{ |>}->}[rr]_-{\eta_{p}} &&
	A_{p}
}
\end{equation*}
Then $\ZDiagImg{r}{p-1}{a}$ is the lift of $a_{p-1}\Comp \eta_{p-1}$ to $\eta_{p}$. Via the inclusion $\eta_p$, we find:
\begin{equation*}
\begin{aligned}
I_p\left(\ZDiagImg{r}{}{A}\right) &= \Img{\ZDiagImg{r}{p-1}{a}} \\
	&= \Img{\eta_p\Comp \ZDiagImg{r}{p-1}{a}} \\
	&= \Img{a_{p-1}\Comp \eta_{p-1}} \\
	&= a_{p-1}(\Img{a_{p-2}\Comp \cdots \Comp a_{p-r-1}}) \\
	&= \Img{a_{p-1}\Comp a_{p-2}\Comp \cdots \Comp a_{p-r-1}} \\
	&= \ZDiagImg{r+1}{p}{A}
\end{aligned}
\end{equation*}
This was to be shown.

We are now ready to iterate the operations 'image quotient diagram' and 'image subdiagram' via transfinite recursion:

\begin{definition}[Iterated image quotient/sub-diagrams $\ZDiagQuo{\tau}{}{A}$ and $\ZDiagImg{\tau}{}{A}$]
\label{def:ImageQuotientSubDiagrams}%
Given a $\ZCat$-diagram $A$, let $\ZDiagQuo{0}{}{A}=A=\ZDiagImg{0}{}{A}$. For a successor ordinal $\tau$, define
\begin{equation*}
\ZDiagQuo{\tau}{}{A}\DefEq Q(\ZDiagQuo{\tau-1}{}{A})  \qquad \text{and}\qquad \ZDiagImg{\tau}{}{A}\DefEq I(\ZDiagImg{\tau-1}{}{A})
\end{equation*}
While for a limit ordinal $\tau$, define
\begin{equation*}
\ZDiagQuo{\tau}{}{A} \DefEq \CoLimOfOver{A \to \cdots \to \ZDiagQuo{\sigma}{}{A}\to \cdots}{\sigma<\tau}  \qquad \text{and}\qquad \ZDiagImg{\tau}{}{A} \DefEq \LimOfOver{\cdots \to \ZDiagImg{\sigma}{}{A}\to \cdots \to A}{\sigma<\tau}
\end{equation*}
These are, respectively, the \Defn{$\tau$-th image quotient diagram} and the \Defn{$\tau$-th image subdiagram of $A$}. %
\index{image!quotient diagram}\index{image!subdiagram}%
\index[not]{$\ZDiagQuo{\tau}{}{A}$ - $\tau$-th image quotient diagram}\index[not]{$\ZDiagImg{\tau}{}{A}$ - $\tau$-th image subdiagram}%
\end{definition}

For cardinality reasons, both image diagrams must stabilize at a sufficiently large ordinal. We will see in a moment that the image quotient diagram $\ZDiagQuo{\omega}{}{A}$ is stable for every $A$. For image subdiagrams the situation is far more complicated: \cite[p.~10]{JMBoardman1999} explains that, for any ordinal $\tau$ there exists a $\ZCat$-diagram $A$ for which the image subdiagram $\ZDiagImg{\tau}{}{A}=0$ while $\ZDiagImg{\alpha}{}{A}\neq 0$ for any $\alpha<\tau$.

\begin{proposition}[$\ZDiagQuo{\omega}{}{A}$ is stable]
\label{thm:Q^omega(A)-Stable}
For every $\ZCat$-diagram of $R$-modules and every ordinal $\tau>\omega$, the map $\ZDiagQuo{\omega}{}{A}\to \ZDiagQuo{\tau}{}{A}$ is an isomorphism.
\end{proposition}
\begin{proof}
We know that $\ZDiagQuo{\omega}{}{A}$ has monomorphic structure maps; see (\ref{thm:ImageQuotient-Props}.iv). Via (\ref{thm:ImageQuotient-Props}.i) the claim follows.
\end{proof}

\begin{definition}[Mittag-Leffler condition / co-Mittag-Leffler condition]
\label{def:Mittag-Leffler/CoMittag-Leffler-Condition}%
A $\ZCat$-diagram $A$ in $\LModules{R}$ satisfies the \Defn{Mittag-Leffler condition} if, for each $p\in\ZNr$, there exists $k\in \NNr$ such that image subobjects satisfy $\ZDiagImg{\tau}{p}{A}=\ZDiagImg{k}{p}{A}$ for all $\tau\geq k$. Similarly, $A$ satisfies the \Defn{co-Mittag-Leffler condition} if, for each $p\in\ZNr$, there exists $k\in \NNr$ such that image quotients satisfy $\ZDiagQuo{\tau}{p}{A}=\ZDiagQuo{k}{p}{A}$ for all $\tau\geq k$, .
\index{Mittag-Leffler condition}\index{co-Mittag-Leffler condition}
\end{definition}

\begin{proposition}[Sufficient conditions for the co-Mittag-Leffler condition]
\label{thm:CoMittagLefflerExamples}
A $\ZCat$-diagram $A$ has the co-Mittag-Leffler property whenever any one of the following conditions holds.
\vspace{-2ex}
\begin{enumerate}[(i)]
\item A terminal segment of $A$ has injective structure maps (e.g. $A$ is eventually stable).
\item The objects in $A$ satisfy the ascending chain condition (i.e. every strictly ascending chain of subobjects of $A$ stabilizes after finitely many steps). For example, this happens whenever each  $A_p$ is a finitely generated module over a Noetherian ring; see \cite[p.~415]{SLang2002}.
\item For each object $A_p$ there exists $k\geq 0$ such that the image of $A_{p} \to A_{p+k}$ satisfies the ascending chain condition. \NoProof
\end{enumerate}
\end{proposition}

\begin{proposition}[Sufficient conditions for the Mittag-Leffler condition]
\label{thm:MittagLefflerExamples}
A $\ZCat$-diagram $A$ has the Mittag-Leffler property whenever any one of the following conditions holds.
\vspace{-2ex}
\begin{enumerate}[(i)]
\item An initial segment of $A$ has surjective structure maps (e.g. $A$ is originally stable).
\item The objects in $A$ satisfy the descending chain condition (i.e. every strictly descending chain of subobjects of $A$ stabilizes after finitely many steps). For example, this happens whenever $A$ takes values in a category of finite dimensional vector spaces over some field, or whenever each $A_p$ is finite.
\item For each object $A_p$ there exists $k\geq 0$ such that the image of $A_{p-k} \to A_{p}$ satisfies the descending chain condition. \NoProof
\end{enumerate}
\end{proposition}

\begin{proposition}[Mittag-Leffler condition implies $\LimOne=0$]
\label{thm:Mittag-Leffler-Lim1=Zero}%
If a $\ZCat$-diagram $A$ in $\ModulesOver{R}$ satisfies the Mittag-Leffler condition (\ref{def:Mittag-Leffler/CoMittag-Leffler-Condition}), then $\LimOneOf{A}=0$. \NoProof%
\index{Mittag-Leffler condition!implies $\LimOne=0$}%
\end{proposition}

\begin{proposition}[Image factorization of $\ZCat$-diagram gives $\CoLim,\Lim,\LimOne$-isomorphisms]
\label{thm:ZDiagramStructureMapGivesIsoCoLim-Lim-LimOne}%
\cite[1.8]{SEilenbergJCMoore1961}\quad Given a $\ZCat$-diagram $(A,a)$, each of the maps $A\to \ZDiagQuo{}{}{A}$ and $\ZDiagImg{}{}{A}\to A$ induces an isomorphism of $\CoLim$, $\Lim$ and $\LimOne$.
\end{proposition}
\begin{proof}
With the notation of diagram \eqref{eq:ImageQuotient-Diagrams}, from $\CoLimOf{\ZDiagKer{}{}{A}}=0=\CoLimOf{ \ZDiagCoKer{}{}{A} }$, the exactness of $\CoLimOver{\ZCat}$ yields the isomorphisms
\begin{equation*}
\CoLimOf{A}\XRA{\cong} \CoLimOf{ \ZDiagQuo{}{}{A} }  \qquad \text{and}\qquad \CoLimOf{ \ZDiagImg{}{}{A} }\XRA{\cong} \CoLimOf{A}.
\end{equation*}
Turning to limits, via direct computation on limits and the Mittag-Leffler property (\ref{thm:Mittag-Leffler-Lim1=Zero}) for $\LimOne$, we have
\begin{equation*}
\LimOf{\ZDiagKer{}{}{A}} = 0 = \LimOneOf{\ZDiagKer{}{}{A}}\quad \text{and}\quad \LimOf{\ZDiagCoKer{}{}{A}} = 0 = \LimOneOf{\ZDiagCoKer{}{}{A}}
\end{equation*}
Now the $6$-term exact $\Lim,\LimOne$-sequences of the short exact sequences $KA\to A\to QA$ and $IA\to A\to RA$ yield the $\Lim$- and $\LimOne$-isomorphisms as claimed.
\end{proof}

Next, we establish properties of the stable image subdiagrams of a $\ZCat$-diagram $A$.

\begin{lemma}[$\LimOf{\ZDiagImg{\tau}{}{A}} \to \LimOf{A}$ is an isomorphism]
\label{thm:ImageDiagGivesLimIso}%
Let $A$ be a $\ZCat$-diagram in $\LModules{R}$. Then, for every ordinal $\tau$, the inclusion $\ZDiagImg{\tau}{}{A}\to A$ induces an isomorphism $\LimOf{\ZDiagImg{\tau}{}{A}} \to \LimOf{A}$.
\end{lemma}
\begin{proof}
We prove this claim by transfinite induction. For $\kappa=0$, $I^0A\to A$ is the identity map. So, the claim holds in this case. Now, consider an ordinal $\tau>0$, and assume the claim is true for every ordinal $\kappa<\tau$. We distinguish two cases:

{\em Case 1:\quad  $\tau$ is a successor ordinal}\quad  In this case, the identity $\ZDiagImg{\tau}{}{A} = I(\ZDiagImg{\tau-1}{}{A})$ allows us to compose the  isomorphism (\ref{thm:ZDiagramStructureMapGivesIsoCoLim-Lim-LimOne}) with the isomorphism of the induction hypothesis to prove the claim:
\begin{equation*}
\LimOf{\ZDiagImg{\tau}{}{A}} \XRA{\cong} \LimOf{\ZDiagImg{\tau-1}{}{A}} \XRA{\cong} \LimOf{A}.
\end{equation*}
{\em Case 2:  $\tau$ is limit ordinal}\quad  In this case, the claim follows because limits commute. 
\end{proof}

Via (\ref{thm:ZDiagramStructureMapGivesIsoCoLim-Lim-LimOne}) we see that, for finite $\tau$, the inclusion $\ZDiagImg{\tau}{}{A}\to A$ also induces an isomorphism $\LimOneOf{\ZDiagImg{\tau}{}{A}}\to \LimOneOf{A}$. Via (\ref{thm:StableImageSubDiagram-Properties}.\ref{thm:StableImageSubDiagram-Properties-EpicStructureMaps}) we infer that this map need not induce an isomorphism if $\tau$ is not finite.

Cardinality considerations show that, for a sufficiently large ordinal $\alpha$, $\ZDiagImg{\alpha}{}{A} = \ZDiagImg{\lambda}{}{A}$  whenever $\alpha<\lambda$. We call $\ZDiagImg{\alpha}{}{A}$ the \Defn{stable image subdiagram} of $A$ and denote it $\ZStableImg{A}$ for short. %
\index[not]{$\ZStableImg{A}$ - stable image subdiagram of $A$}%
\index{stable!image subdiagram of a $\ZCat$-diagram}%

\begin{proposition}[Properties of the stable image subdiagram]
\label{thm:StableImageSubDiagram-Properties}
Given a $\ZCat$-diagram in $\ModulesOver{R}$, its stable image subdiagram has the following properties.
\vspace{-2ex}
\begin{enumerate}[(i)]
\item \label{thm:StableImageSubDiagram-Properties-EpicStructureMaps}%
$\ZStableImg{A}$ has epimorphic structure maps; so $\LimOneOf{\ZStableImg{A}}=0$.
\item \label{thm:StableImageSubDiagram-Properties-ImLim}%
$\Img{\rho_p\from \LimOf{A}\to A_p} = \ZStableImgAt{A}{p}$
\item \label{thm:StableImageSubDiagram-Properties-LimLimOne}%
The maps in the short exact sequence of $\ZCat$-diagrams $\ZStableImg{A}\to A\to A/\ZStableImg{A}$ induce  isomorphisms $\LimOf{\ZStableImg{A}}\to \LimOf{A}$ and $\LimOneOf{A}\to \LimOneOf{A/\ZStableImg{A}}$, while $\LimOf{A/\ZStableImg{A}}=0$.
\end{enumerate}
\end{proposition}
\begin{proof}
\ref{thm:StableImageSubDiagram-Properties-EpicStructureMaps}\quad $\ZStableImg{A}$ has epimorphic structure maps because, for each $p\in \ZNr$, the restriction of $a_p\from A_p\to A_{p+1}$ satisfies
\begin{equation*}
\Img{a_p|\from (\ZStableImgAt{A}{p}\to \ZStableImgAt{A}{p+1}} = \ZStableImgAt{A}{p+1}
\end{equation*}
This happens exactly when $a_p|$ is an epimorphism. By (\ref{thm:Lim^1=0IfStructureMapsAreEpimorphic}), $\LimOneOf{\ZStableImg{A}}=0$.

\ref{thm:StableImageSubDiagram-Properties-ImLim} \quad As $\ZStableImg{A}$ has epimorphic structure maps, the universal maps $\bar{\rho}_p\from \LimOf{ \ZStableImg{A} } \to \ZStableImgAt{A}{p}$ are epimorphisms as well. Further, the inclusion $\ZStableImg{A}\to A$ induces the commutative square below.
\begin{equation*}
\xymatrix@R=5ex@C=3em{
\LimOf{\ZStableImg{A}} \ar[r]^-{\cong} \ar@{-{ >>}}[d]_{\bar{\rho}_p} &
	\LimOf{A} \ar[d]^{\rho_p} \\
\ZStableImgAt{A}{p} \ar@{{ |>}->}[r] &
	A_p
}
\end{equation*}
The map $\LimOf{\ZStableImg{A}} \to \LimOf{A}$ is an isomorphism (\ref{thm:ImageDiagGivesLimIso}). So, the counterclockwise arrow path from $\LimOf{A}$ to $A_p$ is the image factorization of $\rho_p$.

\ref{thm:StableImageSubDiagram-Properties-LimLimOne}\quad The short exact sequence of $\ZCat$-diagrams $\ZStableImg{A} \to A\to A/\ZStableImg{A}$ induces the $6$-term exact sequence below.
\begin{equation*}
\xymatrix@C=2em{
\LimOf{\ZStableImg{A}} \ar[r]^-{\cong} &
	\LimOf{A} \ar[r] &
	\LimOf{A/\ZStableImg{A}} \ar[r] &
	\LimOneOf{\ZStableImg{A}} \ar[r] &
	\LimOneOf{A} \ar@{-{ >>}}[r] &
	\LimOneOf{A/\ZStableImg{A}}
}
\end{equation*}
We know that the map $\LimOf{\ZStableImg{A}} \to \LimOf{A}$ is an isomorphism by (\ref{thm:ImageDiagGivesLimIso}), and that $\LimOneOf{\ZStableImg{A}}=0$ by (i). This implies the claim by exactness.
\end{proof}

\begin{subordinate}
{\bfseries Comment}\quad In view of the isomorphism $\CoLimOf{\ZDiagImg{}{}{A}}\to \CoLimOf{A}$ from (\ref{thm:ZDiagramStructureMapGivesIsoCoLim-Lim-LimOne}), one wonders whether the canonical map $\CoLimOf{\ZStableImg{A}} \to \CoLimOf{A}$ is an isomorphism. - Given that limits don't generally commute with colimits, there is no reason to expect this map  to be an isomorphism. Indeed, here is an example of a $\ZCat$-diagram where it is not:
\begin{equation*}
\cdots \XRA{\times 2} \ZNr \XRA{\times 2} \ZNr \XRA{\times 2} \ZNr\XRA{\times 2} \cdots 
\end{equation*}
Visibly, $\ZStableImg{A}=0$ and, hence, $\CoLimOf{\ZStableImg{A}}=0$. On the other hand, $\CoLimOf{A}\cong \ZPoly{\tfrac{1}{2}}$ is the $2$-divisible subring of the rational numbers $\QNr$.
\end{subordinate}

\section[Filtrations]{Filtrations}
\label{sec:Filtrations}

A \Defn{filtration} of a module $M$ is given by a $\ZCat$-diagram $F$ of submodules
\begin{equation*}
\cdots\subseteq F_{p-1} \subseteq F_p \subseteq F_{p+1} \subseteq \cdots \subseteq M.
\end{equation*}
A filtration $F$ of $M$ is called \Defn{exhaustive} if the universal map $\CoLimOfOver{F}{\ZCat}\to M$ is an isomorphism. %
\index{exhaustive filtration}%
\index{filtration!exhaustive}%
Such a filtration of $M$ determines and is determined by its associated \Defn{cofiltration}
\begin{equation*}
M\longrightarrow \cdots \longrightarrow M/F_{p-1} \longrightarrow M/F_p \longrightarrow M/F_{p+1}\longrightarrow\cdots.
\end{equation*}
The \Defn{$F$\MSComp-completion} of $M$ is defined as $M^{\wedge}_{F}\DefEq \LimOf{M/F}$. %
\index{$F$\MSComp-completion of filtered object}\index[not]{$M^{\wedge}_{F}$ - $F$\MSComp-completion of $M$} %
We say that \Defn{$M$ is $F$\MSComp-complete} if the canonical map $M\to M^{\wedge}_{F}$ is an isomorphism. This happens if and only if $\LimOf{F} = 0=\LimOneOf{F}$.%
\index{$F$\MSComp-complete filtered object}%

Any pair of adjacent (co-)filtrations stages determines a morphism of short exact sequences:
\begin{equation*}
\xymatrix@R=5ex@C=4em{
F_{p-1} \ar@{{ |>}->}[r] \ar@{{ |>}->}[d] &
	M \ar@{=}[d] \ar@{-{ >>}}[r] &
	M/F_{p-1} \ar@{-{ >>}}[d] \\
F_{p} \ar@{{ |>}->}[r] &
	M \ar@{-{ >>}}[r] &
	M/F_p
}
\end{equation*}
So, there is a functorial isomorphism $\CoKer{F_{p-1}\to F_p}\cong \Ker{M/F_{p-1} \to M/F_p}$. When working with spectral sequences, this is useful to keep in mind because then we are dealing with filtered objects whose adjacent filtration stages are related to the $\SSPage{\infty}$-objects of the spectral sequence; see (\ref{thm:E-InfinityExtensionThm}). Here, we see that these $\SSPage{\infty}$-objects may optionally be used in both, the filtration as well as its associated cofiltration.

We call a filtration \Defn{originally vanishing}, \Defn{originally stable}, \Defn{eventually stable}, \Defn{eventually vanishing} if the underlying $\ZCat$-diagram has the property in question; see (\ref{def:Z-DiagramTypes}). %
\index{filtration}%
\index{originally!vanishing filtration}%
\index{originally!stable filtration}%
\index{eventually!stable filtration}%
We say the filtration has an \Defn{anchor} at $p_0\in \ZNr$ if $F_{p_0}$ is known. %
\index{filtration!anchor}%
\index{anchor of a filtration}%
More generally, the filtration is \Defn{Hausdorff} if $\LimOf{F} =0$. For example, a complete filtration is Hausdorff. %
\index{Hausdorff filtration}\index{filtration!Hausdorff}%

Let us now specialize to the situation we need to deal with in the discussion of convergence the spectral sequence associated with an exact couple. Every $\ZCat$-diagram $A$ of $R$-modules determines in a universal manner two filtered objects:
\begin{enumerate}[(1)]
\item $A_{\infty}\DefEq \CoLimOf{A}$\MSComp, filtered by $F_{\bullet}A$, where $F_{p}A$ is the image of the cocone map $\pi_p\from A_p\to A_{\infty}$, $p\in \ZNr$. The associated cofiltration of $A_{\infty}$ satisfies $A_{\infty}/F_{p}A\cong \CoKer{\pi_p}$.%
\index[not]{$\pi_p\from A_p\to A_{\infty}$}
\item $A_{-\infty}\DefEq \LimOf{A}$\MSComp, filtered by $F^{\bullet}A$, where $F^{p}A$ is the kernel of the cone map $\rho_p\from A_{-\infty}\to A_p$. The associated cofiltration of $A_{-\infty}$ satisfies $A_{-\infty}/F^{p}A\cong \ZStableImgAt{A}{p}$, with $\bar{I}A$ the stable image diagram of $A$; see (\ref{thm:StableImageSubDiagram-Properties}).%
\index[not]{$\rho_p\from A_{-\infty}\to A_p$}%
\end{enumerate}
We are most interested in properties of the kernel filtration $F^{\bullet}A$ of $A_{-\infty}$. In general, it fails to be exhaustive, in which case $A_{-\infty}$ has a nonzero image in $A_{\infty}$; see (\ref{thm:KernelFiltrationLim(A)-Props}). Critical now is that we gain additional information about a kernel filtration stage $F^pA$ via the following construction: Let $\ZPointKer{A}{p}$ be the $\bar{\omega}$\MSComp-sequence with
\begin{equation*}
\ZPointKerAt{A}{p}{p-r} \DefEq \Ker{A_{p-r} \longrightarrow A_p}
\end{equation*}
Then (\ref{thm:Kernel/ImageSequencesOfZ-Diagram}) $\LimOf{\ZPointKer{A}{p}} \cong F^{p}A$. Remarkably, it is not difficult to develop conditions under which the canonical map
\begin{equation*}
\Lambda^p\from \LimOneOf{\ZPointKer{A}{p}} \longrightarrow \LimOneOf{\ZPointKer{A}{p+1}}
\end{equation*}
is a monomorphism; see (\ref{thm:Conditions-Lim1(K^pA)->Lim1(K^(p+1)A)Monomorphism}), (\ref{thm:a_pMono->Lim1(K^pA)->Lim1(K^(p+1)A)Monomorphism}), (\ref{thm:PointWiseConditionFor-Lim1(K^pA)->Lim1(K^(p+1)A)Monomorphism}), and (\ref{thm:Omega-Mittag-Leffler->Lim1(K^pA)->Lim1(K^(p+1)A)Monomorphism-All-p}). In Section \ref{sec:Convergence-I}, where we investigate the meaning of the E-infinity objects of a spectral sequence, such conditions are exactly what is needed to establish stability of its E-infinity objects.

\begin{proposition}[Kernel cofiltration of $\LimOf{A}$ complete]
\label{thm:KernelCoFiltration-Complete}
\cite[2.3]{SEilenbergJCMoore1961} If $A$ is a $\ZCat$-diagram then the kernel filtration $F^{\bullet}(A)$ of $\LimOf{A}$ forms a $\ZCat$-diagram satisfying %
\index{kernel filtration of $\LimOfOver{A}{\ZCat}$!intersetion}%
\begin{equation*}
\LimOf{F^{\bullet}(A)} = 0 = \LimOneOf{F^{\bullet}(A)}
\end{equation*}
Moreover, $\LimOf{A}$ is $F^{\bullet}(A)$-complete.
\end{proposition}
\begin{proof}
We know (\ref{thm:StableImageSubDiagram-Properties}) that $\Img{\rho_p} = \ZStableImgAt{A}{p}$ is the $p$-th term of the stable image diagram of $A$. Image factorizing $\rho_p\from \LimOf{A}\to A_p$ via $\gamma_p\from \LimOf{A}\to \ZStableImgAt{A}{p}$, we have the short exact sequence of $\ZCat$-diagrams
\begin{equation*}
\xymatrix@C=3em{
F^p(A) \ar@{{ |>}->}[r]   &
	\LimOf{A} \ar@{-{ >>}}[r]^-{\gamma_p} &
	\ZStableImgAt{A}{p}\qquad\qquad (*)
}
\end{equation*}
The middle term is the constant $\ZCat$-diagram. We claim that $\gamma$ induces an isomorphism of limits. To see this, consider the sequence of $\ZCat$-diagrams:
\begin{equation*}
\xymatrix{
\LimOf{A} \ar@{-{ >>}}[r]^-{\gamma}  &
	\ZStableImg{A} \ar@{{ |>}->}[r]^-{\lambda} &
	A
}
\end{equation*}
By design, $\lambda\gamma=\rho$. Therefore, taking limits yields the composite
\begin{equation*}
\xymatrix{
\LimOf{A} \ar[r]^-{\gamma_{-\infty}}  &
	\LimOf{\ZStableImg{A}} \ar[r]^-{\lambda_{-\infty}} &
	\LimOf{A}
}
\end{equation*}
with $\lambda_{-\infty}\gamma_{-\infty} = \IdMap$. Further, we know (\ref{thm:StableImageSubDiagram-Properties}) 	that $\lambda_{-\infty}$ is an isomorphism. So, $\gamma_{-\infty}$ is an isomorphism as well; i.e. $\LimOf{A}$ is $F^{\bullet}(A)$-complete.

Now consider the $6$-term exact lim-sequence of $(*)$:
\begin{equation*}
\xymatrix{
\LimOf{F^{\bullet}(A)} \ar@{{ |>}->}[r]  &
	\LimOf{A} \ar[r]^-{\gamma_{-\infty}}_-{\cong} &
	\LimOf{\bar{I}A} \ar@{-{ >>}}[r] &
	\LimOneOf{F^{\bullet}(A)} \ar[r] &
	\LimOneOf{ \LimOf{A} } = 0
}
\end{equation*}
By exactness: $\LimOf{F^{\bullet}(A)} = 0 = \LimOneOf{F^{\bullet}(A)}$, as was to be shown.
\end{proof}

\begin{proposition}[Image filtration of $\CoLimOf{A}$ exhaustive]
\label{thm:ImageFiltrationCoLim(A)-Exhaustive}%
If $A$ is a $\ZCat$-diagram, then the image filtration $F_{\bullet}(A)$ of $\CoLimOf{A}$ is exhaustive, and $\CoLimOf{A_{\infty}/F_pA}=0$.
\end{proposition}
\begin{proof}
Consider the image factorization $A_p \to F_pA\to \CoLimOf{A}$ of the structure maps of a colimit cocone for $A$. Via (\ref{thm:CoLim^ZCat-InModules-CommutesKer/Img}), we obtain the image factorization diagram
\begin{equation*}
\xymatrix@R=5ex@C=1em{
\CoLimOf{A} \ar@{=}[rr] \ar@{-{ >>}}[rd] &&
	\CoLimOf{A} \\
& \CoLimOf{F_{\bullet}(A)} \ar@{{ |>}->}[ru]
}
\end{equation*}
Thus the maps in this triangle are isomorphisms. Then  $\CoLimOf{A_{\infty}/F_pA}=0$ follows via exactness of $\CoLimOver{\ZCat}$.
\end{proof}

\begin{lemma}[$F_pA$ as a colimit]
\label{thm:F_pA-CoLimit}%
Given a $\ZCat$-diagram $A$, $F_pA\cong \CoLimOfOver{\Img{a_{p+r}\Comp\cdots \Comp a_p}}{r}$.
\end{lemma}
\begin{proof}
Consider the sequence of $\ZCat$-diagrams:
\begin{equation*}
\xymatrix@R=5ex@C=3em{
\Ker{a_{p+r}\Comp\cdots \Comp a_p} \ar@{{ |>}->}[r] &
	A_p \ar@{-{ >>}}[r] &
	\Img{a_{p+r}\Comp\cdots \Comp a_p} \ar@{{ |>}->}[r] &
	A_{p+r+1}
}
\end{equation*}
As $\CoLimOver{\ZCat}$ is exact (\ref{thm:CoLim^ZExactFunctor}), we obtain the diagram of modules 
\begin{equation*}
\xymatrix@R=5ex@C=3em{
\CoLimOf{\Ker{a_{p+r}\Comp\cdots \Comp a_p}} \ar@{{ |>}->}[r] &
	A_p \ar@{-{ >>}}[r] &
	\CoLimOf{\Img{a_{p+r}\Comp\cdots \Comp a_p}} \ar@{{ |>}->}[r] &
	A_{\infty}
}
\end{equation*}
Thus, the composite $A_p\to \CoLimOf{\Img{a_{p+r}\Comp\cdots \Comp a_p}}\to A_{\infty}$ is an image factorization of $\pi_p\from A_p\to A_{\infty}$, and the claim follows.
\end{proof}

\begin{proposition}[Kernel filtration of $\LimOf{A}=A_{-\infty}$ - properties I]
\label{thm:KernelFiltrationLim(A)-Props}%
For  a $\ZCat$-diagram $A$ of $R$-modules the following hold:%
\index{kernel filtration of $\LimOfOver{A}{\ZCat}$!properties I}%
\vspace{-2ex}
\begin{enumerate}[(i)]
\item This functorial sequence is exact: $\xymatrix@R=5ex@C=1.6em{
0\to\CoLimOf{F^{\bullet}(A)} \ar@{{ |>}->}[r] &
	A_{-\infty} \ar[r]^-{R} &
	A_{\infty} \ar@{-{ >>}}[r] &
	\CoLimOf{\CoKer{\rho_p}}\to 0
}$.
\item $\Img{R}\cong \CoLimOf{\ZStableImg{A}}$ fits into this functorial exact sequence:
\begin{equation*}
\xymatrix@R=5ex@C=2em{
0\to \CoLimOf{\ZStableImg{A}} \ar@{{ |>}->}[r] &
	\LimOf{F_{\bullet}(A)} \ar[r] &
	\CoLimOf{\CoKer{\rho}}\ar[r] &
	\LimOf{A_{\infty}/F_{\bullet}(A)} \ar[r] &
	\LimOneOf{F_{\bullet}(A)}\to 0.
}
\end{equation*}
\end{enumerate}
\end{proposition}
\begin{proof}
The concatenation of short exact sequences below determines a corresponding sequence of $\ZCat$-diagrams:
\begin{equation*}
\xymatrix@R=5ex@C=3em{
F^p(A) \ar@{{ |>}->}[r] &
	A_{-\infty} \ar@{-{ >>}}[rd]_-{\gamma_p} \ar[rr]^-{\rho_p} &&
	A_p \ar@{-{ >>}}[r] &
	\CoKer{\rho_p} \\
&& 	\ZStableImgAt{A}{p} \ar@{{ |>}->}[ru] &
}
\end{equation*}
Part (i) follows as $\CoLimOver{\ZCat}$ is exact. We note that $R=\pi_p\Comp \rho_p$ for any $p\in \ZNr$.

(ii)\quad That $\Img{R}\cong\CoLimOf{\ZStableImg{A}}$ follows with the proof of (i). The stable image diagram of $A$ fits into the commutative diagram below.
\begin{equation*}
\xymatrix@R=5ex@C=4em{
\LimOf{\ZStableImg{A}}\cong A_{-\infty} \ar[r]^-{\tau} \ar@{-{ >>}}[d] \ar@{-{ >>}}@/_3ex/@<-3ex>[dd]_{t} &
	\LimOf{F_{\bullet}(A)} \ar@{{ |>}->}[d] \ar@{{ |>}->}@/^3ex/@<3ex>[dd]^{u} \\
\ZStableImg{A} \ar[r] \ar@{-{ >>}}[d] &
	F_{\bullet}(A) \ar@{{ |>}->}[d] & \\
\CoLimOf{\ZStableImg{A}} \ar@{{ |>}->}[r]_-{s} \ar@{{ |>}.>}[ruu]_(.35){\bar{\tau}} &
	\CoLimOf{F_{\bullet}(A)}=A_{\infty}
}
\end{equation*}
By design, $st=R$. So $0=\tau(\Ker{R})=\Ker{t}$; i.e. $\tau$ factors through $t$ via $\bar{\tau}\from \CoLimOf{\ZStableImg{A}}\to \LimOf{F_{\bullet}(A)}$, unique with $\bar{\tau}t=\tau$. Then $u\bar{\tau}=s$ follows via the epimorphic property of $t$. The short exact sequences of $\ZCat$-diagrams $\ZStableImg{A}\to A\to \CoKer{\rho}$ and $F_{\bullet}(A)\to A_{\infty}\to A_{\infty}/F_{\bullet}(A)$ determine exact sequences:
\begin{equation*}
\xymatrix@R=5ex@C=4em{
\CoLimOf{\ZStableImg{A}} \ar@{{ |>}->}[r]^-{s} \ar@{{ |>}->}[d]_{\bar{\tau}} &
	A_{\infty} \ar@{=}[d] \ar@{-{ >>}}[r] &
	\CoLimOf{\CoKer{\rho}} \ar@{.>}[d]^{T}  \ar[r] &
	0 \ar[d] \\
\LimOf{F_{\bullet}(A)} \ar@{{ |>}->}[r]_-{u} &
	A_{\infty} \ar[r] &
	\LimOf{A_{\infty}/F_{\bullet}(A)} \ar@{-{ >>}}[r] &
	\LimOneOf{F_{\bullet}(A)}
}
\end{equation*}
We just showed that the square on the left commutes. So there exists a map $T$, unique with the property that the center square commutes. Commutativity of the square on the right follows. Now the claim follows via the snake lemma and a bit of diagram chasing.
\end{proof}

\begin{corollary}[Conditions for exhaustive kernel filtration]
\label{thm:ExhaustiveKernelFiltration-Conditions}
A $\ZCat$-diagram $A$ yields an exhaustive kernel filtration of $A_{-\infty}$ if and only if $R\from A_{-\infty}\to A_{\infty}$ is the $0$-map. \NoProof
\end{corollary}

\begin{proposition}[$F(M)$ exhaustive implies kernel filtration of $M_{F}^{\wedge}$ exhaustive]
\label{thm:F(M)Exhaustive->Ker(M_f^Wedge)Exhaustive}
If a filtration $F$ of $M$ is exhaustive, then so is the filtration of $M_{F}^{\wedge}$ by the kernels of $M_{F}^{\wedge}\to M/F_p$.
\end{proposition}
\begin{proof}
Since the maps of the diagram $M/F$ consist of surjections, so do the cone maps $M^{\wedge}_{F}\to M/F_p$. Thus, we obtain this morphism of short exact sequences of $\ZCat$-diagrams:
\begin{equation*}
\xymatrix@R=5ex@C=4em{
F_p \ar@{{ |>}->}[r] \ar[d] &
	M \ar@{-{ >>}}[r] \ar[d] &
	M/F_p \ar@{=}[d] \\
\Ker{t_p} \ar@{{ |>}->}[r] &
	M^{\wedge}_{F} \ar@{-{ >>}}[r]_{t_p} &
	M/F_p
}
\end{equation*}
Taking colimits of the top row yields the short exact sequence
\begin{equation*}
\xymatrix@R=5ex@C=4em{
\CoLimOf{F} \ar@{{ |>}->}[r] &
	M \ar@{-{ >>}}[r] &
	\CoLimOf{M/F}
}
\end{equation*}
As $F$ is exhaustive in $M$, the left hand arrow is an isomorphism, implying that $\CoLimOf{M/F}=0$. Using this information in the short exact sequence of colimits of the bottom row yields the claim.
\end{proof}

\begin{example}[$\ZCat$-diagram with monomorphic structure maps]
\label{exa:Z-Diagram-MonomorphicStructureMaps}
If all structure maps in a $\ZCat$-diagram $A$ are monomorphic, then the following hold:
\vspace{-2ex}
\begin{enumerate}[(i)]
\item $A_{-\infty}\to \LimOf{F_{\bullet}A}$ is an isomorphism.
\item $F^{\bullet}A$ is the $0$-diagram.
\item $\bar{I}_pA \cong \LimOf{A}$ for all $p\in \ZNr$, and $\bar{I}A\cong I^{\omega}A$.
\end{enumerate}
\end{example}
\begin{proof}
(i) holds because $A\to F_{\bullet}A$ is an isomorphism of $\ZCat$-diagrams. Further, $A_{-\infty}\cong \FamIntrsctn{n}{F_nA}$ which turns the cone maps $A_{-\infty}\to A_n$ into inclusions. This implies (ii). To see (iii), note that
\begin{equation*}
I^{\omega}_{p}A \cong \FamIntrsctn{p-k}{A_{p-k}} \cong A_{-\infty} \cong \LimOf{F_{\bullet}A}.
\end{equation*}
Thus $I^{\omega}A$ is a diagram of identity maps, implying that $\ZStableImg{A}\cong \ZDiagImg{\omega}{}{A}$, as claimed.
\end{proof}

\begin{example}[$\ZCat$-diagram with epimorphic structure maps]
\label{exa:Z-Diagram-EpimorphicStructureMaps}
If all structure maps of a $\ZCat$-diagram $A$ are epimorphisms, then the following hold:
\vspace{-2ex}
\begin{enumerate}[(i)]
\item For all $p\in \ZNr$, $A_p\to \CoLimOf{A}$ is an epimorphism, and so $F_pA = \CoLimOf{A}=\LimOf{F_{\bullet}A}$.
\item $\ZStableImg{A} = A$
\item The sequence $\CoLimOf{F^{\bullet}A} \to A_{-\infty}\to A_{\infty}$ is short exact.
\end{enumerate}
\end{example}
\begin{proof}
(i)\quad That $A_p\to \CoLimOf{A}$ is epic follows with (\ref{thm:CoLim(ZCatDiagram)-Properties}). Thus $F_pA=\CoLimOf{A}$ for all $p$, and $\CoLimOf{A}=\LimOf{F_{\bullet}A}$ follows. (ii) holds because $\ZDiagImg{1}{}{A}=A$. 

(iii)\quad  By (\ref{thm:StableImageSubDiagram-Properties}), $\rho_p\from \LimOf{A} \to \ZStableImgAt{A}{p}=A_p$ is epic. So $\CoKer{\rho_{\bullet}}$ is the $0$-diagram. The claim follows via  (\ref{thm:KernelFiltrationLim(A)-Props}.i).
\end{proof}

\begin{example}[$\ZCat$-diagram satisfying the Mittag-Leffler condition]
\label{exa:Z-Diagram-MittagLefflerCondition}
If a $\ZCat$-diagram $A$ satisfies the Mittag-Leffler condition (\ref{def:Mittag-Leffler/CoMittag-Leffler-Condition}), then the following hold:
\vspace{-2ex}
\begin{enumerate}
\item $\ZStableImg{A}=I^{\omega}A$
\item The inclusion $\CoLimOf{I^{\omega}A}\to \CoLimOf{A}$ need not be an isomorphism.
\end{enumerate}
\end{example}
\begin{proof}
(i) follows from the definition of the Mittag-Leffler property. An example for (ii) is given by this $\ZCat$-diagram:
\begin{equation*}
\cdots \to 0 \longrightarrow \ZNr \XRA{\times 2} \ZNr \XRA{\times 2} \ZNr \longrightarrow \cdots
\end{equation*}
It satisfies $I^{\omega}A = 0$ and, hence $\CoLimOf{I^{\omega}A}=0$, while $\CoLimOf{A} \cong \ZPoly{\tfrac{1}{2}}$.
\end{proof}

We now refine the investigation of image subdiagrams of a given $\ZCat$-diagram

\begin{lemma}[Kernel / image sequences of $A$]
\label{thm:Kernel/ImageSequencesOfZ-Diagram}%
Given a $\ZCat$-diagram $A$ in $\ModulesOver{R}$ and $p\in\ZNr$ fixed, the construction below yields a short exact sequence of $\ZCat$-diagrams.
\begin{equation*}
\xymatrix@C=2.5em@R=5ex{
\ZPointKer{A}{p} \ar@{{ |>}->}[d] &
	\cdots \ar[r] &
	\ZPointKerAt{A}{p}{p-2} \ar[r] \ar@{{ |>}->}[d] &
	\ZPointKerAt{A}{p}{p-1} \ar[r] \ar@{{ |>}->}[d] &
	0 \ar@{=}[r] \ar[d] &
	0 \ar[r] \ar[d] &
	\cdots \\
A \ar@{-{ >>}}[d]_{\mu^p} &
	\cdots \ar[r] &
	A_{p-2} \ar[r] \ar@{-{ >>}}[d] &
	A_{p-1} \ar[r]^-{a_{p-1}} \ar@{-{ >>}}[d] &
	A_p \ar[r]^-{a_p} \ar@{=}[d] &
	A_{p+1} \ar[r] \ar@{=}[d] &
	\cdots \\
\ZPointImg{A}{p} &
	\cdots \ar@{{ |>}->}[r] &
	\ZImgDiagItrtdAt{A}{2}{p} \ar@{{ |>}->}[r] &
	\ZImgDiagItrtdAt{A}{1}{p} \ar@{{ |>}->}[r] &
	A_p \ar[r]_-{a_p} &
	A_{p+1} \ar[r] &
	\cdots
}
\end{equation*}
\index[not]{$I_pA$}\index[not]{$K^pA$}
Moreover, $\LimOf{ \ZPointKer{A}{p} }\cong F^pA$ and $\LimOf{\ZPointImg{A}{p}}\cong \ZImgDiagItrtdAt{A}{\omega}{p}$.
\end{lemma}
\begin{proof}
To identify $\LimOf{ \ZPointKer{A}{p}}$ consider the situation depicted in the diagram below.
\begin{equation*}
\xymatrix@R=5ex@C=4em{
\LimOf{\ZPointKer{A}{p}} \ar@{{ |>}->}[r] &
	\LimOf{A} \ar[d]_{\rho_p} \ar[r]^-{\mu^{p}_{-\infty}} &
	\LimOf{\ZPointImg{A}{p} } \ar@{{ |>}->}[d]^{\tilde{\rho}_p} \\
& A_p \ar@{=}[r] &
	A_p
}
\end{equation*}
The sequence on the top is exact, and so $F^pA = \Ker{\rho_p}=\Ker{\tilde{\rho}^p\Comp \mu^{p}_{-\infty}} = \Ker{\mu^{p}_{-\infty}}\cong \LimOf{K^pA}$. The isomorphism $\LimOf{\ZPointImg{A}{p}}\cong \ZImgDiagItrtdAt{A}{\omega}{p}$ is the definition of $\ZImgDiagItrtdAt{A}{\omega}{p}$; see (\ref{def:ImageQuotientSubDiagrams}).
\end{proof}

Given an exact couple, in discussing the relationship between the $\SSPage{\infty}$-terms of its spectral sequence and the adjacent filtration quotients of its universal abutments, we require conditions under which the map $\LimOneOf{K^pA}\to \LimOneOf{K^{p+1}A}$ is a monomorphism.

\begin{lemma}[Conditions: $\LimOneOf{K^pA}\to \LimOneOf{K^{p+1}A}$ monomorphism]
\label{thm:Conditions-Lim1(K^pA)->Lim1(K^(p+1)A)Monomorphism}%
Given a $\ZCat$-diagram $A$ in $\LModules{R}$ and $p\in\ZNr$ fixed, the following conditions are equivalent. %
\index{kernel filtration of $\LimOfOver{A}{\ZCat}$!properties II}
\begin{enumerate}[(i)]
\item The map $\LimOneOf{\ZPointKer{A}{p}}\to \LimOneOf{\ZPointKer{A}{p+1} }$ is a monomorphism.
\item The canonical map $\CoKer{F^pA\to F^{p+1}A}\to \Ker{\LimOf{\ZPointImg{A}{p}}\to \LimOf{\ZPointImg{A}{p+1}} }$ is an isomorphism.
\item $\Ker{a_p}\intrsctn \LimOf{\ZPointImg{A}{p}} \subseteq \Img{ \LimOf{A}\to \LimOf{\ZPointImg{A}{p}} } = \ZStableImgAt{A}{p}$.
\item $\Ker{a_p}\intrsctn \ZImgDiagItrtdAt{A}{\omega}{p} = \Ker{a_p}\intrsctn \ZStableImgAt{A}{p}$.
\item The map $\CoKer{F^pA\to F^{p+1}A}\to \Ker{\LimOf{\ZPointImg{A}{p}}\to \LimOf{\ZPointImg{A}{p+1}} }$ extends to an exact sequence:
{\small
\begin{equation*}
\xymatrix@R=4ex@C=2em{
F^pA \ar@{{ |>}->}[r] &
	F^{p+1}A \ar[r] &
	\LimOf{\ZPointImg{A}{p} } \ar[r] &
	\LimOf{\ZPointImg{A}{p+1}} \ar[r] &
	\CoKer{ \LimOneOf{\ZPointKer{A}{p} }\to \LimOneOf{\ZPointKer{A}{p+1} } } \ar@{->} `r/8pt[d] `/10pt[ll] `^dl[lll] `^r/3pt[rdlll] [rdlll] \\
&&	\LimOneOf{\ZPointImg{A}{p}} \ar@{-{ >>}}[r] &
	\LimOneOf{ \ZPointImg{A}{p+1} } &
}
\end{equation*}}
\end{enumerate}
\end{lemma}
\begin{proof}
Consider the morphism of short exact sequences of $\ZCat$-diagrams:
\begin{equation*}
\xymatrix@C=3em@R=5ex{
\ZPointKer{A}{p} \ar@{{ |>}->}[r] \ar@{{ |>}->}[d] &
	A \ar@{-{ >>}}[r] \ar@{=}[d] &
	\ZPointImg{A}{p} \ar@{-{ >>}}[d] \\
\ZPointKer{A}{p+1} \ar@{{ |>}->}[r] &
	A \ar@{-{ >>}}[r] &
	\ZPointImg{A}{p+1}
}
\end{equation*}
It induces a morphism of exact $6$-term sequences
\begin{equation*}
\xymatrix@C=2em@R=5ex{
\LimOf{\ZPointKer{A}{p}} \ar@{{ |>}->}[r] \ar@{{ |>}->}[d] &
	\LimOf{A} \ar[r] \ar@{=}[d] &
	\LimOf{\ZPointImg{A}{p} } \ar[d]^{a_p|} \ar[r]^-{\partial} &
	\LimOneOf{\ZPointKer{A}{p}} \ar[r] \ar[d] &
	\LimOneOf{A} \ar@{-{ >>}}[r] \ar@{=}[d] &
	\LimOneOf{\ZPointImg{A}{p}} \ar@{-{ >>}}[d] \\
\LimOf{\ZPointKer{A}{p+1}} \ar@{{ |>}->}[r] &
	\LimOf{A} \ar[r] &
	\LimOf{\ZPointImg{A}{p+1}} \ar[r] &
	\LimOneOf{\ZPointKer{A}{p+1}} \ar[r] &
	\LimOneOf{A} \ar@{-{ >>}}[r] &
	\LimOneOf{\ZPointImg{A}{p+1}}
}
\end{equation*}
Using the snake lemma and a bit of diagram chasing shows that the sequence below is exact.
{\small
\begin{equation*}
\xymatrix@C=2em@R=5ex{
\CoKer{ \LimOf{\ZPointKer{A}{p}}\to \LimOf{ \ZPointKer{A}{p+1}} } \ar@{{ |>}->}[r] &
	\Ker{ \LimOf{\ZPointImg{A}{p}}\to \LimOf{\ZPointImg{A}{p+1}} } \ar@{-{ >>}}[r]^-{\tilde{\partial}} &
	\Ker{\LimOneOf{\ZPointKer{A}{p}}\to \LimOneOf{ \ZPointKer{A}{p+1} } }
}
\end{equation*}
}%
Recalling that $F^pA \cong \LimOf{\ZPointKer{A}{p}}$ by (\ref{thm:Kernel/ImageSequencesOfZ-Diagram}), this gives the equivalence of (i) and (ii).

If (i) holds, then (iii) follows as the commutativity of the diagram shows that
\begin{equation*}
\Ker{\LimOf{\ZPointImg{A}{p}}\to \LimOf{\ZPointImg{A}{p+1}} }\subseteq \Ker{\partial} = \Img{ \LimOf{A}\to \LimOf{\ZPointImg{A}{p}} }.
\end{equation*}

If  (iii) holds, then $\tilde{\partial}=0$, which implies (ii). The equivalence of (iii) and (iv) follows by recalling that
\begin{equation*}
\CoKer{\LimOf{\ZPointKer{A}{p}} \to \LimOf{A}} = \ZStableImgAt{A}{p} \subseteq \ZImgDiagItrtdAt{A}{\omega}{p} =\LimOf{ \ZPointImg{A}{p}}.
\end{equation*}
Finally, (i) is equivalent to the exactness of the sequence in (v) in position $\LimOf{\ZPointImg{A}{p}}$. This implies that (i) is equivalent to (v), as the exactness of the remainder of the sequence in (v) follows via the snake lemma and diagram chases.
\end{proof}

We infer conditions under which $\LimOneOf{\ZPointKer{A}{p}}\to \LimOneOf{\ZPointKer{A}{p+1} }$ is a monomorphism.

\begin{corollary}[If $a_p$ monomorphic: $\LimOneOf{K^pA}\to \LimOneOf{K^{p+1}A}$ monomorphism]
\label{thm:a_pMono->Lim1(K^pA)->Lim1(K^(p+1)A)Monomorphism}
If in a $\ZCat$-diagram $A$ in $\ModulesOver{R}$ the structure map $a_p\from A_p\to A_{p+1}$ is a monomorphism, then the map in the caption above is a monomorphism.
\end{corollary}
\begin{proof}
This follows via (iii) $\implies$ (i) in (\ref{thm:Conditions-Lim1(K^pA)->Lim1(K^(p+1)A)Monomorphism}).
\end{proof}

\begin{definition}[$\lambda$-Mittag-Leffler condition]
\label{def:Lambda-Mittag-LefflerCondition}
Given an ordinal $\lambda$, a $\ZCat$-diagram $A$ in $\ModulesOver{R}$ satisfies the $\lambda$-Mittag-Leffler condition if $I^{\lambda+1}A=I^{\lambda}A$. %
\index{Mittag-Leffler condition!for ordinal $\lambda$}
\end{definition}

In other words, $A$ satisfies the $\lambda$-Mittag-Leffler condition if $\ZImgDiagItrtd{A}{\lambda} = \ZStableImg{A}$, the stable image diagram of $A$.

\begin{corollary}[Point-wise condition for $\LimOneOf{K^pA}\to \LimOneOf{K^{p+1}A}$ monomorphism]
\label{thm:PointWiseConditionFor-Lim1(K^pA)->Lim1(K^(p+1)A)Monomorphism}%
If for a $\ZCat$-diagram $A$ the universal map $\LimOf{A}\to \ZImgDiagItrtdAt{A}{\omega}{p}$ is surjective, then $\LimOneOf{\ZPointKer{A}{p}}\to \LimOneOf{\ZPointKer{A}{p+1}}$ is a monomorphism.%
\index{Mittag-Leffler condition!condition for ordinal $\omega$}
\end{corollary}
\begin{proof}
Under the condition stated, we see with (\ref{thm:StableImageSubDiagram-Properties}.\ref{thm:StableImageSubDiagram-Properties-ImLim}),
\begin{equation*}
\LimOf{\ZPointImg{A}{p}} = \ZImgDiagItrtdAt{A}{\omega}{p} = \Img{\LimOf{A}\to \LimOf{\ZPointImg{A}{p}} } = \ZStableImgAt{A}{p}.
\end{equation*}
Thus condition (iv) of (\ref{thm:Conditions-Lim1(K^pA)->Lim1(K^(p+1)A)Monomorphism}) is satisfied, and the claim follows.
\end{proof}

\begin{corollary}[$\omega$-Mittag-Leffler: $\LimOneOf{ \ZPointKer{A}{p} }\to \LimOneOf{ \ZPointKer{A}{p+1} }$ monomorphism, all $p$]
\label{thm:Omega-Mittag-Leffler->Lim1(K^pA)->Lim1(K^(p+1)A)Monomorphism-All-p}%
If a $\ZCat$-diagram $A$ in $\ModulesOver{R}$ satisfies  the $\omega$-Mittag-Leffler condition, then $\LimOneOf{\ZPointKer{A}{p} }\to \LimOneOf{\ZPointKer{A}{p+1} }$ is a monomorphism for all $p\in\ZNr$.%
\index{Mittag-Leffler condition!condition for ordinal $\omega$} \NoProof
\end{corollary}

\section[Comparison in $\ZCat$-Diagrams]{Comparison in $\ZCat$-Diagrams}
\label{sec:ZDiagramComparison}%

This section provides background for the comparison results using spectral sequences in Section \ref{sec:SpecSeq-Comparison-I}. Given a $\ZCat$-diagram $A$, recall from (\ref{sec:Filtrations}) the image filtration $F_{\bullet}$ of $A_{\infty}\DefEq \CoLimOf{A}$ and the kernel filtration $F^{\bullet}$ of $A_{-\infty}\DefEq \LimOf{A}$. Define adjacent filtration quotients $\varepsilon_{p}(A)$ and $\varepsilon^{p}(A)$ via the short exact sequences: %
\index[not]{$\varepsilon_{p}(A)$ - quotient of adjacent im-filtration terms}\index[not]{$\varepsilon^{p}(A)$ - quotient of adjacent ker-filtration terms}%
\begin{equation*}
\xymatrix@R=1ex@C=4em{
F_{p-1}(A) \ar@{{ |>}->}[r] &
	F_p(A) \ar@{-{ >>}}[r] &
	\varepsilon_{p}(A) \\
F^{p}(A) \ar@{{ |>}->}[r] &
	F^{p+1}(A) \ar@{-{ >>}}[r] &
	\varepsilon^{p}(A)
}
\end{equation*}
Abstracting from work with spectral sequences, we require conditions under which a morphism $f\from A\to B$ of $\ZCat$-diagrams induces a monomorphism/epimorphism/isomorphism $A_{\infty}\to B_{\infty}$, respectively $A_{-\infty}\to B_{-\infty}$. We begin with the following technical lemma.

\begin{lemma}[Adjacent filtration quotient comparison: extended]
\label{thm:AdjacentFiltrationQuotientComp-Extended}
Suppose a morphism $f\from A\to B$ of $\ZCat$\MSComp-diagrams of $R$-modules induces a monomorphism/epimorphism/isomorphism
\begin{equation*}
\varepsilon_{p}(A) \longrightarrow \varepsilon_{p}(B)\qquad \text{for every}\quad p\in \ZNr\ .
\end{equation*}
Then $f$ induces a monomorphism/epimorphism/isomorphism
\begin{equation*}
\dfrac{F_p(A)}{F_{p-r}(A)} \longrightarrow \dfrac{F_p(B)}{F_{p-r}(B)}\qquad \text{for every}\quad p\in \ZNr,\ \text{and every}\ r\geq 1.
\end{equation*}
\end{lemma}
\begin{proof}
This follows by induction on $r$ via the snake lemma.
\end{proof}

\begin{proposition}[Monomorphism of $A_{\infty}\to B_{\infty}$ from monomorphism of $\varepsilon_{\bullet}(A)\to \varepsilon_{\bullet}(B)$]
\label{thm:ZMonoCoLim(A)FromMonoEps_.}
Let $f\from A\to B$ be a morphism of $\ZCat$-diagrams. If $f$ induces monomorphisms
\begin{equation*}
\varepsilon_p(A)\to \varepsilon_p(B),\quad \text{for every}\quad p\in\ZNr  \quad \text{and for}\quad \LimOf{F_{\bullet}(A)}\to \LimOf{F_{\bullet}(B)},
\end{equation*}
then $f$ induces a monomorphism $f_{\infty}\from A_{\infty}\to B_{\infty}$.
\end{proposition}
\begin{proof}
{\em Step 1}\quad By (\ref{thm:AdjacentFiltrationQuotientComp-Extended}), for every $p\in\ZNr$ and  $r\geq 0$, $f$ induces a monomorphism
\begin{equation*}
\dfrac{F_p(A)}{F_{p-r}(A) } \longrightarrow \dfrac{F_p(B)}{F_{p-r}(B) }\ .
\end{equation*}
{\em Step 2}\quad $f$ induces monomorphisms $F_p(A)\to F_p(B)$, $p\in \ZNr$. Indeed, for $p\in \ZNr$ fixed, a morphism between the short exact sequences of $\ZCat$-diagrams
\begin{equation*}
\xymatrix@R=5ex@C=4em{
F_{p-r}(-) \ar@{{ |>}->}[r] &
	F_{p}(-) \ar@{-{ >>}}[r] &
	\dfrac{F_p(-)}{F_{p-r}(-)}
}
\end{equation*}
yields the commutative diagram of exact sequences
\begin{equation*}
\xymatrix@R=5ex@C=4em{
\LimOf{F_{p-r}(A)} \ar@{{ |>}->}[r] \ar@{{ |>}->}[d] &
	F_p(A) \ar[r] \ar[d] &
	\LimOf{\tfrac{F_p(A)}{F_{p-r}(A)}} \ar@{{ |>}->}[d] \\
\LimOf{F_{p-r}(B)} \ar@{{ |>}->}[r] &
	F_p(B) \ar[r] &
	\LimOf{\tfrac{F_p(B)}{F_{p-r}(B)}}
}
\end{equation*}
The vertical arrow on the left is a monomorphism by hypothesis. The vertical arrow on the right is a monomorphism as $\Lim$ is left exact. So the vertical
arrow in the middle is a monomorphism as well.

{\em Step 3}\quad From (\ref{thm:ImageFiltrationCoLim(A)-Exhaustive}) we know that the image filtrations of $A_{\infty}$ and $B_{\infty}$ are exhaustive. As $\CoLimOver{\ZCat}$ is left exact the monomorphisms $F_p(A)\to F_p(B)$ from step 2 induce a monomorphism
\begin{equation*}
f_{\infty}\from A_{\infty}\longrightarrow B_{\infty}
\end{equation*}
This was to be shown.
\end{proof}

When trying to formulate conditions under which $f\from A\to B$ induces an epimorphism $f_{\infty}\from A_{\infty}\to B_{\infty}$, we face the problem that $\LimOver{\ZCat}$ fails to be right exact. So, we look for conditions under which it is.

\begin{lemma}[Epimorphism / isomorphism of $A_{\infty}\to B_{\infty}$, I]
\label{thm:ZEpiCoLim(A)-I}
Let $f\from A\to B$ be a morphism of $\ZCat$-diagrams which induces isomorphisms $\varepsilon_{p}(A)\to \varepsilon_{p}(B)$ for all $p\in \ZNr$, and satisfies the following conditions:
\vspace{-1.6ex}
\begin{enumerate}[(i)]
\item $f$ induces an epimorphism $f_{\infty}|\from \LimOf{F_{\bullet}(A)}\to \LimOf{F_{\bullet}(B)}$, and
\item $f$ induces a monomorphism $\LimOneOf{F_{\bullet}(A)}\to \LimOneOf{F_{\bullet}(B)}$,
\vspace{-1.6ex}
\end{enumerate}
Then $f$ induces an epimorphism $f_{\infty}\from A_{\infty}\to B_{\infty}$. If $f_{\infty}|$ in (i) is an isomorphism, then so is $f_{\infty}$.
\end{lemma}
\begin{proof}
{\em Step 1}\quad With (\ref{thm:AdjacentFiltrationQuotientComp-Extended}) we see that, for $p\in \ZNr$ fixed and $r\geq 1$, $f$ induces isomorphisms:
\begin{equation*}
\dfrac{F_p(A)}{F_{p-r}(A) } \XRA{\cong} \dfrac{F_p(B)}{F_{p-r}(B) }\ .
\end{equation*}
{\em Step 2}\quad $f$ induces epimorphisms $F_p(A)\to F_p(B)$ for all $p\in \ZNr$: A morphism between the short exact sequences of $\ZCat$-diagrams
\begin{equation*}
\xymatrix@R=5ex@C=4em{
F_{p-r}(-) \ar@{{ |>}->}[r] &
	F_{p}(-) \ar@{-{ >>}}[r] &
	\dfrac{F_p(-)}{F_{p-r}(-)}
}
\end{equation*}
yields the commutative diagram of exact sequences whose vertical maps have properties as indicated.
\begin{equation*}
\xymatrix@R=5ex@C=4em{
\LimOf{F_{p-r}(A)} \ar@{{ |>}->}[r] \ar@{-{ >>}}[d]_{f_{\infty}|} &
	F_p(A) \ar[r] \ar[d] &
	\LimOf{\tfrac{F_p(A)}{F_{p-r}(A)}} \ar[d]^{\cong} \ar[r] &
	\LimOneOf{F_p(A)}\ar@{{ |>}->}[d] \\
\LimOf{F_{p-r}(B)} \ar@{{ |>}->}[r] &
	F_p(B) \ar[r] &
	\LimOf{\tfrac{F_p(B)}{F_{p-r}(B)}} \ar[r] &
	\LimOneOf{F_p(B)}
}
\end{equation*}
A diagram chase shows that the map $F_p(A)\to F_p(B)$ is an epimorphism, and is an isomorphism if $f_{\infty}|$ is an isomorphism.

{\em Step 3}\quad The functor $\CoLimOver{\ZCat}$ is exact, and to $f_{\infty}\from A_{\infty}\to B_{\infty}$ is an epimorphism/isomorphism under the conditions stated.
\end{proof}

\begin{corollary}[Epimorphism of $A_{\infty}\to B_{\infty}$, II]
\label{thm:ZEpiCoLim(A)-II}%
Let $f\from A\to B$ be a morphism of $\ZCat$-diagrams which induces isomorphisms $\varepsilon_{p}(A)\to \varepsilon_{p}(B)$ for all $p\in \ZNr$. If $A$ is originally stable and $f$ induces an epimorphism $f_{\infty}|\from \LimOf{F_{\bullet}(A)}\to \LimOf{F_{\bullet}(B)}$, then $f$ induces an epimorphism $f_{\infty}\from A_{\infty}\to B_{\infty}$. If $f_{\infty}|$ is an isomorphism, then so is $f_{\infty}$.
\end{corollary}
\begin{proof}
If $A$ is originally stable, then so is $F_{\bullet}(A)$, implying that $\LimOneOf{F_{\bullet}(A)}=0$. So the hypotheses of (\ref{thm:ZEpiCoLim(A)-I}) are satisfied, and the claim follows.
\end{proof}

Let's now turn to comparison results between limits of $\ZCat$-diagrams. Such a limit is filtered by the kernels of its cone map. Given a morphism of $\ZCat$-diagrams, we want to achieve comparison results about the induced map of limits via information about the induced maps of adjacent filtration quotients.

\begin{lemma}[Monomorphism/isomorphism $\CoLimOf{F^pA}\to \CoLimOf{F^pB}$]
\label{thm:Mono/Iso-CoLim(F^p)}%
If a morphism $f\from A\to B$ of $\ZCat$-diagrams induces monomorphisms/isomorphisms $\varepsilon^p(A)\to \varepsilon^p(B)$ for all $p\in \ZNr$. Then $f$ induces a monomorphism/isomorphism $\CoLimOf{F^{\bullet}(A)}\to \CoLimOf{F^{\bullet}(B)}$.
\end{lemma}
\begin{proof}
{\em Step 1}\quad Adapting the argument of (\ref{thm:AdjacentFiltrationQuotientComp-Extended}) to the present situation, we see that $f$ induces monomorphisms/isomorphisms for every $p\in \ZNr$ fixed, and $r\geq 1$:
\begin{equation*}
\dfrac{F^p(A)}{F^{p-r}(A) } \XRA{\text{mono/iso}} \dfrac{F^p(B)}{F^{p-r}(B) }\ .
\end{equation*}
{\em Step 2}\quad For each $p\in \ZNr$ and $r\geq 1$, we therefore have a morphism between short exact sequences  of $\ZCat$-diagrams
\begin{equation*}
\xymatrix@R=5ex@C=4em{
F^{p-r}(-) \ar@{{ |>}->}[r] &
	F^{p}(-) \ar@{-{ >>}}[r] &
	\dfrac{F^p(-)}{F^{p-r}(-)}
}
\end{equation*}
From this morphism, we obtain a morphism of $\Lim$-$\LimOne$ exact sequences:
\begin{equation*}
\xymatrix@R=5ex@C=4em{
0=\LimOf{F^{p-r}(A)} \ar[r] \ar[d] &
	F^p(A) \ar[r]^-{\cong} \ar[d] &
	\LimOf{\tfrac{F^p(A)}{F^{p-r}(A)} } \ar@{{ |>}->}[d]^{\xi} \ar@{-{ >>}}[r] &
	\LimOneOf{F^p(A)} = 0 \ar[d] \\
0=\LimOf{F^{p-r}(B)} \ar[r] &
	F^p(B) \ar[r]_-{\cong} &
	\LimOf{\tfrac{F^p(B)}{F^{p-r}(B)}} \ar@{-{ >>}}[r] &
	\LimOneOf{F^p(B)}=0
}
\end{equation*}
The objects on the left and right vanish by (\ref{thm:KernelCoFiltration-Complete}). The map  $\xi$ is a monomorphism because $\Lim$ is left exact. Thus $f$ induces a monomorphism $F^p(A)\to F^p(B)$ by commutativity. As $\CoLimOver{\ZCat}$ is left exact (\ref{thm:CoLim^ZExactFunctor}), $f_{-\infty}$ restricts to a monomorphism $\CoLimOf{F^{\bullet}(A)}\to \CoLimOf{F^{\bullet}(B)}$.

If the maps $\varepsilon^{p}(A)\to \varepsilon^{p}(B)$ are isomorphisms, then the above argument yields isomorphisms $F^{p}(A)\to F^{p}(B)$ and, hence, an isomorphism $\CoLimOf{F^{\bullet}(A)}\to \CoLimOf{F^{\bullet}(B)}$.
\end{proof}

\begin{corollary}[Monomorphism $A_{-\infty}\to B_{-\infty}$]
\label{thm:ZMonoLim(A)}
Let $f\from A\to B$ be a morphism of $\ZCat$-diagrams which induces monomorphisms $\varepsilon^p(A)\to \varepsilon^p(B)$ for all $p\in \ZNr$. Then $f$ induces a monomorphism $A_{\infty}\to B_{-\infty}$ whenever at least one of the following conditions is satisfied.
\vspace{-2ex}
\begin{enumerate}[(i)]
\item $f$ induces a monomorphism
\begin{equation*}
f_{\infty}|\from \Img{R_A\from A_{-\infty}\to A_{\infty} } \longrightarrow \Img{R_B\from B_{-\infty}\to B_{\infty} }
\end{equation*}
\item There exists $p\in\ZNr$ such that $f$ induces a monomorphism $F_p(A)\to F_p(B)$.
\item $\LimOf{F_{\bullet}(A)}=0$.
\item $A_{\infty}=0$
\item $A$ is eventually vanishing.
\end{enumerate}
\end{corollary}
\begin{proof}
(i)\quad  From (\ref{thm:KernelFiltrationLim(A)-Props}) we obtain this morphism of short exact sequences:
\begin{equation*}
\xymatrix@R=5ex@C=4em{
\CoLimOf{F^p(A)} \ar@{{ |>}->}[r] \ar@{{ |>}->}[d] &
	A_{-\infty} \ar[d] \ar@{-{ >>}}[r] &
	\Img{R_A\from A_{-\infty}\to A_{\infty} } \ar@{{ |>}->}[d]^{f_{\infty}|} \\
\CoLimOf{F^p(B)} \ar@{{ |>}->}[r] &
	B_{-\infty} \ar@{-{ >>}}[r] &
	\Img{R_B\from B_{-\infty}\to B_{\infty} }
}
\end{equation*}
The vertical map on the left is a monomorphism by (\ref{thm:Mono/Iso-CoLim(F^p)}). If $f_{\infty}|$ is a monomorphism as well, then the vertical map in the center is a monomorphism by the snake lemma.

If condition (ii) is satisfied, then (\ref{thm:KernelFiltrationLim(A)-Props}) lets us see that $f_{\infty}|$ is a monomorphism, and so the claim follows with (i). If (iii) is satisfied, then $\Img{R_A}=0$, and the claim follows from (ii). Finally, we have (v) $\implies$ (iv) $\implies$ (iii), and this proves the corollary.
\end{proof}

\begin{corollary}[Isomorphism $A_{-\infty}\to B_{-\infty}$: I]
\label{thm:ZIsoLim(A)-I}
Let $f\from A\to B$ be a morphism of $\ZCat$-diagrams which induces isomorphisms $\varepsilon^p(A)\to \varepsilon^p(B)$ for all $p\in \ZNr$. Then $f$ induces a isomorphism
\begin{equation*}
f_{-\infty}\from A_{-\infty}\longrightarrow B_{-\infty}
\end{equation*}
if and only if $f$ induces a isomorphism $f_{\infty}|\from \Img{R_A\from A_{-\infty}\to A_{\infty} } \to \Img{R_B\from B_{-\infty}\to B_{\infty} }$.
\end{corollary}
\begin{proof}
From (\ref{thm:Mono/Iso-CoLim(F^p)}), we see that $f$ induces an isomorphism $\CoLimOf{F^{\bullet}(A)}\to \CoLimOf{F^{\bullet}(B)}$. Using this information into the diagram of the proof of (\ref{thm:ZMonoLim(A)}.i) shows that $f_{-\infty}\from A_{-\infty}\to B_{-\infty}$ is an isomorphism if and only if $f_{\infty}|$ is one such. - This was to be shown.
\end{proof}

\begin{corollary}[Isomorphism $A_{-\infty}\to B_{-\infty}$, II]
\label{thm:ZIsoLim(A)-II}
Let $f\from A\to B$ be a morphism of $\ZCat$-diagrams which induces isomorphisms $\varepsilon^p(A)\to \varepsilon^p(B)$ for all $p\in \ZNr$. Then $f$ induces an isomorphism $f_{-\infty}\from A_{-\infty}\to B_{-\infty}$ whenever at least one of the following conditions is satisfied.
\begin{enumerate}[(i)]
\item The maps $R_A\from A_{-\infty}\to A_{\infty}$ and $R_B\from B_{-\infty}\to B_{\infty}$are $0$-maps.
\item $\LimOf{F_{\bullet}(A)}=0=\LimOf{F_{\bullet}(B)}$.
\item $A_{\infty}=0=B_{\infty}$
\item $A$ and $B$ are eventually vanishing.
\end{enumerate}
\end{corollary}
\begin{proof}
If $R_A$ and $R_B$ are $0$-maps then (\ref{thm:KernelFiltrationLim(A)-Props}) yields isomorphisms
\begin{equation*}
\CoLimOf{F^{\bullet}(A)}\cong A_{-\infty}  \qquad \text{and}\qquad  \CoLimOf{F^{\bullet}(B)}\cong B_{-\infty}
\end{equation*}
Thus $f_{-\infty}$ is an isomorphism by (\ref{thm:ZIsoLim(A)-I}). Observing that (iv) $\implies$ (iii) $\implies$ (ii) $\implies$ (i) completes the proof.
\end{proof}

Turning to conditions under which $f\from A\to B$ induces an epimorphism $f_{\infty}\from A_{-\infty}\to B_{-\infty}$:

\begin{lemma}[Epimorphism  $\CoLimOf{F^p(A)}\to \CoLimOf{F^p(B)}$]
\label{thm:ZEpiLim(A)}
Let $f\from A\to B$ be a morphism of $\ZCat$-diagrams which induces epimorphisms $\varepsilon^p(A)\to \varepsilon^p(B)$ for all $p\in \ZNr$. For $p\in \ZNr$ and $r\geq 1$, let $\Lambda^{p}_{p-r}\DefEq \Ker{F^p(A)/F^{p-r}(A)\to F^{p}(B)/F^{p-r}(B)}$ if $\LimOneOfOver{\Lambda^{p}_{p-r}}{r}=0$ for all $p$, then $f$ induces an epimorphism $\CoLimOf{F^p(A)}\to \CoLimOf{F^p(B)}$.
\end{lemma}
\begin{proof}
Adapting the argument of (\ref{thm:AdjacentFiltrationQuotientComp-Extended}) to the present situation, we see that $f$ induces epimorphisms for every $p\in \ZNr$ fixed, and $r\geq 1$:
\begin{equation*}
\dfrac{F^p(A)}{F^{p-r}(A) } \XRA{\text{epi}} \dfrac{F^p(B)}{F^{p-r}(B) }\ .
\end{equation*}
Thus, for $p\in \ZNr$ fixed, we have this short exact sequence of $\ZCat$-diagrams:
\begin{equation*}
\xymatrix@R=5ex@C=4em{
\Lambda^{p}_{p-r} \ar@{{ |>}->}[r] &
	\dfrac{F^{p}(A)}{F^{p-r}(A)} \ar@{-{ >>}}[r] &
	\dfrac{F^{p}(A)}{F^{p-r}(B)}
}
\end{equation*}
Under the stated condition, we obtain the short exact sequence
\begin{equation*}
\xymatrix@R=5ex@C=4em{
\LimOfOver{\Lambda^{p}_{p-r}}{r} \ar@{{ |>}->}[r] &
	F^{p}(A) \ar@{-{ >>}}[r] &
	F^{p}(B)
}
\end{equation*}
Right exactness of the $\CoLimOver{\ZCat}$ yields an epimorphism $\CoLimOf{F^p(A)}\to \CoLimOf{F^p(B)}$.
\end{proof}

\begin{corollary}[Sufficient condition for epimorphism  $\CoLimOf{F^p(A)}\to \CoLimOf{F^p(B)}$]
\label{thm:ZEpiLim(A)-Sufficient}
Let $f\from A\to B$ be a morphism of $\ZCat$-diagrams in $\ModulesOver{R}$ which induces epimorphisms $\varepsilon^p(A)\to \varepsilon^p(B)$ for all $p\in \ZNr$. If the structure maps of $A$ have kernels satisfying the descending chain condition, then $f$ induces an epimorphism $\CoLimOf{F^p(A)}\to \CoLimOf{F^p(B)}$.
\end{corollary}
\begin{proof}
We claim that, for $p\in \ZNr$ fixed and $r\geq 1$, the objects $\Lambda^{p}_{p-r}$ in the proof of (\ref{thm:ZEpiLim(A)}) satisfy the Mittag-Leffler condition. To see that, let $p\in \ZNr$ fixed and $r\geq 1$ consider this morphism of short exact sequences:
\begin{equation*}
\xymatrix@R=5ex@C=4em{
F^{p-r}(A) \ar@{{ |>}->}[r] \ar@{{ |>}->}[d] &
	A_{-\infty} \ar@{-{ >>}}[r] \ar@{=}[d] &
	\ZStableImgAt{A}{p-r} \ar@{{ |>}->}[r] \ar@{-{ >>}}[d] &
	A_{p-r} \ar[d]^{a_{p-1}\Comp \cdots\Comp a_{p-r}}\\
F^{p}(A) \ar@{{ |>}->}[r] &
	A_{-\infty} \ar@{-{ >>}}[r] &
	\ZStableImgAt{A}{p} \ar@{{ |>}->}[r] &
	A_{p}
}
\end{equation*}
Then $F^{p-r}(A)/F^{p}(A)$ is isomorphic $\Ker{\ZStableImgAt{A}{p-r}\to \ZStableImgAt{A}{r}}$, which is a subobject of the kernel $K$ of the vertical map on the right. By hypotheses $K$ is a finitely iterated extension of objects which possess the descending chain condition (DC). So, $K$ has the DC and, hence, $F^{p-r}(A)/F^{p}(A)$ has the DC. So, $\Lambda^{p}_{p-r}$ has the DC; i.e. the $\ZCat$-diagram $(\Lambda^{p}_{p-r}\, |\, r\geq 1)$ has the Mittag-Leffler property. So, the hypotheses of (\ref{thm:ZEpiLim(A)}) are satisfied, and the claim follows.
\end{proof}
%%%%%%%%%%%
\begin{footnotesize}
%\bibliographystyle{alpha}
%\bibliography{X:/RefDB/References}

%%%%%%%%%%%%% Indexes
\cleardoublepage
\phantomsection
\addcontentsline{toc}{part}{List of Notation}
\printindex[not]%   makeindex -o ECsSpecSeqs.nnd ECsSpecSeqs.ndx
\cleardoublepage
\phantomsection
\addcontentsline{toc}{part}{Index}
\printindex
\end{footnotesize}
}% end of \sffamily
\end{document}